\documentclass[aos]{imsart}

\RequirePackage[numbers]{natbib}%% uncomment this for author-\year citations
\usepackage{booktabs}
\usepackage{hyperref}
\usepackage{csquotes} 
\usepackage{anyfontsize}
\usepackage{xcolor}
\definecolor{ps}{RGB}{0,0,200}
\definecolor{yl}{RGB}{200,0,0}

\usepackage[colorinlistoftodos]{todonotes}
\usepackage{amsmath}
\usepackage{amsthm}
\usepackage{amssymb}
\usepackage{amsfonts}
\usepackage{mathrsfs}
\usepackage{sansmath}
\usepackage{bbm}
\usepackage{bm}
\usepackage{nicefrac}
\usepackage{mathtools}
\usepackage{physics}

\usepackage{algorithm}
\usepackage{algpseudocode}

\usepackage{graphicx}
\usepackage{subfigure}

\newcounter{algsubstate}
\renewcommand{\thealgsubstate}{\alph{algsubstate}}

\usepackage{dsfont}
\usepackage{array}
\usepackage{comment}
\usepackage{relsize}
\usepackage{enumitem}
\makeatother
\usepackage{hyperref}
\hypersetup{
	colorlinks=true,
	linkcolor=black,
	filecolor=black,
	citecolor = blue,      
	urlcolor=cyan,
}
\usepackage[nameinlink]{cleveref}

\startlocaldefs

\newcommand{\Rbm}{\mathbf{R}}
\newcommand{\vbf}{\mathbf{v}}
\newcommand{\Vbf}{\mathbf{V}}
\newcommand{\bbeta}{\bm{\beta}}

\newcommand{\E}{\mathbb{E}}
\newcommand{\Xstar}{\mathsf{B}^\star}
\newcommand{\D}{\mathsf{D}}
\newcommand{\Ls}{\mathsf{L}}
\newcommand{\Ps}{\mathsf{P}}
\newcommand{\Es}{\mathsf{E}}

\newcommand{\PP}{\mathbb{P}}
\newcommand{\R}{\mathbb{R}}

\newcommand{\Gs}{\mathsf{G}}

\newcommand{\Hs}{\mathsf{H}}
\newcommand{\Us}{\mathsf{U}}
\newcommand{\Vs}{\mathsf{V}}
\newcommand{\Xs}{\mathsf{X}}
\newcommand{\Ys}{\mathsf{Y}}
\newcommand{\Zs}{\mathsf{Z}}
\newcommand{\Ss}{\mathsf{S}}

\newcommand{\N}{\mathbb{N}}

\newcommand{\Dc}{\mathfrak{D}}

\newcommand{\V}{\mathbb{V}}

\newcommand{\co}{c_0}

\newcommand{\pfrak}{\mathfrak{p}}

\newcommand{\bUm}{\mathbf{U}}

\newcommand{\Jbm}{\mathbf{J}}
\newcommand{\Sbm}{\mathbf{S}}

\newcommand{\y}{\mathbf{y}}

\newcommand{\xonet}{\hat{\mathbf{x}}_{1t}}

\newcommand{\ronetm}{\mathbf{r}_{1, t-1}}
\newcommand{\ronet}{\mathbf{r}_{1t}}
\newcommand{\rtwot}{\mathbf{r}_{2t}}
\newcommand{\xtwot}{\hat{\mathbf{x}}_{2t}}
\newcommand{\xtwotm}{\hat{\mathbf{x}}_{2,t-1}}
\newcommand{\rtwotm}{\mathbf{r}_{2,t-1}}

\newcommand{\ampst}{\mathbf{s}^t}
\newcommand{\ampxt}{\mathbf{x}^t}
\newcommand{\ampxtpo}{\mathbf{x}^{t+1}}
\newcommand{\ampxone}{{\mathbf{x}}^{1}}
\newcommand{\ampyt}{\mathbf{y}^t}
\newcommand{\ampytpo}{\mathbf{y}^{t+1}}
\newcommand{\ampstpo}{\mathbf{s}^{t+1}}
\newcommand{\ampyone}{{\mathbf{y}}^{1}}

\newcommand{\toW}{\stackrel{W_2}{\to}}

\renewcommand{\Tr}{\operatorname{Tr}}

\newcommand{\adj}{\widehat{\mathsf{adj}}}

\newcommand{\tauh}{\hat{\tau}_*}
\newcommand{\bhetah}{\hat{\bm{\beta}}^u}
\newcommand{\bhetahi}{\hat{{\beta}}^u_i}

\newcommand{\taustar}{\tau_{*}}
\newcommand{\hv}{\vec{h}}
\newcommand{\bigsum}{\mathlarger{\mathlarger{\sum}}}
\newcommand{\X}{\mathbf{X}}

\newcommand{\Zetr}{\mathsf{C}^\star}
\newcommand{\alphr}{\alphstar}
\newcommand{\rstar}{\mathbf{r}_{*}}
\newcommand{\rstst}{\mathbf{r}_{**}}
\newcommand{\hatrstst}{\hat{\mathbf{{r}}}_{**}}

\newcommand{\pcrdb}{\bhetah_{\mathsf{pcr}}}
\newcommand{\pcrdbi}{\hat{\beta}^u_{\mathsf{pcr},i}}
\newcommand{\pcrdbI}{\hat{\bm{\beta}}^u_{\mathsf{pcr},\mathcal{I}}}
\newcommand{\pcr}{\hat{\bm{\theta}}_{\mathsf{pcr}}}
\newcommand{\pcri}{\hat{\theta}_{\mathsf{pcr},i}}
\newcommand{\pcrt}{\hat{\bm{\beta}}_{\mathsf{al}}}

\newcommand{\pcrtI}{\hat{\bm{\beta}}_{\mathsf{al},\mathcal{I}}}
\newcommand{\Js}{\mathcal{J}}

\newcommand{\Jsb}{\bar{\mathcal{J}}}
\newcommand{\ynew}{\y_\mathsf{new}}
\newcommand{\Xnew}{\X_\mathsf{new}}
\newcommand{\Qnew}{\Qbm_\mathsf{new}}
\newcommand{\Dnew}{\Dbm_\mathsf{new}}
\newcommand{\epnew}{\epsilon_\mathsf{new}}
\newcommand{\pcrc}{\hat{\bm{\beta}}_{\mathsf{co}}}

\newcommand{\pcrcI}{\hat{\bm{\beta}}_{\mathsf{co},\mathcal{I}}}

\newcommand{\tauhpcr}{\hat{\tau}_{*}}
\newcommand{\alphstar}{\bm{\upsilon}^\star}
\newcommand{\alphstari}{\upsilon^\star_i}

\newcommand{\hatbt}{\hat{\bm{\beta}}}
\newcommand{\hjatbtj}{\hat{{\beta}}_j}
\newcommand{\hiatbti}{\hat{{\beta}}_i}
\newcommand{\zetr}{\bm{\zeta}^\star}
\newcommand{\zetri}{{\zeta}^\star_i}
\newcommand{\zetrj}{{\zeta}^\star_j}
\newcommand{\Dbm}{\mathbf{D}}
\newcommand{\Obm}{\mathbf{O}}
\newcommand{\Qbm}{\mathbf{Q}}
\newcommand{\epbm}{\bm{\varepsilon}}
\newcommand{\epi}{{\varepsilon}_i}
\newcommand{\st}{\bm{\beta}^\star}
\newcommand{\sti}{{\beta}^\star_i}
\newcommand{\tauc}{\hat{\tau}_{*}}

\newcommand{\stal}{\st_{\mathsf{al}}}

\newcommand{\stalI}{\bm{\beta}^\star_{\mathsf{al},\mathcal{I}}}

\newcommand{\ebit}{\mathbf{e}_b}

\newcommand{\magzetr}{\hat{\omega}}

\newcommand{\adg}{\breve{\mathsf{adj}}}

\newtheorem{Theorem}{Theorem}[section]
\newtheorem{Proposition}{Proposition}[section]
\newtheorem{Lemma}[Proposition]{Lemma}
\newtheorem{Corollary}[Proposition]{Corollary}

\theoremstyle{definition}
\newtheorem{Assumption}{Assumption}
\newtheorem*{Assumption*}{Assumption}
\newtheorem{Definition}[Proposition]{Definition}
\newtheorem{Remark}[Proposition]{Remark}
\newtheorem{Conjecture}[Proposition]{Conjecture}
\newtheorem{Claim}[Proposition]{Claim}

\newtheorem{Example}[Proposition]{Example}

\Crefname{Theorem}{Theorem}{Theorem}
\Crefname{Proposition}{Proposition}{Proposition}
\Crefname{Lemma}{Lemma}{Lemma}
\Crefname{Corollary}{Corollary}{Corollary}
\Crefname{Assumption}{Assumption}{Assumption}
\Crefname{Remark}{Remark}{Remark}
\Crefname{Notation}{Notation}{Notation}
\Crefname{Definition}{Definition}{Definition}
\Crefname{Conjecture}{Conjecture}{Conjecture}
\Crefname{Claim}{Claim}{Claim}
\Crefname{Example}{Example}{Example}
\Crefname{Algorithm}{Algorithm}{Algorithm}

\def\O{\mathbb{O}}

\def\1{\mathbf{1}}
\DeclareMathOperator{\Haar}{Haar}
\DeclareMathOperator{\diag}{diag}

\endlocaldefs

\begin{document}

\begin{frontmatter}
\title{Spectrum-Aware Debiasing: A Modern Inference Framework with Applications to Principal Components Regression}
%\title{ Spectrum-Aware Debiasing with Applications to Non-i.i.d.~Heavy-tailed Data and Principal Components Regression}
\runtitle{Spectrum-Aware Debiasing}
%Debiasing regularized linear estimators with spectrum-aware adjustment
\begin{aug}
%%%%%%%%%%%%%%%%%%%%%%%%%%%%%%%%%%%%%%%%%%%%%%%
%% Only one address is permitted per author. %%
%% Only division, organization and e-mail is %%
%% included in the address.                  %%
%% Additional information can be included in %%
%% the Acknowledgments section if necessary. %%
%% ORCID can be inserted by command:         %%
%% \orcid{0000-0000-0000-0000}               %%
%%%%%%%%%%%%%%%%%%%%%%%%%%%%%%%%%%%%%%%%%%%%%%%
\author[A]{\fnms{Yufan}~\snm{Li}\ead[label=e1]{yufan\_li@g.harvard.edu}}
\and
\author[A]{\fnms{Pragya}~\snm{Sur}\ead[label=e2]{pragya@fas.harvard.edu}}
%%%%%%%%%%%%%%%%%%%%%%%%%%%%%%%%%%%%%%%%%%%%%%
%% Addresses                                %%
%%%%%%%%%%%%%%%%%%%%%%%%%%%%%%%%%%%%%%%%%%%%%%
\address[A]{Department of Statistics,
Harvard University \printead[presep={,\ }]{e1,e2}}
\end{aug}

\maketitle
\begin{abstract}
Debiasing is a fundamental concept in high-dimensional statistics.  While degrees-of-freedom adjustment is the state-of-the-art debiasing technique in high-dimensional linear regression, it largely remains limited to independent, identically distributed samples and sub-Gaussian covariates. These limitations hinder its wider practical use. In this paper, we break this barrier and introduce Spectrum-Aware Debiasing----a novel inference method that applies to challenging high-dimensional regression problems with structured row-column dependencies, heavy tails, asymmetric properties, and latent low-rank structures.   Our method achieves debiasing through a rescaled gradient descent step, where the rescaling factor is derived from the spectral properties of the sample covariance matrix.  This spectrum-based approach enables accurate debiasing in much broader contexts. We study the common modern regime where the number of features and samples scale proportionally. We establish asymptotic normality of our proposed estimator (suitably centered and scaled) under various convergence notions when the covariates are right-rotationally invariant. {We further prove a spectral universality result, extending our guarantees to a much broader class of covariate distributions.} Furthermore, we devise a consistent estimator for the asymptotic variance. 

Our work has two notable by-products: first, Spectrum-Aware Debiasing rectifies the bias in principal components regression (PCR), providing the first debiased PCR estimator in high dimensions. Second, we introduce a principled test for checking the presence of alignment between the signal and the eigenvectors of the sample covariance matrix. This test is independently valuable for statistical methods developed using approximate message passing, leave-one-out, random matrix theory, or convex Gaussian min-max theorems. We demonstrate the utility of our method through diverse simulated and real data experiments.
\end{abstract}


\end{frontmatter}

\section{Introduction}
Regularized estimators constitute a basic staple of high-dimensional regression. These estimators incur a regularization bias, and characterizing this bias is imperative for accurate uncertainty quantification. This motivated debiased versions of these estimators \cite{zhang2014confidence,javanmard2018debiasing,van2014asymptotically} that remain unbiased asymptotically around the signal of interest. To describe debiasing, consider the setting of a canonical linear model where one observes a sample of size $n$ satisfying  
$$
\y=\X \st+\epbm.
$$
Here $\y \in \mathbb{R}^n$ denotes the vector of outcomes, $\X \in \mathbb{R}^{n \times p}$ the design matrix, $\st \in \mathbb{R}^p$
the unknown coefficient vector, 
and $\epbm$ the unknown noise vector. Suppose $\boldsymbol{\hatbt}$ denotes the estimator obtained by minimizing 
$\mathcal{L}(\bm{\cdot} \; ;\X,\y):\R^p \mapsto \R_+$ given by 
\begin{equation}\label{deflasso}
\mathcal{L}(\bm{\beta};\X,\y):=\frac{1}{2}\|\y-\X \bm{\beta}\|^2+\sum_{i=1}^p h\left(\beta_i\right), \qquad \bm{\beta} \in \mathbb{R}^p,
\end{equation}
where  $h: \mathbb{R} \mapsto [0,+\infty)$ is some convex penalty function. Commonly used penalties include the ridge $h(b)=\lambda b^2, \lambda>0$, the Lasso $h(b)=\lambda |b|, \lambda>0$, the Elastic Net $h(b)=\lambda_1 |b|+\lambda_2 b^2, \lambda_1,\lambda_2>0$, etc. The debiased version of $\boldsymbol{\hatbt}$ takes the form 
\begin{equation}\label{eq:debiased}
    \bhetah=\hatbt+\frac{1}{\adj} \boldsymbol{M} \X^{\top}(\y-\X \hat{\boldsymbol{\beta}}),
\end{equation}
for suitable choices of $\boldsymbol{M} \in\mathbb{R}^{p\times p} $ and adjustment coefficient $\adj>0$\footnote{We adopt a scaling where $\|\X\|_{\mathrm{op}}$ and $\frac{1}{\sqrt{p}}\|\st\|_2$ remain at a constant order as $n$ and $p$ tend to infinity. Prior literature (e.g. \cite{bellec2019biasing}) often adopts a scaling where $\frac{1}{\sqrt{p}}\|\X\|_{\mathrm{op}}$ and $\|\st\|_2$ maintains constant order as $n$ and $p$ approach infinity. These scalings should be viewed as equivalent up to a change of variable. }. At a high level, one expects the debiasing term $\frac{1}{\adj} \boldsymbol{M} \X^{\top}(\y-\X \hat{\boldsymbol{\beta}})$ will compensate for the regularization bias and lead to asymptotic normality in entries of $\bhetah- \st$, whereby one can develop associated inference procedures. 

Classical statistics textbooks tell us that when the dimension $p$ is fixed and the sample size $n$ approaches infinity, the debiased estimator $\bhetah$ reduces to the well-known one-step estimator. In this case, Gaussianity of $\bhetah- \st$ follows from \cite[Theorem 5.45]{van2000asymptotic}  by  choosing $\mathbf{M}=(\X^{\top} \X)^{-1}$, the inverse of the sample covariance matrix,  and $\adj=1$, requiring no adjustment. Early work on ultra high-dimensional problems ($p \gg n$) \cite{van2014asymptotically,zhang2014confidence,javanmard2018debiasing,CaiGuoMini} established that when the signal $\boldsymbol{\beta}^{\star}$ is sufficiently sparse, 
the Lasso can be debiased by taking $\mathbf{M}$ as suitable ``high-dimensional''  substitutes of $(\X^{\top} \X)^{-1}$ and setting $\adj=1$. However, later work uncovered that an adjustment of $\adj<1$ is necessary to relax sparsity assumptions on $\boldsymbol{\beta}^{\star}$ or to debias general regularized estimators beyond the Lasso. For instance, \cite{javanmard2014hypothesis,bellec2022biasing} established under the proportional regime ($n/p\to w>0$) that when the signal is not sufficiently sparse, the adjustment for the Lasso should be $\adj=1-\hat{s}/n$ with $\mathbf{M}=\bm{\Sigma}^{-1}$, where $\hat{s}$ denotes the number of non-zero entries in $\hat{\bm{\beta}}$ and $\bm{\Sigma}$ is the covariance matrix of i.i.d. Gaussian rows of $\mathbf{X}$. This correction term was named the ``degrees-of-freedom adjustment'' since $\hat{s}$ corresponds to the degrees-of-freedom of the estimator $\hat{\bbeta}$ \cite{DOFOG}. 

\begin{figure}
    \centering
    \includegraphics[width=0.73\linewidth]{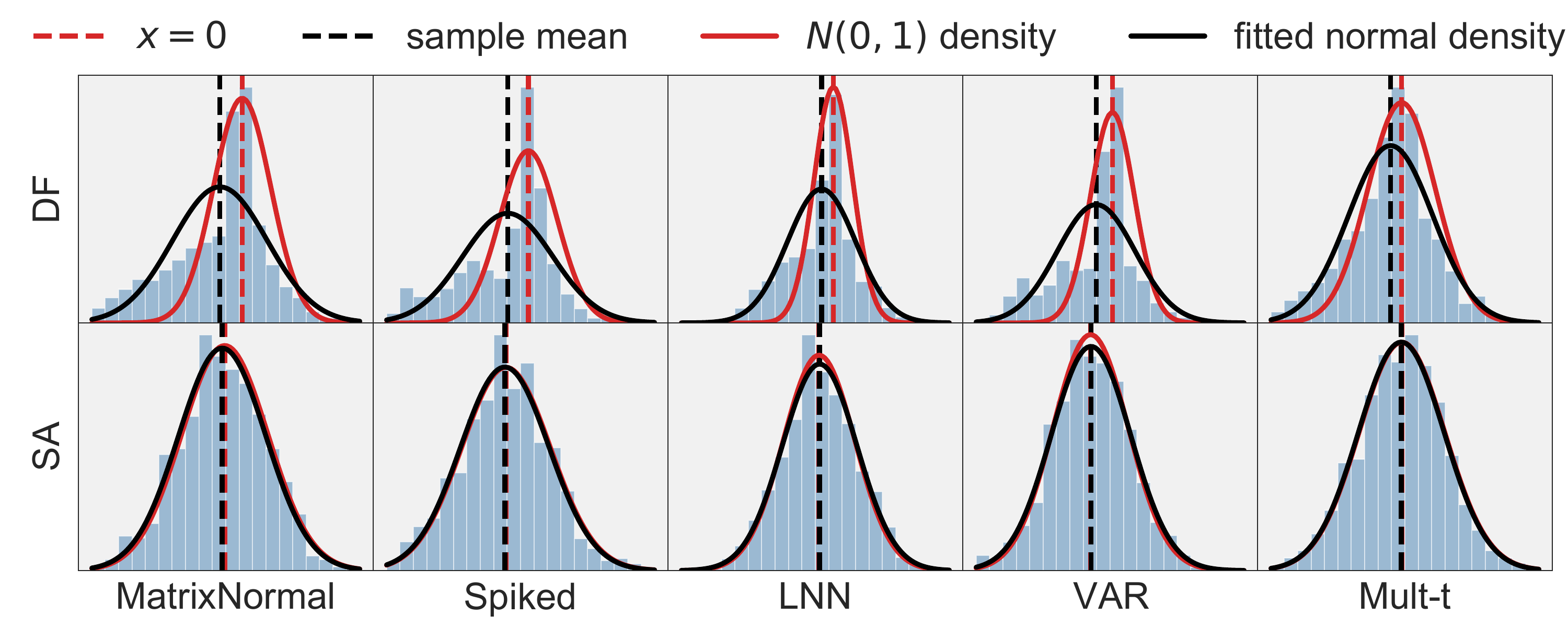}
    \vspace{-1mm}
    \caption{Histograms of empirical distribution of $(\tauh^{-1/2}(\bhetahi-\sti))_{i=1}^p$ comparing Degrees-of-Freedom Debiasing \cite{bellec2019biasing} with our Spectrum-Aware Debiasing, where $\bhetah$ is the debiased Elastic-Net estimator with tuning parameters $\lambda_1=1, \lambda_2=0.1$. The first row uses the Degrees-of-Freedom Debiasing from \cite{bellec2019biasing} with $\mathbf{M}=\mathbf{I}_p$ (denoted DF). The second row uses our Spectrum-Aware Debiasing (denoted SA) as in \Cref{algodebias}. Entries of the signal $\st$ are i.i.d. draws from $ 0.24 \cdot N(-20,1)+0.06\cdot N(10,1)+0.7\cdot \delta_0$ where $\delta_0$ is Dirac-delta function at $0$. Thereafter, the signal is fixed and we generate the responses using $\y=\X \st+\epbm$ where $\varepsilon_i \stackrel{i.i.d.}{\sim} N(0,1)$. The solid black curve indicates a normal density fitted to the blue histograms whereas the dotted black line indicates the empirical mean corresponding to the histogram. See the corresponding QQ plot in \Cref{figQQA} from Appendix. All designs are re-scaled so that the average of eigenvalues of $\X^\top \X$ is 1. The design matrices are of shape $n=500, p=1000$.}
    \label{fig1}
\end{figure}

Degrees-of-Freedom Debiasing introduced a novel perspective. However, it relied on some strict assumptions, namely independent and identically distributed (i.i.d.) data with Gaussian covariates that follow $\mathbf{X}_i \sim N(\bm{0},\boldsymbol{\Sigma})$.
Furthermore,  Degrees-of-Freedom Debiasing  used $\mathbf{M}=\boldsymbol{\Sigma}^{-1}$; thus, even when the i.i.d. assumption holds, implementing this estimator required knowledge of the true covariance matrix $\boldsymbol{\Sigma}$ or an accurate estimate.
In summary, Degrees-of-Freedom Debiasing suffers three key limitations: (i) it is  restricted to Gaussian type distributions (see \cite{han2023universality} for extension to sub-Gaussians), thus failing to capture heavy-tailed or 
asymmetrically distributed covariates; (ii) it is ineffective in scenarios with heterogeneity or dependency among samples, and (iii) it faces challenges in choosing \(\mathbf{M}\) when precise estimates of \(\boldsymbol{\Sigma}\) are unavailable. These limitations restrict the broader applicability of Degrees-of-Freedom Debiasing  to real-world settings  that violate these strict assumptions.

We exemplify this issue in \Cref{fig1}, where we consider the following design distributions: (i) $\mathsf{MatrixNormal}$: $\X$ drawn from a matrix normal distribution with row and column correlations, i.e.~$\X\sim N(\rm{0},\mathbf{\Sigma}^{\mathrm{(col)}}\otimes \mathbf{\Sigma}^{\mathrm{(row)}}),$ where $\mathbf{\Sigma}^{\mathrm{(col)}}_{ij}=0.5^{|i-j|}$ and $\bm{\Sigma}^{\mathrm{(row)}}$ follows an inverse-Wishart distribution with identity  scale and degrees-of-freedom $1.1p$, which is chosen to be close to $p$ to encourage heavy-tails in the covariates; (ii) $\mathsf{Spiked}$: $\X$ contains latent structure, i.e. $\X=\alpha \cdot  \mathbf{V}\mathbf{W}^\top +n^{-1}N(\rm{0}, \mathbf{I}_n \otimes \mathbf{I}_p)$ where $\alpha=10$ and $\mathbf{V} \in \mathbb{R}^{n \times m}, \mathbf{W} \in \mathbb{R}^{p \times m}$ are drawn randomly from Haar matrices of dimensions $n,p$, and then we retain $ m=50$ columns; (iii) $\mathsf{LNN}$: $\X$ formed by product of multiple random matrices (see \cite{hanin2020products} for connections to linear neural networks), i.e. $\X=\X_1\cdot \X_2 \cdot \X_3 \cdot \X_4 $ where $\X_i$'s have i.i.d. entries from $N(\rm{0},1)$; (iv) $\mathsf{VAR}$: rows of $\X$ drawn from a vector time series with the $i$-th row given by $\X_{i,\bullet}=\sum_{k=1}^{\tau\vee i}\alpha_k \X_{i-k,\bullet}+\epbm_i$ where $\tau=3,\alpha=\qty(0.4, 0.08, 0.04)$ and $\epbm_i \sim N(\mathbf{0}, \mathbf{\Sigma})$ with $\mathbf{\Sigma}$ drawn from an inverse-Wishart distribution with the same parameters as in (i); 
(v) $\mathsf{Mult}$-$\mathsf{t}$: rows of $\X$ drawn independently from a multivariate t-distribution with identity scale and degrees-of-freedom $3$. Figure \ref{fig1} plot histograms of the empirical distribution of $\bhetah-\bbeta^\star$ scaled by an estimate $\tauh$ of its standard deviation. The topmost panel uses the Degrees-of-Freedom Debiasing formula for $\bhetah$. We observe that the histograms in this panel deviate substantially from the overlaid standard Gaussian density. Degrees-of-Freedom Debiasing thus fails in these challenging settings. 

To underscore the difficulties posed by these examples, note that cases (i)-(iv) involve non-i.i.d.~designs and (i),(iv),(v) involve heavy-tailed covariates. As discussed later, the failure observed is primarily attributable to these structural deviations rather than finite-sample effects.

In this paper, we propose a new debiasing formula that addresses the shortcomings of previous techniques and enables accurate debiasing in the aforementioned settings. To develop our method, we leverage the insight that a debiasing procedure effective for a wide range of scenarios must thoughtfully utilize the spectral characteristics of the data. To accomplish this, we explore an alternative path for modeling the randomness in the design. Instead of assuming that the rows of the design are i.i.d. Gaussian vectors, we require that the singular value decomposition of $\X$ satisfies certain natural structure that allows dependence among samples and potentially heavy-tailed distributions. Specifically, we assume that $\X$ is right-rotationally invariant (Definition \ref{def:Rotinv}).

Right-rotationally invariant designs have been widely studied in signal processing, information theory, statistical physics, and high-dimensional statistics \cite{takeda2006analysis,ma2017orthogonal,rangan2019vector,takeuchi2019rigorous,dudeja2020analysis,10.1214/21-AOS2101,gerbelot2020asymptotic,takeuchi2020convolutional,takeuchi2021bayes,venkataramanan2022estimation,gerbelot2022asymptotic,liu2022memory,li2023random,xu2023capacity,pmlr-v258-luo25b}, and they serve as useful prototypes for fundamental high-dimensional phenomena in compressed sensing. { Roughly speaking, if the right singular vectors $\Obm$ of a design matrix $\X$ are Haar-distributed, then $\X$ lies in the class of right-rotationally invariant designs, regardless of distribution of its eigenvalues. This generality lets us handle design distributions not covered by degree-of-freedom–based debiasing methods, including designs (i)–(v) in \Cref{fig1}. Since the right-rotational invariance assumption preserves the spectral information of $\X^{\top}\X$, we also expect methods developed under this assumption to exhibit improved robustness when applied to real-data designs. This is illustrated in \Cref{figPCRA}, where we evaluate our PCR–Spectrum–Aware Debiasing method based on right-rotational invariance assumption on six real datasets spanning image data, financial data, socio-economic data and so forth. Furthermore, recent advances indicate that a wide variety of covariate distributions fall within the same universality class as right-rotationally invariant designs, provided the eigenvectors of the sample covariance are sufficiently “generic,” even if not exactly Haar \cite{donoho2009observed,dudeja2022spectral,wang2022universality}. { In Appendix, Section \ref{bigsecuni}, we extend our results to this broader “spectral universality class” \cite{dudeja2022spectral,wang2022universality}, which encompasses right-rotationally invariant designs, i.i.d. designs and their linear transforms as well as other challenging design distributions.} A more detailed discussion of the technical challenges associated with right-rotational invariance, along with a review of relevant prior work, is provided in Appendix, \Cref{RIDint}.  }

We discover that for right-rotationally invariant designs,  the accurate debiasing formula is given by
\begin{equation}\label{SAestimator}
    \bhetah=\hatbt+\adj^{-1} \X^{\top}(\y-\X \hat{\boldsymbol{\beta}}),
\end{equation}
where $\adj$ solves the equation 
\begin{equation}\label{eq:adjspectrum}
\frac{1}{p} \mathlarger{\mathlarger{\sum}}_{i=1}^p \frac{1}{\left(d_i^2-\adj \right)\left(\frac{1}{p} \sum_{j=1}^p \qty(\adj+h^{\prime \prime}\left(\hjatbtj\right))^{-1} \right)+1}=1.
\end{equation}
Here, $\{d_i^2\}_{1 \leq i \leq p}$ represents the eigenvalues of the sample covariance matrix $\X^{\top}\X$, and $h^{\prime \prime}$ denotes the second derivative of the penalty function $h$ used in calculating the regularized estimator $\bhetah$. At points of non-differentiability {(e.g. $x=0$ for Lasso and Elastic Net)}, we extend $h''$ by $+\infty$ (cf. \Cref{asspenalty}). The solution $\adj$ of \eqref{eq:adjspectrum} is unique for any $p\ge 1$ under mild assumptions (cf. \Cref{uniquewell}).  We refer to $\adj$ as the  ``Spectrum-Aware adjustment'' and the debiasing approach in \eqref{SAestimator} as ``Spectrum-Aware Debiasing'' since $\adj$ depends on the eigenvalues of $\X^\top \X$. Figure \ref{fig1} illustrates the efficacy of our method. The second panel shows the empirical distribution of  $\bhetah-\bbeta^\star$,  scaled by an appropriate estimate of its standard deviation, when $\bhetah$ is given by our Spectrum-Aware formula \eqref{SAestimator}. Note the remarkable agreement with the overlaid standard Gaussian density. { We emphasize that this debiasing formula relies on a fundamentally different structured dependency assumption than Degrees-of-Freedom Debiasing. The type of dependency it captures is  incomparable to that of anisotropic Gaussians.\footnote{for the analogue of anisotropic Gaussian-type dependence in the context of right rotationally invariant designs, see Appendix \ref{section:CONJE}.} Nevertheless, our method can capture quite diverse dependency structures, as demonstrated in Figure \ref{fig1}. Crucially, it operates without requiring an estimate of the population feature covariance matrix.}

%that, see Appendix G  As a result, it can more robustly handle structured dependencies—when they are well modeled by the right-rotationally invariant assumption—without requiring an estimate of $\boldsymbol{\Sigma}$, as demonstrated in Figure \ref{fig1}.}
% We should mention that, unlike Degrees-of-Freedom Debiasing, this debiasing formula can handle data with structured dependencies without requiring an estimate of $\boldsymbol{\Sigma}$, as demonstrated in Figure \ref{fig1}. Nevertheless, we should be clear that 

Despite the strengths of Spectrum-Aware Debiasing, we observe that it falls short when $\X$ contains outlier eigenvalues and/or the signal aligns with some eigenvectors of $\X$. To address these issues, we introduce an enhanced procedure that integrates classical Principal Components Regression (PCR) ideas with Spectrum-Aware Debiasing. In this approach, we employ PCR to handle the outlier eigenvalues while using a combination of PCR and Spectrum-Aware Debiasing to estimate the parts of the signal that do not align with an eigenvector. We observe that this hybrid PCR-Spectrum-Aware approach works exceptionally well in challenging settings where  these issues are present.  

We next summarize our main contributions below.
 
\begin{itemize}
\item[(i)] We establish that our proposed debiasing formula is well-defined, that is, \eqref{eq:adjspectrum} admits a unique solution (Proposition \ref{COR}). Then we establish that $\bhetah - \st$, with this choice of $\adj$, converges to a mean-zero Gaussian with some variance $\tau_{*}$ in a Wasserstein-2 sense (Theorem \ref{NEIGMAIN}; Wasserstein-2 convergence notion introduced in Definition \ref{def:Wass}). Under an exchangeability assumption on $\bbeta^\star$, we strengthen this result to convergence guarantees on finite-dimensional marginals of $\bhetah-\bbeta^{\star}$ (Corollary \ref{sgoods}).
\item[(ii)] We develop a consistent estimator for $\tau_*$ (Theorem \ref{NEIGMAIN}) by developing new algorithmic insights and new proof techniques that can be of independent interest in the context of vector approximate message passing algorithms \cite{rangan2019vector,schniter2016vector,fletcher2018plug} (details in Section \ref{neig}).
 \item[(iii)] To establish the aforementioned points, we imposed two strong assumptions: (a) the signal $\st$ is independent of $\X$ and cannot align with any subspace spanned by a small number of eigenvectors of $\X^{\top}\X$; (b) $\X^{\top}\X$ does not contain outlier eigenvalues. To mitigate these, we develop a PCR-Spectrum-Aware Debiasing approach (Section \ref{section:pcar}) that applies when these assumptions are violated. We prove asymptotic normality for this approach in Theorem \ref{PCRTHM}.
 \item[(iv)] We demonstrate the utility of our debiasing formula in the context of hypothesis testing and confidence interval construction with explicit guarantees on quantities such as the false positive rate, false coverage proportion, etc. (Sections \ref{sec:infvanilla} and  \ref{section:inference}). 
\item[(v)] As a by-product, our PCR-Spectrum-Aware approach introduces the first methodology for debiasing the classical PCR estimator (Theorem \ref{PCRTHM}), which would otherwise exhibit a shrinkage bias due to omission of low-variance principal components. We view this as a contribution in and of itself to the PCR literature since inference followed by PCR is under-explored despite the widespread usage of PCR. 

\item[(vi)] As a further byproduct, we introduce a hypothesis test to identify alignment between principal components of the design matrix and the unknown regression coefficient $\st$. This may be of independent interest in the context of statistical methods developed based on approximate message passing/leave-one-out/convex Gaussian min-max theorems.

\item[(vii)] { On the technical front, we rigorously characterize the risk of regularized estimators under right-rotationally invariant designs (cf. \Cref{thm:empmain}), and extend these results to a broader spectral universality class (cf. \Cref{thm:empmainuniv}). We prove existence and uniqueness of the solution associated with our fixed-point equations under appropriate conditions (cf. \Cref{fixexist}). We establish the Cauchy convergence of VAMP iterates (cf. \Cref{Corc}). We further extend our results to the challenging case of the Lasso under suitable sparsity conditions (cf. \Cref{Smdfkf})---this requires substantial arguments beyond those for strongly convex penalties (cf. \Cref{sec:Lasso}).} { We note that analogs of the leave-one-out approach \cite{mezard1987spin,talagrand2003spin,bean2013optimal,el2018impact,sur2019likelihood,sur2019modern,chen2021spectral,jiang2022new} and Stein's method \cite{stein1981estimation,chatterjee2010spin,bellec2019biasing,anastasiou2023stein}, both of which form fundamental proof techniques for Gaussian designs, are nonexistent or under-developed for rotationally invariant designs. Therefore, our approach adopts an algorithmic proof strategy inspired by prior work from the senior authors and others in the Gaussian case.}
%These results significantly advance our technical understanding of regularized regression with right-rotationally invariant designs.}
%, yet have not previously been established with full mathematical rigor in the literature \cite{rangan2019vector,gerbelot2020asymptotic,gerbelot2022asymptotic,takahashi2018statistical}.}

% {with extension to a broader universality class.} One should compare these results to \cite{bayati2011lasso} that developed a risk characterization of the Lasso in high dimensions under Gaussian design matrices. 
\item[(viii)] Finally, we demonstrate the applicability of our Spectrum-Aware approach across a wide variety of covariate distributions, ranging from settings with heightened levels of correlation or heterogeneity among the rows or a combination thereof (Figure \ref{figPCRA}, top-left experiment), to diverse real data designs (Figure \ref{figPCRA}, bottom-left experiment).
We observe that PCR-Spectrum-Aware Debiasing demonstrates superior performance across the board. 
\end{itemize}

In the remaining Introduction, we walk the readers through some important discussion points, before we delve into our main results. In Section \ref{section:intuitridge}, we provide some intuition for our Spectrum-Aware construction using the example of the ridge estimator, since it admits a closed form and is simple to study. In \Cref{twoissues}, we describe how the debiasing methods tend to fail when the design $\X$ contains outlier eigenvalues and/or the signal aligns with some eigenvectors of $\X$. In Section \ref{sec:debiasedPCR}, we discuss a novel PCR-Spectrum-Aware Debiasing approach which addresses the aforementioned two issues and an associated hypothesis test for alignment between signal and principal components.

\subsection{Intuition via ridge estimator}\label{section:intuitridge}

To motivate Spectrum-Aware Debiasing, let us focus on the simple instance of a ridge estimator that admits the closed-form
\begin{equation}\label{betaridge}
    \hatbt=\left(\X^{\top} \X+\lambda_2 \mathbf{I}_p\right)^{-1} \X^{\top}\y, \quad \lambda_2>0.
\end{equation}
Recall that we seek a debiased estimator of the form $\bhetah=\hatbt+\adj^{-1}\X^\top (\y-\X\hatbt)$. Suppose we plug in \eqref{betaridge}, leaving $\adj$ unspecified for the moment. If we denote the singular value decomposition of  $\X$ to be $\mathbf{Q}^{\top}\mathbf{D}\mathbf{O}$, we obtain that
\begin{equation}\label{matrixcenterD}
\E [\bhetah\mid \X, \st]=\underbrace{\left[\left(1+\frac{\lambda_2}{\adj}\right) \sum_{i=1}^p\left(\frac{d_i^2}{d_i^2+\lambda_2}\right) \mathbf{o}_i \mathbf{o}_i^{\top}\right]}_{=:\Vbf} \st,
\end{equation}
where $\mathbf{o}_i^{\top}\in \R^{p}$ denotes the $i$-th row of $\Obm$ and recall that  $d_i^2$'s denote the eigenvalues of $\X^\top\X$. 

For $\bhetah$ to be unbiased, it appears necessary to choose $\adj$ so that it centers $\mathbf{V}$ around the identity matrix $\mathbf{I}_p$. We thus choose $\adj$ to be solution to the equation
\begin{equation}\label{ideq}
\left(1+\frac{\lambda_2}{\adj}\right) \frac{1}{p} \sum_{i=1}^p \frac{d_i^2}{d_i^2+\lambda_2}=1.
\end{equation}
This choice guarantees that the average of the eigenvalues of $\Vbf$ equals 1. Solving for $\adj$, we obtain
\begin{equation}\label{eq:choiceadj}
\adj=\left(\qty(\frac{1}{p} \sum_{i=1}^p \frac{\lambda_2 d_i^2}{d_i^2+\lambda_2})^{-1}-\frac{1}{\lambda_2}\right)^{-1}.
\end{equation}
This is precisely our Spectrum-Aware adjustment formula for the ridge estimator! However, it is not hard to see that centering $\mathbf{V}$ does not guarantee debiasing in general: for instance, $\bhetah$ would have an inflation bias if $\st$ completely aligns with the top eigenvector $\mathbf{o}_1$. To ensure suitable debiasing, one requires $\X$ and $\st$ to satisfy additional structure. To this end, if we further assume that $\Obm$ is random, independent of $\st$, and satisfies
$
\mathbb{E}\left(\mathbf{o}_i \mathbf{o}_i^{\top}\right)=\frac{1}{p}\cdot \mathbf{I}_p.
$
we would obtain, after choosing $\adj$ following \eqref{eq:choiceadj}, that
\begin{equation}\label{matrixcenter}
\begin{aligned}
    \mathbb{E}\left[\bhetah\mid \st\right]&=\mathbb{E} \underbrace{\left[\left(1+\frac{\lambda_2}{\adj}\right) \sum_{i=1}^p\left(\frac{d_i^2}{d_i^2+\lambda_2}\right) \mathbf{o}_i \mathbf{o}_i^{\top}\right]}_{=:\Vbf} \st \stackrel{(\star)}{=}\st,
\end{aligned}
\end{equation}
This motivates us to impose the following assumption on $\Obm$.
\begin{Assumption*}
  $\Obm$ is drawn uniformly at random from the set of all orthogonal matrices of dimension $p$, independent of $\st$ (this is the orthogonal group of dimension $p$ that we denote as $\O(p)$), in other words, $\Obm$ is drawn from the Haar measure on $\O(p)$.
\end{Assumption*}

We operate under this assumption since it ensures $(\star)$ holds and our Spectrum-Aware adjustment turns out to be the correct debiasing strategy in this setting. 
Meanwhile, the degrees-of-freedom adjustment \cite{bellec2019biasing} yields the correction factor
\begin{equation*}
\adg=1-n^{-1} \operatorname{Tr} \bigg(\X\left(\X^{\top} \X+\lambda_2 \mathbf{I}_p\right)^{-1} \X^{\top}\bigg)=1-\frac{1}{n} \sum_{i=1}^p \frac{d_i^2}{d_i^2+\lambda_2}.
\end{equation*}
Notably, $\adj$ and $\adg$ may be quite different. Unlike $\adj$, $\adg$ may not center the spectrum of $\mathbf{V}$, and does not yield $\E (\bhetah\mid \st)=\st$ in general. However, it is important to note that they coincide asymptotically and $\adg$ would provide accurate debiasing if one assumes that the empirical distribution of $\left(d_i^2\right)_{i=1}^p$ converges weakly to the Marchenko-Pastur law (cf. \Cref{appendix:adjadgcc} from Appendix), a property that many design matrices do not satisfy. In other words, Degrees-of-Freedom Debiasing is sub-optimal in the sense that it implicitly makes the assumption that the spectrum of $\X^\top \X$ converges to the Marchenko-Pastur law, rather than using the actual spectrum. We provide examples of designs where Degrees-of-Freedom Debiasing fails in Figure \ref{fig1}. In contrast, $\adj$ is applicable under much broader settings as it accounts for the \textit{actual spectrum} of $\X^\top \X$. \Cref{fig1} shows the clear strengths of our approach over Degrees-of-Freedom Debiasing.

\subsection{Practical issues and PCR-Spectrum-Aware Debiasing} 
\label{twoissues}

Our discussion in \Cref{section:intuitridge} precludes two crucial settings that could occur in practice. 
Continuing our discussion on ridge regression, recall that $\E [\bhetah\mid \st]=\E[\Vbf \st]$ for $\Vbf$ defined in \eqref{matrixcenterD}, and we chose $\adj$ to center the spectrum of $\Vbf$ at 1 so that $\E[\Vbf] = \mathbf{I}_p$ under our assumptions. Thus our choice of  $\adj$ leads to the following, 
\begin{equation}\label{huererere}
    \Vbf \approx \mathbf{I}_p+\mathsf{unbiased \;component}.
\end{equation}
This ensures that $\bhetah$ remains centered around $\st$. However, to achieve this, we implicitly assumed that $\st$ does not align with any of the $\mathbf{o}_i$'s. Potential issues may arise when this assumption is violated.   For instance, if $\st$ perfectly aligns with the top eigenvector $\mathbf{o}_1$, we would obtain
$$\E [\bhetah\mid\X,\st]=\left(\frac{1}{p} \sum_{i=1}^p \frac{ d_i^2}{d_i^2+\lambda_2}\right)^{-1}\frac{d_1^2}{d_1^2+\lambda_2}\st.$$
This results in an inflation bias since 
$
\frac{d_1^2}{d_1^2+\lambda_2}>\frac{1}{p} \sum_{i=1}^p \frac{ d_i^2}{d_i^2+\lambda_2}.
$
Similar problems arise if $\st$ aligns with other eigenvectors, and the resulting bias could lead to inflation or shrinkage depending on the set of aligned eigenvectors. We refer to this as the \textit{alignment issue}. Another common issue arises when the top few eigenvalues of the sample covariance matrix $\X^{\top}\X$
are significantly separated from the bulk of the spectrum.  In this case, after centering the spectrum of $\Vbf$, the variance of the ``$\mathsf{unbiased \;component}$"  in \eqref{huererere} will be large, making the debiasing procedure unstable. We refer to these eigenvalues as \textit{outlier eigenvalues}.

In practice, these issues often arise simultaneously due to a small number of dominant principal components (PCs) that align with the signal. These PCs tend to distort desirable statistical properties that underlie Spectrum-Aware Debiasing. To address this, we propose a PCR-Spectrum-Aware Debiasing framework that integrates ideas from  Principal Components Regression (PCR) with Spectrum-Aware Debiasing. In this enhanced method, we employ PCR to handle the outlier eigenvalues and the aligned eigenvectors, and then use Spectrum-Aware Debiasing on a transformed version of the original data to correct for shrinkage bias incurred from discarding low-variance PCs. We observe that this hybrid PCR-Spectrum-Aware approach works exceptionally well in challenging settings where alignment and outlier eigenvalue issues may both occur. In Figure \ref{figPCRA}, we demonstrate the efficacy of our PCR-Spectrum-Aware approach in situations with extremely strong correlations, heterogeneities, and heavy tails in the design matrix.

\subsection{Notable outcomes: Alignment Testing and Debiased PCR} \label{sec:debiasedPCR}
Our theory for Spectrum-Aware Debiasing has two significant by-products. In modern high-dimensional inference, calculating the precise asymptotic risk of regularized estimators has emerged as a prominent research area. Technical tools such as approximate message passing \cite{donoho2009message,bayati2011dynamics,bayati2011lasso,zdeborova2016statistical,javanmard2013state,sur2019modern,barbier2019optimal,feng2022unifying}, the convex Gaussian min-max theorem \cite{thrampoulidis2014gaussian,stojnic2013framework}, random matrix theory \cite{Dicker,dobriban2018high,hastie2022surprises,cheng2024dimension,adlam2020understanding,li2024understanding}, and the cavity or leave-one-out method \cite{mezard1987spin,talagrand2003spin,el2013robust,el2018impact, bean2013optimal,sur2019likelihood,sur2019modern,chen2021spectral,jiang2022new} have proven invaluable for this purpose. These tools have facilitated the discovery of novel high-dimensional phenomena that other mathematical techniques simply fail to capture \cite{donoho2016high,el2013robust,bean2013optimal,el2018impact, sur2019likelihood,sur2019modern,candes2020phase,zhao2022asymptotic,liang2022precise,10.1093/imaiai/iaad042,jiang2022new,zhou2022non}. Consequently, they have inspired new high-dimensional estimators that outperform traditional ones by a margin \cite{sur2019modern,celentano2023challenges}. 
Despite such remarkable progress, these technical tools suffer a crucial limitation. They typically assume that the design matrices are random and independent of the true signal, implying that the PCs are random vectors in generic position relative to the true signal. To the best of our knowledge, a principled test to validate this assumption has so far eluded the literature. In this paper, we introduce the first formal hypothesis test for PC-signal alignment, utilizing our PCR-Spectrum-Aware approach (see \Cref{testCor} and the subsequent discussion). We hope this serves as a foundation for more systematic investigations into this issue, thereby enhancing the applicability of statistical methods developed based on  approximate message passing algorithms/leave-one-out/convex Gaussian min-max theorems.

As a second outcome, our work contributes to an extensive and growing body of work on PCR methodologies \cite{jolliffe1982note,  hubert2003robust,bair2006prediction, howley2006effect,fan2016projected, agarwal2019robustness,silin2022canonical,bing2021prediction}. Similar to the ridge and Lasso estimators, the traditional PCR estimator exhibits shrinkage bias due to the discarding of low-variance PCs \cite{farebrother1978class, frank1993statistical,george1996multiple,bickel2006regularization,druilhet2008shrinkage, jolliffe2016principal}. To the best of our knowledge, no previous work has investigated how this bias can be eliminated in high dimensions and its implications for inference. We develop the  first approach for debiasing the classical PCR estimator, complete with formal high-dimensional guarantees. We will next formally introduce Spectrum-Aware Debiasing and discuss its properties.

\subsection{Organization} We organize the rest of the paper as follows. 
In Section \ref{sec:assumption}, we introduce our assumptions and preliminaries. In Sections \ref{subsectionmrdeb} and \ref{section:pcar}, we introduce our Spectrum-Aware and PCR-Spectrum-Aware methods with formal guarantees. Finally in Section \ref{sec:conc}, we conclude with potential directions for future work.

\section{Assumptions and Preliminaries}\label{sec:assumption}
In this section, we introduce our assumptions and preliminaries that we require for the sequel.
\subsection{Design matrix, signal and noise}
We first formally define right-rotationally invariant designs. 
\begin{Definition}[Right-rotationally invariant designs]\label{def:Rotinv}
     Consider the singular value decomposition $\X=\Qbm^{\top} \Dbm \Obm$ where $\Qbm \in \mathbb{R}^{n \times n}$ and $\Obm \in \mathbb{R}^{p \times p}$ are orthogonal and $\Dbm \neq 0 \in \mathbb{R}^{n \times p}$ is diagonal. We say a design matrix $\X \in \R^{n\times p}$ is right-rotationally invariant if $\Qbm,\Dbm$ are deterministic, and $\Obm$ is uniformly distributed on the orthogonal group.
\end{Definition}

%\textcolor{red}{Can you give 2-3 line summary of the new examples here ---and then point to the appendix?}
We work in a high-dimensional regime where $p$ and $n(p)$ both diverge and $n(p)/p \to \delta \in (0,+\infty)$. Known as proportional asymptotics, this regime has gained increasing popularity in recent times owing to the fact that asymptotic results derived under this assumption demonstrate remarkable finite sample performance (cf.~extensive experiments in \cite{sur2019likelihood,sur2019modern,candes2020phase,zhao2022asymptotic,liang2022precise,jiang2022new} and the references cited therein). In this setting, we consider a sequence of problem instances $ \qty{\y(p), \X(p), \st(p), \epbm(p)}_{p\ge 1}$ such that $\y(p), \epbm (p)\in \R^{n(p)}, \X(p) \in \R^{n(p) \times p}, \st(p) \in \R^p$ and $\y (p)=\X(p) \st(p)+\epbm(p)$. In the sequel, we drop the dependence on $p$ whenever it is clear from context. 

For a vector $\boldsymbol{v} \in \mathbb{R}^p$, we call its empirical distribution to be the probability distribution that puts equal mass $1/p$ to each coordinate of the vector. Some of our convergence results will be in terms of empirical distributions of sequences of random vectors. Specifically, we will use the notion of Wasserstein-2 convergence frequently so we introduce this next.

\begin{Definition}[Convergence of empirical distribution under Wasserstein-2 distance]\label{def:Wass}
For a matrix $\left(\vbf_{1}, \ldots, \vbf_{k}\right)=$ $\left(v_{i, 1}, \ldots, v_{i, k}\right)_{i=1}^{n} \in \mathbb{R}^{n \times k}$ and a random vector $\left(\mathsf{V}_{1}, \ldots, \mathsf{V}_{k}\right)$, we write
$$ 
\left(\vbf_{1}, \ldots, \vbf_{k}\right) \stackrel{W_2}{\rightarrow}\left(\mathsf{V}_{1}, \ldots, \mathsf{V}_{k}\right)
$$
to mean that the empirical distribution of the columns 
 of
$\left(\vbf_{1}, \ldots, \vbf_{k}\right)$ converge to $(\Vs_1,\ldots,\Vs_k)$ in
Wasserstein-$2$ distance. This means that
for any continuous function
$f: \mathbb{R}^{k} \rightarrow \mathbb{R}$ satisfying
	\begin{equation}\label{eq:wasslip}
		\left|f\left(v_{1}, \ldots, v_{k}\right)\right| \leq C\left(1+\left\|\left(v_{1}, \ldots, v_{k}\right)\right\|^{2}\right)
	\end{equation}
for some $C>0$, we have $$\lim_{n \to \infty}
\frac{1}{n} \sum_{i=1}^{n} f\left(v_{i, 1}, \ldots, v_{i,
k}\right)=\mathbb{E}\left[f\left(\mathsf{V}_{1}, \ldots,
\mathsf{V}_{k}\right)\right],$$
where $\mathbb{E}\left[\left\|\left(\Vs_1, \ldots, \Vs_k\right)\right\|^2\right]<\infty$. See 
in \Cref{section:wass} from Appendix for a review of the properties of the Wasserstein-2 convergence.
\end{Definition}

\begin{Assumption}[Measurement matrix]\label{AssumpD}
 We assume that $\X \in \R^{n\times p}$ is right-rotationally invariant (\Cref{def:Rotinv}) and independent of $\epbm$. For the eigenvalues, we assume that as $n,p \rightarrow \infty$,
\begin{equation}\label{Dconve}
    \mathbf{d}:=\Dbm^\top \bm{1}_{n \times 1} \stackrel{W_2}{\to} \D,
\end{equation}
where $\D^2$ has non-zero mean with compact support\footnote{Throughout, we define support of a random variable $X$ as the smallest closed set $A$ such that $\mathbb{P}(X\in A)=1$.} $\operatorname{supp}(\D^2) \subseteq [0,\infty)$. We denote $d_-:=\min(x:x \in \operatorname{supp}(\D^2))$. Furthermore, we assume that as $p\to \infty$,
\begin{equation}\label{eqtwe}
    d_+:=\limsup_{p\to \infty}\max_{i\in [p]} d_i^2< +\infty .
\end{equation}
\end{Assumption}

\begin{Remark}\label{remarkRA}
    The constraint \eqref{eqtwe} states that $\X^\top \X$ has bounded operator norm. It has important practical implications. It  prevents the occurrence of outlier eigenvalues, where a few prominent eigenvalues of $\X^\top \X$  deviate significantly from the main bulk of the spectrum. 
\end{Remark}

We work with Assumption \ref{AssumpD} for part of the sequel, in particular, Section \ref{subsectionmrdeb}. But later in Section \ref{section:pcar}, we relax restriction \eqref{eqtwe}.

Since our debiasing procedure relies on the spectrum of $\X^{\top}\X$, analyzing its properties requires a thorough understanding of the properties of  $\D$ (from \eqref{Dconve}), the limit of the empirical spectral distribution of $\X^{\top}\X$. Often these properties can be expressed using two important quantities---the Cauchy and the R-transform. We define these next. For technical reasons, we will define these transforms corresponding to the law of $-\D^2$.
\begin{Definition}[Cauchy- and R-transform]
Under Assumption \ref{AssumpD}, let $G:(-d_-,\infty) \to (0,\infty)$
and $R:(0,G(-d_-)) \to (-\infty,0)$ be the Cauchy- and
R-transforms of the law of $-\D^2$, defined as 
\begin{equation}\label{eq:CauchyR}
G(z)=\E\left[\frac{1}{z+\D^2}\right], \qquad R(z)=G^{-1}(z)-\frac{1}{z},
\end{equation}
where $G^{-1}(\cdot)$ is the inverse function of $G(\cdot)$. See properties and well-definedness of these in \Cref{lem:cauchy} from Appendix. We set $G(-d_-)=\lim_{z \to -d_-} G(z)$.
\end{Definition}

We next move to discussing our assumptions on the signal.

\begin{Assumption}[Signal and noise]\label{AssumpPrior}
We assume throughout that $\epbm\sim N(0, \sigma^2 \cdot \mathbf{I}_p)$ for potentially \textit{unknown} noise level $\sigma^2>0$. We require that $\st$ is either deterministic or independent of $\Obm,\epbm$. In the former case, we assume that $\st \stackrel{W_2}{\to} \Xstar$ where $\Xstar$ is a random variable with finite variance. In the latter case, we assume the same convergence holds almost surely. 
\end{Assumption}

\begin{Remark}\label{remarkRB}
    The independence condition between $\st$ and $\Obm$, along with the condition that $\Obm$ is uniformly drawn from the orthogonal group enforces that $\st$ cannot align with a small number of these eigenvectors. Once again, we require these assumptions in Section \ref{subsectionmrdeb} but we relax these later in Section \ref{section:pcar}. 
\end{Remark}

{
\begin{Remark}
    The assumption on the signal $\st \stackrel{W_2}{\to} \Xstar$ may be relaxed using recent non-asymptotic AMP/VAMP theories \cite{cademartori2024non,li2024non,li2022non}. We leave this to future works. 
\end{Remark}
}

\begin{Remark}
   We believe the assumption on the noise can be relaxed in many settings. For instance, if we assume $\Qbm$ (\Cref{def:Rotinv}) to be uniformly distributed on the orthogonal group independent of $\Obm$ and $\st$, one may work with the relaxed assumption that  $\epbm \toW \mathsf{E}$ for any random variable $\mathsf{E}$ with mean 0 and variance $\sigma^2$. This encompasses many noise distributions beyond Gaussians. Even without such an assumption on $\Qbm$, allowing for sub-Gaussian noise distributions should be feasible invoking universality results. However, in this paper, we prefer to focus on fundamentally breaking the i.i.d. Gaussian assumptions on $\X$ in prior works. In this light, we work with the simpler Gaussian assumption on the noise.
\end{Remark}

In the next segment, we describe the penalty functions that we work with. 

\subsection{Penalty function}\label{asspenalty}
As observed in the vast majority of literature on high-dimensional regularized regression, the proximal map of the penalty function plays a crucial role in understanding properties of $\boldsymbol{\hat{\beta}}$. We introduce this function next.

Let the proximal map associated to $h$ be
$$
\forall v>0, x,y \in \mathbb{R}, \quad \operatorname{Prox}_{v h}(x) \equiv \underset{y \in \mathbb{R}}{\arg \min }\left\{h(y)+\frac{1}{2 v}(y-{x})^2\right\}.
$$

\begin{Assumption}[Penalty function]\label{Assumph}
We assume that $h: \mathbb{R} \mapsto[0,+\infty)$ is non-constant, proper and closed convex function. Furthermore, we assume that $h(x)$ is twice continuously differentiable 
except for a finite set $\Dc$ of points, and that $h^{\prime \prime}(x)$ and $\operatorname{Prox}_{v h}^{\prime}(x)$ have been extended at their respective undefined points using \Cref{Extend} below. 
\end{Assumption}

Note that convexity of $h$ in \Cref{Assumph} implies that for some $\co\ge 0, \forall x,y\in \R, t\in [0,1]$, 
\begin{equation}\label{eq:strong}
h(t\cdot x+(1-t)\cdot y) \leq t\cdot h(x)+(1-t)\cdot h(y)-\frac{1}{2} \co \cdot t(1-t)\cdot (x-y)^2.
\end{equation}
Here, $h$ is said to be strongly convex if $c_0>0$. 

\begin{Lemma}[Extension at non-differentiable points]\label{Extend}
 Fix any $v>0$. Under \Cref{Assumph}, $x\mapsto \operatorname{Prox}_{v h}(x)$ is continuously differentiable at all but a finite set $\mathcal{C}$ of points. Extending functions $x\mapsto h^{\prime \prime}(x)$ and $x\mapsto \operatorname{Prox}_{v h}^{\prime}(x)$ on $\Dc$ and $\mathcal{C}$ by $+\infty$ and $0$ respectively, we have that for all $x\in \R$,
\begin{equation}\label{eq:Jacprox}
    \operatorname{Prox}_{v h}^{\prime}(x)=\frac{1}{1+v h^{\prime \prime}\left(\operatorname{Prox}_{v h}(x)\right)}\in \qty[0,\frac{1}{1+v\co}], \quad h^{\prime\prime}(x) \in [\co, +\infty].
\end{equation}
After the extension, for any $w>0$, $x\mapsto \frac{1}{w+h^{\prime \prime}\left(\operatorname{Prox}_{v h}(x)\right)}$ is piecewise continuous with finitely many discontinuity points on which it takes value $0$. 
\end{Lemma}

We defer the proof to \Cref{appendix:prox} in Appendix. We considered performing this extension since our debiasing formula involves the second derivative of $h(\cdot)$. The extension allows us to handle cases where the second derivative may not exist everywhere. As an example, we compute the extension for the elastic net penalty and demonstrate the form our debiasing formula takes after plugging in this extended version of $h(\cdot).$

\begin{Example}[Elastic Net penalty] \label{Eg:EN}
    Consider the elastic-net penalty
\begin{equation}\label{elaspen}
    h(x)=\lambda_1|x|+\frac{\lambda_2}{2} x^2, \lambda_1\ge 0, \lambda_2\ge 0.
\end{equation}
This is twice continuously differentiable except at $x=0$ (i.e. $\Dc=\qty{0}$). Fix any $v>0$. Its $\operatorname{Prox}_{v h}(x)=\frac{1}{1+\lambda_2 v} \operatorname{ST}_{\lambda_1 v}\left(x\right)$ is continuously differentiable except at $x=$ $\pm \lambda_1 v$. Here, $\mathrm{ST}_{\lambda v}(x):=\operatorname{sgn}(x)(|x|-\lambda v)_{+}$ is the soft-thresholding function. Per Lemma \ref{Extend}, the extended $h^{\prime \prime}, \operatorname{Prox}_{vh}'$ are $$h^{\prime \prime}(x)=\left\{\begin{array}{c}+\infty, \text { if } x=0 \\ \lambda_2, \text { otherwise }\end{array}\right., \quad \operatorname{Prox}_{vh}'(x)=\frac{1}{1+\lambda_2 v} \mathbb{I}\left(|x|>\lambda_1 v\right)$$ respectively, so that \eqref{eq:Jacprox} holds for all $x\in \R $. Note also that for any $w>0, x\mapsto \frac{1}{1+wh^{\prime \prime}\left(\operatorname{Prox}_{v h}(x)\right)}=\frac{1}{1+\lambda_2 w} \mathbb{I}(|x|>\lambda_1 v)$ is piecewise continuous and takes value 0 on both of its discontinuity points. It follows that our adjustment \eqref{eq:adjspectrum} can be written as
\begin{equation}\label{gammasolveaELAST}
\frac{1}{p} \mathlarger{\mathlarger{\sum}}_{i=1}^p \frac{1}{\left(d_i^2\adj^{-1}-1 \right)\left(\frac{\hat{s}}{p} \qty(1+\adj^{-1}\lambda_2)^{-1} \right)+1}=1,
\end{equation}
where $\hat{s}=\left|\left\{j: \hjatbtj \neq 0\right\}\right|$. 

As a sanity check, 
if one sets $\lambda_2=0$ and solves the population version of the above equation
\begin{equation}\label{gammasolveaELASTLas}
\E \frac{1}{\left(\D^2 \adj^{-1}-1 \right)\cdot \frac{\hat{s}}{p}+1}=1
\end{equation}
with $\D^2$ drawn from the Marchenko-Pastur law, then one recovers the well-known degrees-of-freedom adjustment for the Lasso: $\adj=1-\hat{s}/n.$
\end{Example}

The following assumption is analogous to \cite[Assumption 3.1]{bellec2019biasing} for the Gaussian design: we require either $h$ to be strongly convex or $\X^\top \X$ to be non-singular with smallest eigenvalues bounded away from $0$. 
\begin{Assumption}\label{Assumpgp}
    Either $\co>0$ or $d_-:=\lim_{p\to\infty} \min_{i\in p}(d_i^2)\ge c_1$ for some constant $c_1>0$. 
\end{Assumption}

{ However, we show that for the Lasso with $h(x)=\lambda_1 |x|,\lambda_1>0$ (where $\co=0$), we may be able to drop the requirement that $d_->0$ if the penalty strength $\lambda_1>0$ is sufficiently large. We defer the discussion to \Cref{secmainLasso}. }

\subsection{Fixed-point equation}
Our general approach to study the regularized estimator $\boldsymbol{\hat{\beta}}$ is by introducing a more tractable surrogate $\boldsymbol{\hat{\beta}}^t$. As detailed in the Appendix, \Cref{section:asympchar}, we construct this surrogate using an iterative algorithmic scheme known as Vector Approximate Message Passing algorithm (VAMP) \cite{rangan2019vector}. Thus to study the surrogate, one needs to study the VAMP algorithm carefully. One can describe the properties of this algorithm using a system of fixed point equations in four variables. We use 
$\gamma_*, \eta_*, \taustar, \tau_{**}, \in(0,+\infty)$ to denote these variables, and define the system here:

\begin{subequations}\label{fp}
\begin{align}
& \frac{\gamma_*}{\eta_*}=\E \operatorname{Prox}_{\gamma_{*}^{-1} h}^{\prime}\left(\Xstar+\sqrt{\taustar} \Zs\right), \label{RCa}\\
& \tau_{**}=\frac{\eta_*^2}{\left(\eta_*-\gamma_*\right)^2}\left[\E\left(\operatorname{Prox}_{\gamma_*^{-1} h}\left(\Xstar+\sqrt{\taustar} \Zs\right)-\Xstar\right)^2-\left(\frac{\gamma_*}{\eta_*}\right)^2 \taustar\right], \label{RCb}\\
& \gamma_*=-R\left(\eta_*^{-1}\right) 
\label{RCc},\\
& \taustar=\left(\frac{\eta_*}{\gamma_*}\right)^2\left[\E\left[\frac{\sigma^2 \D^2+\tau_{**}\left(\eta_*-\gamma_*\right)^2}{\left(\D^2+\eta_*-\gamma_*\right)^2}\right]-\left(\frac{\eta_*-\gamma_*}{\eta_*}\right)^2 \tau_{**}\right], \label{RCd}
\end{align}
\end{subequations}
where $\Zs\sim N(0,1)$ is independent of $\Xstar$. We remind the reader that $x\mapsto \operatorname{Prox}_{\gamma_*^{-1} h}^{\prime}(x)$ is well-defined on $\R$ by the extension described in \Cref{Extend}. 

The following assumption ensures that at least one solution exists. 
\begin{Assumption}[Existence of fixed points] \label{Assumpfix}
    There exists a solution $\gamma_*, \eta_*, \taustar, \tau_{**}\in (0,+\infty)$ and $\eta_*>\gamma_*$ such that \eqref{fp} holds. 
\end{Assumption}

{ We now provide sufficient conditions under which \Cref{Assumpfix} holds. While the system of fixed-point equations \eqref{fp} plays a central role in the theory of regularized high-dimensional linear regression with right-rotationally invariant designs \cite{gerbelot2020asymptotic,gerbelot2022asymptotic}, there have been no rigorous mathematical results establishing the existence of solutions. The following proposition demonstrates that \Cref{Assumpfix} holds for a class of strongly convex penalties satisfying \Cref{Assumpproxclass}.

\begin{Assumption}\label{Assumpproxclass}
We assume that the proximal operator $\operatorname{Prox}_{v h}(x)$ satisfies the following  properties: 
\begin{itemize}
\item[(i)] \textbf{Monotonicity.}
For any $b\in \R, \alpha>0,v>0$, 
\begin{equation} \label{monoroa}
\frac{\partial}{\partial v} \mathbb{E} \operatorname{Prox}_{v h}^{\prime}\left(b+\frac{v}{\alpha} \mathsf{Z}\right)\le 0
\end{equation}
where $\mathsf{Z}\sim N(0,1)$. 
\item [(ii)] \textbf{Asymptotic linearity. }There exists continuous functions $r(v)$ defined on $v\in (0,+\infty)$ such that for each fixed $v>0$,
\begin{equation}\label{asumpline}
\lim _{x\rightarrow -\infty}\operatorname{Prox}^{\prime}_{v h}(x) = \lim _{x\rightarrow +\infty}\operatorname{Prox}^{\prime}_{v h}(x) =r(v)>0. 
\end{equation}
\end{itemize}
\end{Assumption}

\begin{Remark}
    Ridge, Lasso, Elastic Net (cf. ~\Cref{Eg:EN}) and Huber Norm penalty (cf. \cite{zadorozhnyi2016huber}), defined as, for tuning parameters $u,\delta>0, \lambda_1\ge 0, \lambda_2\ge 0$,
$$h(x)=\lambda_1g(x)+\frac{\lambda_2}{2} x^2, \quad g(x)=
\begin{cases}
\dfrac{u}{2}\,x^{2}, & |x|\le\delta,\\[6pt]
u\delta\bigl(|x|-\tfrac{\delta}{2}\bigr), & |x|>\delta,
\end{cases}$$
all satisfy \Cref{Assumpproxclass}. See \Cref{smdff}. 
\end{Remark}

The proof of \Cref{fixexist} under strong convexity is in \Cref{existENex} while the Lasso case is shown in \Cref{fixexistLasso}. For the Lasso case, \Cref{fixexist} imposes a mild condition that $\D^2$ places nonzero mass at zero or it satisfies $d_->0$. This rules out the edge case where $\D^2$ has no mass at zero but the eigenvalues of $\X^\top \X$ are not bounded away from zero. When $\X$ has i.i.d. sub-Gaussian entries and $n/p\to\varsigma$, $\D^2$ follows the Marchenko–Pastur law; the condition then excludes the edge case $\varsigma= 1$.

\begin{Proposition}\label{fixexist}
    Let $\mathsf{D}^2$ be the random variable defined in \Cref{AssumpD}, $h$ is strongly convex, satisfying \Cref{Assumph} with $c_0>0$, and that proximal operator $\operatorname{Prox}_{v h}(x)$ satisfies \Cref{Assumpproxclass}. Then \Cref{Assumpfix} holds. 
    
    Furthermore, under a mild condition on \(\D^2\) — specifically, if it either has non-zero probability mass at zero or satisfies \(d_- > 0\) (cf. \Cref{AssumpDLasso}) — Assumption \ref{Assumpfix} also holds for the Lasso penalty, i.e., when \(h\) is as in \Cref{Eg:EN} with \(\lambda_1 > 0\) and \(\lambda_2 = 0\).
\end{Proposition}
}
\begin{Remark}[Existence implies uniqueness]
    Under Assumptions \ref{AssumpD}--\ref{Assumpgp}, the existence of a solution implies uniqueness,  as we show in \Cref{justifuniqe} from Appendix. 
\end{Remark}

\begin{Assumption}[Feasibility of noise-level estimation]\label{Assumpfix2}
When the noise-level $\sigma^2$ is unknown, we require that $\gamma_*,\eta_*$ defined in \eqref{fp} and $\D^2$ defined in \Cref{AssumpD}  satisfy 
\begin{equation}\label{Asfix2}
    \delta\cdot \frac{\mathbb{E}\D^2 \cdot\left(1-\left(\frac{\eta_*-\gamma_*}{\D^2+\left(\eta_*-\gamma_*\right)}\right)^2\right)}{\mathbb{E}\D^2 \cdot \mathbb{E}\left(1-\left(\frac{\eta_*-\gamma_*}{\D^2+\left(\eta_*-\gamma_*\right)}\right)^2\right)} \neq 1.
\end{equation}
\end{Assumption}

\begin{Remark}
\Cref{Assumpfix2} serves as a technical condition to rule out degenerate scenarios where estimating  $\sigma^2$ is impossible. For example, this condition is not satisfied when $n=p$ and $\X=\mathbf{I}_p$: in this case, our sole observation is $\y=\st+\epbm$ and it is indeed impossible to estimate $\sigma^2$. We provide a consistent estimator for the left-hand side of \eqref{Asfix2} in \eqref{werwerlf}, facilitating the verification of \Cref{Assumpfix2}. 
\end{Remark}

\section{Spectrum-Aware Debiasing}\label{subsectionmrdeb}
Recall that our debiasing formula involved $\adj$ obtained by solving \eqref{eq:adjspectrum}. To ensure our estimator is well-defined, we need to establish that this equation has a unique solution. In this section, we address this issue, establish asymptotic normality of our debiased estimator (suitably centered and scaled), and present a consistent estimator for its asymptotic variance. 

\subsection{Well-definedness of our debiasing formula}\label{uniquewell} 
To show that \eqref{eq:adjspectrum} admits a unique solution, we define the function $g_p:(0,+\infty)\mapsto \R$ as
\begin{equation}\label{defgp}
g_p(\gamma)=\frac{1}{p} \mathlarger{\mathlarger{\sum}}_{i=1}^p \frac{1}{\left(d_i^2-\gamma\right)\left(\frac{1}{p} \sum_{j=1}^p \frac{1}{\gamma+h^{\prime \prime}\left(\hjatbtj\right)}\right)+1}.
\end{equation}
Here $h''(\cdot)$ refers to the extended version we defined using Lemma \ref{Extend} { where one should plug in $h^{\prime \prime}(x)=+\infty$ if $h$ is not twice continuously differentiable at $x$.  }

The following Proposition is restated from \Cref{CORweak} in Appendix. 
\begin{Proposition} \label{COR}
    Fix $p\ge 1$ and suppose that \Cref{Assumph} holds. Then, the function $\gamma\mapsto g_p(\gamma)$ is well-defined, strictly increasing for any $\gamma>0$, and 
    \begin{equation}\label{gammasolvea}
        g_p(\gamma)=1
    \end{equation}
    admits a unique solution in $(0,+\infty)$ if and only if there exists some $i\in [p]$ such that $h^{\prime\prime}(\hiatbti)\neq +\infty$ and at least one of the following holds: (i) $\left\|h^{\prime \prime}(\hatbt)\right\|_0=p$; (ii)  $\X^\top  \X$ is non-singular; (iii) $\norm{d}_0+\norm{h^{\prime\prime}(\hatbt)}_0>p$.
\end{Proposition}

{
\begin{Remark}
The assumptions of \Cref{COR} hold under \Cref{AssumpD}---\ref{Assumpfix} for all $p$ sufficiently large. See the proof of \Cref{ptconv}. Furthermore, if $h$ is the Lasso penalty, \Cref{Assumpfix} maybe dropped and the assumptions of \Cref{COR} hold under the assumptions of \Cref{prop:sdsLasso}. See the proof of \Cref{prop:sdsLasso}.
\end{Remark}

\begin{Remark}
We emphasize that the appearance of $h^{\prime \prime}$ in \eqref{defgp} does not preclude interesting cases such as the Lasso or Elastic Net where the penalty is non-differentiable only on a finite set.  As in \Cref{Extend}, one may simply replace $h^{\prime \prime}(x_0)$ to be $+\infty$ and hence the corresponding summand term $\frac{1}{\gamma+h^{\prime \prime} (x_0)}$ with 0 if $h$ is not differentiable at $x_0$. That said, using Lemma \ref{Extend}, we could also express \eqref{eq:adjspectrum}  in terms of  $\operatorname{Prox}_{v h}^{\prime}(x)$ for a suitable constant $v$ (instead of $h''$). The latter formulation is more common in the previous debiasing literature \cite{bellec2022biasing}. But in the way we have set things up, these formulations are equivalent. 

% As a side note, using \Cref{Extend}, it can be shown that it is asymptotically equivalent to obtain $\adj$ from solving the following equation instead
% \begin{equation*}
% \frac{1}{p} \sum_{i=1}^p \frac{1}{\left(\frac{d_i^2}{\adj}-1\right)\left(\frac{1}{p} \sum_{j=1}^p \operatorname{Prox}_{\adj^{-1} h}^{\prime}\left(\widehat{\boldsymbol{\beta}}+\adj^{-1} \mathbf{X}^{\top}(\mathbf{y}-\mathbf{X} \widehat{\boldsymbol{\beta}})\right)\right)+1}=1,
% \end{equation*}
% which refers to $\operatorname{Prox}_{\adj^{-1} h}^{\prime}(\cdot)$ operator as opposed to $h^{\prime \prime}(\cdot)$. 

\end{Remark}
}

\begin{Remark}\label{remarkuniquefind}
    To find the unique solution of $g_p(\gamma)=1$, we recommend using Newton's method initialized at $\gamma=\frac{1}{p}\sum_{i=1}^p d_i^2$. In rare cases where Newton's method fails to converge, we suggest using a bisection-based method, such as the Brent's method, to solve \eqref{eq:adjspectrum} on the interval $\left[0, \max_{i\in [p]} d_i^2\right]$, where convergence is guaranteed  (by Jensen's inequality, the solution must be upper bounded by $\max_{i\in [p]} d_i^2$). For numerical stability, we suggest re-scaling the design matrix $\X$ such that average of its eigenvalues equals 1, i.e.  $\X_{\mathsf{rescaled}} \gets \qty(\frac{1}{p}\sum_{i=1}^p d_i^2)^{-1/2}\cdot \X$. 
\end{Remark}

\subsection{The procedure} In this section, we introduce our Spectrum-Aware Debiasing procedure (\Cref{algodebias}). 

\begin{Definition}[Spectrum-Aware Debiasing]\label{algodebias}
Given observed data $(\X,\y)$ and a penalty function $h$, our procedure for Spectrum-Aware Debiasing  computes the regularized estimator $\hatbt$ and eigenvalues $(d_i^2)_{i=1}^p$ of the sample covariance matrix $\X^\top \X$. Subsequently, it solves for $\adj$ from
\begin{equation}
\frac{1}{p} \mathlarger{\mathlarger{\sum}}_{i=1}^p \frac{1}{\left(d_i^2-\adj \right)\left(\frac{1}{p} \sum_{j=1}^p \qty(\adj+h^{\prime \prime}\left(\hjatbtj\right))^{-1} \right)+1}=1.
\end{equation}
where $h''(\cdot)$ refers to the extended version we defined using Lemma \ref{Extend} (see also \Cref{COR} and \Cref{remarkuniquefind}). Finally, we generate the debiased estimator as follows
\begin{equation}
    \bhetah=\hatbt+\adj^{-1} \X^{\top}(\y-\X \hat{\boldsymbol{\beta}}).
\end{equation}
\end{Definition}

\subsection{Asymptotic normality}
\Cref{NEIGMAIN} below states that the empirical distribution of $(\tauh^{-1/2}(\bhetahi-\sti))_{i=1}^p$ converges to a standard Gaussian. The proof and discussion of technical novelty is deferred to \Cref{section:asympchar}.  
\begin{Theorem}[Asymptotic normality of $\bhetah$]\label{NEIGMAIN}
Suppose that \Cref{AssumpD}---\ref{Assumpfix2} hold. Then, we have that almost surely as $p\to \infty$,
\begin{equation*}
        \tauh^{-1/2}(\bhetah-\st)\stackrel{W_2}{\to} N(0,1).
\end{equation*}
\end{Theorem}

{ \begin{Remark}
    We prove that the asymptotic normality result in \Cref{NEIGMAIN} continue to hold under a broader spectral universality class defined in \Cref{uniclass}. This result is stated in \Cref{universethm}. 
\end{Remark}}

Next, we describe a consistent estimator for the asymptotic variance $\taustar$. We require some intermediate quantities that depend on the observed data and the choice of the penalty. We define these next. Later in \Cref{section:asympchar}, we will provide intuition as to why these intermediate quantities are important and how we construct the variance estimator. 

\begin{Definition}[Scalar statistics]
    Let $\adj(\X,\y,h) \in (0,+\infty)$ be the unique solution to \eqref{eq:adjspectrum}. We define the following scalar statistics
\begin{equation}\label{DEFEFD}
    \begin{aligned}
&\hat{\eta}_* (\X,\y,h) \gets \left(\frac{1}{p} \bigsum_{j=1}^p \frac{1}{\adj+h^{\prime \prime}\left(\hjatbtj\right)}\right)^{-1},\\
& \hat{\tau}_{**}(\X,\y,h) \gets \frac{\left\| \qty(\mathbf{I}_n+ \frac{1}{\hat{\eta}_*-\adj} \X\X^\top)\qty(\y-\X\hatbt)\right\|^2-n \hat{\sigma}^2}{ \sum_{i=1}^p d_i^2},\\
& \tauh(\X,\y,h) \gets \frac{1}{p}\bigsum_{i=1}^p \frac{\hat{\eta}_*^2 d_i^2 \hat{\sigma}^2+\qty(d_i^2-\adj+2\hat{\eta}_*)\qty(\adj-d_i^2) \qty(\hat{\eta}_*-\adj)^2 \hat{\tau}_{**}}{\qty(d_i^2-\adj+\hat{\eta}_*)^2\qty(\adj)^2},
\end{aligned}
\end{equation}
where $\hat{\sigma}^2$ is an estimator for the noise level $\sigma^2$ (see \Cref{nofkdka} below). Note that the quantities in \eqref{DEFEFD} are well-defined for any $p$ (i.e. no zero-valued denominators) if there exists some $i\in [p]$ such that $h^{\prime \prime} (\hat{\beta}_i)\neq +\infty$ and there exists some $j\in [p]$ such that $h^{\prime \prime} (\hat{\beta}_j)\neq 0$. Going forward, we suppress the dependence on $\X,\y,h$ for convenience.
\end{Definition}

\begin{Remark}\label{nofkdka}
    The computation of $\tauh$ and $\hat{\tau}_{**}$ in \eqref{DEFEFD} requires an estimator $\hat{\sigma}^2$ for the noise level $\sigma^2$ when it is not already known. We provide a consistent estimator in \eqref{wemr} that we use in all our numerical experiments.
\end{Remark}

% \begin{algorithm}[t]
% \caption{Spectrum-Aware Debiasing procedure}\label{algodebias}
% \begin{algorithmic}[1]
% \Require Response and design $(\y,\X)$ and a penalty function $h$

% \State \textbf{Find} minimizer $\hatbt$ of \eqref{deflasso}
% \State \textbf{Compute} the eigenvalues $(d_i^2)_{i=1}^p$ of $\X^\top\X$
% \State \textbf{Find} solution $\adj(\X,\y,h)$ of \eqref{gammasolvea}
% \Ensure debiased estimator \begin{equation}\label{de-based}
%         \bhetah\qty(\X,\y,h)=\hatbt+\adj^{-1} \X^\top (\y-\X \hatbt)
%     \end{equation}
%     and the associated variance estimator $\hat{\tau}_{*}(\X,\y,h)$ from \eqref{DEFEFD}
% \end{algorithmic}
% \end{algorithm}

We illustrate \Cref{NEIGMAIN} in \Cref{fig1} under five different right-rotationally-invariant designs (cf. \Cref{remark:RO} in Appendix) with non-trivial correlation structures, and compare with Degrees-of-Freedom Debiasing with $\mathbf{M}=\mathbf{I}_p$.  The corresponding QQ-plot can be found in \Cref{figQQA} in Appendix. We observe that our method outperforms Degrees-of-Freedom Debiasing by a margin. 

We next develop a different result that characterizes the asymptotic behavior of finite-dimensional marginals of $\bhetah$. Corollary \ref{sgoods} below establishes this under an additional exchangeability assumption on $\st$. To state the corollary, we recall to readers the standard definition of exchangeability for a sequence of random variables.  

\begin{Definition}[Exchangeability]\label{exchangedef}
     We call a sequence of random variables $\left(\mathsf{V}_i\right)_{i=1}^p$ exchangeable if for any permutation $\pi$ of the indices $1,...,p$, the joint distribution of the permuted sequence $\left(\mathsf{V}_{\pi(i)}\right)_{i=1}^p$ is the same as the original sequence. 
\end{Definition}

\Cref{sgoods} below is a consequence of \Cref{NEIGMAIN}. We defer its proof to \Cref{appendix:single} in Appendix. A numerical demonstration is included in \Cref{margex} from Appendix.
\begin{Corollary}\label{sgoods}
Fix any finite index set $\mathcal{I} \subset [p]$. Suppose that \Cref{AssumpD}---\ref{Assumpfix2} hold, and $\left(\st\right)_{j=1}^{p}$ is exchangeable independent of $\X, \epbm$. Then as $p\to \infty$, we have
\begin{equation}\label{werbaoz}
    \frac{\bhetah_{\mathcal{I}}-\st_{\mathcal{I}}}{\sqrt{\hat{\tau}_{*}}} \Rightarrow N(\rm{0},\mathbf{I}_{|\mathcal{I}|})
\end{equation}
where $\Rightarrow$ denotes weak convergence. 
\end{Corollary}
Corollary \ref{sgoods} is naturally useful for constructing confidence intervals for finite-dimensional marginals of $\st$ with associated false coverage proportion guarantees. 

\subsection{Inference}\label{sec:infvanilla}
\begin{figure}
    \centering
    \includegraphics[width=0.71\linewidth]{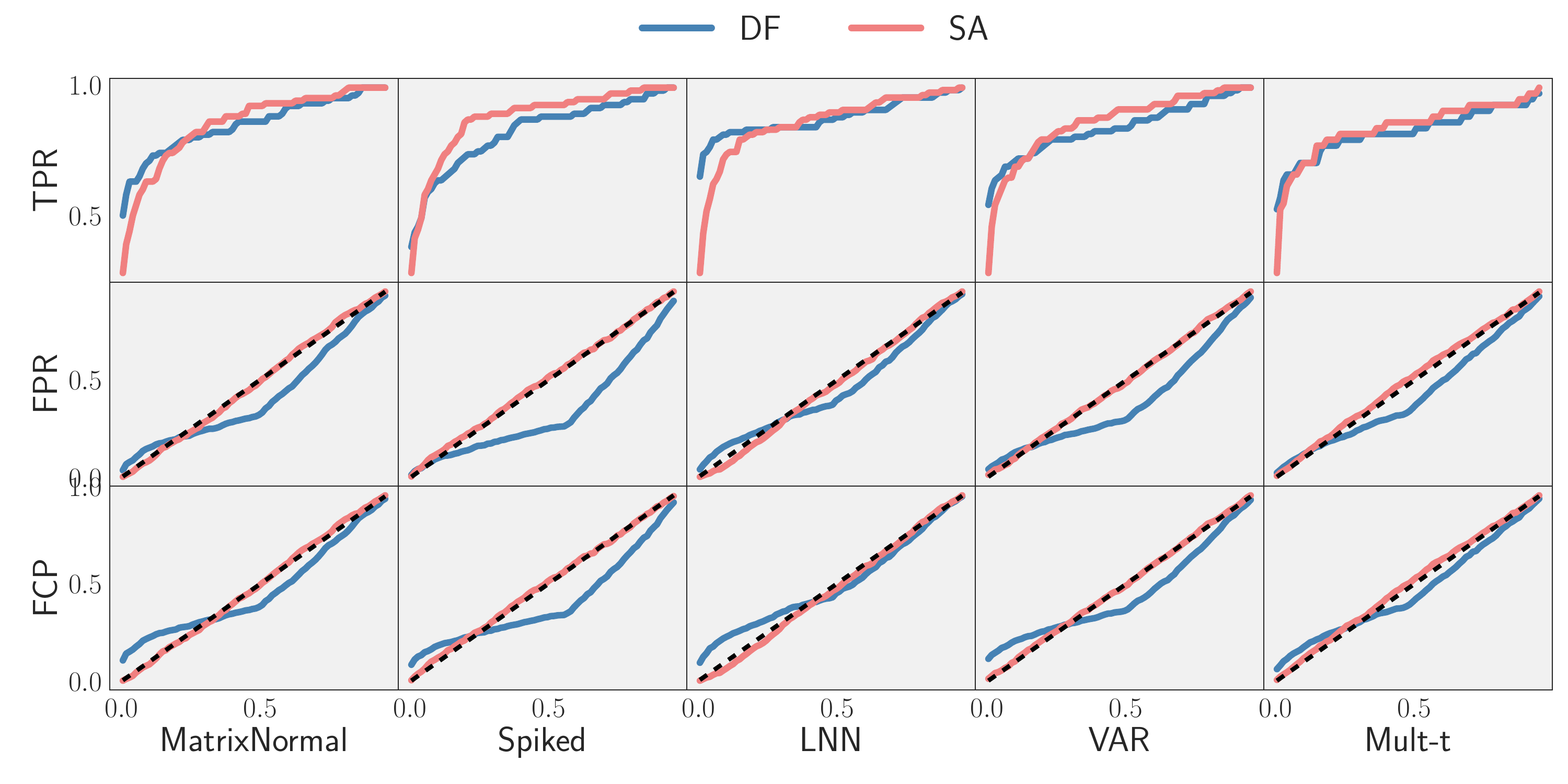}
    \vspace{-0.1cm}
    \caption{The above plots the TPR and FPR of the hypothesis testing procedure defined in \eqref{pvaldef} with significance level $\alpha$ and FCP of the constructed confidence intervals \eqref{defCI} with $b=\Phi^{-1}(1-\alpha/2), a=\Phi^{-1}(\alpha/2)$ as $\alpha$ on the x-axis varies from $0$ to $1$, for both degrees-of-freedom ($\mathsf{DF}$, blue) adjustment and Spectrum-Aware ($\mathsf{SA}$, red) adjustment. The setting here is the same as in \Cref{fig1}. }
    \label{fig2}
\end{figure}
In this section, we discuss applications of our Spectrum-Aware Debiasing approach to hypothesis testing and construction of confidence intervals. Consider the null hypotheses $H_{i, 0}: \sti=0$ for all $i \in[p]$. We define $\mathrm{p}$-values $P_i$ and decision rule $T_i$ ($T_i=1$ means rejecting $H_{0,i}$) for the test $H_{0, i}$ based on the definitions
\begin{equation}\label{pvaldef}
P_i \qty(\bhetahi, \tauh)=2\left(1-\Phi\left(\left|\frac{\bhetahi}{\sqrt{\hat{\tau}_{*}}}\right|\right)\right), \quad T_i(\bhetahi, \tauh)=\left\{\begin{array}{cc}
1, & \text { if } P_i\qty(\bhetahi,\tauh) \leq \alpha \\
0, & \text { if } P_i \qty(\bhetahi,\tauh)>\alpha
\end{array},\right.
\end{equation}
where $\Phi$ denotes the standard Gaussian CDF and $\alpha \in[0,1]$ is the significance level. We define the false positive rate (FPR) and true positive rate (TPR) below
$$
\mathsf{FPR}(p):=\frac{\sum_{j=1}^p \mathbb{I}\left(P_j \leq \alpha, \beta_j^{\star}=0\right)}{\sum_{j=1}^p \mathbb{I}\left(\beta_j^{\star}=0\right)}, \quad \mathsf{TPR}(p):=\frac{\sum_{j=1}^p \mathbb{I}\left(P_j \leq \alpha,\left|\beta_j^{\star}\right|>0\right)}{\sum_{j=1}^p \mathbb{I}\left(\beta_j^{\star}>0\right)}
$$
when their respective denominators are non-zero. Fix $\alpha \in[0,1]$. We can construct confidence intervals 
\begin{equation}\label{defCI}
    \mathsf{CI}_i(\bhetahi, \tauh)=\left(\bhetahi+a \sqrt{\hat{\tau}_{*}}, \bhetahi+b \sqrt{\hat{\tau}_{*}}\right), \qquad \forall i \in[p]
\end{equation}
for any $a, b \in \mathbb{R}$ such that $\Phi(b)-\Phi(a)=1-\alpha$. One can define the associated false coverage proportion (FCP)
$$
\mathsf{FCP}(p):=\frac{1}{p} \sum_{i=1}^p \mathbb{I}\left(\sti \notin \mathsf{CI}_i\right).
$$
for any $p\ge 1$. \Cref{NEIGMAIN} directly yield guarantees on the FPR, TPR and FCP as shown in  \Cref{weeren} below. We defer the proof to \Cref{appendix:testing} in Appendix.

\begin{Corollary}\label{weeren}
Suppose that \Cref{AssumpD}---\ref{Assumpfix} hold. We have the following.
\begin{itemize}
\item [(a)] Suppose that $\mathbb{P}\left(\Xstar =0\right)>0$ and there exists some $\mu_0 \in(0,+\infty)$ such that $$\mathbb{P}\left(\abs{\Xstar} \in\left(\mu_0,+\infty\right) \cup\{0\}\right)=1.$$ Then for any fixed $i$ such that $\sti=0$, we have $\lim _{p \rightarrow \infty} \mathbb{P}\left(T_{i}=1\right)=\alpha,$ and the false positive rate satisfies that almost surely $\lim _{p\to \infty} \mathsf{FPR}(p)=\alpha.$ Refer also to \Cref{TPRlimit} from Appendix  for the exact asymptotic limit of TPR. 
\item [(b)] The false coverage proportion satisfies that almost surely $\lim _{p\to \infty} \mathsf{FCP}(p)=\alpha$. 
\end{itemize}
\end{Corollary}

We demonstrate \Cref{weeren} in \Cref{fig2}. We note that the FPR and FCP values obtained from Degrees-of-Freedom Debiasing diverge from the intended $\alpha$ values, showing a clear misalignment with the 45-degree line. In contrast, the Spectrum-Aware Debiasing method aligns rather well with the specified $\alpha$ values, and this occurs without much compromise on the TPR level.

\section{PCR-Spectrum-Aware Debiasing}\label{section:pcar}

\subsection{Outliers and PC alignment}\label{challengedfs}
The assumptions made in our previous section exclude important scenarios where the  design may contain outlier eigenvalues or the signal may align with an eigenvector of the sample covariance matrix. To address this challenge, we propose an enhanced Spectrum-Aware procedure which relaxes Assumptions \ref{AssumpD} and \ref{AssumpPrior} to \Cref{AssumpPriorAL} below. To this end, denote $\mathcal{N}:=\left\{i \in [p]: d_i^2>0\right\}, N:=|\mathcal{N}|$. We let $\Js$ be a user-chosen, finite index set $\Js \subseteq \mathcal{N}$ that should ideally contain outlier eigenvalues and alignment eigenvectors (See \Cref{woremakrs}). We denote its size as $J:=\abs{\Js}$. 
\begin{Assumption}\label{AssumpPriorAL}
We assume that $\Js$ is of finite size\footnote{Finite size means that $J_1$ does not grow with $n,p$.} and for some real-valued vectors $\alphr \in \R^{J},\zetr\in \R^{p}$,
    \begin{equation}\label{stwerif}
        \st=\stal+\zetr, \qquad  \stal=\sum_{i=1}^{J} \alphstari \cdot \mathbf{o}_{\Js(i)}.
    \end{equation}
where we used $\Js(i)$ to denote the $i$-th index in $\Js$. Both $\alphstar$ and $\zetr$ are unknown, and they can be either deterministic or random independent of $\Obm,\epbm$.  If $\zetr$ is deterministic, we assume that 
$\zetr \stackrel{W_2}{\to} \Zetr$ as $n,p\to \infty$, where $\Zetr$ is a random variable with finite variance. If $\zetr$ is random, we assume the same convergence holds almost surely. Furthermore, we assume that \Cref{AssumpD} holds except that, instead of \eqref{Dconve} and \eqref{eqtwe}, we only require eigenvalues outside of $\Js$ to be bounded and converge in empirical measure, $$\mathbf{d}_{\Js^c} \stackrel{W_2}{\to} \D, \qquad \limsup_{p\to \infty} \max_{i\in [p]\setminus \Js } d_i^2<+\infty$$ where $\mathbf{d}_{\Js^c}$ denotes a sub-vector of $\mathbf{d}=\mathbf{D}^\top \bm{1}_{n\times 1}$ with entries indexed by $\Js$ removed. Finally, we require that $\limsup_{p\to \infty} \max_{i\in \Js } d_i^{-2}/p\to 0.$
\end{Assumption}
Under Assumption \ref{AssumpPriorAL}, $\stal$ is the \textit{alignment component} that aligns to $\Js(i)$-th Principal Component (PC) $\mathbf{o}_{\Js(i)}$ if the corresponding $\alphstari$ is non-zero, while $\zetr$ is the \textit{complement component} that is independent of the design. Note that when $\Js=\emptyset$, \Cref{AssumpPriorAL} reduces to Assumptions \ref{AssumpD} and \ref{AssumpPrior} precisely. Finally, we note that the condition $\limsup_{p\to \infty} \max_{i\in \Js } d_i^{-2}/p\to 0$ is mild: it simply requires that the smallest eigenvalues contained in $\Js$ does not converge to 0 at a faster than $O(1/p)$ rate. 
\begin{Remark}
    \Cref{AssumpPriorAL} does not impose any constraints on $\alphstar\in \R^{J}$. For example, it is permitted that $\alphstar=0$ or that $p^{-1}\norm{\alphstar}^2$ diverges as $p\to \infty$. Note that \Cref{AssumpPriorAL} also permits $\zetr=0$ but $p^{-1}\norm{\zetr}^2$ cannot diverge. 
\end{Remark}

\begin{Remark}\label{woremakrs}
$\Js$ needs to be a finite index set that contains indices of both outlier eigenvalues and alignment eigenvectors. The outlier eigenvalues can be determined by observing the spectrum of $\X^\top \X$. The alignment set is generally not observed. We thus proposed an alignment test in \Cref{Altest} for detecting the alignments. However, we remark that eigenvectors that are both dominant and align with the signal tend to distort the debiasing procedure most severely. So often just including top few PCs in $\Js$ can significantly improve inference. 
\end{Remark}

We develop a debiasing approach that recovers both components of $\bbeta^{\star}$ from \eqref{stwerif}. Our approach uses classical PCR to consistently estimate the aligned component $\stal$ and uses Spectrum-Aware Debiasing to produce a debiased estimator of $\zetr$. 

\subsection{The PCR algorithms}
\begin{figure}[t] 
  \centering
  \includegraphics[width=0.49\columnwidth]{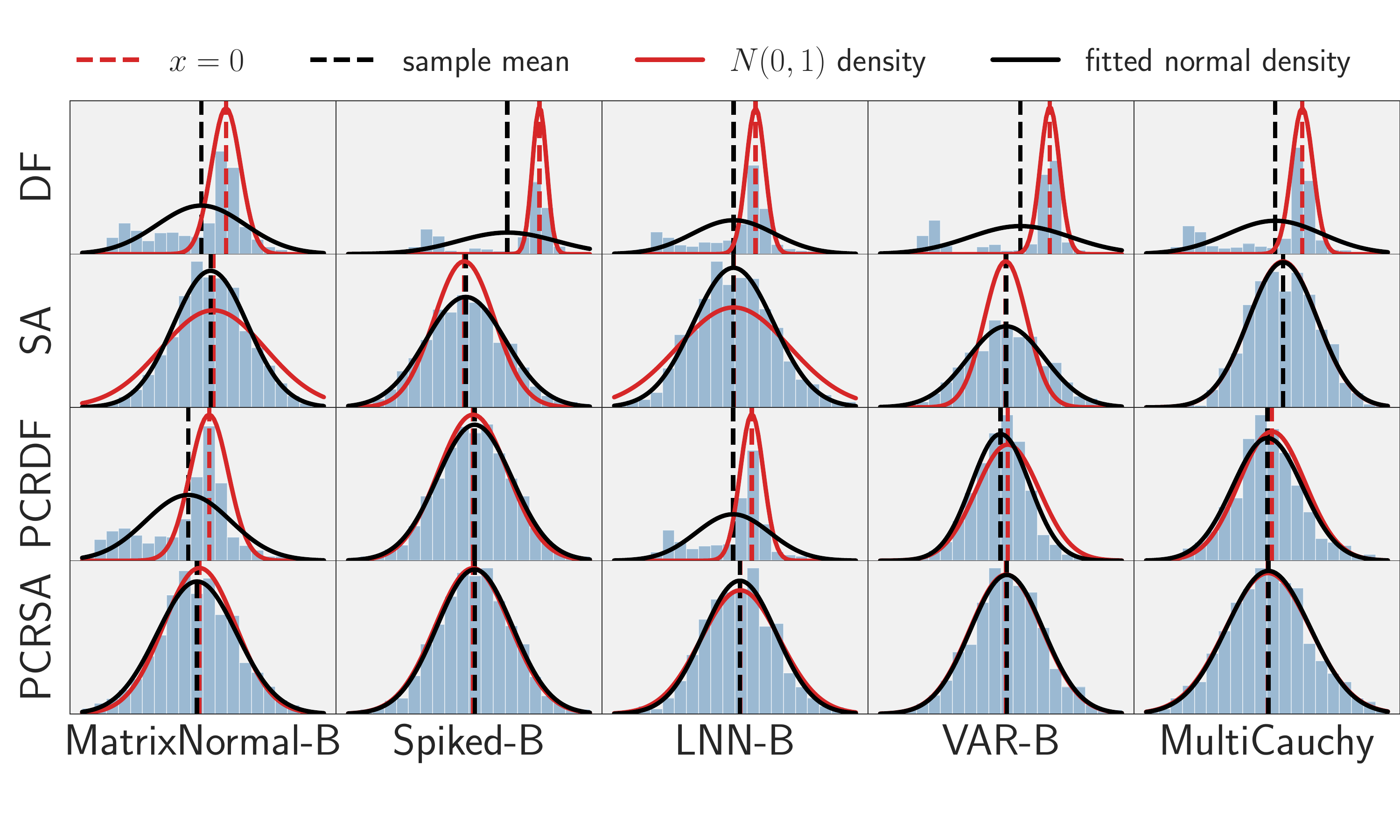} 
  \includegraphics[width=0.49\columnwidth]{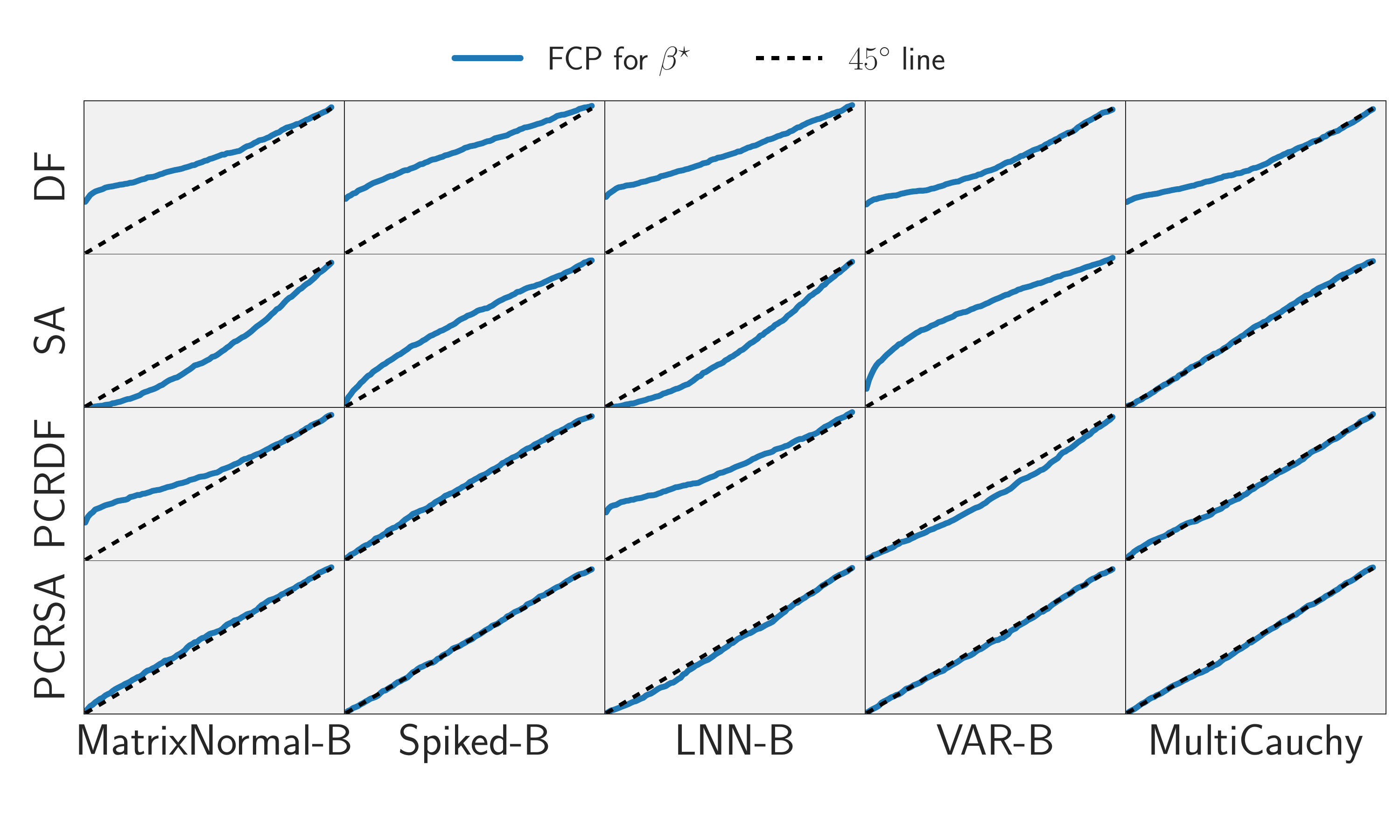} 
  
  % Second row
  \includegraphics[width=0.49\columnwidth]{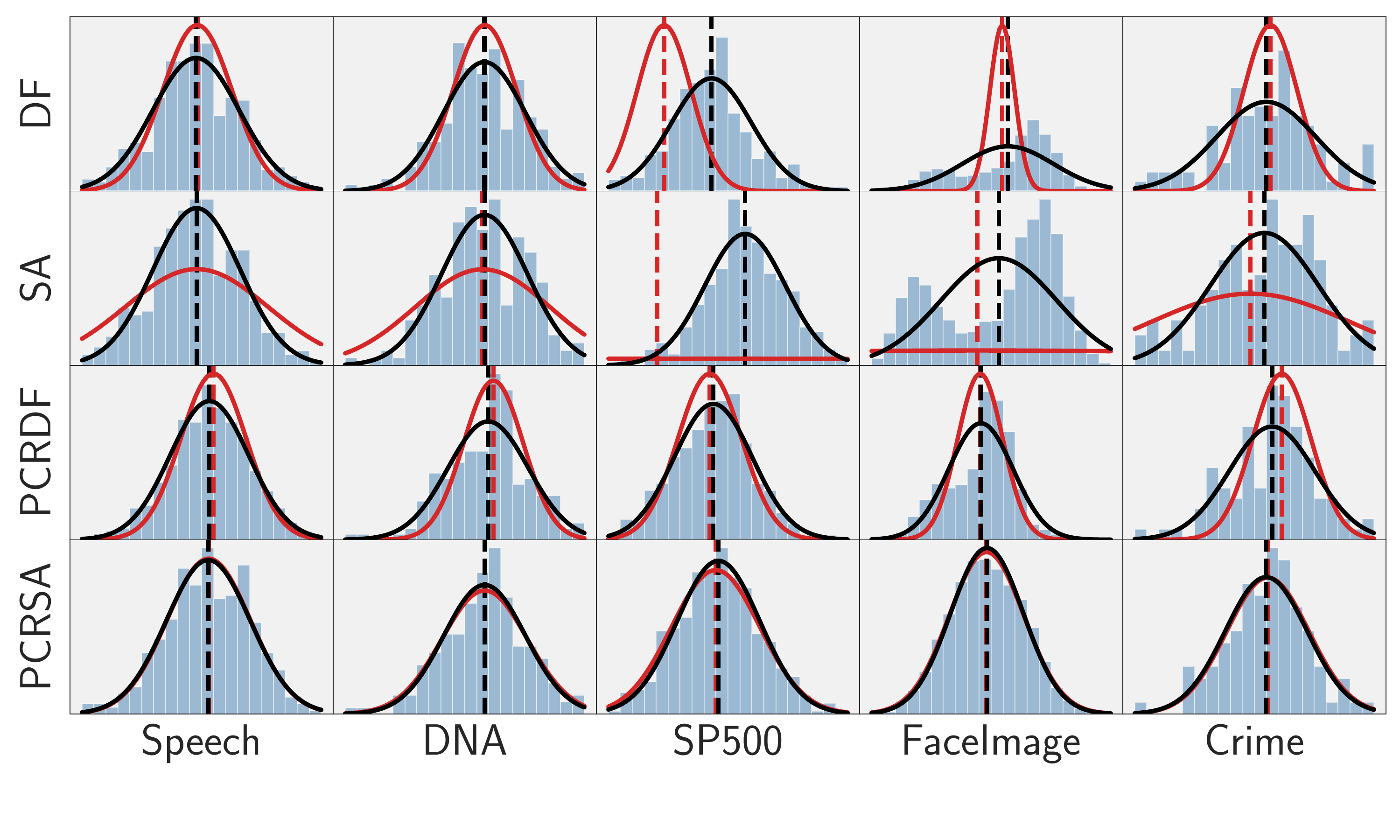} 
  \includegraphics[width=0.49\columnwidth]{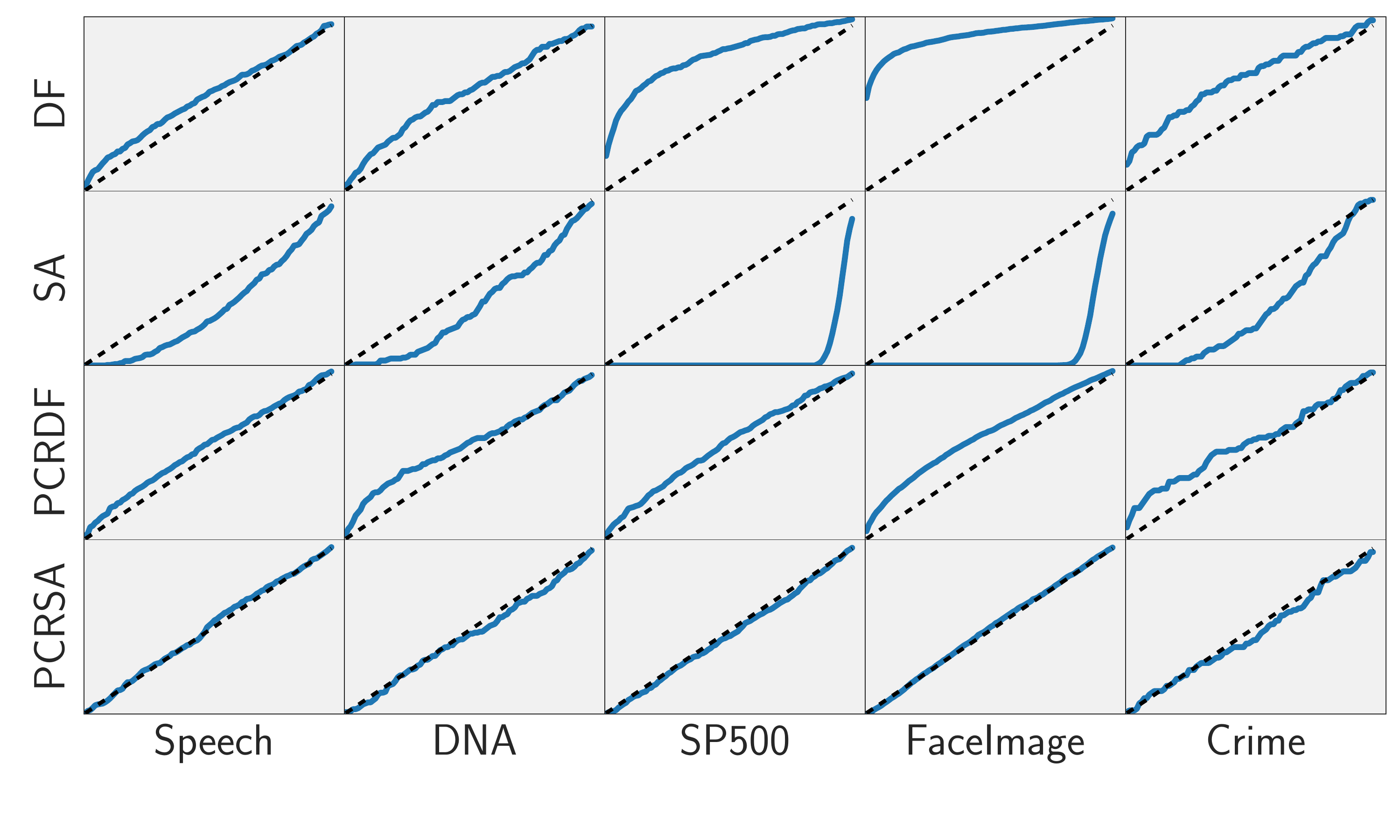} 
  \caption{\textbf{top-left}: Same setting as \Cref{fig1} except for specific changes to the design distribution parameters that lead to more challenging scenarios (see details of $\mathsf{MatrixNormal}$-$\mathsf{B}$,...,$\mathsf{MultiCauchy}$ in \Cref{GroupS} from Appendix). Rows 1--4 correspond to: (i) DF: Degrees-of-Freedom Debiasing as described in \cite{bellec2019biasing}; (ii) SA: Spectrum-Aware Debiasing as described in \Cref{algodebias}; (iii): PCRDF: PCR-Degrees-of-Freedom Debiasing, that is, the procedure obtained from \Cref{PCRSADef}  on substituting Degrees-of-Freedom Debiasing for Spectrum-Aware Debiasing in the complement PCR step; (iv) PCRSA: PCR-Spectrum-Aware Debiasing as described in \Cref{PCRSADef}. The true signals $\st$ for these designs are generated such that they align with the 2nd, 4th, and 6th PCs. Concretely, we generate the signal as follows: $\st=\stal+\zetr$ where the components of $\zetri$ are i.i.d.~draws satisfying $\zetri \sim 0.24\cdot N(-20,1)+0.06\cdot N(10,1)+0.7\cdot \delta_0$ and $\stal=\sum_{i=1}^{J} \alphstari \cdot \mathbf{o}_{\Js(i)}$ with $\alphstari=5\cdot \sqrt{p}, i \in \{2,4,6\}$ and $0$ otherwise. We set $\Js$ to be the top 20 PCs for all designs, except MultiCauchy where we use the top 100 PCs. Penalty $h$ used in complement PCR step (described in \Cref{pcrcompe}) is identical to that used in \Cref{fig1}. See the corresponding QQ plot in \Cref{figQQB} from Appendix. \textbf{bottom-left}: Same setting as top-left except that the designs are taken from real datasets (real data descriptions are in \Cref{GroupD} from Appendix). The dataset sizes are as follows. $\mathsf{Speech}$: $200\times 400$, $\mathsf{DNA}$: $100\times 180$, $\mathsf{SP500}$: $300\times 496$ and $\mathsf{FaceImage}$: $1348\times 2914$ and $\mathsf{Crime}$: $50 \times 99$.  All designs are centered and standardized (across rows) and then rescaled so that the average of eigenvalues of sample covariance matrix is $1$. The signal is generated in the same way as in the top-left. See the corresponding QQ plot in \Cref{figQQD} from Appendix. We set $\Js$ to be the top 10 PCs for all designs, except FaceImage where we once again use the top 100 PCs.
  \textbf{top-right}: Under the setting of top-left, we plot the false coverage proportion (FCP) of the confidence intervals for $(\sti)_{i=1}^p$, as we vary the targeted FCP level on the x-axis $\alpha$ from $0$ to $1$. The $y$-axis also ranges from 0 to 1. \textbf{bottom-right}: analogous FCP plots under the setting of bottom-left.}
  \label{figPCRA}
\end{figure}

\subsection{PCR procedures}\label{assB}
Given the index set $\Js$, we describe PCR procedures that separately estimate the alignment component $\stal$ and the complement component $\zetr$. 

\subsubsection{Classical PCR}
The alignment component $\stal$ can be readily recovered using the traditional PCR method. The method computes the following:
\begin{equation}\label{pcrdef}
\pcr(\Js):=\left(\tilde{\mathbf{X}}_\Js^{\top} \tilde{\mathbf{X}}_\Js\right)^{-1} \tilde{\mathbf{X}}_\Js^{\top}\y\in \mathbb{R}^J,
\end{equation}
where $\tilde{\mathbf{X}}_\Js:=\X \Obm_\Js^{\top}\in \R^{n \times J}$ represents the basis-transformed design matrix and $\Obm_{\Js}\in \R^{J\times p}$ comprises rows of $\Obm$ indexed by $\Js$. The alignment PCR estimator is then given by $\pcrt:=\Obm_{\Js}^\top \pcr (\Js) \in \R^p$.
\Cref{PCRTHM} (a) shows that $\pcrt$ is a consistent estimator of $\stal$. 
This is the traditional PCR estimator, but it suffers from a shrinkage bias since it only recovers $\stal$.  To obtain an asymptotically unbiased estimator for $\st$, it is essential to debias $\pcrt$. We accomplish this in the following section.

\subsubsection{Complement PCR}\label{pcrcompe}
%In the last paragraph, we developed $\pcrt$ to estimate the alignment component $\stal$. We now
We leverage our Spectrum-Aware Debiasing theory to devise a modified PCR procedure that provides an accurate 
estimate of the complement component $\zetr$. {We collect the indices in $\Js^c$ corresponding to \textit{positive} eigenvalues not used by alignment PCR ($\Jsb$ differs from $\Js^c$ as it excludes zero eigenvalues)}
$$\Jsb:= \{i\in [p]: d_i^2>0, i\notin \Js\}.$$ 
Here, $|\Jsb|=N-J$ where $N=\mathrm{rank}(\X)$. As a first step, we calculate a PCR estimator using the PCs indexed by $\Jsb$. That is, we calculate
${\pcr}^\bot=\pcr(\Jsb)$ using the definition in  \eqref{pcrdef}. Next, we construct a  new dataset as follows
\begin{equation}\label{xnewYnew}
    \begin{aligned}
        &\Xnew := \qty(\Dbm_{\Jsb}^\top \Dbm_{\Jsb})^{1/2} \Obm_{\Jsb}, \;\; \;\; \ynew := \qty(\Dbm_{\Jsb}^\top \Dbm_{\Jsb})^{1/2} \pcr^\bot,
    \end{aligned}
\end{equation}
where $\Dbm_{\Jsb} \in \R^{n \times (N-J)}$, $ \Obm_{\Jsb}\in \R^{(N-J)\times p}$  respectively consist of the columns of $\Dbm$ and the rows of $\Obm$   indexed by $\Jsb$. We employ Spectrum-Aware Debiasing on this new dataset. The resulting estimator, which we call complement PCR, is $\pcrc= \bhetah(\Xnew, \ynew, h)$, which is calculated from \eqref{SAestimator} and \eqref{eq:adjspectrum} with respect to the new dataset $(\Xnew, \ynew)$. We establish in \Cref{PCRTHM} (b) that $\pcrc$ is approximately Gaussian centered at the complement signal component $\zetr$, with variance $\tauc=\tauh(\Xnew, \ynew, h)$ obtained using \eqref{DEFEFD} on the new dataset.

\subsubsection{PCR-Spectrum-Aware Debiasing or Debiased PCR} \label{PCRSADef}
Combining our estimators from the previous sections, we obtain a debiased estimator for the full signal 
$\st$ given by 
$\pcrdb:= \pcrt+\pcrc$. Since this estimator utilizes ideas from the classical PCR as well as our Spectrum-Aware Debiasing approaches, we name it   PCR-Spectrum-Aware Debiasing. 
If the index set $\Js$ includes all outlier PCs and PCs aligned with $\st$, the procedure successfully removes shrinkage bias of the classical PCR estimator. It achieves this by ``repurposing" discarded PCs to construct the complement component estimator  $\pcrc$. \Cref{algodebiaspcr}  from Appendix presents the entire procedure in detail.

\subsection{Asymptotic normality}\label{section:asympcrcf}
We now state the asymptotic properties of the debiased PCR procedure. The proof of the theorem below is deferred to \Cref{appendix:pcrdeb} in Appendix.  

\begin{Theorem}\label{PCRTHM}
Suppose Assumptions \ref{Assumph}---\ref{AssumpPriorAL} hold. Then, almost surely as $p\to \infty$, we have the following: (a) Alignment PCR: $\frac{1}{p}\norm{\pcrt(\Js)-\stal}^2\to 0$; (b) Complement PCR: $\tauc^{-1/2}\qty(\pcrc(\Jsb)-\zetr)\stackrel{W_2}{\to} N(0,1)$; (c) Debiased PCR: $\tauc^{-1/2}\qty(\pcrdb-\st)\stackrel{W_2}{\to} N(0,1).$
% \begin{align}
%     &\frac{1}{p}\norm{\pcrt(\Js)-\stal}^2\to 0, \tag{a. Alignment PCR}\\
%     &\tauc^{-1/2}\qty(\pcrc-\zetr)\stackrel{W_2}{\to} N(0,1), \tag{b. Complement PCR}\\
%     &\tauc^{-1/2}\qty(\pcrdb-\st)\stackrel{W_2}{\to} N(0,1). \tag{c. Debiased PCR}
% \end{align}

% \begin{itemize}
%     \item[(a) \textbf{Alignment PCR.}] $\quad \qquad \qquad \frac{1}{p}\norm{\pcrt(\Js)-\stal}^2\to 0,$

%     \item[(b) \textbf{Complement PCR.}] $\qquad \qquad \tauc^{-1/2}\qty(\pcrc(\Jsb)-\zetr)\stackrel{W_2}{\to} N(0,1),$
%     \item[(c) \textbf{Debiased PCR.}] $\;\;\;\; \quad \qquad \qquad \tauc^{-1/2}\qty(\pcrdb-\st)\stackrel{W_2}{\to} N(0,1)$.
    
% \end{itemize}
\end{Theorem}

\begin{Remark}
    Given exchangeability of entries of $\zetr$, we may obtain results analogous to \Cref{sgoods} for finite or single coordinate inference. We defer the results to \Cref{appendix:finitecoorpcr} in Appendix.
\end{Remark}

{ \begin{Remark}
    We prove that a variant of \Cref{PCRTHM} continues to hold under a broader spectral universality class defined in \Cref{uniclass}. This result is stated in \Cref{univPCRthm}. 
\end{Remark}}

We demonstrate \Cref{PCRTHM} using two sets of design matrices. Our first set (top panel of \Cref{figPCRA}) represents more challenging variants of the settings from \Cref{fig1}. These designs contain high correlation, heterogeneity, or both. They also contain outlier eigenvalues and the signal $\st$ aligns with a few top eigenvectors. Specifically, the top panel presents the following right-rotationally invariant designs: (i) $\mathsf{MatrixNormal}$-$\mathsf{B}$: stronger row- and column-wise correlations than $\mathsf{MatrixNormal}$; (ii) $\mathsf{Spiked}$-$\mathsf{B}$: larger and fewer spikes than $\mathsf{Spiked}$; (iii) $\mathsf{LNN}$-$\mathsf{B}$: matrix product with larger exponents and stronger correlations than $\mathsf{LNN}$; (iv) $\mathsf{VAR}$-$\mathsf{B}$: stronger row dependencies than $\mathsf{VAR}$; (v) $\mathsf{MultiCauchy}$: heavier tails than $\mathsf{Multi}$-$t$. Detailed description of these design distributions are given in \Cref{GroupS} from Appendix. Our second set of experiments (bottom panel of \Cref{figPCRA}) uses real data designs from five domains: speech audio \cite{goldstein2016comparative} , DNA \cite{noordewier1990training}, stock returns (S\&P 500) \cite{SP500}, face images \cite{LFWTech}, and crime metrics \cite{misc_communities_and_crime_183}. Further details about these design matrices are included in \Cref{GroupD} from Appendix.

\subsection{Inference}\label{section:inference}
Theorem \ref{PCRTHM} motivates an inference procedure similar to \Cref{sec:infvanilla}. For a specified level $\alpha \in [0,1]$, the confidence intervals  
\begin{equation} \label{defCIPCR}
    \mathsf{CI}_i \qty(\pcrdbi, \tauh)=\left(\pcrdbi+a \sqrt{\tauh}, \pcrdbi+b \sqrt{\tauh}\right)
\end{equation}
admit the false coverage proportion guarantee $\mathsf{FCP}(p):=\frac{1}{p} \sum_{i=1}^p \mathbb{I}\left(\sti \notin \mathsf{CI}_i\right) \rightarrow \alpha$, when $a,b$ satisfy $\Phi(b) - \Phi(a) = 1-\alpha$. 
The right column of \Cref{figPCRA} displays the FCP of these confidence intervals in the settings discussed following Theorem \ref{PCRTHM}. 
PCR-Spectrum-Aware Debiasing achieves an FCP that aligns exceptionally well with the intended $\alpha$ values across these challenging settings, outperforming other methods. 

% Analogous inference procedures can be developed for the components $\zetr$ in \eqref{stwerif} using \Cref{PCRTHM}, (b). 

\subsection{Alignment testing} \label{Altest} A  fundamental challenge that modern data analysis presents relates to alignment of a part of the signal with eigenvectors of the sample covariance matrix. 
Such alignment distorts the performance of inference procedures unless they explicitly account for it.  
As a by-product,  our Spectrum-Aware Debiasing theory provides a formal test for alignment, in other words, for testing $H_{i,0}^{\alphstar}:\alphstari=0$ vs $H_{i,1}^{\alphstar}:\alphstari \neq 0$, where $\alphstari$ is given by \eqref{stwerif}. 
Below, $\alphstar$ refers to the vector with $i$-th entry $\alphstari$. \Cref{testCor} below is proved in \Cref{apptestcor} from Appendix.
\begin{Corollary}\label{testCor}
Suppose that the assumptions in \Cref{PCRTHM} hold. Then as $p\to \infty$, $$\hat{\bm{\Gamma}}^{-1/2}\qty(\pcr-\alphstar)\Rightarrow N(\bm{0},\mathbf{I}_J),$$ where $\pcr$ is given by \eqref{pcrdef}, $\hat{\bm{\Gamma}}=\hat{\sigma}^2\cdot \qty(\Dbm^\top_\Js \Dbm_\Js)^{-1}+\magzetr \cdot \mathbf{I}_{J}$ with $\Dbm_\Js\in \R^{n\times J}$ representing columns of $\Dbm$ indexed by $\Js$, $\hat{\sigma}^2$ a consistent estimator for the noise variance $\sigma^2$ given in \eqref{wemr} and $\magzetr:=p^{-1}{\norm{\pcrc}^2}-\tauhpcr$. 
\end{Corollary}
Corollary \ref{testCor} motivates the p-values $P_i :=2-2\cdot \Phi\qty(\abs{\pcri/s_i}),s_i:=\sqrt{\hat{\sigma}^2\cdot d^{-2}_{\Js(i)}+\magzetr}$. Since the quantities $\qty(P_i)_{i=1}^J$ are asymptotically independent, the Benjamini-Hochberg procedure \cite{benjamini1995controlling} can be used to control the False Discovery Rate (FDR), which is the expected ratio of PCs falsely identified as aligned with $\st$ out of all PCs identified as aligned with $\st$. 

 We demonstrate the efficacy of our alignment test in \Cref{figtest}. Panel (i) displays Benjamini-Hochberg adjusted p-values for testing alignment in the setting of the real-data designs considered in \Cref{figPCRA}, bottom-row. Panel (iii) shows the true alignment angles between the underlying signal and the top six PCs. Our test accurately identifies alignment where present. In this setting, alignment detection is relatively easy, as the true alignment angles between the top PCs and signal, where present, are all significantly smaller than $90^\circ$ ($\angle (\st, \mathbf{o}_i) \lesssim 75^\circ$). We illustrate in \Cref{figtest}, panels (ii) and (iv) how our alignment test performs when alignments become less pronounced and therefore harder to detect. Our method remains effective in rejecting all strong alignments present ($\angle (\st, \mathbf{o}_i) \lesssim 75^\circ$). While it is less decisive in rejecting weak alignments ($75^\circ \lesssim\angle (\st, \mathbf{o}_i) \lesssim 85^\circ$), the overall detected alignment pattern, as reflected in the small p-values shown in panel (ii), closely matches the true alignment pattern displayed in panel (iv)\footnote{Note that in the setting of \Cref{figtest}, we artificially aligned the signal with the 2nd, 4th and 6th PCs. However, the bottom row of \Cref{figtest} suggests that the signal $\st$ also aligns with the 1st PC for the SP500, FaceImage and Crime designs. This additional alignment was not introduced deliberately; however, it exists due to the following reason.  Recall the signal decomposition from \Cref{stwerif} given by $\st=\stal+\zetr$. In our setting here, $\stal$ is a linear combination of the 2nd, 4th, 6th PCs, while we generated $\zetr$ such that its entries have non-zero mean ($\bm{1}_p^\top \zetr\neq 0$). Coincidentally, the top PC of the SP500, FaceImage and Crime designs aligns with \(\bm{1}_p\) due to correlation among the covariates. Thus, the intended alignment-complement decomposition of the signal is  mis-specified, and the model is able to correctly identify the alignment of $\zetr$ with $\bm{1}_p$.}. We conducted similar experiments for the simulated designs in the top row of \Cref{figPCRA} in \Cref{dnf} from Appendix. 

\begin{figure}[t]
\centering

\begin{minipage}[b]{0.21\columnwidth}\centering
\includegraphics[width=\linewidth]{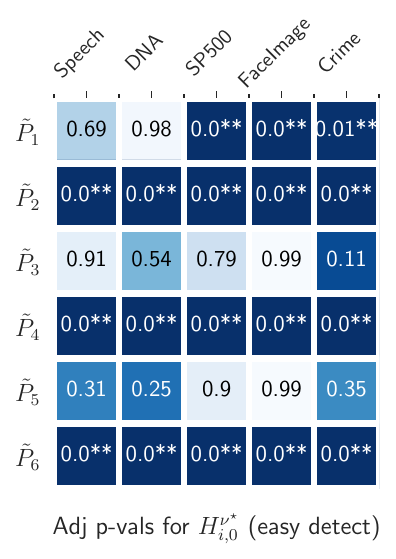}\\[-0.4em]
\textit{(i)}
\end{minipage}\hspace{0.01cm}
\begin{minipage}[b]{0.26\columnwidth}\centering
\includegraphics[width=\linewidth]{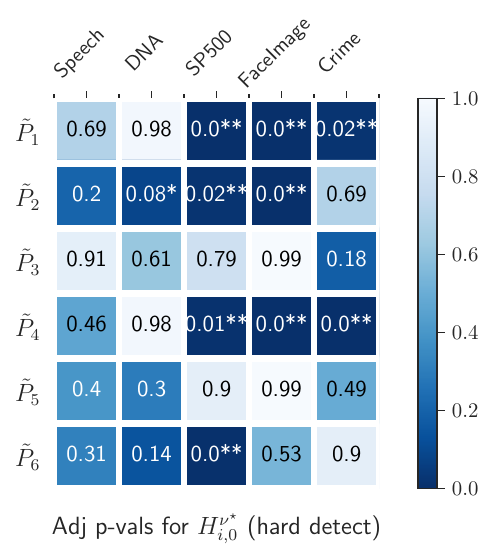}\\[-0.4em]
\textit{(ii)}
\end{minipage}\hspace{0.3cm}
\begin{minipage}[b]{0.21\columnwidth}\centering
\includegraphics[width=\linewidth]{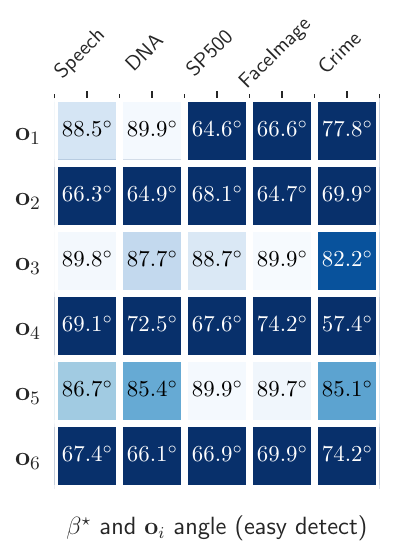}\\[-0.4em]
\textit{(iii)}
\end{minipage}\hspace{0.01cm}
\begin{minipage}[b]{0.26\columnwidth}\centering
\includegraphics[width=\linewidth]{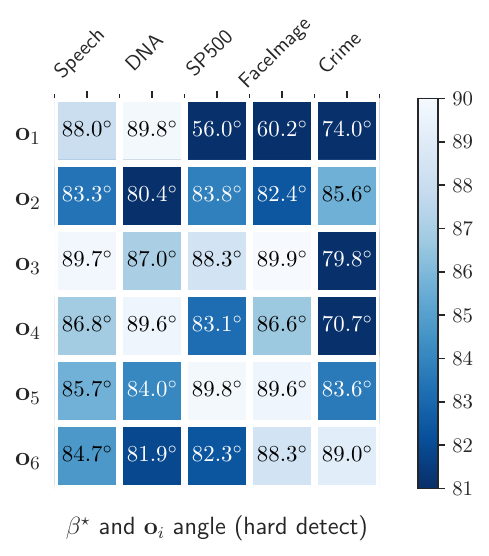}\\[-0.4em]
\textit{(iv)}
\end{minipage}

\vspace{-3mm}
\caption{Panels \textbf{(i)} and \textbf{(iii)} use the \Cref{figPCRA} (bottom-left) setting with $v_i^\star=5\sqrt{p}$ for $i\in\{2,4,6\}$; \textbf{(ii)} and \textbf{(iv)} weaken the alignment to $v_i^\star=\sqrt{p}$. Panels \textbf{(i)}–\textbf{(ii)} report Benjamini–Hochberg–adjusted $p$-values for $H^{\alpha^\star}_{i,0}$ via Corollary~\ref{testCor} (**: FDR 0.05, *: 0.1). Panels \textbf{(iii)}–\textbf{(iv)} show the corresponding true alignment angles $\angle(\mathbf{o}_i,\st)$.}
\label{figtest}
\end{figure}

% \begin{figure}[t]
%     \centering
%     \includegraphics[width=\linewidth]{Image_pnas/panel_pcr_sim_FCP.pdf}
%     \vspace{-6mm}
%     \caption{Under the setting of \Cref{figPCRA}, we plot the false coverage proportion (FCP) of the confidence intervals for $(\sti)_{i=1}^p$ defined in \eqref{defCIPCR}, as we vary the targeted FCP level $\alpha$ from $0$ to $1$. The $x$-axis spans $\alpha$ values from 0 to 1, while the $y$-axis ranges from 0 to 1. The dotted black line is the 45-degree line.}
%     \label{figstCIA}
% \end{figure}

\section{Discussion}\label{sec:conc}
We conclude our paper with a discussion of two main points. 
First, we clarify that although our method can handle various dependencies in the design through the right-rotational invariance assumption, it does not extend to anisotropic Gaussian designs where the rows of $\X$ are sampled from $N(\mathbf{0},\mathbf{\Sigma})$ with an arbitrary covariance matrix $\mathbf{\Sigma}$ (unless $\mathbf{\Sigma}$ is right-rotationally invariant). Moreover, contrasting with \cite{bellec2019biasing}, our Spectrum-Aware adjustment \eqref{eq:adjspectrum} does not apply directly to non-separable penalties, e.g. SLOPE, group Lasso, etc. We note that the current framework can be expanded to address both these issues. In \Cref{section:CONJE} from Appendix, we suggest a debiased estimator for ``ellipsoidal designs'' $\X=\Qbm^\top \Dbm \Obm \mathbf{\Sigma}^{1/2}$ and non-separable convex penalties. We also conjecture its asymptotic normality using the non-separable VAMP formalism \cite{fletcher2018plug}. We leave a detailed study of this extensive class of estimators to future works. 

We discuss another potential direction of extension, that of relaxing the exchangeability assumption in \Cref{sgoods} and \Cref{CORPCRTHM} from Appendix that  establish inference guarantees on finite-dimensional marginals. One may raise a related question, that of constructing confidence intervals for $\mathbf{a}^{\top}\st$ for a given choice of $\mathbf{a}$. 
Under Gaussian design assumptions, such guarantees were obtained using the leave-one-out method as in \cite[Section 4.6]{celentano2020lasso} or Stein's method as in \cite{bellec2019biasing} without requiring the exchangeability assumption (at the cost of  other assumptions on $\st$ and/or $\boldsymbol{\Sigma}$). Unfortunately, these arguments no longer apply under right-rotational invariant designs owing to the presence of a global dependence structure. 
Thus, establishing such guarantees without exchangeability can serve as an exciting direction for future research. 

\begin{acks}[Acknowledgments]
P.S. was funded partially by NSF DMS-2113426. The authors would like to thank Florent Krzakala and Cedric Gerbelot for clarification on the contributions in \cite{gerbelot2020asymptotic,gerbelot2022asymptotic}, and Boris Hanin for references on linear neural networks. 
\end{acks}

\bibliographystyle{imsart-number} 
\bibliography{ref}

\newpage
\thispagestyle{empty}
\vspace{2em}
\appendix
\begin{center}
\large
\textbf{SUPPLEMENTARY MATERIAL: ``SPECTRUM\textendash{}AWARE DEBIASING: A MODERN INFERENCE FRAMEWORK\\
WITH APPLICATIONS TO PRINCIPAL COMPONENTS REGRESSION''} \\[1.2em]

    % \large \textbf{Supplementary Material for Spectrum-aware Debiasing: A Modern Inference Framework with Applications to Principal Components Regression}\\[1.2em]

\large\scshape
By Yufan Li\textsuperscript{1,a}\quad and\quad Pragya Sur\textsuperscript{1,b}\\[1em]

\itshape \small
\textsuperscript{1}Department of Statistics, Harvard University\quad
\textsuperscript{a}\href{mailto:yufan_li@g.harvard.edu}{yufan\_li@g.harvard.edu};
\textsuperscript{b}\href{mailto:pragya@fas.harvard.edu}{pragya@fas.harvard.edu}
\end{center}
\vspace{1em}
\section{Preliminary}
\subsection{Empirical Wasserstein-2 convergence}\label{section:wass}
We will use below the following fact. See \cite[Appendix E]{fan2022approximate} and references within for its justification. 
\begin{Proposition}\label{wasfact}
To verify $(\vbf_1,\ldots, \vbf_k) \stackrel{W_2}{\rightarrow} (\Vs_1,\ldots, \Vs_k)$, it suffices to check that $$\lim_{n \to \infty}
\frac{1}{n} \sum_{i=1}^{n} f\left(v_{i, 1}, \ldots, v_{i,
k}\right)=\mathbb{E}\left[f\left(\mathsf{V}_{1}, \ldots,
\mathsf{V}_{k}\right)\right]$$ holds for every function $f: \mathbb{R}^k \rightarrow \mathbb{R}$ satisfying, for some constant $C>0$, the pseudo-Lipschitz condition
$
\left|f(\vbf)-f\left(\vbf^{\prime}\right)\right| \leq C\left(1+\|\vbf\|_2+\left\|\vbf^{\prime}\right\|_2 \right)\left\|\vbf-\vbf^{\prime}\right\|_2.
$  Meanwhile, this condition implies \eqref{eq:wasslip}. 
\end{Proposition}

The following results are from \cite[Appendix E]{fan2022approximate}. 
\begin{Proposition}\label{prop:iidW}
Suppose $\Vbf \in \R^{n \times t}$ has i.i.d.\ rows equal in law to $\Vs \in \R^t$,
which has finite mixed moments of all orders. Then $\Vbf \toW \Vs$ almost surely as $n
\to \infty$. Furthermore, if $\mathbf{E} \in \R^{n \times k}$ is deterministic with $\mathbf{E} \toW \Es$, then $(\Vbf,\mathbf{E}) \toW (\Vs,\Es)$ almost surely
where $\Vs$ is independent of $\Es$.
\end{Proposition}
\begin{Proposition}\label{prop:contW}
Suppose $\Vbf \in \R^{n \times k}$ satisfies $\Vbf \toW \Vs$ as $n \to \infty$,
and $g:\R^k \to \R^l$ is continuous with $\|g(\vbf)\| \leq C(1+\|\vbf\|)^\pfrak$ for some $C>0$ and
$\pfrak \geq 1$. Then $g(\Vbf) \toW g(\Vs)$ where $g(\cdot)$ is applied
row-wise to $\Vbf$. 
\end{Proposition}
\begin{Proposition}\label{prop:combW}
Suppose $\Vbf \in \R^{n \times k}$, $\mathbf{W} \in \R^{n \times l}$, and $\mathbf{M}_n,\mathbf{M} \in \R^{k
\times l}$ satisfy $\Vbf \toW \Vs$, $\mathbf{W}  \toW 0$, and $\mathbf{M}_n \to \mathbf{M}$ entrywise as $n \to
\infty$. Then $\Vbf \mathbf{M}_n+\mathbf{W}  \toW \Vs^\top \cdot \mathbf{M}$. 
\end{Proposition}

\begin{Proposition}\label{prop:asW}
Fix $\pfrak \geq 1$ and $k \geq 0$. Suppose $\Vbf \in \mathbb{R}^{n \times k}$ satisfies ${\Vbf} \stackrel{W_2}{\rightarrow} \Vs$, and $f: \mathbb{R}^k \rightarrow \mathbb{R}$ is a function satisfying \eqref{eq:wasslip} that is continuous everywhere except on a set having probability 0 under the law of $\Vs$. Then
$
\frac{1}{n} \sum_{i=1}^n f({\Vbf})_i \rightarrow \mathbb{E}[f(\Vs)].
$
\end{Proposition}

\begin{Proposition}\label{prop:orthoW}
Fix $l \geq 0$, let $\Obm \sim \Haar(\O(n-l))$, and let $\vbf \in \R^{n-l}$ 
and $\bm{\Pi} \in \R^{n \times (n-l)}$ be deterministic, where $\bm{\Pi}$ has orthonormal
columns and $n^{-1}\| \vbf \|^2 \to \sigma^2$ as $n \to \infty$. Then
$\bm{\Pi} \Obm \vbf \toW \Zs \sim N(0,\sigma^2)$ almost surely.
Furthermore, if $ \mathbf{E} \in \R^{n \times k}$ is deterministic with $ \mathbf{E} \toW \Es$,
then $(\bm{\Pi} \Obm \vbf, \mathbf{E} ) \toW (\Zs,\Es)$ almost surely
where $\Zs$ is independent of $\Es$.
\end{Proposition}

\subsection{Proximal map}\label{appendix:prox}
We collect a few useful properties of proximal map. 
\begin{Proposition}\label{prop:proxp}
    Under \Cref{Assumph}, we have that for any $v>0$,
\begin{itemize}
\item[(a)] for any $x,y \in \R$, $y=\operatorname{Prox}_{vh}(x) \iff x-y \in v\partial h(y)$ where $\partial h$ is the subdifferential of $h$;
\item[(b)] Proximal map is firmly non-expansive: for any $x,y \in \R$, $\qty|\operatorname{Prox}_{vh}(x)-\operatorname{Prox}_{vh}(y)|^2\le (x-y)(\operatorname{Prox}_{vh}(x)-\operatorname{Prox}_{vh}(y))$. This implies that $x\mapsto \operatorname{Prox}_{vh}(x)$ is 1-Lipschitz continuous.
\item[(c)] We have the following inequality \begin{equation*}
\left|\operatorname{Prox}_{v_1 h}(x)-\operatorname{Prox}_{v_2 h}(x)\right| \leq \frac{\left|v_1-v_2\right|}{\left(1+v_1 c_0\right) v_2}\left|x-\operatorname{Prox}_{v_2 h}(x)\right|.
\end{equation*}
\item[(d)] For any $v>0$, it cannot be true that $\operatorname{Prox}_{vh}^{\prime}(x) =1 $ almost everywhere.
\end{itemize}
\end{Proposition}

\begin{proof}[Proof of \Cref{prop:proxp}]
    We skip proofs of (a) and (b) which are well-known properties of proximal operator. We now prove (c). Let \(y_1=\operatorname{Prox}_{v_1h}(x)\) and \(y_2=\operatorname{Prox}_{v_2h}(x)\). By the definition of the proximal mapping, there exist subgradients \(u_1\in\partial h(y_1)\) and \(u_2\in\partial h(y_2)\) such that 
\[
x-y_1=v_1\,u_1 \quad \text{and} \quad x-y_2=v_2\,u_2.
\]
Subtracting these two relations yields
\[
y_2-y_1=v_1(u_1-u_2)+(v_1-v_2)u_2.
\]
Since we are in the scalar setting, we can multiply both sides by \(y_1-y_2\) (note that \((y_2-y_1)(y_1-y_2)=-|y_1-y_2|^2\)) to obtain
\[
-|y_1-y_2|^2=v_1(u_1-u_2)(y_1-y_2)+(v_1-v_2)u_2(y_1-y_2).
\]
The convexity assumption on \(h\) with parameter \(c_0\le 0\) implies that for any \(u_1\in\partial h(y_1)\) and \(u_2\in\partial h(y_2)\) we have
\[
(u_1-u_2)(y_1-y_2)\ge c_0\,(y_1-y_2)^2.
\]
Substituting this inequality into the previous display yields
\[
-|y_1-y_2|^2\ge v_1\,c_0\,(y_1-y_2)^2+(v_1-v_2)u_2(y_1-y_2).
\]
Rearranging terms and taking absolute values, we deduce
\[
(1+v_1c_0)|y_1-y_2|^2\le |v_1-v_2|\;|u_2|\;|y_1-y_2|.
\]
Assuming \(y_1\neq y_2\) so that we can cancel a factor of \(|y_1-y_2|\), it follows that
\[
|y_1-y_2|\le \frac{|v_1-v_2|}{1+v_1c_0}\,|u_2|.
\]
Finally, recalling that the optimality condition for \(y_2\) gives \(x-y_2=v_2\,u_2\), we have \(|u_2|=\frac{|x-y_2|}{v_2}\). Substituting this expression into the inequality above yields
\[
|y_1-y_2|\le \frac{|v_1-v_2|}{(1+v_1c_0)v_2}\,|x-y_2|,
\]
which is the desired result. 

To see (d), note that since \Cref{Assumph} requires $h(x)$ to be non-constant, it suffices to show the following: If $\operatorname{Prox}^{\prime}_{vh}(x)=1$ almost everywhere $v>0$, then $h(x)$ is constant. Define 
\[
T(x) = \operatorname{Prox}_{v h}(x).
\]
Since \(T'(x)=1\) almost everywhere, integrating over an interval shows that
\[
T(b)-T(a)=b-a\quad \text{for all } a,b\in\mathbb{R}.
\]
Thus, \(T\) is an affine function of the form
\[
T(x)=x+c,
\]
for some constant \(c\in\mathbb{R}\).

Now, by the definition of the proximal operator, for every \(x\in\mathbb{R}\) the optimality condition (in terms of subgradients) implies that
\[
0\in \partial h\bigl(T(x)\bigr) + \frac{1}{v}\bigl(T(x)-x\bigr).
\]
Substituting \(T(x)=x+c\) gives
\[
0\in \partial h(x+c) + \frac{c}{v},\quad \text{for all } x\in\mathbb{R}.
\]
Letting \(y=x+c\), we deduce that for every \(y\in\mathbb{R}\)
\[
-\frac{c}{v} \in \partial h(y).
\]

By our assumption, $h$ is twice continuously differentiable except for a finite set of points. Integrating, we obtain
\[
h(y) = -\frac{c}{v}y + b,
\]
for some constant \(b\in\mathbb{R}\). However, \(h\) is assumed to be nonnegative. Hence, we must have \(c=0\).

\end{proof}

\begin{proof}[Proof of \Cref{Extend}]
Under \Cref{Assumph}, for any $v>0$, $x\mapsto \operatorname{Prox}_{v h}(x)$ is continuous, monotone increasing in $x$, and continuously differentiable at any $x$ such that $\operatorname{Prox}_{v h}(x) \notin \Dc$ and
\begin{equation}
    \operatorname{Prox}_{v h}^{\prime}(x)=\frac{1}{1+v h^{\prime \prime}\left(\operatorname{Prox}_{v h}(x)\right)}.
\end{equation}
This follows from the assumption that $h(x)$ is twice continuously differentiable on $\Dc^c$ and the implicit differentiation calculation shown in \cite[Appendix B1]{gerbelot2020asymptotic}. For $x\in \{x: \operatorname{Prox}_{v h}(x) \in \Dc\}$, $\operatorname{Prox}_{vh}(x)$ is differentiable and has derivative equal to 0 except for a finite set of points. To see this, note that preimage $\operatorname{Prox}^{-1}_{v h}(\y)$ for $y\in \Dc$ is either a singleton set or a closed interval of the form $[x_1, x_2]$ for $x_1\in \R\cup\{-\infty\},x_2\in \R\cup\{+\infty\}$ and $x_1<x_2$, using continuity and monotonicity of $x\mapsto \operatorname{Prox}_{vh}(x)$. This implies that $ \{x: \operatorname{Prox}_{v h}(x) \in \Dc\}$ is a union of finite number of singleton sets and a finite number of closed intervals. Furthermore, $ \operatorname{Prox}_{v h}(x)$ is constant on each of the closed intervals. It follows that $ \operatorname{Prox}_{v h}(x)$ is differentiable and has derivative equal to $0$ on the interiors of the closed intervals, and that $\mathcal{C}$ is union of some of the singleton sets and all of the finite-valued endpoints of the closed intervals. 

We extend functions $h^{\prime \prime}(x)$ and $\operatorname{Prox}_{v h}^{\prime}(x)$ on $\Dc$ and $\mathcal{C}$ respectively in the following way: (i) For $y_0 \in \Dc$ such that $\operatorname{Prox}^{-1}_{v h}(y_0)$ is a closed interval with endpoints  $x_1\in \R\cup\{-\infty\},x_2\in \R\cup\{+\infty\}$ and $x_1<x_2$, we set $h^{\prime \prime }(y_0) \gets +\infty$ and $\operatorname{Prox}^{\prime}_{vh}(x)\gets 0$ for all $x\in [x_1, x_2]$ (ii) For $y_0\in \Dc$ such that $\operatorname{Prox}^{-1}_{v h}(y_0)$ is a singleton set and its sole element $x_0$ is contained in $\mathcal{C}$, we set $h^{\prime\prime}(y_0)\gets +\infty, \operatorname{Prox}_{vh}^{\prime}(x_0)\gets 0$; (iii) For $y_0\in \Dc$ such that $\operatorname{Prox}^{-1}_{v h}(y_0)$ is a singleton set $\{x_0\}$ and that $x \mapsto \operatorname{Prox}_{v h}(x)$ is differentiable at $x_0$ with 0 derivative, we set $h(y_0)\gets +\infty$. 

We show that it is impossible to have some $y_0 \in \Dc$ such that $\operatorname{Prox}_{v h}^{-1}\left(y_0\right)$ is a singleton set $\left\{x_0\right\}$ and that $x \mapsto \operatorname{Prox}_{v h}(x)$ is differentiable at $x_0$ with non-zero derivative. This means that all $\y \in \Dc$ belongs to cases (i), (ii) and (iii) above. Suppose to the contrary. We know from the above discussion that there exists some $\mathfrak{e}>0$ such that $\operatorname{Prox}_{v h}^{\prime}(x)$ is continuous on $\left(x_0, x_0+\mathfrak{e}\right)$ and $\left(x_0-\mathfrak{e}, x_0\right)$. We claim that $x \mapsto \operatorname{Prox}_{v h}^{\prime}(x)$ is continuous at $x_0$. To see this, note that for any $\Delta>0$, we can find $\epbm \in(0, \mathfrak{e})$ such that
\begin{itemize}
    \item there exists some $x_{+} \in\left(x_0, x_0+\epsilon \right)$ such that for any $x \in\left(x_0, x_0+\epsilon \right)$, 
    
    $$\left | \operatorname{Prox}_{v h}^{\prime}(x)-\operatorname{Prox}_{v h}^{\prime}\left(x_{+}\right)\right|<\frac{\Delta}{5}, \quad \left| \frac{\operatorname{Prox}_{v h}\left(x_0\right)-\operatorname{Prox}_{v h}\left(x_{+}\right)}{x_0-x_{+}}-\operatorname{Prox}_{v h}^{\prime}\left(x_{+}\right) \right |<\frac{\Delta}{5}$$
    
    \item there exists some $x_{-} \in\left(x_0-\epsilon, x_0\right)$ such that for any $x \in\left(x_0-\epsilon, x_0\right)$, 
    
    $$\left | \operatorname{Prox}_{v h}^{\prime}(x)-\operatorname{Prox}_{v h}^{\prime}\left(x_{-}\right)\right|<\frac{\Delta}{5},\quad \left| \frac{\operatorname{Prox}_{v h}\left(x_0\right)-\operatorname{Prox}_{v h}\left(x_{-}\right)}{x_0-x_{-}}-\operatorname{Prox}_{v h}^{\prime}\left(x_{-}\right) \right|<\frac{\Delta}{5}$$
    
    \item for any $x \in\left(x_0-\epsilon, x_0\right) \cup\left(x_0, x_0+\epsilon \right)$,$$\left|\operatorname{Prox}_{v h}^{\prime}\left(x_0\right)-\frac{\operatorname{Prox}_{v h}\left(x_0\right)-\operatorname{Prox}_{v h}(x)}{x_0-x}\right|<\frac{\Delta}{5}.$$
\end{itemize}
Then for any $x \in\left(x_0-\epsilon, x_0+\epsilon\right)$, we have $\left|\operatorname{Prox}_{v h}^{\prime}\left(x_0\right)-\operatorname{Prox}_{v h}^{\prime}(x)\right|<\Delta$ by triangle inequality. This proves the claim. Now, since $x \mapsto \operatorname{Prox}_{v h}(x)$ is continuously differentiable on $\left(x_0-\mathfrak{e}, x_0+\mathfrak{e}\right)$ and $\operatorname{Prox}_{v h}^{\prime}\left(x_0\right) \neq 0$, inverse function theorem implies that $y \mapsto \operatorname{Prox}_{v h}^{-1}(y)$ is a well defined, real-valued function and it is continuous differentiable on some open interval $U$ containing $y_0$. This implies that $h$ is differentiable at any $y \in U$ and that $y\mapsto \operatorname{Prox}_{v h}^{-1}(y)=y+v h^{\prime}(y)$ is continuously differentiable. But this would imply that $h$ is twice continuously differentiable on $U$ which contradicts the assumption that $y_0 \in \Dc$.

Note that we have assigned $+\infty$ to $h^{\prime\prime}$ on $\Dc$ and $0$ to $\operatorname{Prox}^{\prime}_{v h}$ on $\mathcal{C}$. Piecewise continuity of $x\mapsto \frac{1}{w+h^{\prime \prime}\left(\operatorname{Prox}_{v h}(x)\right)}$ for any $w>0$ follows from the discussion above. 
\end{proof}

\subsection{Properties of R- and Cauchy transform}
The following shows that the Cauchy- and R-transforms of $-\D^2$ are
well-defined by (\ref{eq:CauchyR}), and reviews their properties.
\begin{Lemma}\label{lem:cauchy}
Let $G(\cdot)$ and $R(\cdot)$ be the Cauchy- and R-transforms of $-\D^2$
under Assumption \ref{AssumpD}.
\begin{itemize}
\item[(a)] The function $G: (-d_{-}, \infty) \to \R$ is positive and
strictly decreasing. Setting $G\left(-d_-\right) := \lim _{z \to -d_-} G(z) \in (0,\infty]$, $G$ admits a functional inverse $G^{-1}:(0,G(-d_-)) \to (-d_-,\infty)$.
\item[(b)] The function $R:\left(0,G\left(-d_-\right)\right) \to
\mathbb{R}$ is negative and strictly increasing.
\item[(c)] For any $z \in\left(0, G\left(-d_{-}\right)\right), R^{\prime}(z)=-\left(\mathbb{E} \frac{1}{\left(\D^2+R(z)+\frac{1}{z}\right)^2}\right)^{-1}+\frac{1}{z^2}$.
\item[(d)] For any $z \in\left(0, G\left(-d_{-}\right)\right),-\frac{z R^{\prime}(z)}{R(z)}\in (0,1)$.
\item[(e)] For any $z \in\left(0, G\left(-d_{-}\right)\right),z^2R'(z)\in (0,1)$.
\item[(f)] For all sufficiently small $z\in (0,G(-d_-))$, R-transform admits convergent series expansion given by 
\begin{equation}\label{expandR}
R(z)=\sum_{k \geq 1} \kappa_k z^{k-1}
\end{equation}
where $\left\{\kappa_k\right\}_{k \geq 1}$ are the free cumulants of the law of $-\D^2$ and $\kappa_1=-\E \D^2$ and $\kappa_2=\V(\D^2)$. 
\end{itemize}
\end{Lemma}

\begin{proof}
See \cite[Lemma G.6]{li2023random} for (a) and (b), To see (c), for any $z \in\left(0, G\left(-d_{-}\right)\right)$, differentiating $R(z)=G^{-1}(z)-z^{-1}$ yields
$$
-z R^{\prime}(z)=z\left(\mathbb{E} \frac{1}{\left(\D^2+G^{-1}(z)\right)^2}\right)^{-1}-\frac{1}{z}
$$
To see (d), 
$$
\begin{aligned}
& -\frac{z R^{\prime}(z)}{R(z)}=\frac{z\left(\mathbb{E} \frac{1}{\left(\D^2+G^{-1}(z)\right)^2}\right)^{-1}-\frac{1}{z}}{G^{-1}(z)-\frac{1}{z}}<1 \\
& \Leftrightarrow z\left(\mathbb{E} \frac{1}{\left(\D^2+G^{-1}(z)\right)^2}\right)^{-1}>G^{-1}(z) \\
& \Leftrightarrow \mathbb{E} \frac{G^{-1}(z)}{\left(\D^2+G^{-1}(z)\right)^2}<z=\mathbb{E} \frac{1}{\D^2+G^{-1}(z)} \\
& \Leftrightarrow \mathbb{E} \frac{-\D^2}{\left(\D^2+G^{-1}(z)\right)^2}<0
\end{aligned}
$$
where we used in the second line that $R(z)=G^{-1}(z)-1/z<0$ from (b). Note that the last line is true since $\D^2 \neq 0$ with positive probability. (e) trivially follows from (c). (f) follows from \cite[Notation 12.6, Proposition 13.15]{nica2006lectures}.
\end{proof}

\subsection{DF adjustment coincide with Spectrum-Aware adjustment under Marchenko-Pastur law}\label{appendix:adjadgcc}

\begin{Lemma}\label{adjadgconv}
    If the empirical distribution of the eigenvalues of $\X^\top \X$ weakly converges Marchenko-Pastur law, then $\abs{\adj-\adg}\to 0.$
\end{Lemma}

\begin{proof}[Proof of \Cref{adjadgconv}]
By weak convergence,
$$
\frac{1}{p} \sum_{i=1}^p \frac{-1}{d_i^2+\lambda_2} \rightarrow G\left(-\lambda_2\right)
$$
where $z \mapsto G(z)$ is the Cauchy transform of Marchenko-Pastur law\footnote{Here, $G(z):=$ $\int \frac{1}{z-x} \mu(d x)$ where $\mu(\cdot)$ is measure associated to Marchenko-Pasteur law. }. Then we have that
\begin{equation}\label{qudraticeq}
\adj \rightarrow \lambda_2\left(\frac{1}{1+\lambda_2 G\left(-\lambda_2\right)}-1\right)^{-1}, \quad \adg \rightarrow 1-\delta^{-1}\left(1+\lambda_2 G\left(-\lambda_2\right)\right)
\end{equation}
Observe that the limiting values of $\adj$ and $\adg$ above are equal if and only if the following holds
\begin{equation}\label{webao}
1+\left(\lambda_2+1-\delta^{-1}\right) G\left(-\lambda_2\right)-\delta^{-1} \lambda_2 \qty(G\left(-\lambda_2\right))^2=0.
\end{equation}
Here, \eqref{webao} indeed holds true since $G\left(-\lambda_2\right)$ is one of the root of the quadratic equation \eqref{webao}. This is by referencing the explicit expression of the Cauchy transform of the Marchenko-Pastur law (cf. \cite[Lemma 3.11]{bai2010spectral}). 
\end{proof}

\subsection{VAMP algorithm}\label{vamp}
For $\sigma^2=1$, the VAMP algorithm consists of iteration as follows: for $t \geq 1$,
$$
\begin{aligned}
& \xonet=\operatorname{Prox}_{\gamma_{1, t-1}^{-1}}\left(\ronetm\right), \quad \eta_{1 t}^{-1}=\gamma_{1, t-1}^{-1} \nabla \cdot \operatorname{Prox}_{\gamma_{1, t-1}^{-1} h}\left(\ronetm\right) \\
& \gamma_{2 t}=\eta_{1 t}-\gamma_{1, t-1}, \quad \rtwot=\left(\eta_{1 t} \xonet-\gamma_{1, t-1} \ronetm\right) / \gamma_{2 t} \\
& \xtwot=\left(\X^{\top} \X+\gamma_{2 t} \mathbf{I}_p\right)^{-1}\left(\X^{\top} \y+\gamma_{2 t} \rtwot\right), \quad \eta_{2 t}^{-1}=\frac{1}{p} \operatorname{Tr}\left[\left(\X^{\top} \X+\gamma_{2 t} \mathbf{I}_p \right)^{-1}\right] \\
& \gamma_{1 t}=\eta_{2 t}-\gamma_{2 t}, \quad \ronet=\left(\eta_{2 t} \xtwot-\gamma_{2 t} \rtwot\right) / \gamma_{1 t}
\end{aligned}
$$
The algorithm can be initialized at $r_{10} \in \mathbb{R}^p, \gamma_{10}, \tau_{10} >0$ such that $\left(r_{10}, \st\right) \stackrel{W_2}{\to}\left(\mathsf{R}_{10}, \Xstar \right)$ and $\mathsf{R}_{10}-\Xstar \sim N\left(0, \tau_{10}\right)$. This algorithm is first introduced in \cite{rangan2019vector} and the iterates $\xonet, \xtwot$ are supposed to track $\hatbt$. The performance of this algorithm is characterized by state evolution iterations: for $t\ge 1$,
\begin{equation}
\begin{aligned}
& \bar{\alpha}_{1 t}=\mathbb{E} \operatorname{Prox}_{\gamma_{1,1-1}^{\prime}}^{\prime}\left(\Xstar+N\left(0, \tau_{1, t-1} \right)\right), \quad \bar{\eta}_{1 t}^{-1}=\bar{\gamma}_{1, t-1}^{-1} \bar{\alpha}_{1 t} \\
& \bar{\gamma}_{2 t}=\bar{\eta}_{1 t}-\bar{\gamma}_{1, t-1}, \quad \tau_{2 t}=\frac{1}{\left(1-\bar{\alpha}_{1 t}\right)^2}\left[\mathcal{E}_1\left(\bar{\gamma}_{1, t-1}, \tau_{1, t-1}\right)-\bar{\alpha}_{1 t}^2 \tau_{1, t-1}\right] \\
& \bar{\alpha}_{2 t}=\bar{\gamma}_{2 t} \mathbb{E} \frac{1}{\D^2+\bar{\gamma}_{2 t}}, \quad \bar{\eta}_{2 t}^{-1}=\bar{\gamma}_{2 t}^{-1} \bar{\alpha}_{2 t} \\
& \bar{\gamma}_{1, t}=\bar{\eta}_{2 t}-\bar{\gamma}_{2 t}, \quad \tau_{t}=\frac{1}{\left(1-\bar{\alpha}_{2 t}\right)^2}\left[\mathcal{E}_2\left(\bar{\gamma}_{2 t}, \tau_{2 t}\right)-\bar{\alpha}_{2 t}^2 \tau_{2 t}\right]
\end{aligned}
\end{equation}
where 
$$
\begin{aligned}
& \mathcal{E}_1\left(\gamma_1, \tau\right):=\mathbb{E}\left(\operatorname{Prox}_{\gamma_1^{-1} h}\left(\Xstar+N\left(0, \tau\right)\right)-\Xstar \right)^2, \quad \mathcal{E}_2\left(\gamma_2, \tau_2\right):=\mathbb{E}\left[\frac{\D^2+\tau_2 \gamma_2^2}{\left(\D^2+\gamma_2\right)^2}\right].
\end{aligned}
$$

\section{Fixed point equation}\label{appendix:misce}
\subsection{An auxiliary lemma}
\begin{Lemma}\label{noid}
    Under \Cref{Assumph} and \ref{Assumpfix}, 
    \begin{equation}
        \begin{aligned}
            &\PP \qty(\operatorname{Prox}_{\gamma_{*}^{-1} h}^{\prime}\left(\sqrt{\taustar}\Zs+\Xstar\right)\neq 0)>0, \quad \PP \qty(\operatorname{Prox}_{\gamma_{*}^{-1} h}^{\prime}\left(\sqrt{\taustar}\Zs+\Xstar\right)\neq 1)>0\\
            & \PP \qty(h^{\prime \prime}\left(\operatorname{Prox}_{\gamma_*^{-1} h}\left(\sqrt{\taustar} \Zs+\Xstar\right)\right)\neq +\infty )>0, \quad \PP \qty(h^{\prime \prime}\left(\operatorname{Prox}_{\gamma_*^{-1} h}\left(\sqrt{\taustar} \Zs+\Xstar\right)\right)\neq 0 )>0
        \end{aligned}
    \end{equation}
    where $\Zs\sim N(0,1)$ is independent of $\Xstar$. 
\end{Lemma}
\begin{proof}[Proof of \Cref{noid}]
    Note that $\operatorname{Prox}_{\gamma_{*}^{-1} h}^{\prime}\left(\sqrt{\taustar} \Zs+\Xstar\right)\neq 0$ with positive probability or else $\frac{\gamma_*}{\eta_*}=\mathbb{E} \operatorname{Prox}_{\gamma_{*}^{-1} h}^{\prime}\left(\sqrt{\taustar} \Zs+\Xstar\right)=0$ which violates \Cref{Assumpfix}. Meanwhile, $\operatorname{Prox}_{\gamma_{*}^{-1} h}^{\prime}\left(\sqrt{\taustar} \Zs+\Xstar\right)\neq 1$ with positive probability or else $\operatorname{Prox}_{\gamma_{*}^{-1} h}^{\prime}\left(x\right) =1$ almost everywhere, violating \Cref{prop:proxp}, (d).  
    The inequalities in the second line follows immediately from \eqref{eq:Jacprox} and the first line we have just shown. 
\end{proof}

\subsection{Uniqueness of fixed points given existence}\label{justifuniqe}
Suppose that \Cref{AssumpD}---\ref{Assumpfix} hold. Our proof of \Cref{thm:empmain} and \Cref{neig} does not require $(\gamma_*, \eta_*, \tau_*, \tau_{**})$ to be a unique solution of \eqref{fp}, only that it is one of the solutions. However, if there are two different solutions of \eqref{fp}, it would lead to a contradiction in \Cref{neig}. More concretely, suppose that there exists two different solutions of \eqref{fp}: $x^{(1)}:=\qty (\gamma_*^{(1)}, \eta_*^{(1)}, \tau_*^{(1)}, \tau_{**}^{(1)})$ and $x^{(2)}:=\qty (\gamma_*^{(2)}, \eta_*^{(2)}, \tau_*^{(2)}, \tau_{**}^{(2)})$. By \Cref{neig}, we would have $\qty(\adj, \hat{\eta}_*, \tauh, \hat{\tau}_{**})$ converges almost surely to both $x^{(1)}$ and $x^{(2)}$, hence the contradiction. 

{
\subsection{Existence of fixed points} \label{existENex}

\begin{Proposition}\label{smdff}
    Proximal operators of Elastic Net, Lasso, Ridge penalty
    \begin{equation*}
    h(x)=\lambda_1|x|+\frac{\lambda_2}{2} x^2, \lambda_1\ge 0, \lambda_2\ge 0.
\end{equation*}
and Huber-norm penalty (cf. \cite{zadorozhnyi2016huber})
$$h(x)=\lambda_1g(x)+\frac{\lambda_2}{2} x^2, \lambda_1\ge 0, \lambda_2\ge 0$$
where
    \[
g(x)=
\begin{cases}
\dfrac{u}{2}\,x^{2}, & |x|\le\delta,\\[6pt]
u\delta\bigl(|x|-\tfrac{\delta}{2}\bigr), & |x|>\delta,
\end{cases}
\qquad u,\delta>0.
\]
satisfy the \Cref{Assumpproxclass}.
\end{Proposition}

\begin{proof}[Proof of \Cref{smdff}]
Note that Lasso, Ridge and Elastic Net, we have that
\begin{equation*}
\begin{aligned}
&\E \operatorname{Prox}_{vh}^{\prime}\left(b+\frac{v}{\alpha} \mathsf{Z}\right)
=\frac{1}{1+\lambda_2 v}\mathbb{P}\left(\left|\frac{v}{\alpha} \mathsf{Z}+b\right| \geq v \lambda_1\right)
\\
&\qquad =1-\qty(\Phi \left(-\alpha \left(v^{-1}b-\lambda_1\right)\right)-\Phi\left(-\alpha\left(v^{-1}b+\lambda_1\right)\right))
\end{aligned}
\end{equation*}
Then, \eqref{monoroa} follows from the following: we also have that
\begin{equation*}
\begin{aligned}
&\frac{d}{d v^{-1}} \Phi \left(-\alpha \left(v^{-1}b-\lambda_1\right)\right)-\Phi\left(-\alpha\left(v^{-1}b+\lambda_1\right)\right) \\
& =\left(-\alpha b\right) \phi\left(-\alpha\left(v^{-1}b-\lambda_1\right)\right)-\left(-\alpha b\right) \phi\left(-\alpha\left(v^{-1}b+\lambda_1\right)\right) \\
& =\left(\alpha b\right)\left(\phi\left(\alpha\left(\lambda_1+v^{-1}b\right)\right)-\phi\left(\alpha\left(\lambda_1-v^{-1}b\right)\right)\right)\le 0 .
\end{aligned}
\end{equation*}
Meanwhile, we have that
$$\lim_{x\to \pm \infty} \operatorname{Prox}_{vh}^{\prime}\left(x\right)=\frac{1}{1+\lambda_2 v}$$
which satisfies \eqref{asumpline} for any $\lambda_1, \lambda_2\in [0,+\infty)$. 

For the Huber-norm penalty, For simplicity we restrict to the case $\lambda_2=0$; the general case $\lambda_2>0$ is identical up to an overall factor of $(1+\lambda_2\,v)^{-1}$ multiplying the proximal‐derivative. In this case we have that
\[
h(x)=
\begin{cases}
\dfrac{u}{2}\,x^{2}, & |x|\le\delta,\\[6pt]
u\delta\bigl(|x|-\tfrac{\delta}{2}\bigr), & |x|>\delta,
\end{cases}
\qquad u,\delta>0.
\]
We first show that for every \(v>0\) one has
\[
\frac{\partial}{\partial v}\,
\mathbb{E}\!\bigl[
  \operatorname{Prox}^{\prime}_{v h}\!\bigl(b+\tfrac{v}{\alpha}\mathsf{Z}\bigr)
\bigr]\le 0. 
\]
Set \(\mathsf{X}:=b+\tfrac{v}{\alpha}\mathsf{Z}\).  The scalar proximal
operator of \(v h\) is
\[
\operatorname{Prox}_{v h}(x)=
\begin{cases}
\dfrac{x}{1+v\,u}, & |x|\le\delta+v\,u\,\delta,\\[6pt]
x-\operatorname{sgn}(x)\,v\,u\,\delta, & |x|>\delta+v\,u\,\delta,
\end{cases}
\]
whose derivative equals
\[
\operatorname{Prox}^{\prime}_{v h}(x)=
\begin{cases}
\dfrac{1}{1+v\,u}, & |x|\le\delta+v\,u\,\delta,\\[6pt]
1, & |x|>\delta+v\,u\,\delta.
\end{cases}
\]
Consequently
\begin{equation}\label{eq:Eproxprime}
\begin{aligned}
    \mathbb{E}\bigl[\operatorname{Prox}^{\prime}_{v h}(\mathsf{X})\bigr]
&=\mathbb{P}\!\bigl(|\mathsf{X}|>\delta+v\,u\,\delta\bigr)
+\frac{1}{1+v\,u}\,
  \mathbb{P}\!\bigl(|\mathsf{X}|\le\delta+v\,u\,\delta\bigr)\\
&=1-\frac{v\,u}{1+v\,u}\,
  \mathbb{P}\!\bigl(|\mathsf{X}|\le\delta+v\,u\,\delta\bigr).
\end{aligned}
\end{equation}

Introduce the thresholds
\[
R:=\bigl(-b+\delta\bigr)\frac{\alpha}{v}+\alpha\,u\,\delta,
\qquad
L:=-\bigl(b+\delta\bigr)\frac{\alpha}{v}-\alpha\,u\,\delta,
\]
so that \(\{|\mathsf{X}|\le\delta+v\,u\,\delta\}=\{L\le\mathsf{Z}\le R\}\)
and
\(\mathbb{P}(|\mathsf{X}|\le\delta+v\,u\,\delta)=\Phi(R)-\Phi(L)\),
where \(\Phi,\phi\) are the standard normal c.d.f.\ and pdf.  Define
\[
g(v):=\frac{v\,u}{1+v\,u}\,\bigl[\Phi(R)-\Phi(L)\bigr],
\quad\text{so that}\quad
\mathbb{E}\bigl[\operatorname{Prox}^{\prime}_{v h}(\mathsf{X})\bigr]=1-g(v).
\]

A direct calculation shows
\[
g^{\prime}(v)=
\frac{u}{(1+v\,u)^{2}}
\Bigl\{
  \Phi(R)-\Phi(L)
  -(1+v\,u)\Bigl[
      (-b+\delta)\frac{\alpha}{v}\,\phi(R)
      +(b+\delta)\frac{\alpha}{v}\,\phi(L)
  \Bigr]
\Bigr\}.
\]
By integration by parts on \([L,R]\),
\[
\Phi(R)-\Phi(L)
=
\bigl[\phi(t)\,t\bigr]_{L}^{R}
+\int_{L}^{R}t^{2}\phi(t)\,dt,
\]
and substituting yields
\[
\Phi(R)-\Phi(L)
-(1+v\,u)
\Bigl[
    (-b+\delta)\tfrac{\alpha}{v}\phi(R)
    +(b+\delta)\tfrac{\alpha}{v}\phi(L)
\Bigr]
=
\int_{L}^{R}t^{2}\phi(t)\,dt
+\alpha\,u\,b\,[\phi(R)-\phi(L)].
\]
Both terms on the right are non-negative: the integral is strictly
positive, and since \(R+L=-2b\alpha/v\) we have
\(\phi(R)\ge\phi(L)\) exactly when \(b\ge0\) (and the reverse if \(b<0\)),
so \(b\,[\phi(R)-\phi(L)]\ge0\) in either case.  Thus the brace is
positive, giving \(g^{\prime}(v)>0\).  Finally,
\eqref{eq:Eproxprime} implies
\[
\frac{\partial}{\partial v}\,
\mathbb{E}\bigl[\operatorname{Prox}^{\prime}_{v h}(\mathsf{X})\bigr]
=-g^{\prime}(v)\le0,
\]
completing the proof.

Meanwhile, we have that
$$\lim_{x\to \pm \infty} \operatorname{Prox}_{vh}^{\prime}\left(x\right)=1$$
which satisfies \eqref{asumpline} for any $u,\delta$.

\end{proof}

Below, we prove existence of fixed points for strongly convex penalties. The Lasso case is deferred to \Cref{fixexistLasso} in \Cref{sec:Lassoexist}.

\begin{proof}[Proof of \Cref{fixexist} (strongly convex penalty)]

First, eliminate the variable $\tau_{**}$ from \eqref{fp} via \eqref{RCb} and introduce change of variable $\tau_*=\gamma_*^{-2} \alpha_*^{-2}$ for some new variable $\alpha_*>0$. We then obtain a new system of fixed equation
\begin{subequations}\label{fpstren}
\begin{align}
& \gamma_*^{-1}=\frac{1}{-R\left(\eta_*^{-1}\right)} \label{RSa}\\
& \eta_*^{-1}=\gamma_*^{-1} \mathbb{E} \operatorname{Prox}_{\gamma_*^{-1} h}^{\prime}\left(\Xstar+\frac{\gamma_*^{-1}}{\alpha_*} \mathsf{Z}\right) \label{RSb} \\
& 1=\alpha_*^2 R^{\prime}\left(\eta_*^{-1}\right)\mathbb{E}\left(\operatorname{Prox}_{\gamma_*^{-1} h}\left(\Xstar+\frac{\gamma_*^{-1}}{\alpha_*} \mathsf{Z}\right)-\Xstar\right)^2+\sigma^2\frac{\alpha_*^2}{\gamma_*^{-1}}\left[1+\frac{\eta_*^{-1} R^{\prime}\left(\eta_*^{-1}\right)}{R\left(\eta_*^{-1}\right)}\right] \label{RSc}
\end{align}
\end{subequations}
Note that \Cref{Assumpfix} holds if and only if we can find a solution $\gamma_*^{-1}, \eta_*^{-1}, \alpha_*>0$ for the above. 

Denote 
$$\gamma_{+}^{-1}:=\lim _{z \rightarrow G\left(-d_{-}\right)} \frac{1}{-R(z)}.$$
When $G(-d_{-})<+\infty$, we have 
\begin{equation}\label{cc1}
    \gamma_{+}^{-1}=\frac{1}{\frac{1}{G\left(-d_{-}\right)}+d_{-}} \quad \text{ and } \quad \gamma_{+}^{-1} \in\left(\frac{1}{\E \D^2}, G\left(-d_{-}\right)\right]
\end{equation}
using the definition of $R(z)=G^{-1}(z)-1/z$ and the fact that $z \mapsto \frac{1}{-R(z)}$ is strictly increasing on its domain $\left(0, G\left(-d_{-}\right)\right)$ by \Cref{lem:cauchy}, (b). When  $G(-d_{-})=+\infty$ and $d_->0$, we have that 
\begin{equation}\label{cc2}
    \gamma_+^{-1}=1/{d_-}.
\end{equation}
In the two cases above, or equivalently $G(-d_{-})<+\infty$ or $ d_{-}>0$, we have that
\begin{equation}\label{sdnfnfnn}
    \gamma_+^{-1}<+\infty \quad \text{ and }\quad \gamma_+^{-1} \le G(-d_-).
\end{equation}
When $G(-d_{-})=+\infty$ and $ d_{-}=0$, we have $\gamma_{+}^{-1}=+\infty$, again using definition of $R(z)$. We highlight from the above that
\begin{equation}\label{iff}
    \gamma_+^{-1}=+\infty \text{  if and only if  }  G(-d_{-})=+\infty \text{ and } d_{-}=0.
\end{equation}
We also have that $\lim _{z \rightarrow 0} \frac{1}{-R(z)}=\frac{1}{\E \D^2}$ using \Cref{lem:cauchy}, (f). 

Let us define the functions $f_1:\left[\frac{1}{\E \D^2}, \gamma_{+}^{-1}\right) \mapsto\left[0, G\left(-d_{-}\right)\right)$ as the inverse function of $z \mapsto \frac{1}{-R(z)}$, i.e. $f_1(z)=R^{-1}(-\frac{1}{z})$. Note that $f_1$ is well-defined and strictly increasing on its domain. It also satisfies
\begin{equation}\label{thething}
f_1\left(\frac{1}{\mathbb{E} \D^2}\right)=0,\qquad  \lim _{\gamma^{-1} \rightarrow \gamma_{+}^{-1}} f_1\left(\gamma^{-1}\right)=G\left(-d_{-}\right).
\end{equation}
Let us define function $f_2:(0,+\infty) \times(0,+\infty) \mapsto(0,+\infty)$ such that
\begin{equation*}
f_2\left(\gamma^{-1}, \alpha\right)=\gamma^{-1} \mathbb{E} \operatorname{Prox}_{\gamma^{-1} h}^{\prime}\left(\Xstar+\frac{\gamma^{-1}}{\alpha} \mathsf{Z}\right). 
\end{equation*}
Now we study the equation (in terms of $\gamma^{-1}$ )
\begin{equation}\label{importanteq}
f_1\left(\gamma^{-1}\right)=f_2\left(\gamma^{-1}, \alpha\right).
\end{equation}
Observe that this equation amounts to eliminating $\eta_*^{-1}$ and solving for $\gamma_*^{-1}$ in terms $\alpha_*$ from \eqref{RSa} and \eqref{RSb}. We claim that for any fixed $\alpha>0$, there is at least one solution $\gamma^{-1}(\alpha) \in\left(\frac{1}{\E \D^2}, \gamma_{+}^{-1}\right)$\footnote{When $\gamma_{+}^{-1}=+\infty$, this statement is interpreted as: there is at least one solution $\gamma^{-1}(\alpha) \in\left(\frac{1}{\E \D^2}, +\infty\right)$.}. To see the claim, note that
\begin{equation*}
f_2\left(\frac{1}{\E \D^2}, \alpha\right)= \frac{1}{\E \D^2}\E \operatorname{Prox}_{vh}^{\prime}\qty(\Xstar+\frac{1}{\alpha \E \D^2 } \mathsf{Z})\ge 0\stackrel{(*)}{=}f_1\left(\frac{1}{\E \D^2}\right),
\end{equation*}
where $(*)$ follows from \eqref{thething}. Thus, a sufficient condition for $f_1\left(\gamma^{-1}\right)=f_2\left(\gamma^{-1}, \alpha\right)$ to have a solution on $\left[\frac{1}{\mathbb{E} \D^2}, \gamma_{+}^{-1}\right)$ is 
\begin{equation}\label{ribenweak}
    \lim_{\gamma^{-1}\to \gamma_{+}^{-1} } f_2\left(\gamma^{-1}, \alpha\right)<\lim_{\gamma^{-1} \to \gamma_{+}^{-1}}f_1\left(\gamma^{-1}\right).
\end{equation}
We establish a stronger result for later use,
\begin{equation}\label{riben}
    \lim_{\gamma^{-1}\to \gamma_{+}^{-1} } \sup_{\alpha \in (0,+\infty)}f_2\left(\gamma^{-1}, \alpha\right)<\lim_{\gamma^{-1} \to \gamma_{+}^{-1}}f_1\left(\gamma^{-1}\right).
\end{equation}

When $\gamma_{+}^{-1}<+\infty$, the claim follows from combining the following two facts: (i) by \Cref{Extend},
$$\lim_{\gamma^{-1} \to \gamma_{+}^{-1}}\sup_{\alpha \in (0,+\infty)}f_2\left(\gamma^{-1}, \alpha\right) \le \frac{\gamma^{-1}_+}{1+\gamma^{-1}_+ c_0}<\gamma^{-1}_+$$
and (ii) by \eqref{sdnfnfnn}, 
$$
\lim_{\gamma^{-1} \to \gamma_{+}^{-1}}f_1\left(\gamma^{-1}\right)=G(-d_-)\ge \gamma_+^{-1}.
$$
When $\gamma_{+}^{-1}=+\infty$, the claim follows from combining the following two facts: (i) by \Cref{Extend},
$$\lim_{\gamma^{-1} \to \gamma_{+}^{-1}}\sup_{\alpha \in (0,+\infty)}f_2\left(\gamma^{-1}, \alpha\right) \le \frac{1}{c_0}$$
and (ii) by \eqref{thething}, 
$$ \lim_{\gamma^{-1} \to \gamma_{+}^{-1}}f_1\left(\gamma^{-1}\right)=+\infty.$$

Thus, we have shown that for any $\alpha>0$, we can find a solution $\gamma^{-1}(\alpha)$ and $\eta^{-1}(\alpha)=f_1(\gamma^{-1}(\alpha))=f_2(\gamma^{-1}(\alpha),\alpha)$ that solves \eqref{RSa} and \eqref{RSb}. Furthermore, we can show that the solutions $\gamma^{-1}(\alpha)$ and $\eta^{-1}(\alpha)$ are unique and continuous in $\alpha$. To see this, we may write the \eqref{importanteq} as
\begin{equation}\label{smdf}
    \gamma R^{-1}(-\gamma)=\mathbb{E} \operatorname{Prox}_{\gamma^{-1} h}^{\prime}\left(\Xstar+\frac{\gamma^{-1}}{\alpha} \mathsf{Z}\right) 
\end{equation}
The LHS is a strictly decreasing function in $\gamma$: with $y:=R^{-1}(\gamma )$, 
\begin{equation*}
\frac{d}{d \gamma} \gamma R^{-1}(-\gamma)=R^{-1}(-\gamma)-\frac{\gamma}{R^{\prime}\left(R^{-1}(-\gamma)\right)}=y\left(1-\frac{-R(y)}{y R^{\prime}(y)}\right)<0
\end{equation*}
while the RHS is a non-decreasing function in $\gamma$ following \Cref{Assumpproxclass}, (i). The uniqueness and continuity then follows from implicit function theorem. 

The plan is to plug $\gamma^{-1}(\alpha)$ and $\eta^{-1}(\alpha)$ into the RHS of \eqref{RSc} to obtain the function $v:(0,+\infty)\mapsto (0,+\infty)$
$$\begin{gathered}v(\alpha)=\alpha^2 R^{\prime}\left(\eta^{-1}(\alpha)\right)\left[\mathbb{E}\left(\operatorname{Prox}_{\gamma^{-1}(\alpha) h}\left(\Xstar+\frac{\gamma^{-1}(\alpha)}{\alpha} \mathsf{Z}\right)-\Xstar\right)^2\right] \\ +\sigma^2 \alpha^2 \frac{1}{\gamma^{-1}(\alpha)}\left[1+\frac{\eta^{-1}(\alpha) R^{\prime}\left(\eta^{-1}(\alpha)\right)}{R\left(\eta^{-1}(\alpha)\right)}\right]\end{gathered}$$
and show that the RHS of \eqref{RSc}, i.e. $v(\alpha)$, diverges to $+\infty$ as $\alpha \to +\infty$ and goes to some value less than $1$ as $\alpha \to 0$. 

First consider any positive increasing sequence $\left(\alpha_m\right)_{m=1}^{+\infty}$ such that $\alpha_m \rightarrow+\infty$ as $m \rightarrow \infty$. We have that $$C_1:=\limsup _{m \rightarrow \infty} \gamma^{-1}\left(\alpha_m\right)<\gamma_{+}^{-1}$$ 
which follows from \eqref{riben}. In other words, $C_1$ must be a finite constant bounded away from $\gamma_+^{-1}$ when $\gamma_+^{-1}<+\infty$. 

It follows from this and monotonicity of $f_1$ that
\begin{equation*}
\limsup _{m \rightarrow \infty} \eta^{-1}\left(\alpha_m\right)=\limsup_{m \rightarrow \infty} f_1\left(\gamma^{-1}\left(\alpha_m\right)\right)<G\left(-d_{-}\right)
\end{equation*}
from which we conclude that
\begin{equation*}
C_2:=\liminf _{m \rightarrow \infty} 1+\frac{\eta^{-1}\left(\alpha_m\right) R^{\prime}\left(\eta^{-1}\left(\alpha_m\right)\right)}{R\left(\eta^{-1}\left(\alpha_m\right)\right)}>0
\end{equation*}
This follows from the fact that $\lim _{x \rightarrow 0} 1+\frac{x R^{\prime}(x)}{R(x)}=1$ using \Cref{lem:cauchy}, (f) and continuity of the function $x \mapsto$ $1+\frac{x R^{\prime}(x)}{R(x)}$ on $\left(0, G\left(-d_{-}\right)\right)$. Note that by the above discussion, we have $\liminf _{\alpha \rightarrow+\infty} \frac{v(\alpha)}{\alpha^2} \geq \sigma^2 \frac{C_2}{C_1}$ by lower-bounding second summand in $v(\alpha)$ which then implies that
\begin{equation}\label{liminffixe}
\liminf _{\alpha \rightarrow+\infty} v(\alpha) \rightarrow+\infty.
\end{equation}

Now consider any positive decreasing sequence $\left(\alpha_m\right)_{m=1}^{+\infty}$ such that $\alpha_m \rightarrow 0$ as $m \rightarrow \infty$. We first show that the second summand of $v(\alpha_m)$ vanishes as $\alpha_m \rightarrow 0$. Using \Cref{lem:cauchy}, (d) and $\gamma^{-1}(\alpha_m)\ge \frac{1}{\E \D^2}$, we have that
\begin{equation*}
\limsup _{m \rightarrow+\infty} \frac{1}{\gamma^{-1}\left(\alpha_m\right)}\left[1+\frac{\eta^{-1}\left(\alpha_m\right) R^{\prime}\left(\eta^{-1}\left(\alpha_m\right)\right)}{R\left(\eta^{-1}\left(\alpha_m\right)\right)}\right] \leq \E \D^2
\end{equation*}
which then implies 
\begin{equation}\label{baolde}
\lim _{m \rightarrow+\infty} \frac{\sigma^2 \alpha_m^2}{\gamma^{-1}\left(\alpha_m\right)}\left[1+\frac{\eta^{-1}\left(\alpha_m\right) R^{\prime}\left(\eta^{-1}\left(\alpha_m\right)\right)}{R\left(\eta^{-1}\left(\alpha_m\right)\right)}\right]=0.
\end{equation}
as required. 

We now proceed to show that the first summand of $v(\alpha_m)$ converges to a constant less than 1 as $\alpha_m \rightarrow 0$. We first state two facts: (1) on compact interval $[\frac{1}{\E \D^2}, C]$ for some $C>\frac{1}{\E \D^2}$, the function
$\gamma^{-1} \mapsto \mathbb{E} \operatorname{Prox}_{\gamma^{-1} h}^{\prime}\left(\Xstar+\frac{\gamma^{-1}}{\alpha} \mathsf{Z}\right) $
converges uniformly to the function $\gamma^{-1} \mapsto r(\gamma^{-1})$ as $\alpha \to 0$ and (2) the equation $\gamma R^{-1}(-\gamma)=r(\gamma^{-1})$ has a unique solution $\gamma^{-1}=\gamma_0^{-1}\in [\frac{1}{\E \D^2},\gamma_+^{-1})$. Fact (1) follows from an application of Dini's theorem as well as dominated convergence theorem, where the former uses monotonicity and asymptotic linearity properties from \Cref{Assumpproxclass}, (i) and (ii). To see fact (2), recall that $\gamma\mapsto \gamma R^{-1}(-\gamma)$ is strictly decreasing. Meanwhile, $\gamma \mapsto r(\gamma^{-1})$ is continuous and non-decreasing function on  interval $\qty[\frac{1}{\E \D^2}, C]$ for any $C>\frac{1}{\E \D^2}$. This follows from the uniform convergence in fact (1). We also have that 
$$ R^{-1}(-\E \D^2 )=0\le \frac{1}{\E \D^2} r \qty(\frac{1}{\E \D^2})$$
and
$$\lim_{\gamma^{-1} \to \gamma_+^{-1}} R^{-1}(-\gamma) > \lim_{\gamma^{-1} \to \gamma_+^{-1}} \gamma^{-1} r \qty(\gamma^{-1})$$
where the second line is due to \eqref{riben} and the uniform convergence in fact (1). Fact (2) follows. Combining fact (1) and (2), we have that 
\begin{equation}\label{taica}
    \gamma^{-1}(\alpha_m)\to \gamma^{-1}_0
\end{equation}
for $\gamma^{-1}_0\in [\frac{1}{\E\D^2},\gamma_+^{-1})$. 
This implies that 
$$\eta_0^{-1}:=\lim_{m\to \infty} \eta^{-1}(\alpha_m)=\gamma^{-1}(\alpha_m)\mathbb{E} \operatorname{Prox}_{\gamma^{-1}(\alpha_m) h}^{\prime}\left(\Xstar+\frac{\gamma^{-1}(\alpha_m)}{\alpha_m} \mathsf{Z}\right)\to  \gamma_0^{-1}r(\gamma^{-1}_0).$$
where 
\begin{equation}\label{xxxdf}
    \eta_0^{-1}=f_1(\gamma_0^{-1}) \in [0,G(-d_-)).
\end{equation}

Now, we also have that as $m \to \infty$, almost surely
\begin{equation}\label{sjdfff1}
\left|\alpha_m \operatorname{Prox}_{\gamma^{-1}\left(\alpha_m\right) h}\left(\Xstar+\frac{\gamma^{-1}\left(\alpha_m\right)}{\alpha_m} \mathsf{Z}\right)-\alpha_m  \operatorname{Prox}_{\gamma_0^{-1} h}\left(\Xstar+\frac{\gamma_0^{-1}}{\alpha_m} \mathsf{Z}\right)\right| \rightarrow 0
\end{equation}
which follows from \Cref{prop:proxp} (b), (c), and \eqref{taica}. Meanwhile, we have that as $m \to \infty$, almost surely,
\begin{equation}\label{sjdfff2}
\left|\alpha_m \operatorname{Prox}_{\gamma_0^{-1} h}\left(\Xstar+\frac{\gamma_0^{-1}}{\alpha_m} \mathsf{Z}\right)-r(\gamma_0^{-1}) \gamma_0^{-1}  \mathsf{Z}\right| \rightarrow 0
\end{equation}
following from an application of L'Hôpital's rule which uses \Cref{Assumpproxclass}, (ii). Combining \eqref{sjdfff1}, \eqref{sjdfff2} and dominated convergence theorem, we have that

\begin{equation}\label{DsfKW}
\begin{aligned}
\lim _{m \rightarrow+\infty} \alpha_m^2 R^{\prime}\left(\eta^{-1}\left(\alpha_m\right)\right) \mathbb{E}\left(\operatorname{Prox}_{\gamma^{-1}\left(\alpha_m\right) h}\left(\Xstar+\frac{\gamma^{-1}\left(\alpha_m\right)}{\alpha_m} \mathsf{Z}\right)-\Xstar\right)^2=\eta^{-2}_0 R^{\prime}\left(\eta^{-1}_0\right) \\.
\end{aligned}
\end{equation}
Now note that
\begin{equation*}
\eta^{-2}_0 R^{\prime}\left(\eta^{-1}_0\right)<1
\end{equation*}
using \Cref{lem:cauchy}, (e) and \eqref{xxxdf}. Using this and \eqref{DsfKW}, we may then conclude that
\begin{equation*}
\lim _{m \rightarrow+\infty} \alpha_m^2 R^{\prime}\left(\eta^{-1}\left(\alpha_m\right)\right) \mathbb{E}\left(\operatorname{Prox}_{\gamma^{-1}\left(\alpha_m\right) h}\left(\Xstar+\frac{\gamma^{-1}\left(\alpha_m\right)}{\alpha_m} \mathsf{Z}\right)-\Xstar\right)^2<1
\end{equation*}
which along with \eqref{baolde} implies that
\begin{equation}\label{supvlim}
\limsup_{\alpha \rightarrow 0} v(\alpha)<1.
\end{equation}
Combine \eqref{liminffixe} and \eqref{supvlim}. By continuity of $\alpha \mapsto v(\alpha)$ on $(0,+\infty)$, we know that there exists a solution $\alpha_* \in(0,+\infty)$ to the equation $v\left(\alpha_*\right)=1$. Therefore, a solution of \eqref{fpstren} is $(\gamma^{-1}, \eta^{-1}, \alpha)=\left(\gamma^{-1}\left(\alpha_*\right), \eta^{-1}\left(\alpha_*\right), \alpha_*\right)$ by construction. This concludes the proof.
\end{proof}
}

\section{Proofs for Spectrum-Aware Debiasing} \label{section:asympchar}
%\section{Proof Outline and Novelties} 
Proof of our main result, \Cref{NEIGMAIN}, relies on three main steps: (i) a characterization of the empirical distribution of a population version of $\hat{\bbeta}$, (ii) connecting this population version with our data-driven Spectrum-Aware estimator, (iii) developing a consistent estimator of the asymptotic variance. We next describe our main technical novelties for step (i) in Section \ref{subsec:riskmath}, and that for steps (ii) and (iii) in \Cref{subsec:tauest}.

\subsection{Result A: Distributional characterizations}\label{subsec:riskmath}
\Cref{NEIGMAIN} relies on the characterization of certain properties of $\hatbt$ and the following two quantities:

% \begin{table}[t]
% \caption{The table prints Benjamini-Hochberg adjusted p-values for the hypothesis tests $H_i:\alphstari=0,i=1,...,5$ corresponding to the experiments from \Cref{figPCRC} for the designs specified in \Cref{GroupD}. Below we use $**$ to indicate rejection under FDR level $0.01$ and $*$ rejection under FDR level $0.05$ }
%     \centering
%     \begin{tabular}{llllll}
%     \toprule
%     {} & Speech &    DNA &     SP500 &  FaceImage &     Crime \\
%     \midrule
% $\upsilon_{1}^\star$  &   0.71 &  0.92 &  0.00 ** &   0.00 ** &  0.00 ** \\
% $\upsilon_{2}^\star$  &   0.71 &  0.90 &     0.42 &      0.62 &     0.10 \\
% $\upsilon_{3}^\star$  &   0.91 &  0.90 &     0.95 &      0.77 &  0.01 ** \\
% $\upsilon_{4}^\star$  &   0.91 &  0.90 &     0.21 &   0.00 ** &  0.00 ** \\
% $\upsilon_{5}^\star$  &   0.99 &  0.92 &     0.31 &      0.98 &   0.05 * \\
%     \bottomrule
%     \end{tabular}
%     \label{tabCc}
% \end{table}

% \begin{figure}[H]
%     \centering
%     \includegraphics[width=\linewidth]{Image_pnas/panel_pcr_real_FCP.pdf}
%     \vspace{-6mm}
%     \caption{Under the setting of \Cref{figPCRC}, the above plots the false coverage proportion (FCP) of confidence intervals defined in \eqref{defCIPCR} for $(\sti)_{i=1}^p$, as we vary targeted FCP level $\alpha$ from $0$ to $1$. The $x$-axis spans $\alpha$ values from 0 to 1, while the $y$-axis ranges from 0 and 1. The dotted black line is the 45-degree line for reference.}
%     \label{figstCIC}
% \end{figure}

\begin{equation}\label{defr1r2}
    \rstar:=\hatbt+\frac{1}{\gamma_*} \X^{\top}(\y-\X \hatbt), \quad \rstst:=\hatbt+\frac{1}{\eta_*-\gamma_*} \X^{\top}(\X \hatbt-\y).
\end{equation}
Here, $\rstar$ can be interpreted as the population version of the debiased estimator $\bhetah$ and $\rstst$ as an auxiliary quantity that arises in the intermediate steps in our proof. The following theorem characterizes the empirical distribution of the entries of $\hatbt$ and $\rstar$. We prove it in \Cref{section:DIST} from Appendix.

\begin{Theorem}[Distributional characterizations]\label{thm:empmain}
Under Assumptions \ref{AssumpD}--\ref{Assumpfix}, almost surely as $n,p\to \infty$,
\begin{equation}\label{waka1}
\left(\hatbt, \rstar, \st\right) \stackrel{W_2}{\rightarrow}\left(\operatorname{Prox}_{\gamma_*^{-1} h}\left(\sqrt{\taustar} \Zs+\Xstar \right), \sqrt{\taustar} \Zs+\Xstar, \Xstar \right),
\end{equation}
where $\Zs\sim N(0,1)$ is independent of $\Xstar$. Furthermore, almost surely as $p\to \infty$
\begin{equation}\label{waka2}
    \begin{aligned}
        &\frac{1}{p}\left\|\X \rstst-\y\right\|^2 \to \tau_{**} \cdot \mathbb{E}\D^2+\sigma^2\cdot \delta, \\
        &\frac{1}{p} \norm{\y-\X \hatbt}^2\to \tau_{**} \cdot \E \frac{\D^2(\eta_*-\gamma_*)^2}{(\D^2+\eta_*-\gamma_*)^2}+\sigma^2\cdot \qty(\frac{n-p}{p}+\E \qty(\frac{\eta_*-\gamma_*}{\D^2+\eta_*-\gamma_*})^2).
    \end{aligned}
\end{equation}
\end{Theorem}

We now discuss the proof novelties for \Cref{thm:empmain}. \Cref{section:DIST} from Appendix contains this proof. 

We base our proof on the approximate message passing (AMP) machinery (cf.~\cite{donoho2009message,zdeborova2016statistical,sur2019modern,feng2022unifying,montanari2022short} 
for a non-exhaustive list of references). In this approach, one constructs an AMP algorithm in terms of fixed points ($\eta_*,\gamma_*, \tau_{*}, \tau_{**}$ in our case) and shows that its iterates $\hat{\mathbf{v}}^t$ converge to our objects of interest $\hat{\mathbf{v}}$ ($\hat{\mathbf{v}}$ can be  $\hatbt$ or $\rstar$ in our case) in the following sense: almost surely
\begin{equation}\label{garbc}
\lim _{t \rightarrow \infty} \lim_{p\to \infty} \frac{\left\|\hat{\mathbf{v}}^t-\hat{\mathbf{v}}\right\|^2}{p}=0. 
\end{equation}
AMP theory provides a precise characterization of the following limit involving the algorithmic iterates for any fixed $t$: $\lim_{p \rightarrow \infty } \|\hat{\mathbf{v}}^t-\mathbf{v}_0 \|^2/p,$
where $\mathbf{v}_0$ is usually a suitable function of $\bm{\beta}_{\star}$ around which one expects $\hat{\mathbf{v}}$ should be centered. 
Thus plugging this in \eqref{garbc} yields properties of the object of interest $\hat{\mathbf{v}}$. Within this theory, the framework that characterizes $\lim_{p \rightarrow \infty } \|\hat{\mathbf{v}}^t-\mathbf{v}_0 \|^2/p$ is known as \emph{state evolution} \cite{bayati2011lasso,javanmard2013state}. Despite the existence of this solid machinery, \eqref{garbc}
 requires a case-by-case proof, and for many settings, this presents deep challenges.

We use the above algorithmic proof strategy, but in case of our right-rotationally invariant designs to which the original AMP algorithms fail to apply.
To alleviate this, \cite{rangan2019vector} proposed vector approximate message passing algorithms. We use these algorithms to create our $\hat{\mathbf{v}}^t$'s. Subsequently, proving \eqref{garbc} presents the main challenge. To this end, one is required to show the following Cauchy convergence property of the VAMP iterates: almost surely, $\lim _{(s, t) \rightarrow \infty}\left(\lim _{p \rightarrow \infty} \frac{1}{p}\left \|\hat{\mathbf{v}}^t-\hat{\mathbf{v}}^s\right \|^2\right)=0.$
We prove this using a Banach contraction argument (cf. \eqref{eq:gprimebound} from Appendix). Such an argument saw prior usage in the context of  Bayes optimal learning in \cite{li2023random}. 
However, they studied a ``matched" problem where the signal prior (analogous to $\Xstar$ in our setting) is known to the statistician and she uses this exact prior during the estimation process. Arguments under such matched Bayes optimal problems do not translate to our case, and proving \eqref{eq:gprimebound} presents novel difficulties in our setting. To mitigate this, we leverage a fundamental property of the R-transform, specifically that $-zR'(z)/R(z)<1$ for all $z$, and discover and utilize a crucial interplay of this property with the non-expansiveness of the proximal map (see \Cref{prop:proxp} (b) from Appendix).

\begin{Remark}[Comparison with \cite{gerbelot2020asymptotic,gerbelot2022asymptotic}]\label{rem:cedric} In their seminal works,
\cite{gerbelot2020asymptotic,gerbelot2022asymptotic} initiated the first study of the risk of regularized estimators under right-rotationally invariant designs. They stated a version of Theorem \ref{thm:empmain} with a partially non-rigorous argument.  In their approach, an auxiliary $\ell_2$ penalty of sufficient magnitude is introduced to ensure contraction of AMP iterates. Later, they remove this penalty through an analytical continuation argument. However, this proof suffers two limitations. The first one relates to the non-rigorous applications of the AMP state evolution results. For instance, \cite[Lemma 3]{gerbelot2022asymptotic} shows that for each fixed value of $p$,
$
\lim _{t \rightarrow \infty} \frac{\left \|\hat{x}^t-\hat{x}\right \|^2}{p}=0.
$
However, in \cite[Proof of Lemma 4]{gerbelot2022asymptotic}, the authors claim that this would imply \eqref{garbc} upon exchanging limits with respect to $t$ and $p$. Such an exchange of limits is  non-rigorous since the correctness of AMP state evolution is established for a finite number of iterations ($t<T,T$ fixed) as $p\to \infty$. The limit in $T$ is taken after $p$. The other limitation lies in the analytic continuation approach that requires multiple exchanges of limit operations \cite[Appendix H]{gerbelot2022asymptotic} that seem difficult to justify and incur intractable assumptions \cite[Assumption 1 (c), (e)]{gerbelot2022asymptotic} (in particular, it is unclear how to verify the existence claim in Assumption 1 (c) beyond Gaussian designs). Our alternative approach establishes contraction without the need for a sufficiently large $\ell_2$-regularization component, as in \cite{gerbelot2020asymptotic,gerbelot2022asymptotic}, and thereby avoids the challenges associated with the  analytic continuation argument. 
\end{Remark}

\subsection{Result B: Consistent estimation of fixed points}\label{subsec:tauest}
Note that the population debiased estimator $\rstar$ cannot be used to conduct inference since $\gamma_*$ is unknown. Furthermore, the previous theorem says roughly that $\rstar - \bbeta^{\star}$ behaves as a standard Gaussian with variance $\taustar$, without providing any estimator for $\taustar$. We address these two points here. In particular, we will see that addressing these points ties us to establishing consistent estimators for the solution to the fixed points defined in \eqref{fp}.
The theorem below shows that $(\adj, \hat{\eta}_*, \hat{\tau}_*, \hat{\tau}_{**})$ from \eqref{DEFEFD} serve as consistent estimators of the fixed points $(\gamma_*,\eta_*,\tau_*, \tau_{**})$, and $\bhetah,\hatrstst$ as consistent estimators of $\rstar$ and $\rstst$, where $\hatrstst$ is defined as in \eqref{NEFTF22} below. For the purpose of the discussion below, we note that $\hat{\tau}_{**}$ from \eqref{DEFEFD} can be written as follows.
\begin{equation}\label{NEFTF22}
    \hat{\tau}_{**}(p) :=\frac{\frac{1}{p}\left\|\X \hatrstst-\y\right\|^2-\frac{n}{p}\cdot \sigma^2}{\frac{1}{p} \sum_{i=1}^p d_i^2}; \quad \hatrstst:=\hatbt+\frac{1}{\hat{\eta}_*-\adj} \X^{\top}(\X \hatbt-\y).
\end{equation}
 Furthermore, recall that when the noise level $\sigma^2$ is unknown, one requires an estimator for $\sigma^2$ to calculate $\tauh, \hat{\tau}_{**}$ in \eqref{DEFEFD}.  We define such an estimator below and show that that it estimates $\sigma^2$ consistently. 
\begin{equation}\label{wemr}
    \hat{\sigma}^2(\X,\y,h) \gets \frac{\|\mathbf{y}-\mathbf{X} \widehat{\boldsymbol{\beta}}\|^2-\frac{\left\|\left(\mathbf{I}_n+\frac{1}{\hat{\eta}_*-\mathrm{adj}} \mathbf{X X}^{\top}\right)(\mathbf{y}-\mathbf{X} \widehat{\boldsymbol{\beta}})\right\|^2}{\sum_{i=1}^p d_i^2} \sum_{i=1}^p \frac{\left(\hat{\eta}_*-\adj\right)^2 d_i^2}{\left(d_i^2+\hat{\eta}_*-\adj\right)^2}}{\sum_{i=1}^p \frac{\left(\hat{\eta}_*-\adj\right)^2 \cdot\left(\sum_{j=1}^p d_j^2-n d_i^2\right)}{\left(d_i^2+\hat{\eta}_*-\adj\right)^2 \cdot\left(\sum_{j=1}^p d_j^2\right)}+n-p}.
\end{equation}
Note this is well-defined when
\begin{equation}\label{werwerlf}
    \frac{n}{p}\cdot \frac{\frac{1}{p}\sum_{i=1}^p d_i^2\cdot \qty (1-\qty(\frac{\hat{\eta}_*-\adj}{d_i^2+\hat{\eta}_*-\adj})^2 )}{\frac{1}{p}\sum_{i=1}^p d_i^2\cdot \frac{1}{p}\sum_{i=1}^p \qty (1-\qty(\frac{\hat{\eta}_*-\adj}{d_i^2+\hat{\eta}_*-\adj})^2 )} \neq 1.
\end{equation}
In particular, the LHS of \eqref{werwerlf} consistently estimates the  LHS of \eqref{Asfix2} in \Cref{Assumpfix2}.

\begin{Theorem}[Consistent estimation of fixed points]\label{neig}
Suppose that \Cref{AssumpD}---\ref{Assumpfix2} hold. Then, the estimators in \eqref{DEFEFD} and \eqref{NEFTF22} are well-defined for any $p$ and we have that almost surely as $p\to \infty$,
$$
\begin{aligned}
    & \adj \left(p\right) \rightarrow \gamma_*, \quad  \hat{\eta}_*\left(p\right) \rightarrow \eta_*, \quad \hat{\tau}_{*}\left(p\right) \rightarrow \taustar, \quad  \hat{\tau}_{**}\left(p\right) \rightarrow \tau_{**}, \quad \hat{\sigma}^2\left(p\right) \rightarrow \sigma^2,\\
    & \frac{1}{p}\norm{\bhetah(p)-\rstar}^2\rightarrow 0, \quad \frac{1}{p}\norm{\hatrstst(p)-\rstst}^2\rightarrow 0.
\end{aligned}
$$
We note that if $\sigma^2$ is known and one sets $\hat{\sigma}^2(p)=\sigma^2$, the above holds without requiring \Cref{Assumpfix2}. 
\end{Theorem}
It is not hard to see that  \Cref{thm:empmain} combined with \Cref{neig} proves our main result \Cref{NEIGMAIN}.

We now discuss the proof of \Cref{neig}. See \Cref{section:DEBIA} from Appendix for the proof details.

First, let us present some heuristics for how one might derive 
 the consistent estimators $\qty(\adj, \hat{\eta}_*, \hat{\tau}_*, \hat{\tau}_{**})$. We start from \eqref{RCa}. Using \Cref{Extend}, it can be written as
\begin{equation}\label{veree1}
    \frac{\gamma_*}{\eta_*}=\E \frac{1}{1+\gamma_*^{-1} h^{\prime \prime} \qty(\operatorname{Prox}_{\gamma_*^{-1}h} (\Xstar+\sqrt{\taustar} \Zs) )}.
\end{equation}
Recall that we have established \Cref{thm:empmain} that shows  $p\to \infty$, almost surely, 
$\hatbt \stackrel{W_2}{\to} \operatorname{Prox}_{\gamma_*^{-1} h}\left(\sqrt{\taustar} \Zs+\Xstar \right).$
Combining this and \eqref{veree1}, we expect that
\begin{equation}\label{veree3}
    \frac{1}{\eta_*} \approx \frac{1}{p} \sum_{i=1}^p \frac{\gamma_*^{-1}}{1+\gamma_*^{-1} h^{\prime \prime} (\hiatbti) }.
\end{equation}
Using the definition of R-transform, we can rewrite \eqref{RCc} as $\eta_*^{-1}=\E \frac{1}{\D^2+\eta_*-\gamma_*}$ which, along with \eqref{Dconve}, implies that
$\frac{1}{\eta_*} \approx \frac{1}{p} \sum_{i=1}^p \frac{1}{d_i^2+\eta_*-\gamma_*}.$ Combining this and \eqref{veree3} to eliminate $\eta_*$, we obtain that 
\begin{equation}\label{safdas}
\frac{1}{p} \sum_{i=1}^p \frac{1}{\left(d_i^2-\gamma_* \right)\left(\frac{1}{p} \sum_{j=1}^p \qty(\gamma_*+h^{\prime \prime}\left(\hjatbtj\right))^{-1} \right)+1}\approx 1.
\end{equation}
Setting $\approx$ above to equality, we obtain our exact equation for the Spectrum-Aware adjustment factor, i.e.~ \eqref{gammasolvea}. One thus expects intuitively that $\adj$ consistently estimates $\gamma_*$. To establish the consistency rigorously, we recognize and establish the monotonicity of the LHS of \eqref{safdas} as a function of $\gamma_*$, and study its point-wise limit. We direct the reader to \Cref{well} and \Cref{ptconv} from Appendix for more details. 

Once we have established the consistency of $\adj$ as an estimator for $\gamma_*$, we substitute $\adj$ back into \eqref{veree3} to obtain a consistent estimator $\hat{\eta}_*$ for $\eta_*$. It is important to note that the definition of $\rstst$, as given in \eqref{defr1r2}, only involves the fixed points $\eta_*$ and $\gamma_*$. As a result, we can utilize $\adj$ and $\hat{\eta}_*$ to produce a consistent estimator $\hatrstst$ for $\rstst$. Now note that \eqref{waka2} would give us a system of linear equation
\begin{equation}
    \mqty(\frac{1}{p}\left\|\X \rstst-\y\right\|^2 \\ \frac{1}{p} \norm{\y-\X \hat{\mathbf{\beta}}}^2) \approx \mqty(\mathbb{E}\D^2 &  \delta \\ \E \frac{\D^2(\eta_*-\gamma_*)^2}{(\D^2+\eta_*-\gamma_*)^2} & \quad \frac{n-p}{p}+\E \qty(\frac{\eta_*-\gamma_*}{\D^2+\eta_*-\gamma_*})^2) \mqty(\tau_{**} \\ \sigma^2).
\end{equation}
The estimators $(\hat{\tau}_{**}, \hat{\sigma}^2)$ in \eqref{DEFEFD} for $(\tau_{**}, \sigma^2)$ are solved from the two linear equations above with the 2-by-2 matrix on RHS replaced by its sample version. Note that \eqref{Asfix2} is required to ensure the 2-by-2 matrix is non-singular.  Now with estimators for $\gamma_*, \eta_*, \sigma^2$ and $\tau_{**}$, we can construct the estimator $\tauh$ for $\tau_*$ using \eqref{RCd} and \eqref{Dconve}.

\subsection{Proof result A: Distribution characterization}\label{section:DIST}
In this section, we prove \Cref{thm:empmain} using VAMP algorithm as proof device. We define the version of VAMP algorithm we will use in \Cref{subsectionVAMP}, prove Cauchy convergence of its iterates in \Cref{subsectionCauchc}, and prove \Cref{thm:empmain} in \Cref{subsectionlimi}. To streamline the presentation, proofs of intermediate claims are collected in \Cref{SupportPI}. We also assume without loss of generality that $$\sigma^2=1$$ 
for the remainder of this section. The general case for arbitrary $\sigma^2>0$ follows from a simple rescaling argument.

\subsubsection{The oracle VAMP algorithm}\label{subsectionVAMP}
We review the oracle VAMP algorithm defined in \cite{gerbelot2020asymptotic} and present an extended state evolution result for the algorithm. This algorithm is obtained by initializing the VAMP algorithm introduced in \cite{rangan2019vector} at stationarity $\mathbf{r}_{10}=\st +N(\rm{0},\taustar \mathbf{I}_p), \gamma_{10}^{-1}=\gamma_*^{-1}$. See \Cref{vamp} for a review. Then for $t\ge 1$, we have iterates
\begin{subequations}\label{empvamp}
    \begin{align}
& \xonet=\operatorname{Prox}_{\gamma_*^{-1} h}\left(\ronetm\right) \label{x1t} \\
& \rtwot=\frac{1}{\eta_*-\gamma_*}\left(\eta_* \xonet-\gamma_* \ronetm\right) \label{r2t} \\
& \xtwot=\left(\X^{\top} \X+\left(\eta_*-\gamma_*\right) \mathbf{I}_p\right)^{-1}\left(\X^{\top} \y+\left(\eta_*-\gamma_*\right) \rtwot\right) \label{x2t}\\
& \ronet=\frac{1}{\gamma_*}\left(\eta_* \xtwot-\left(\eta_*-\gamma_*\right) \rtwot\right) \label{r1t}
\end{align}
\end{subequations}
\begin{Remark}
    Note that the above definition assumes existence of fixed point $\eta_*, \gamma_*, \tau_*, \tau_{**} \in (0,+\infty)$, i.e. \Cref{Assumpfix}. We however do not require the fixed point to be unique. Our proof may proceed by defining the oracle VAMP algorithm above with respect to any one of the fixed points. 
\end{Remark}

Let us define functions $F: \R\times \R \to \R$ and $F':\R \times \R \to \R$
\begin{equation}\label{Fdef}
\begin{aligned}
& F(q, x):=\frac{\eta_*}{\eta_*-\gamma_*} \operatorname{Prox}_{\gamma_*^{-1} h}(q+x)-\frac{\gamma_*}{\eta_*-\gamma_*} q-\frac{\eta_*}{\eta_*-\gamma_*} x \\
& F^\prime(q, x):=\frac{\eta_*}{\eta_*-\gamma_*} \operatorname{Prox}^\prime_{\gamma_*^{-1} h}(q+x)-\frac{\gamma_*}{\eta_*-\gamma_*}  
\end{aligned}
\end{equation}
Note that for any fixed $x$, $F^\prime(q,x)$ equals to the derivative of $q \mapsto F(q,x)$ whenever the derivative exists, and at the finitely many points where $q\mapsto F(q,x)$ is not differentiable $F^\prime(q,x)$ equals to $0$ (cf. \Cref{Extend}). We also define some quantities
\begin{equation}\label{Quant}
\begin{aligned}
& \bm{\Lambda}:=\frac{\eta_*\left(\eta_*-\gamma_*\right)}{\gamma_*}\left(\Dbm^{\top} \Dbm+\left(\eta_*-\gamma_*\right) \mathbf{I}_p\right)^{-1}-\left(\frac{\eta_*-\gamma_*}{\gamma_*}\right) \cdot \mathbf{I}_p \\
& \bm{\xi}:=\Qbm \epbm, \quad \ebit :=\frac{\eta_*}{\gamma_*}\left(\Dbm^{\top} \Dbm+\left(\eta_*-\gamma_*\right) \mathbf{I}_p\right)^{-1} \Dbm^{\top} \bm{\xi}, \quad \mathbf{e}:=\Obm^{\top} \ebit 
\end{aligned}
\end{equation}

We note some important properties of these quantities, which are essentially consequence of \Cref{AssumpD} and \eqref{fp}. We defer the proof of \Cref{prop:ppt} to \Cref{appendix:ovamp}.
\begin{Proposition}\label{prop:ppt}
Under \Cref{AssumpD}---\ref{Assumph} and \ref{Assumpfix}, almost surely,
\begin{subequations}\label{dddfef}
\begin{align}
& \lim _{p \rightarrow \infty} \frac{1}{p} \operatorname{Tr}(\bm{\Lambda})=0 , \quad \kappa_*:=\lim _{p \rightarrow \infty} \frac{1}{p} \operatorname{Tr}\left(\bm{\Lambda}^2\right)=\mathbb{E}\left(\frac{\eta_*\left(\eta_*-\gamma_*\right)}{\gamma_*\left(\D^2+\left(\eta_*-\gamma_*\right)\right)}-\frac{\eta_*-\gamma_*}{\gamma_*}\right)^2 \\
& b_*:=\lim _{p \rightarrow \infty} \frac{1}{p}\left\|\ebit \right\|^2=\frac{1}{\gamma_*}-\frac{\kappa_*}{\eta_*-\gamma_*}=\left(\frac{\eta_*}{\gamma_*}\right)^2 \mathbb{E} \frac{\D^2}{\left(\D^2+\eta_*-\gamma_*\right)^2},\quad \taustar=b_*+\kappa_* \tau_{**} \label{eq:somide}\\
& \mathbb{E} F^{\prime}\left(\sqrt{\taustar}\Zs, \Xstar \right)=0, \quad \mathbb{E} F\left(\sqrt{\taustar}\Zs, \Xstar \right)^2= \tau_{**} \label{eq:denoiserpp2}
\end{align}
\end{subequations}
where $\Zs\sim N(0,1)$ is independent of $\Xstar$. Moreover, the function $(q,x)\mapsto F(q,x)$ is Lipschitz continuous on $\R\times \R$. 
\end{Proposition}

Then, one can show that by eliminating $\xonet, \xtwot$ and introducing a change of variables
\begin{equation}\label{changeofv}
    \ampxt=\rtwot-\st, \quad \ampyt=\ronet-\st-\mathbf{e}, \quad \ampst=\Obm \ampxt
\end{equation}
\eqref{empvamp} is equivalent to the following iterations: with initialization $\mathbf{q}^0\sim N(\rm{0},\taustar \cdot \mathbf{I}_p), \ampxone=F(q_0,\st)$, for $t=1,2,3,\ldots,$
\begin{equation}\label{eq:xsy}
\ampst=\Obm \ampxt, \qquad \ampyt=\Obm^\top \bm{\Lambda} \ampst, \qquad \ampxtpo=F(\ampyt+\mathbf{e},\st).
\end{equation}
The following Proposition will be needed later. Its proof is deferred to \Cref{appendix:ovamp}. 
\begin{Proposition}\label{prop:AMPparamconverge}
Suppose Assumptions \ref{AssumpD}--\ref{AssumpPrior} hold.
Define random variables
\[\Xi\sim N(0,1), \qquad
\mathsf{P}_0\sim N(0,\taustar), \qquad \Es \sim N(0,b_*)
\]
independent of each other and of $\D$, and set
	$$
\mathsf{L}=\frac{\eta_{*}-\gamma_{*}}{\gamma_{*}}\left(\frac{\eta_*}{\D^{2}+\eta_{*}-\gamma_{*}}-1\right),\;
\mathsf{E}_{b}=\frac{\eta_{*}}{\gamma_{*}} \frac{\D
\Xi}{\D^{2}+\eta_{*}-\gamma_{*}}, \;
\Hs=(\Xstar,\D,\D\Xi,\Ls,\Es_b,\Es,\Ps_0).$$
Then $\kappa_*=\E \Ls^2$ and $b_*=\E \Es_b^2$. Furthermore, almost surely as $n,p\to \infty$,
$$\mathbf{H}:= \left(\st, \Dbm^{\top} \bm{1}_{n \times 1}, \Dbm^{\top} \bm{\xi}, \operatorname{diag}(\bm{\Lambda}), \ebit , \mathbf{e}, \mathbf{q}^0\right)\stackrel{W_2}{\to} \Hs.$$ 

\end{Proposition}

Now we state the state evolution for the VAMP algorithm. Its proof is deferred to \Cref{appendix:ovamp}. 
\begin{Proposition}\label{thm:ampSE}
Suppose \Cref{AssumpD}---\ref{Assumph} and \ref{Assumpfix} hold. Further assume that the function $x\mapsto \operatorname{Prox}_{\gamma_*^{-1} h}^{\prime}(x)$ defined in \Cref{Extend} is non-constant. 
Let $\Hs=(\Xstar,\D,\D\Xi,\Ls,\Es_b,\Es,\Ps_0)$ be as defined in
\Cref{prop:AMPparamconverge}. Set $\Xs_1=F(\Ps_0,\Xstar)$, set
$\Delta_1=\E[\Xs_1^2] \in \R^{1 \times 1}$, and define iteratively
$\Ss_t,\Ys_t,\Xs_{t+1},\Delta_{t+1}$ for $t=1,2,3,\ldots$ such that
\[(\Ss_1,\ldots,\Ss_t) \sim N(\rm{0},\Delta_t),
\qquad (\Ys_1,\ldots,\Ys_t) \sim N(\rm{0},\kappa_* \Delta_t)\]
are Gaussian vectors independent of each other and of $\Hs$, and
\[\Xs_{t+1}=F(\Ys_t+\Es,\Xstar), \qquad
\Delta_{t+1}=\mathbb{E}\left[\left(\mathsf{X}_{1}, \ldots,
\mathsf{X}_{t+1}\right)\left(\mathsf{X}_{1}, \ldots,
\mathsf{X}_{t+1}\right)^{\top}\right] \in \R^{(t+1) \times (t+1)}.\]
Then for each $t \geq 1$, $\Delta_t \succ 0$ strictly,
$\tau_{**}=\E \Xs_t^2$, and $\kappa_* \tau_{**}=\E \Ys_t^2$.

Furthermore, let $\mathbf{X}_{t}=\left(\ampxone, \ldots, \ampxt\right) \in
\mathbb{R}^{p \times t}$, $\mathbf{S}_{t}=\left(\mathbf{s}^{1}, \ldots, \ampst\right) \in
\mathbb{R}^{p \times t}$, and $\mathbf{Y}_{t}=\left(\ampyone, \ldots, \ampyt\right) \in
\mathbb{R}^{p \times t}$ collect the iterates of \eqref{eq:xsy}, starting from
the initialization $\ampxone=F(\mathbf{q}^0,\st)$. Then for any fixed $t \geq 1$,
almost surely as $p,n \rightarrow \infty$,
	$$
	\left(\mathbf{H},\X_t, \mathbf{S}_t,\mathbf{Y}_t\right)
\stackrel{W_2}{\to}\left(\mathsf{H},
\mathsf{X}_{1}, \ldots, \mathsf{X}_t, 
\mathsf{S}_{1}, \ldots, \mathsf{S}_{t}, \mathsf{Y}_{1}, \ldots,
\mathsf{Y}_{t}\right).
	$$ 
\end{Proposition}

Noting that each matrix $\Delta_t$ is the upper-left submatrix of
$\Delta_{t+1}$, let us denote the entries of these matrices as
$\Delta_t=(\delta_{rs})_{r,s=1}^t$. We also denote $\delta_*:=\tau_{**}$ and $\sigma_*^2:=\kappa_* \tau_{**}$. 

\begin{Remark} In case where $\operatorname{Prox}_{\gamma_*^{-1} h}^{\prime}(x)$ is constant in $x$ (e.g. ridge penalty), the iterates converges in one iteration and the above result holds for $t\le 1$. 
\end{Remark}

Proof of the following Corollary is deferred to \Cref{appendix:ovamp}.
\begin{Corollary}\label{SEovamp}
    Under \Cref{AssumpD}---\ref{Assumph} and \ref{Assumpfix}, almost surely as $p,n\to \infty$
    \begin{equation}\label{consA}
        \begin{aligned}
            \left(\xonet, \ronet, \st\right) \stackrel{W_2}{\to}\left(\operatorname{Prox}_{\gamma_*^{-1} h}\left(\sqrt{\taustar} \mathsf{Z}+\Xstar\right), \sqrt{\taustar} \mathsf{Z}+\Xstar, \Xstar\right).
        \end{aligned}
    \end{equation}
    Furthermore, almost surely as $p,n\to \infty$,
\begin{equation}\label{consC}
\begin{aligned}
    &\frac{1}{p}\left\|\X \rtwot-\y\right\|^2 \rightarrow \tau_{**} \mathbb{E} \D^2+\delta \\
    & \frac{1}{p} \norm{\y-\X \hat{\mathbf{x}}_{2 t}}^2 \to \tau_{**} \cdot \E \frac{\D^2(\eta_*-\gamma_*)^2}{(\D^2+\eta_*-\gamma_*)^2}+\frac{n-p}{p}+\E \qty(\frac{\eta_*-\gamma_*}{\D^2+\eta_*-\gamma_*})^2.
\end{aligned}
\end{equation}
\end{Corollary}

\subsubsection{Cauchy convergence of VAMP iterates}\label{subsectionCauchc}
The following Proposition is analogous to \cite[Proposition 2.3]{fan2021replica} and \cite[Lemma B.2.]{li2023random} in the context of rotationally invariant spin glass and Bayesian linear regression. However, it requires observing a simple but crucial property of the R-transform (i.e. $-zR'(z)/R(z)<1$ for all $z$ on the domain) and its interplay with the non-expansiveness of the proximal map. We defer the proof to \Cref{appendix:ovamp}.

\begin{Proposition}\label{prop:convsmallbeta}
Under \Cref{AssumpD}---\ref{Assumph} and \ref{Assumpfix},
$$\lim_{\min (s, t) \rightarrow \infty} \delta_{s t}=\delta_*$$
where $\delta_{st}=\E \mathsf{X}_s \mathsf{X}_t$.
\end{Proposition}

We can then obtain the convergence of vector iterates for the oracle VAMP algorithm. We defer the proof to \Cref{appendix:ovamp}.
\begin{Corollary}\label{Corc}
Under \Cref{AssumpD}---\ref{Assumph} and \ref{Assumpfix}, for $j=1,2$,
\begin{equation}
\begin{gathered}
\lim _{(s, t) \rightarrow \infty}\left(\lim _{p \rightarrow \infty} \frac{1}{p}\left\|\ampxt-\mathbf{x}^s\right\|^2\right)=\lim _{(s, t) \rightarrow \infty}\left(\lim _{p \rightarrow \infty} \frac{1}{p}\left\|\ampyt-\y^s\right\|^2\right) \\
\qquad \qquad \qquad =\lim _{(s, t) \rightarrow \infty}\left(\lim _{p \rightarrow \infty} \frac{1}{p}\left\|\mathbf{r}_{j t}-\mathbf{r}_{j s}\right\|^2\right)=\lim _{(s, t) \rightarrow \infty}\left(\lim _{p \rightarrow \infty} \frac{1}{p}\left\|\hat{ \mathbf{x} }_{j t}-\hat{ \mathbf{x} }_{j s}\right\|^2\right)=0
\end{gathered}
\end{equation}
where the inner limits exist almost surely for each fixed $t$ and $s$.
\end{Corollary}

\subsubsection{Characterize limits of empirical distribution}\label{subsectionlimi}
Recall definition of $\rstar,\rstst$ from \eqref{defr1r2}. The following is a direct consequence of the Cauchy convergence of the VAMP iterates and the strong convexity in the penalized loss function. We defer the proof to \Cref{appendix:trackiovamp}. 
\begin{Proposition}\label{prop:sds}
Under Assumptions \ref{AssumpD}--\ref{Assumpfix}, for $j=1,2$,
\begin{equation}\label{tianshdf}
    \lim _{t \rightarrow \infty} \lim _{p \rightarrow \infty} \frac{1}{p}\left\|\hatbt-\hat{\mathbf{x}}_{j t}\right\|_2^2=\lim _{t \rightarrow \infty} \lim _{p \rightarrow \infty} \frac{1}{p}\left\|\mathbf{r}_{j t}-\mathbf{r}_{j *}\right\|_2^2=0.
\end{equation}
where the inner limits exist almost surely for each fixed $t$.
\end{Proposition}

Combining \Cref{prop:sds} and \Cref{SEovamp} yields the proof of \Cref{thm:empmain}. 

\begin{proof}[Proof of \Cref{thm:empmain}]
We prove \eqref{waka1} first. Fix function $\psi:\R^3 \mapsto \R$ satisfying, for some constant $C>0$, the pseudo-Lipschitz condition
$$
\left|\psi (\vbf)-\psi \left(\vbf^{\prime}\right)\right| \leq C\left(1+\|\vbf \|_2+\left\|\vbf^{\prime}\right\|_2\right)\left\|\vbf-\vbf^{\prime}\right\|_2.
$$ 
For any fixed $t$, we have
\begin{equation*}
\begin{aligned}
& \left|\frac{1}{p} \sum_{i=1}^p \psi\left(\hat{x}_{1 t, i}, r_{1 t, i}, \sti\right)-\frac{1}{p} \sum_{i=1}^p \psi\left(\hiatbti, r_{*, i}, \sti\right)\right| \\
& \leq \frac{C}{p} \sum_{i=1}^p\left(\left|\hat{x}_{1 t, i}-\hiatbti\right|^2+\left|r_{1 t, i}-r_{*, i}\right|^2\right)^{\frac{1}{2}}\\
& \qquad \qquad \times \left(1+\sqrt{\hat{x}_{1 t, i}^2+r_{1 t, i}^2+\beta_i^{\star 2}}+\sqrt{\hiatbti^2+r_{*, i}^2+\beta_i^{\star 2}}\right)\\
& \stackrel{(\star)}{\leq} C\left(\frac{1}{p} \sum_{i=1}^p\left|\hat{x}_{1 t, i}-\hiatbti\right|^2+\left|r_{1 t, i}-r_{*, i}\right|^2\right)^{\frac{1}{2}}\\
&\qquad \qquad \times \left(\frac{1}{p} \sum_{i=1}^p\left(1+\sqrt{\hat{x}_{1 t, i}^2+r_{1 t, i}^2+\beta_i^{\star 2}}+\sqrt{\hiatbti^2+r_{*, i}^2+\beta_i^{\star 2}}\right)^2\right)^{\frac{1}{2}} \\
& \leq C\left(\frac{1}{p}\left\|\xonet-\hatbt\right\|_2^2+\frac{1}{p}\left\|\xonet-\hatbt\right\|_2^2\right)^{\frac{1}{2}}\\
& \qquad \qquad \times \left(3+\frac{3}{p}\left(\left\|\xonet\right\|_2^2+3\left\|\ronet\right\|_2^2+2\left\|\st\right\|_2^2+2\left\|\ronet-\rstar\right\|_2^2\right)\right)^{\frac{1}{2}}
\end{aligned}
\end{equation*}
where $(\star)$ is by Cauchy-Schwarz inequality. This, along with \Cref{prop:sds}, \Cref{AssumpPrior}, \Cref{SEovamp} implies that
\begin{equation}\label{fkt1}
\lim _{t \rightarrow \infty} \lim _{p \rightarrow \infty}\left|\frac{1}{p} \sum_{i=1}^p \psi\left(\hat{x}_{1 t, i}, r_{1 t, i}, \sti\right)-\frac{1}{p} \sum_{i=1}^p \psi\left(\hiatbti, r_{*, i}, \sti\right)\right|=0
\end{equation}
Using \Cref{SEovamp} and \Cref{wasfact}, we have that
\begin{equation}\label{fkt2}
\lim _{p \rightarrow \infty}\left|\frac{1}{p} \sum_{i=1}^p \psi\left(\hat{x}_{1 t, i}, r_{1 t, i}, \sti\right)-\mathbb{E} \psi\left(\operatorname{Prox}_{\gamma_*^{-1} h}\left(\sqrt{\taustar} \Zs+\Xstar\right), \sqrt{\taustar} \Zs+\Xstar, \Xstar\right)\right|=0
\end{equation}
By triangle inequality, we also have
$$
\begin{aligned}
&\left|\mathbb{E} \psi\left(\operatorname{Prox}_{\gamma_*^{-1} h}\left(\sqrt{\taustar} \Zs+\Xstar\right), \sqrt{\taustar} \Zs+\Xstar, \Xstar\right)-\frac{1}{p} \sum_{i=1}^p \psi\left(\hiatbti, r_{*, i}, \sti\right)\right| \\
& \leq\left|\frac{1}{p} \sum_{i=1}^p \psi\left(\hat{x}_{1 t, i}, r_{1 t, i}, \sti\right)-\mathbb{E} \psi\left(\operatorname{Prox}_{\gamma_*^{-1} h}\left(\sqrt{\taustar} \Zs+\Xstar\right), \sqrt{\taustar} \Zs+\Xstar, \Xstar\right)\right| \\
&\qquad +\left|\frac{1}{p} \sum_{i=1}^p \psi\left(\hat{x}_{1 t, i}, r_{1 t, i}, \sti\right)-\frac{1}{p} \sum_{i=1}^p \psi\left(\hiatbti, r_{*, i}, \sti\right)\right|
\end{aligned}.
$$
Taking $p$ and then $t$ to infinity on both sides of the above, by \eqref{fkt1} and \eqref{fkt2}, 
\begin{equation*}
\lim _{p \rightarrow \infty}\left|\frac{1}{p} \sum_{i=1}^p \psi\left(\hiatbti, r_{*, i}, \sti\right)-\mathbb{E} \psi\left(\operatorname{Prox}_{\gamma_*^{-1} h}\left(\sqrt{\taustar} \Zs+\Xstar \right), \sqrt{\taustar} \Zs+\Xstar, \Xstar\right)\right|=0
\end{equation*}
where we used the fact that lhs does not depend on $t$. An application of \Cref{wasfact} with $\pfrak=2,k=3$ completes the proof for \eqref{waka1}. 

To see first result in \eqref{waka2}, note that
$$
\begin{aligned}
\bigg | \frac{1}{p}\left\|\X \rstst-\y\right\|^2 & -\frac{1}{p}\left\|\X \rtwot-\y\right\|^2 \bigg| = \bigg | \frac{1}{p}\left\langle \X \rstst-2\y+\X \rtwot, \X \rstst-\X \rtwot\right\rangle \bigg| \\
& \leq \frac{1}{p}\left\|\X \rstst-2\y+\X \rtwot\right\|_2\left\|\X \rstst-\X \rtwot\right\|_2 \\
& \leq \frac{1}{p}\left(\|\X\|_{\mathrm{op}}\left(\left\|\rstst-\st\right\|_2+\left\|\rtwot-\st\right\|\right)+2\|\epbm\|_2\right)\|\X\|_{\mathrm{op}}\left\|\rstst-\rtwot\right\|_2.
\end{aligned}
$$
Using this inequality and $\|\X\|_{\mathrm{op}}=\max _{i \in[p]}\left|d_i\right| \rightarrow \sqrt{d_{+}}$ (cf. \Cref{AssumpD}), we obtain that almost surely
\begin{equation}\label{win1}
    \lim _{t \rightarrow \infty} \limsup _{p \rightarrow \infty}\left|\frac{1}{p}\left\|\X \rstst-\y\right\|^2-\frac{1}{p}\left\|\X \rtwot-\y\right\|^2\right|=0.
\end{equation}
From triangle inequality, we have
$$
\begin{aligned}
\bigg | \frac{1}{p}\left\|\X \rstst-\y\right\|^2 & -\left(\tau_{**} \mathbb{E} \D^2+\delta\right) \bigg | \\
& \leq\left|\frac{1}{p}\left\|\X \rtwot-\y\right\|^2-\left(\tau_{**} \mathbb{E} \D^2+\delta\right)\right|+\left|\frac{1}{p}\left\|\X \rstst-\y\right\|^2-\frac{1}{p}\left\|\X \rtwot-\y\right\|^2\right|.
\end{aligned}
$$
Apply limit operation $\lim _{t \rightarrow \infty} \limsup _{p \rightarrow \infty}$ on both sides. Using \eqref{win1}, \eqref{consC} and the fact that the LHS does not depend on $t$, we have that almost surely
$$
\limsup _{p \rightarrow \infty}\left|\frac{1}{p}\left\|\X \rstst-\y\right\|^2-\left(\tau_{**} \mathbb{E} \D^2+\delta\right)\right|=0.
$$
The proof of the second result in \eqref{waka2} is analgous using \Cref{SEovamp}. This completes the proof.
\end{proof}

\subsection{Prove result B: Consistent estimation}\label{section:DEBIA}
We prove existence and uniqueness of the solution to the adjustment equation \eqref{gammasolvea} in \Cref{subsectionspsd}, show that the adjustment equation converges to a population limit in \Cref{subsectionWe}, and prove \Cref{NEIGMAIN} in \Cref{subsectionConsf}. To streamline the presentation, proofs of intermediate claims are collected in \Cref{SupportPII}. 

\subsubsection{Properties of the adjustment equation}\label{subsectionspsd}
Recall definition of function $g_p:(0,+\infty)\mapsto \R$ from \eqref{defgp}. We outline in \Cref{existslams} the conditions under which it is well-defined, strictly increasing and the equation 
\begin{equation}\label{solution}
        g_p(\gamma)=1
    \end{equation}
admits a unique solution on $(0,+\infty)$. The proof is deferred to \Cref{appendix:sampleadjeq}.

\begin{Lemma}\label{existslams}
    Fix $p\ge 1$. Assume that $h^{\prime \prime}(\hat{\beta_j})\ge 0$ for all $j\in [p]$. We then have the following statements:    
    \begin{itemize}
        \item [(a)] If $d_i\neq 0$ for all $i$, the function $\gamma\mapsto g_p(\gamma )$ is well-defined. If for some $i\in [p], d_i=0$, the function $\gamma\mapsto g_p(\gamma )$ is well-defined if and only if $\norm{h^{\prime\prime}(\hatbt)}_0>0$.
        \item [(b)] Given that $g_p$ is well-defined, it is strictly increasing if there exists some $j\in [p]$ such that $h^{\prime \prime}\left(\hjatbtj\right)\neq +\infty$, or else $ g_p(\gamma)=1,\forall \gamma \in (0,+\infty)$. 
        \item[(c)] Given that $\left\|h^{\prime \prime}(\hatbt)\right\|_0=p$ or for all $i, d_i\neq 0$, by which $g_p$ is well-defined from (a), \eqref{solution} has a unique solution if and only if there exists some $j\in [p]$ such that $h^{\prime \prime}\left(\hjatbtj\right)\neq +\infty$. 
        \item [(d)] Given that $\left\|h^{\prime \prime}(\hatbt)\right\|_0<p$ and for some $i$, $d_i= 0$, $g_p$ is well-defined and \eqref{solution} has a unique solution on $(0,+\infty)$ if and only if $\norm{d}_0+\norm{h^{\prime\prime}(\hatbt)}_0>p$. 
    \end{itemize}
\end{Lemma}

The following assumption is made to simplify the conditions outlined in \Cref{existslams}. 
\begin{Assumption}\label{Assumpgpweak}
    Fix $p\ge 1$ and suppose that \Cref{Assumph} holds. If $\left\|h^{\prime \prime}(\hatbt)\right\|_0=p$ or that $\X^\top  \X$ is non-singular, we require only that there exists some $i\in [p]$ such that $h^{\prime\prime}(\hiatbti)\neq +\infty$. Otherwise, we require in addition that $\norm{d}_0+\norm{h^{\prime\prime}(\hatbt)}_0>p$.
\end{Assumption}

The following is a direct consequence of \Cref{existslams} which in turn has \Cref{COR} as a special case. 
\begin{Proposition} \label{CORweak}
    Fix $p\ge 1$ and suppose that \Cref{Assumph} holds. Then, \Cref{Assumpgpweak} holds if and only if the function $\gamma\mapsto g_p(\gamma)$ is well-defined for any $\gamma>0$, strictly increasing, and the equation \eqref{solution} admits a unique solution contained in $(0,+\infty).$ 
\end{Proposition}

\subsubsection{Population limit of the adjustment equation}\label{subsectionWe}
From now on, we use notation for the following random variable
$$\Us:= h^{\prime \prime}\left(\operatorname{Prox}_{\gamma_*^{-1} h}\left(\sqrt{\taustar} \Zs+\Xstar\right)\right).$$
Define $g_\infty:(0,+\infty)\mapsto \R$ by
$$
g_\infty(\gamma)=\mathbb{E} \frac{1}{\left(\D^2-\gamma\right) \mathbb{E} \frac{1}{\gamma+\Us}+1}.
$$
which is well-defined under \Cref{Assumph}, \ref{Assumpfix} as shown in \Cref{well} below. We defer its proof to \Cref{appendix:popuadjeq}.

\begin{Lemma}\label{well}
   Under \Cref{Assumph}, \ref{Assumpfix}, $g_\infty$ is well-defined on  and strictly increasing on $(0,+\infty)$. The equation $g_\infty(\gamma)=1$ admits a unique solution $\gamma_*$ on $(0,+\infty)$. 
\end{Lemma}

\begin{Remark}
    We emphasize that the proof of \Cref{well} does not require \eqref{fp} admits a unique solution, only that a solution exists. 
\end{Remark}

We can show that the LHS of the sample adjustment equation converges to the LHS of the population adjustment equation. We defer its proof to \Cref{appendix:popuadjeq}.
\begin{Proposition}\label{ptconv}
    Under \Cref{AssumpD}---\ref{Assumpfix}, almost surely for all sufficiently large $p$, $g_p$ is well-defined and strictly increasing on $(0,+\infty)$ $g_p$ and equation \eqref{solution} admits a unique solution on $(0,+\infty)$. Furthermore, for any $\gamma>0$, almost surely,  
    \begin{equation}\label{desired}
        \lim_{p\to \infty} g_p(\gamma)=g_\infty (\gamma).
    \end{equation}
\end{Proposition}

\subsubsection{Consistent estimation of fixed points}\label{subsectionConsf}
% Recall definition of $\adj$ from \Cref{COR} and definitions of $\hat{\eta}_*, \hat{\tau}_*,  \hat{\tau}_{**}, \hatrstst$ from \eqref{DEFEFD}. We also define consistent estimator of $b_*$ and $\kappa_*$ as follows
% \begin{equation}\label{gammadef}
%     \begin{aligned}
% &\hat{b}_* (p):=\left(\frac{\hat{\eta}_*}{\adj}\right)^2 \frac{1}{p} \sum_{i=1}^p \frac{d_i^2}{\left(d_i^2+\hat{\eta}_*-\adj\right)^2} \\
% &\hat{\kappa}_*(p) :=\left(\frac{\hat{\eta}_*-\adj}{\adj}\right)^2\left(\frac{1}{p} \sum_{i=1}^p\left(\frac{\hat{\eta}_*}{d_i^2+\hat{\eta}_*-\adj}\right)^2-1\right)
% \end{aligned}
% \end{equation}
% whereupon we have $\hat{\tau}_{*}(p)=\hat{b}_*+\hat{\kappa}_* \hat{\tau}_{**}$. All of the above are well-defined under \Cref{Assumph}, \ref{Assumpgp} which ensures that $\adj, \hat{\eta}_*-\adj>0$ (cf. \Cref{COR} or \Cref{CORweak}).  On the other hand, if \Cref{Assumph}, \ref{Assumpgp} do not both hold, we say that the above estimators are not well-defined.

We are now ready to prove \Cref{neig} which shows that the quantities defined in \eqref{DEFEFD} indeed converges to their population counterparts.

\begin{proof}[Proof of \Cref{neig}]
We first show that $\lim _{p \rightarrow \infty} \adj\left(p\right) \rightarrow \gamma_*$ almost surely. Fix any $0<\epsilon<\gamma_*$. Note that almost surely
$$
\begin{aligned}
& \lim _{p\rightarrow \infty} g_{p}\left(\gamma_*-\epsilon\right)=g_{\infty}\left(\gamma_*-\epsilon\right)<g_{\infty}\left(\gamma_*\right)=1, \\
& \lim _{p \rightarrow \infty} g_{p}\left(\gamma_*+\epsilon\right)=g_{\infty}\left(\gamma_*+\epsilon\right)>g_{\infty}\left(\gamma_*\right)=1
\end{aligned}
$$
as a direct consequence of \Cref{ptconv} and that $g_\infty$ is strictly increasing (cf. \Cref{well}). It follows that almost surely for all $p$ sufficiently large
\begin{equation}\label{dddde}
    g_{p}\left(\gamma_*-\epsilon\right)<1, \quad g_{p}\left(\gamma_*+\epsilon\right)>1.
\end{equation}
Since $g_p$ is increasing and continuous almost surely for all sufficiently large $p$, \eqref{dddde} implies that almost surely for all $p$ sufficiently large $\left|\adj\left(p\right)-\gamma_*\right|<\epsilon$. This completes the proof for $\lim _{p\rightarrow \infty} \adj\left(p\right) \rightarrow \gamma_*$. The consistency of $\hat{\eta}_*$ immediately follows. To show $\bhetah(p) \stackrel{W_2}{\to} \Xstar+\sqrt{\taustar}\Zs$ almost surely as $p\to \infty$, note that $\bhetah(p)-\rstar \stackrel{W_2}{\to} 0$ by consistency of $\adj$ and $\limsup_{p\to \infty }p^{-1}\left\|\X^{\top}(\y-\X \hatbt)\right\|_2^2<+\infty$, and the claims follow from \eqref{waka1} and an application of \Cref{prop:combW}. A similar argument shows that $\frac{1}{p}\norm{\hatrstst-\rstst}^2\rightarrow 0$ almost surely as $p\to \infty$. The consistency statements for $\hat{\sigma}^2, \hat{\tau}_{**}, \hat{\tau}_{*}$ follow from results above, \eqref{waka2}, \eqref{RCd} and \Cref{Assumpfix2}. 
\end{proof}

\subsection{Supporting proofs for result A}\label{SupportPI}

\subsubsection{Oracle VAMP proofs}\label{appendix:ovamp}

\begin{proof}[Proof of \Cref{prop:ppt}]
By \Cref{AssumpD}, and \Cref{Assumpfix}, \eqref{RCc},
$$\lim _{p \rightarrow \infty} \frac{1}{p} \operatorname{Tr}(\bm{\Lambda})=\mathbb{E}\left(\eta_*-\gamma_*\right)\left(\frac{\eta_*}{\gamma_*\left(\D^2+\left(\eta_*-\gamma_*\right)\right)}-\frac{1}{\gamma_*}\right)=0.$$
The limiting values of $\kappa_*:=\lim _{p \rightarrow \infty} \frac{1}{p} \operatorname{Tr}\left(\bm{\Lambda}^2\right)$ and $b_*:=\lim _{p \rightarrow \infty} \frac{1}{p}\left\|\ebit \right\|^2$ is found analogously under \Cref{AssumpD}. The identity $\taustar=b_*+\kappa_* \tau_{**}$ is obtained by rewriting \eqref{RCd} using definitions of $b_*,\kappa_*$. Using \eqref{RCa}, we have that
$$
\mathbb{E} F^{\prime}\left(\sqrt{\taustar} \mathsf{Z}, \Xstar\right)=\frac{\eta_*}{\eta_*-\gamma_*}\left(\mathbb{E} \operatorname{Prox}_{\gamma_{*}^{-1} h}^{\prime}\left(\Xstar+\sqrt{\taustar} \mathsf{Z}\right)-\frac{\gamma_*}{\eta_*}\right)=0. 
$$
The Lipschitz continuity of $(q,x)\mapsto F(q,x)$ on $\R$ follows from \Cref{prop:proxp}, (b). To show $\mathbb{E} F\left(\sqrt{\taustar} \Zs, \Xstar\right)^2=\tau_{**}$, note that
$$
\begin{aligned}
& \mathbb{E} F\left(\sqrt{\taustar} \Zs, \Xstar\right)^2=\mathbb{E}\left(\frac{\eta_*}{\eta_*-\gamma_*}\left(\operatorname{Prox}_{\gamma_*^{-1} h}\left(\sqrt{\taustar} \Zs+\Xstar\right)-\Xstar\right)-\frac{\gamma_*}{\eta_*-\gamma_*} \sqrt{\taustar} \Zs\right)^2 \\
& =\left(\frac{\eta_*}{\eta_*-\gamma_*}\right)^2 \mathbb{E}\left(\operatorname{Prox}_{\gamma_*^{-1} h}\left(\sqrt{\taustar} \Zs+\Xstar\right)-\Xstar\right)^2+\left(\frac{\gamma_*}{\eta_*-\gamma_*}\right)^2 \tau_{*} \\
& \quad-2 \frac{\gamma_*}{\eta_*-\gamma_*} \frac{\eta_*}{\eta_*-\gamma_*} \mathbb{E}\left(\sqrt{\taustar} \Zs\left(\operatorname{Prox}_{\gamma_*^{-1} h}\left(\sqrt{\taustar} \Zs+\Xstar\right)-\Xstar\right)\right) \\
& \stackrel{(a)}{=}\left(\frac{\eta_*}{\eta_*-\gamma_*}\right)^2 \mathbb{E}\left(\operatorname{Prox}_{\gamma_*^{-1} h}\left(\sqrt{\taustar} \Zs+\Xstar\right)-\Xstar\right)^2\\
& \qquad \qquad +\left(\frac{\gamma_*}{\eta_*-\gamma_*}\right)^2 \tau_{*}-2 \frac{\gamma_*}{\eta_*-\gamma_*} \frac{\eta_*}{\eta_*-\gamma_*} \frac{\gamma_*}{\eta_*} \tau_{*} \\
& =\left(\frac{\eta_*}{\eta_*-\gamma_*}\right)^2 \mathbb{E}\left(\operatorname{Prox}_{\gamma_*^{-1} h}\left(\sqrt{\taustar} \Zs+\Xstar\right)-\Xstar\right)^2-\left(\frac{\gamma_*}{\eta_*-\gamma_*}\right)^2 \tau_{*} \\
& \stackrel{(b)}{=} \tau_{**}
\end{aligned}
$$
where in $(a)$ we used Stein's lemma and \eqref{RCa} for the following
$$
\mathbb{E}\left(\Zs\left(\operatorname{Prox}_{\gamma_*^{-1} h}\left(\sqrt{\taustar} \Zs+\Xstar\right)-\Xstar\right)\right)=\mathbb{E}\left(\operatorname{Prox}_{\gamma^{-1} h}^{\prime}\left(\sqrt{\taustar} \Zs+\Xstar\right)\right)=\frac{\gamma_*}{\eta_*}
$$
and in (b) we used \eqref{RCc}. We remark that although the function $x\mapsto \operatorname{Prox}_{\gamma_*^{-1} h}(x)$ may not be differentiable on a finite set of points, Stein's lemma can still be applied (cf. \cite[Lemma 1]{stein1981estimation}). 
\end{proof}

\begin{proof}[Proof of \Cref{prop:AMPparamconverge}]
Note that $\bm{\xi}=\Qbm \epbm \sim N(\rm{0},\mathbf{I}_{n})$. Then $\Dbm^\top \bm{\xi} \in \R^n$
may be written as the entrywise product of $\Dbm^\top \bm{1}_{n \times 1} \in \R^p$
and a vector $\bar{\bm{\xi}} \sim N(\rm{0},\mathbf{I}_{p})$, both when $p \geq n$
and when $n \leq p$. The almost-sure convergence $H \toW \Hs$ is then
a straightforward consequence of Propositions \ref{prop:iidW}, \ref{prop:contW},
and \ref{prop:orthoW}, where all random variables of $\Hs$ have finite moments
of all orders under Assumptions \ref{AssumpD} and \ref{AssumpPrior}.
The identities $\kappa_*=\E \Ls^2$ and $b_*=\E \Es_b^2$ follows from definitions of $\kappa_*,b_*$ in \Cref{prop:ppt}. 
\end{proof}

\begin{proof}[Proof of \Cref{thm:ampSE}]
We have $\delta_{11}=\E \Xs_1^2=\delta_*$ by the last identity of
(\ref{eq:denoiserpp2}). Supposing that $\delta_{tt}=\E \Xs_t^2=\delta_*$, we
have by definition $\E \Ys_t^2=\kappa_* \delta_{tt}=\sigma_*^2=\delta_*\kappa_*$.
Since $\Ys_t$ is independent of $\Es$, we have $\Ys_t+\Es \sim N(0,\sigma_*^2+b_*)$
where this variance is $\sigma_*^2+b_*=\taustar$ by last identity of \eqref{eq:somide}. Then $\E \Xs_{t+1}^2=\delta_*$ by the last identity of
(\ref{eq:denoiserpp2}), so $\E \Xs_t^2=\delta_*$ and $\E \Ys_t^2=\sigma_*^2$ for
all $t \geq 1$.

Noting that $\Delta_t$ is the upper-left submatrix of $\Delta_{t+1}$, let us
denote
\[\Delta_{t+1}=\begin{pmatrix} \Delta_t & \delta_t \\ \delta_t^\top & \delta_*
\end{pmatrix}\]
We now show by induction on $t$ the following three statements:
\begin{enumerate}
\item $\Delta_t \succ 0$ strictly.
\item We have
	\begin{equation}\label{eq:extraSE}
	    \mathsf{Y}_t=\sum_{k=1}^{t-1} \mathsf{Y}_{k}\left(\Delta_{t-1}^{-1}
\delta_{t-1}\right)_k+\mathsf{U}_t, \quad \mathsf{S}_t=\sum_{k=1}^{t-1}
\mathsf{S}_{k}\left(\Delta_{t-1}^{-1} \delta_{t-1}\right)_{k}+\mathsf{U}^{\prime}_t
	\end{equation}
	where $\mathsf{U}_t,\mathsf{U}_t'$ are Gaussian variables
with strictly positive variance, independent of
$\mathsf{H}$, $\left(\mathsf{Y}_{1}, \ldots, \mathsf{Y}_{t-1}\right)$, and
$\left(\mathsf{S}_{1}, \ldots, \mathsf{S}_{t-1}\right)$.
\item $\left(\mathbf{H},\X_{t+1}, \mathbf{S}_t,\mathbf{Y}_t\right)
\stackrel{W_2}{\to}\left(\mathsf{H},
\mathsf{X}_{1}, \ldots, \mathsf{X}_{t+1}, 
\mathsf{S}_{1}, \ldots, \mathsf{S}_{t}, \mathsf{Y}_{1}, \ldots,
\mathsf{Y}_{t}\right)$.
\end{enumerate}

We take as base case $t=0$, where the first two statements are vacuous,
and the third statement requires $(\mathbf{H},\ampxone) \toW
(\Hs,\Xs_1)$ almost surely as $p \to \infty$.
Recall that $\ampxone=F(\mathbf{p}^0,\st)$, and that
$F(p,\beta)$ is Lipschitz by Proposition \Cref{prop:ppt}.
Then this third statement follows from
Propositions \ref{prop:AMPparamconverge} and \ref{prop:contW}.

Supposing that these statements hold for some $t \geq 0$,
we now show that they hold for $t+1$. To show the first statement $\Delta_{t+1}
\succ 0$, note that for $t=0$ this follows from $\Delta_1=\delta_*>0$ by \Cref{Assumpfix}. For $t \geq 1$,
given that $\Delta_{t}\succ0$, $\Delta_{t+1}$ is singular if and only if there exist constants $\alpha_{1}, \ldots, \alpha_{t} \in \mathbb{R}$ such that
$$
\Xs_{t+1}=F\left(\mathsf{Y}_{t}+\mathsf{E}, \Xstar\right)=\sum_{r=1}^t \alpha_{r} \mathsf{X}_{r}
$$
with probability 1. From the induction hypothesis,
$\mathsf{Y}_{t}=\sum_{k=1}^{t-1} \mathsf{Y}_{k}\left(\Delta_{r}^{-1}
\delta_{r}\right)_{k}+\mathsf{U}_t$ where $\mathsf{U}_t$ is independent of
$\mathsf{H},\mathsf{Y}_{1}, \ldots, \mathsf{Y}_{t-1}$ and hence also of
$\Es,\Xstar,\mathsf{X}_1,...,\mathsf{X}_t$. We now show that for any realized values
$(e_{0}, x_{0}, w_{0})$ of $$\left(\mathsf{E}+\sum_{k=1}^{t-1}
\mathsf{Y}_{k}\left(\Delta_{r}^{-1} \delta_{r}\right)_{k}, \quad \Xstar, \quad
\sum_{r=1}^t \alpha_{r} \mathsf{X}_{r}\right),$$ we have that
$\mathbb{P}\left(F\left(\mathsf{U}_t+e_{0}, x_{0}\right) \neq w_{0}\right)>0$. This would imply that $\Delta_{t+1}\succ 0$. Suppose to the contrary, we then have that 
$$
\mathbb{P}\left(\frac{\eta_*}{\eta_*-\gamma_*} \operatorname{Prox}_{\gamma_{\gamma_*^{-1}}}\left(\mathsf{U}_t+e_0+x_0\right)-\frac{\gamma_*}{\eta_*-\gamma_*} \mathsf{U}_t=w_0+\frac{\eta_*}{\eta_*-\gamma_*} x_0+\frac{\gamma_*}{\eta_*-\gamma_*} e_0\right)=1.
$$
Since $\Us_t$ is Gaussian with strictly positive variance, the above implies that the function
$$
u \mapsto \frac{\eta_*}{\eta_*-\gamma_*} \operatorname{Prox}_{\gamma_*^{-1} h}\left(u+e_0+x_0\right)-\frac{\gamma_*}{\eta_*-\gamma_*} u
$$
is constant almost everywhere. This in turn is equivalent to that
$
\operatorname{Prox}_{\gamma_*^{-1} h}(u)=C+\frac{\gamma_*}{\eta_*} u
$
almost everywhere for some constant $C\in \R $ by a change of variable. Noting that $u \mapsto \operatorname{Prox}_{\gamma_*^{-1} h}(u)$ is continuous, we thus have that $\operatorname{Prox}_{\gamma_*^{-1} h}(u)=C+\frac{\gamma_*}{\eta_*} u$ for all $u \in \mathbb{R}$. This implies that $\operatorname{Prox}_{\gamma_*^{-1} h}(u)$ is continuously differentiable and has constant derivative $\frac{\gamma_*}{\eta_*}$, which contradicts to the assumption that $x\mapsto \operatorname{Prox}_{\gamma_*^{-1} h}^{\prime}(x)$ is non-constant. We thus have proved the first inductive statement that $\Delta_{t+1}\succ 0$.

 To study the empirical limit of $s_{t+1}$, let $\mathbf{U}=\left(\mathbf{e}_{b}, \mathbf{S}_{t},
\bm{\Lambda} \mathbf{S}_{t}\right)$ and $\mathbf{V}=\left(\mathbf{e},\X_{t}, \mathbf{Y}_{t}\right)$.
(For $t=0$, this is simply $\mathbf{U}=\ebit $ and $\mathbf{V}=\mathbf{e}$.) By the induction
hypothesis, the independence of $(\Ss_1,\ldots,\Ss_t)$ with $(\Es_b,\Ls)$,
and the identities $\E \Es_b^2=b_*$ and $\E \Ls=0$ and $\E \Ls^2=\kappa_*$,
almost surely as $p \to \infty$,
\begin{equation*}
	\frac{1}{p}\left(\mathbf{e}_{b}, \mathbf{S}_{t}, \bm{\Lambda} \mathbf{S}_{t}\right)^{\top}\left(\mathbf{e}_{b}, \mathbf{S}_{t}, \bm{\Lambda} \mathbf{S}_{t}\right) \rightarrow\left(\begin{array}{ccc}
		b_{*} & 0 & 0 \\
		0 & \Delta_{t} & 0 \\
		0 & 0 & \kappa_{*} \Delta_{t}
	\end{array}\right)\succ0
\end{equation*}
So almost surely for sufficiently large $p$, conditional on
$(\mathbf{H},\X_{t+1},\mathbf{S}_t,\mathbf{Y}_t)$, the law of $\ampstpo$ is given by its law conditioned on
$\mathbf{U}=\Obm V$, which is (see \cite[Lemma B.2]{fan2022tap})
\begin{equation}\label{eq:SEconv3}
	\ampstpo\big|_{\mathbf{U}=\Obm \mathbf{V}}=\Obm \ampxtpo\big|_{\mathbf{U}=\Obm \mathbf{V}} \stackrel{L}{=}\bUm\left(\bUm^\top
\bUm\right)^{-1} \Vbf^{\top} \ampxtpo+ \bm{\Pi}_{\mathbf{U}^{\perp}} \tilde{\Obm} \bm{\Pi}_{\mathbf{V}^{\perp}}^{\top} \ampxtpo
\end{equation}
where $\tilde{\Obm} \sim \Haar(\mathbb{O}(p-(2 t+1)))$ and $\bm{\Pi}_{\mathbf{U}^{\perp}},
\bm{\Pi}_{\mathbf{V}^{\perp}} \in \mathbb{R}^{p \times(p-(2 t+1))}$ are matrices with
orthonormal columns spanning the orthogonal complements of the column spans of
$\mathbf{U},\mathbf{V}$ respectively. We may replace $\ampstpo$ by the right side of
\eqref{eq:SEconv3} without affecting the joint law of $\left(\mathbf{H},\X_{t+1}, \mathbf{S}_t,
\mathbf{Y}_{t},\ampstpo\right)$.

For $t=0$, we have $\E \Xs_1\Es=0$ since $\Xs_1$ is independent of $\Es$.
For $t \geq 1$, by the definition of $\Xs_{t+1}$, the
condition $\E F'(\Ps,\Xstar)=0$ from
(\ref{eq:denoiserpp2}), and Stein's lemma, we have $\E \Xs_{t+1}\Es=0$
and $\E \Xs_{t+1}\Ys_r=0$ for each $r=1,\ldots,t$. Then
by the induction hypothesis, almost surely as $p \rightarrow \infty$,
$$
\left(p^{-1}\bUm^\top \bUm\right)^{-1} \rightarrow\left(\begin{array}{ccc}
	b_{*} & 0 & 0 \\
	0 & \Delta_{t} & 0 \\
	0 & 0 & \kappa_{*} \Delta_{t}
\end{array}\right)^{-1}, \quad p^{-1}\Vbf^{\top} \mathbf{x}^{t+1} \rightarrow\left(\begin{array}{c}
	0 \\
	\delta_{t} \\
	0
\end{array}\right).
$$
Then by (\ref{eq:SEconv3}) and
Propositions \ref{prop:combW} and \ref{prop:orthoW}, it follows that
$$
\begin{aligned}
    &\left(\mathbf{H},\X_{t+1}, \mathbf{S}_{t}, \mathbf{Y}_{t}, \ampstpo\right) \\
    &\qquad \qquad \stackrel{W_2}{\to}\left(\mathsf{H}, \mathsf{X}_{1}, \ldots, \mathsf{X}_{t+1}, \mathsf{S}_{1}, \ldots, \mathsf{S}_{t}, \mathsf{Y}_{1}, \ldots \mathsf{Y}_{t}, \sum_{r=1}^{t} \mathsf{S}_{r}\left(\Delta_{t}^{-1} \delta_{t}\right)_{r}+\mathsf{U}^{\prime}_{t+1}\right)
\end{aligned}
$$
where $\Us^{\prime}_{t+1}$ is the Gaussian limit of the second term on the right
side of (\ref{eq:SEconv3}) and is independent of $\mathsf{H},
\mathsf{X}_{1}, \ldots, \mathsf{X}_{t+1}, \mathsf{S}_{1}, \ldots,
\mathsf{S}_{t}, \mathsf{Y}_{1}, \ldots \mathsf{Y}_{t}$.
We can thus set $\Ss_{t+1}:=\sum_{r=1}^{t} \Ss_{r}\left(\Delta_{t}^{-1}
\delta_{t}\right)_{r}+\Us^{\prime}_{t+1}$. Then $(\Ss_1,\ldots,\Ss_{t+1})$ is
multivariate Gaussian and remains independent of $\Hs$ and
$(\Ys_1,\ldots,\Ys_t)$. Since $p^{-1}\|\ampstpo\|^{2}=p^{-1}\|\ampxtpo\|^{2}
\rightarrow \delta_*$ almost surely as $p \rightarrow \infty$ by the induction
hypothesis, we have $\E\mathsf{S}_{t+1}^2=\delta_*$.
From the form of $\Ss_{t+1}$, we may check also
$\E\Ss_{t+1}(\Ss_1,\ldots,\Ss_t)=\delta_t$, so 
$(\Ss_1,\ldots,\Ss_{t+1})$ has covariance $\Delta_{t+1}$ as desired.
Furthermore $\sum_{r=1}^{t} \mathsf{~S}_{r}\left(\Delta_{t}^{-1}
\delta_{t}\right)_{r} \sim N\left(0, \delta_{t}^{\top} \Delta_{t}^{-1}
\delta_{t}\right)$. From $\Delta_{t+1} \succ 0$ 
and the Schur complement formula, $\delta_*-\delta_{t}^{\top} \Delta_{t}^{-1}
\delta_{t}>0$ strictly.
Then $\Us^{\prime}_{t+1}$ has strictly positive variance, since the variance of $\sum_{r=1}^{t} \mathsf{S}_{r}\left(\Delta_{t}^{-1} \delta_{t}\right)_{r}$ is less than the variance of $\Ss_{t+1}$. 
This proves the second equation in \eqref{eq:extraSE} for $t+1$.

Now, we study the empirical limit of $\ampytpo$. Let $\mathbf{U}=\left(\mathbf{e},\X_{t+1},
\mathbf{Y}_{t}\right)$, $\mathbf{V}=\left(\mathbf{e}_{b}, \mathbf{S}_{t+1}, \bm{\Lambda} \mathbf{S}_{t}\right)$. Similarly
by the induction hypothesis and the empirical convergence of $(\mathbf{H},\mathbf{S}_{t+1})$
already shown, almost surely as $p \rightarrow \infty$,
$$
\frac{1}{p}\left(\ebit , \mathbf{S}_{t+1}, \bm{\Lambda} \mathbf{S}_{t}\right)^{\top}\left(\ebit , \mathbf{S}_{t+1}, \bm{\Lambda} \mathbf{S}_{t}\right) \rightarrow\left(\begin{array}{ccc}
	b_{*} & 0 & 0 \\
	0 & \Delta_{t+1} & 0 \\
	0 & 0 & \kappa_{*} \Delta_{t}
\end{array}\right)\succ 0.
$$
Then the law of $\ampytpo$ conditional on $(\mathbf{H},\X_{t+1},\mathbf{S}_{t+1},\mathbf{Y}_t)$ is given by
its law conditioned on $\mathbf{U}=\Obm^\top \Vbf$, which is
\begin{equation}\label{eq:SEconv4}
	\ampytpo\big|_{\mathbf{U}=\Obm^\top \Vbf}=\Obm^{\top}\bm{\Lambda} \ampstpo\big|_{\mathbf{U}=\Obm^{\top} \Vbf}  \stackrel{L}{=}
\bUm\left(\Vbf^{\top} \Vbf\right)^{-1} \Vbf^{\top} \bm{\Lambda} \ampstpo+\bm{\Pi}_{\mathbf{U}^{\perp}} \tilde{\Obm}
\bm{\Pi}_{\mathbf{V}^{\perp}}^{\top} \bm{\Lambda} \ampstpo
\end{equation}
where $\tilde{\Obm} \sim \Haar(\mathbb{O}(p-(2 t+2)))$. From the convergence of
$(\mathbf{H},\mathbf{S}_{t+1})$ already shown, almost surely as $p \rightarrow \infty$,
$$
\left(n^{-1} \Vbf^{\top} \Vbf\right)^{-1} \rightarrow\left(\begin{array}{ccc}
	b_{*} & 0 & 0 \\
	0 & \Delta_{t+1} & 0 \\
	0 & 0 & \kappa_{*} \Delta_{t}
\end{array}\right)^{-1}, \quad n^{-1} \Vbf^{\top} \bm{\Lambda} \ampstpo \rightarrow\left(\begin{array}{c}
	0 \\
	0 \\
	\kappa_{*} \delta_{t}
\end{array}\right).
$$
Then by (\ref{eq:SEconv4}) and Propositions \ref{prop:combW} and
\ref{prop:orthoW},
$$
\begin{aligned}
    &\left(\mathbf{H},\X_{t+1}, \mathbf{S}_{t+1}, \mathbf{Y}_{t}, \ampytpo\right) \\ & \qquad \qquad \stackrel{W_2}{\to}\left(\mathsf{H}, \mathsf{X}_{1}, \ldots, \mathsf{X}_{t+1}, \mathsf{~S}_{1}, \ldots, \mathsf{S}_{t+1}, \mathsf{Y}_{1}, \ldots \mathsf{Y}_t, \sum_{r=1}^{t} \mathsf{Y}_{r}\left(\Delta_{t}^{-1} \delta_{t}\right)_{r}+\mathsf{U}_{t+1}\right)
\end{aligned}
$$
where $\Us_{t+1}$ is the limit of the second term on the right side of
(\ref{eq:SEconv4}), which is
Gaussian and independent of $ \mathsf{H},\Ss_{1}, \ldots,
\Ss_{t+1}, \Ys_{1}, \ldots \Ys_{t}$. Setting
$\Ys_{t+1}:=\sum_{r=1}^{t} \Ys_{r}\left(\Delta_{t}^{-1}
\delta_{t}\right)_{r}+\Us_{t+1}$, it follows that $(\Ys_1,\ldots,\Ys_{t+1})$
remains independent of $\Hs$ and $(\Ss_1,\ldots,\Ss_{t+1})$.
We may check that $\E \Ys_{t+1}(\Ys_1,\ldots,\Ys_t)=\kappa_* \delta_t$, and
we have also $n^{-1}\|\ampytpo\|^{2}=n^{-1}\|\bm{\Lambda}
\ampstpo\|^{2} \rightarrow \kappa_{*}\delta_*$ so $\E\mathsf{Y}_{t+1}^2
=\kappa_{*}\delta_*$. From $\Delta_{t+1} \succ 0$ and the Schur complement
formula, note that $\sum_{r=1}^{t}
\mathsf{Y}_{r}\left(\Delta_{t}^{-1} \delta_{t}\right)_{r}$ has variance
$\kappa_{*} \delta_{t}^{\top} \Delta_{t}^{-1} \delta_{t}$ which is strictly
smaller than $\kappa_* \delta_*$, so $\Us_{t+1}$ has strictly positive variance.
This proves the first equation in \eqref{eq:extraSE} for $t+1$, and completes
the proof of this second inductive statement.

Finally, recall $\mathbf{x}^{t+2}=F\left(\ampytpo+\mathbf{e}, \st\right)$ where $F$ is Lipschitz.
Then by Proposition \ref{prop:contW}, almost surely
$$
\left(\mathbf{H},\X_{t+2}, \mathbf{S}_{t+1}, \mathbf{Y}_{t+1}\right)
\stackrel{W_2}{\to}\left(\mathsf{H}, \mathsf{X}_{1}, \ldots,
\mathsf{X}_{t+2}, \mathsf{~S}_{1}, \ldots, \mathsf{S}_{t+1}, \mathsf{Y}_{1},
\ldots, \mathsf{Y}_{t+1}\right)
$$
where $\mathsf{X}_{t+2}=F\left(\mathsf{Y}_{t+1}+\mathsf{E}, \Xstar\right)$,
showing the third inductive statement and completing the induction.
\end{proof}

\begin{proof}[Proof of \Cref{SEovamp}]
    \eqref{consA} is a direct consequence of \Cref{thm:ampSE}, \eqref{changeofv}, \Cref{prop:contW}, \Cref{prop:combW}, \eqref{x1t} and the fact that proximal map is 1-Lipschitz. To see the first result in \eqref{consC}, note that
    $$
\X \rtwot-\y=\Qbm^{\top} \Dbm \Obm\left(\rtwot-\st\right)-\epbm=\Qbm^{\top} \Dbm \ampst-\epbm
$$
and thus almost surely
$$
\lim _{p \rightarrow \infty} \frac{1}{p}\left\|\X \rtwot-\y\right\|^2=\lim _{p \rightarrow \infty} \frac{1}{p} (\ampst)^{\top} \Dbm^{\top} \Dbm\ampst+\frac{1}{p}\|\epbm\|_2^2- \frac{2}{p} (\ampst)^{\top} \Dbm^{\top} \Qbm \epbm=\tau_{**} \mathbb{E}\D^2+\delta.
$$

To see the second result in \eqref{consC}, we first note the identities
\begin{equation}\label{identv1}
    \hat{\mathbf{x}}_{2 t}-\st=\mathbf{O}^{\top}\left(\mathbf{D}^{\top} \mathbf{D}+\left(\eta_*-\gamma_*\right) \cdot \mathbf{I}_p\right)^{-1}\left[\mathbf{D}^{\top} \mathbf{Q} \epbm+\left(\eta_*-\gamma_*\right) \cdot \mathbf{s}^t\right]
\end{equation}
We also have that
\begin{equation}\label{identv2}
    \begin{aligned}
        \frac{1}{p} \norm{\y-\X \hat{\mathbf{x}}_{2 t}}^2=\frac{1}{p}\norm{\epbm}^2&+\frac{1}{p}\qty(\hat{\mathbf{x}}_{2 t}-\st)\Obm^\top\Dbm^\top \Dbm \Obm(\hat{\mathbf{x}}_{2 t}-\st)\\
        & -2\epbm^\top \Qbm^\top \Dbm \Obm(\hat{\mathbf{x}}_{2 t}-\st)
    \end{aligned}
\end{equation}
Using \eqref{identv1}, \eqref{identv2} above, \Cref{AssumpD},\ref{AssumpPrior}, and \Cref{thm:ampSE}, we obtain that almost surely as $p\to\infty$
\begin{equation}
    \frac{1}{p} \norm{\y-\X \hat{\mathbf{x}}_{2 t}}^2 \to \tau_{**} \cdot \E \frac{\D^2(\eta_*-\gamma_*)^2}{(\D^2+\eta_*-\gamma_*)^2}+\frac{n-p}{p}+\E \qty(\frac{\eta_*-\gamma_*}{\D^2+\eta_*-\gamma_*})^2
\end{equation}
as required. 
\end{proof}

\begin{proof}[Proof of \Cref{prop:convsmallbeta}] 
Recall that $\delta_{tt}=\delta_*$ for all $t \geq 1$ from Theorem
\ref{thm:ampSE}. Then
	$\delta_{s t}=\mathbb{E}\left[\mathsf{X}_{s} \mathsf{X}_{t}\right] \leq
\sqrt{\mathbb{E}\left[\mathsf{X}_{s}^{2}\right]
\mathbb{E}\left[\mathsf{X}_{t}^{2}\right]} =\delta_*$ for all $s,t \geq 1$.
For $s=1$ and any $t \geq 2$, observe also that
\begin{equation}\label{eq:delta1t}
	\begin{gathered}
		\delta_{1t}=\mathbb{E} \Xs_{1} \mathsf{X}_{t}=\mathbb{E}\left[F\left(\mathsf{P}_{0}, \Xstar\right) F\left(\mathsf{Y}_{t-1}+\mathsf{E}, \Xstar\right)\right]=\mathbb{E}\left[\mathbb{E}\left[F\left(\mathsf{P}_{0}, \Xstar\right) F\left(\mathsf{Y}_{t-1}+\mathsf{E}, \Xstar\right) \mid \Xstar\right]\right] \\
		=\mathbb{E}[\mathbb{E}\left[F\left(\mathsf{P}_{0}, \Xstar\right) \mid \Xstar\right]^{2}] \geq 0
	\end{gathered}
\end{equation}
where the last equality holds because
$\mathsf{P}_{0}$, $\mathsf{Y}_{t-1}+\Es$, and $\Xstar$ are independent,
with $\Ps_0$ and $\Ys_{t-1}+\Es$ equal in law
(by the identity $\sigma_*^2+b_*=\taustar$). Consider now the map $\delta_{st} \mapsto \delta_{s+1,t+1}$.
Recalling that $\E \Ys_t^2=\sigma_{*}^{2}$ and $\E \Ys_s\Ys_t=\kappa_*
\delta_{st}$, we may represent
	$$
	\left(\mathsf{Y}_{s}+\mathsf{E}, \mathsf{Y}_{t}+\mathsf{E}\right)\stackrel{L}{=}\left(\sqrt{\kappa_{*} \delta_{s t}+b_{*}} \mathsf{G}+\sqrt{\sigma_{*}^{2}-\kappa_{*} \delta_{s t}} \mathsf{G}^{\prime}, \sqrt{\kappa_{*} \delta_{s t}+b_{*}} \mathsf{G}+\sqrt{\sigma_{*}^{2}-\kappa_{*} \delta_{s t}} \mathsf{G}^{\prime \prime}\right)
	$$
	where $\mathsf{G}, \mathsf{G}^{\prime}, \mathsf{G}^{\prime \prime}$ are
jointly independent standard Gaussian variables.
Denote
$$\mathsf{P}_{\delta}^{\prime}:=\sqrt{\kappa_{*} \delta+b_{*}} \cdot \mathsf{G}+\sqrt{\sigma_{*}^{2}-\kappa_{*} \delta} \cdot \mathsf{G}^{\prime}, \quad \mathsf{P}_{\delta}^{\prime \prime}:=\sqrt{\kappa_{*} \delta+b_{*}} \cdot \mathsf{G}+\sqrt{\sigma_{*}^{2}-\kappa_{*} \delta} \cdot \mathsf{G}^{\prime \prime}$$
and define $g:\left[0,\delta_*\right] \to \mathbb{R}$ by
$g(\delta):=\mathbb{E}\left[F\left(P_\delta',\Xstar\right)
F\left(P_\delta'',\Xstar\right)\right]$. Then $\delta_{s+1,t+1}=g(\delta_{st})$.

We claim that for any $\delta \in [0,\delta_*]$, we have $g(\delta) \geq 0$,
$g'(\delta) \geq 0$, and $g''(\delta) \geq 0$. The first bound $g(\delta) \geq 0$
follows from
$$
g(\delta)=\mathbb{E}\Big[\E[F\left(\mathsf{P}_{\delta}^{\prime}, \Xstar\right)
F\left(\mathsf{P}_{\delta}^{\prime \prime}, \Xstar\right) \mid
\Xstar,\Gs]\Big]=\mathbb{E}\left[\mathbb{E}\left[F\left(\mathsf{P}_{\delta}^{\prime},
\Xstar\right) \mid \Xstar, \mathsf{G}\right]^{2}\right] \geq 0,
$$ 
because $\Ps_\delta',\Ps_\delta''$ are independent and equal in law conditional
on $\Gs,\Xstar$. Differentiating in $\delta$ and applying
Gaussian integration by parts,
	$$
	\begin{aligned}
		&g^{\prime}(\delta) =2 \mathbb{E}\left[F^{\prime}\left(\mathsf{P}_{\delta}^{\prime}, \Xstar\right) F\left(\mathsf{P}_{\delta}^{\prime \prime}, \Xstar\right)\left(\frac{\kappa_{*}}{2 \sqrt{\kappa_{*} \delta+b_{*}}} \cdot \mathsf{G}-\frac{\kappa_{*}}{2 \sqrt{\sigma_{*}^{2}-\kappa_{*} \delta}} \cdot \mathsf{G}^{\prime}\right)\right] \\
		&=\frac{\kappa_{*}}{\sqrt{\kappa_{*} \delta+b_{*}}} \mathbb{E}\left[F^{\prime}\left(\mathsf{P}_{\delta}^{\prime}, \Xstar\right) F\left(\mathsf{P}_{\delta}^{\prime \prime}, \Xstar\right) \mathsf{G}\right]-\frac{\kappa_{*}}{\sqrt{\sigma_{*}^{2}-\kappa_{*} \delta}} \mathbb{E}\left[F^{\prime}\left(\mathsf{P}_{\delta}^{\prime}, \Xstar\right) F\left(\mathsf{P}_{\delta}^{\prime \prime}, \Xstar\right) \mathsf{G}^{\prime}\right] \\
		&=\kappa_{*}\mathbb{E}\left[F^{\prime \prime}\left(\mathsf{P}_{\delta}^{\prime}, \Xstar\right) F\left(\mathsf{P}_{\delta}^{\prime \prime}, \Xstar\right)+F^{\prime}\left(\mathsf{P}_{\delta}^{\prime}, \Xstar\right) F^{\prime}\left(\mathsf{P}_{\delta}^{\prime \prime}, \Xstar\right)\right]
		-\kappa_{*}\mathbb{E}\left[F^{\prime \prime}\left(\mathsf{P}_{\delta}^{\prime}, \Xstar\right) F\left(\mathsf{P}_{\delta}^{\prime \prime}, \Xstar\right)\right] \\
		&=\kappa_{*}
\mathbb{E}\left[F^{\prime}\left(\mathsf{P}_{\delta}^{\prime}, \Xstar\right)
F^{\prime}\left(\mathsf{P}_{\delta}^{\prime \prime}, \Xstar\right)\right].
	\end{aligned}
	$$
Then $g'(\delta)=\kappa_*\E\left[\E[F'(\Ps_\delta',\Xstar) \mid
\Gs,\Xstar]^2\right] \geq 0$, and a similar argument shows $g''(\delta) \geq 0$. Observe that at $\delta=\delta_*$, we have
$\mathsf{P}_{\delta_{*}}^{\prime}=\mathsf{P}_{\delta_{*}}^{\prime
\prime}=\sqrt{\sigma_*^2+b_{*}} \cdot \mathsf{G}=\sqrt{\taustar} \mathsf{G}$
which is equal in law to $\Ps \sim N(\rm{0},\taustar)$. Then
$g(\delta_{*})=\mathbb{E}[F(\Ps,\Xstar)^{2}]=\delta_{*}$
by \Cref{prop:ppt}. So $g:[0,\delta_*] \to
[0,\delta_*]$ is a non-negative, increasing, convex function with a fixed point
at $\delta_*$. We claim that
\begin{equation}\label{eq:gprimebound}
g'(\delta_*)<1
\end{equation}
This then implies that $\delta_*$ is the unique fixed point of $g(\cdot)$
over $[0,\delta_*]$, and $\lim_{t \to \infty} g^{(t)}(\delta)=\delta_*$ for
any $\delta \in [0,\delta_*]$. Observe from (\ref{eq:delta1t})
that $\delta_{1t}=\delta_{12}$ for all $t \geq 2$, so
$\delta_{t,t+s}=g^{(t-1)}(\delta_{1,1+s})=g^{(t-1)}(\delta_{12})$
for any $s \geq 1$. Then $\lim_{\min(s,t) \to \infty} \delta_{st}=\delta_*$
follows.

It remains to show (\ref{eq:gprimebound}). Using \Cref{prop:ppt},
\begin{equation}\label{combo1}
    \begin{aligned}
g^{\prime}\left(\delta_*\right)=\kappa_* \mathbb{E}[ & \left.F^{\prime}\left(\mathsf{P}_{\delta_*}^{\prime}, \Xstar\right)^2\right]=\kappa_* \mathbb{E}\left[\left(\frac{\eta_*}{\eta_*-\gamma_*} \operatorname{Prox}_{\gamma_*^{-1} h}^{\prime}\left(\mathsf{P}_{\delta_*}^{\prime}+\Xstar\right)-\frac{\gamma_*}{\eta_*-\gamma_*}\right)^2\right] \\
& =\left(\frac{\eta_*}{\eta_*-\gamma_*}\right)^2 \kappa_* \mathbb{E}\left[\left(\operatorname{Prox}_{\gamma_*^{-1} h}^{\prime}\left(\mathsf{P}_{\delta_*}^{\prime}+\Xstar\right)-\frac{\gamma_*}{\eta_*}\right)^2\right] \\
& =\left(\frac{\eta_*}{\gamma_*}\right)^2\left(\mathbb{E} \frac{\eta_*^2}{\left(\D^2+\eta_*-\gamma_*\right)^2}-1\right) \mathbb{E}\left[\left(\operatorname{Prox}_{\gamma_*^{-1} h}^{\prime}\left(\mathsf{P}_{\delta_*}^{\prime}+\Xstar\right)\right)^2-\left(\frac{\gamma_*}{\eta_*}\right)^2\right].
\end{aligned}
\end{equation}
Using \Cref{lem:cauchy} (c), we obtain that
\begin{equation}\label{combo2}
\begin{aligned}
& R^{\prime}\left(\eta_*^{-1}\right)=-\left(\mathbb{E} \frac{1}{\left(\D^2+\eta_*-\gamma_*\right)^2}\right)^{-1}+\eta_*^2 \implies  \frac{\eta_*^2}{\eta_*^2-R^{\prime}\left(\eta_*^{-1}\right)}=\mathbb{E} \frac{\eta_*^2}{\left(\D^2+\eta_*-\gamma_*\right)^2}.
\end{aligned}
\end{equation}
Note also that by Jensen's inequality and \eqref{RCc} that 
\begin{equation}\label{combo5}
    \mathbb{E} \frac{\eta_*^2}{\left(\D^2+\eta_*-\gamma_*\right)^2}-1\ge 0
\end{equation}
By \Cref{noid} and \eqref{eq:Jacprox}, we have
\begin{equation*}
\mathbb{E}\left[\left(\operatorname{Prox}_{\gamma_*^{-1} h}^{\prime}\left(\mathsf{P}_{\delta_*}^{\prime}+\Xstar\right)\right)^2\right]<\mathbb{E} \operatorname{Prox}_{\gamma_*^{-1} h}^{\prime}\left(\mathsf{P}_{\delta_*}^{\prime}+\Xstar\right)=\frac{\gamma_*}{\eta_*}.
\end{equation*}
This implies that
\begin{equation}\label{combo3}
0 \leq \mathbb{E}\left[\left(\operatorname{Prox}_{\gamma_*^{-1} h}^{\prime}\left(\mathsf{P}_{\delta_*}^{\prime}+\Xstar\right)\right)^2-\left(\frac{\gamma_*}{\eta_*}\right)^2\right]<\frac{\gamma_*}{\eta_*}-\left(\frac{\gamma_*}{\eta_*}\right)^2.
\end{equation}
Combining \eqref{combo1},\eqref{combo2},\eqref{combo5} and \eqref{combo3} above, we obtain that
$$
g^{\prime}\left(\delta_*\right)<\left(\frac{R^{\prime}\left(\eta_*^{-1}\right)}{\eta_*^2-R^{\prime}\left(\eta_*^{-1}\right)}\right)\left(\frac{\eta_*}{\gamma_*}-1\right).
$$
To show the rhs is less than 1, we observe that
\begin{equation}\label{combb}
    \begin{aligned}
&\left(\frac{R^{\prime}\left(\eta_*^{-1}\right)}{\eta_*^2-R^{\prime}\left(\eta_*^{-1}\right)}\right)\left(\frac{\eta_*}{\gamma_*}-1\right)<1 \Leftrightarrow \frac{R^{\prime}\left(\eta_*^{-1}\right)}{\eta_*^2-R^{\prime}\left(\eta_*^{-1}\right)}<\frac{\eta_* \gamma_*}{\eta_*^2-\eta_* \gamma_*} \stackrel{(i)}{\Leftrightarrow} R^{\prime}\left(\eta_*^{-1}\right) \\
&<\eta_* \gamma_* \stackrel{(i i)}{\Leftrightarrow}-\frac{\eta_*^{-1} R^{\prime}\left(\eta_*^{-1}\right)}{R\left(\eta_*^{-1}\right)}<1
\end{aligned}
\end{equation}
where in $(i)$ we used that $x \mapsto \frac{x}{\eta_*^2-x}$ is strictly increasing and in $(ii)$ we used \eqref{RCc}. Finally, we conclude the proof by noting that the rhs of \eqref{combb} holds true by \Cref{lem:cauchy}, (d).
\end{proof}

\begin{proof}[Proof of \Cref{Corc}]
Note that 
$$
\begin{aligned}
& \lim _{(s, t) \rightarrow \infty}\left(\lim _{p \rightarrow \infty} \frac{1}{p}\left\|\ampxt-\mathbf{x}^s\right\|^2\right)=\lim _{(s, t) \rightarrow \infty}\left(\delta_{s s}+\delta_{t t}-2 \delta_{s t}\right)=0 \\
&\lim _{(s, t) \rightarrow \infty}\left(\lim _{p \rightarrow \infty} \frac{1}{p}\left\|\ampyt-\y^s\right\|^2\right)=\lim _{(s, t) \rightarrow \infty} \kappa_*\left(\delta_{s s}+\delta_{t t}-2 \delta_{s t}\right)=0
\end{aligned}
$$
using \Cref{prop:convsmallbeta}. The convergence of iterates $\ronet, \rtwot$ follows from $\rtwot=\ampxt+\st, \ronet=\ampyt+\st+\mathbf{e}$. The convergence of $\xonet, \xtwot$ follows from the fact they can be expressed as Lipschitz function applied to iterates $\ronetm$ and $\rtwot$, i.e. \eqref{x1t} and \eqref{x2t}. 
\end{proof}

\subsubsection{Track regularized estimator using VAMP iterates}\label{appendix:trackiovamp}

Let us first prove the following lemma
\begin{Lemma}\label{subgConv}
Recall the objective function $\mathcal{L}$ defined in \eqref{deflasso}. The vector
$$
\mathcal{L}^{\prime}\left(\xonet\right):=\X^{\top}\left(\X \xonet-\y\right)+\gamma_*\left(\ronetm-\xonet\right)
$$
is a subgraident of $\mathcal{L}$ at $\xonet$. We also have that almost surely,
$$\lim _{t \rightarrow \infty} \lim _{p \rightarrow \infty} \frac{1}{p}\left\|\mathcal{L}^{\prime}\left(\xonet\right)\right\|_2^2=0.$$
\end{Lemma}

\begin{proof}[Proof of \Cref{subgConv}]
Let $\partial h$ denotes sub-gradients of $h$. We have that $$
\mathcal{L}^{\prime}\left(\xonet\right) =\X^{\top}\left(\X \xonet-\y\right)+\gamma_*\left(\ronetm-\xonet\right)\in \X^{\top}\left(\X \xonet-\y\right)+\partial h\left(\xonet\right)
$$
because
$$
\xonet=\operatorname{Prox}_{\gamma_*^{-1} h}\left(\ronetm\right) \Leftrightarrow \ronetm-\xonet \in \gamma_*^{-1} \partial h\left(\xonet\right).
$$

Now note that
$$
\begin{aligned}
 \mathcal{L}^{\prime}\left(\xonet\right)&=\left(\X^{\top} \X-\gamma_* I\right) \xonet-\X^{\top} \y+\gamma_* \ronetm \\
& \stackrel{(a)}{=}\left(1-\frac{\gamma_*}{\eta_*}\right)\left(\X^{\top} \X+\gamma_* \mathbf{I}_p\right)\left(\rtwot-\rtwotm \right)+\left(\X^{\top} \X+\left(\eta_*-\gamma_*\right) \mathbf{I}_p\right) \xtwotm \\
& \qquad \qquad -\X^{\top} \y-\left(\eta_*-\gamma_*\right) \rtwotm  \\
& \stackrel{(b)}{=}\left(1-\frac{\gamma_*}{\eta_*}\right)\left(\X^{\top} \X+\gamma_* \mathbf{I}_p\right)\left(\rtwot-\rtwotm \right)
\end{aligned}
$$
where we used in $(a)$
\begin{equation}\label{aux2}
    \xonet=\left(1-\frac{\gamma_*}{\eta_*}\right)\left(\rtwot-\rtwotm \right)+\xtwotm 
\end{equation}
 which follows from \eqref{r1t},\eqref{r2t} and in $(b)$,
$$
\left(\X^{\top} \X+\left(\eta_*-\gamma_*\right) \mathbf{I}_p\right) \xtwotm =\X^{\top} \y+\left(\eta_*-\gamma_*\right) \rtwotm 
$$
which follows from \eqref{x2t}. It then follows from \Cref{prop:convsmallbeta} that almost surely
$$
\lim _{t \rightarrow \infty} \lim _{p \rightarrow \infty} \frac{1}{p}\left\|\mathcal{L}^{\prime}\left(\xonet\right)\right\|_2^2\le \lim _{t \rightarrow \infty} \lim _{p \rightarrow \infty}\left(1-\frac{\gamma_*}{\eta_*}\right)\left\|\X^{\top} \X+\gamma_* \mathbf{I}_p \right\|_{\mathrm{op}}^2 \cdot \frac{1}{p}\left\|\rtwot-\rtwotm \right\|_2^2=0
$$
as required.

\end{proof}

\begin{proof}[Proof of \Cref{prop:sds}]
Let us first consider the case $\co>0$ from \Cref{Assumpgp}. From strong convexity of the penalty function, almost surely, for all sufficiently large $p$,
\begin{equation}\label{eqconv}
    \mathcal{L}\left(\xonet\right) \geq \mathcal{L}(\hatbt) \geq \mathcal{L}\left(\xonet\right)+\left\langle\mathcal{L}^{\prime}\left(\xonet\right), \hatbt-\xonet\right\rangle+\frac{1}{2} \co \left\|\hatbt-\xonet\right\|_2^2
\end{equation}
where $\mathcal{L}^{\prime}\left(\xonet\right)$ is the subgradient of $\mathcal{L}$ defined in \Cref{subgConv}.

By Cauchy-Schwartz inequality, we have that
\begin{equation}\label{aux1}
    \left\|\hatbt-\xonet\right\|_2 \leq \frac{2}{c_0}\left\|\mathcal{L}^{\prime}\left(\xonet\right)\right\|_2
\end{equation}
which along with \Cref{subgConv} implies that
\begin{equation}\label{x1tcon}
    \lim _{t \rightarrow \infty} \lim _{p \rightarrow \infty} \frac{1}{p}\left\|\hatbt-\xonet\right\|_2^2=0
\end{equation}
By \eqref{aux2} and \Cref{prop:convsmallbeta}, we also have that
\begin{equation}\label{x2tcon}
    \lim _{t \rightarrow \infty} \lim _{p \rightarrow \infty} \frac{1}{p}\left\|\hatbt-\xtwot\right\|_2^2=0.
\end{equation}
Rearranging \eqref{x1t}---\eqref{r1t}, we have
$$
\rtwot=\xtwot+\frac{1}{\left(\eta_*-\gamma_*\right)} \X^{\top}\left(\X \xtwot-\y\right), \quad \ronet=\xtwot+\frac{1}{\gamma_*} \X^{\top}\left(\y-\X \xtwot\right)
$$
which along with \eqref{x1tcon}, \eqref{x2tcon} implies that
$$
\lim _{t \rightarrow \infty} \lim _{p \rightarrow \infty} \frac{1}{p}\left\|\ronet-\rstar\right\|_2^2=\lim _{t \rightarrow \infty} \lim _{p \rightarrow \infty} \frac{1}{p}\left\|\rtwot-\rstar\right\|_2^2=0.
$$
The proof for the other case in \Cref{Assumpgp}—that is, when \(\lim_{p\to \infty} \min_{i\in [p]} d_i^2 > c_1\)—is almost identical; the only difference is that for all sufficiently large \(p\), \eqref{eqconv} and \eqref{aux1} hold with \(c_0\) replaced by \(c_1\).
\end{proof}

\subsection{Supporting proofs for result B} \label{SupportPII}
\subsubsection{Properties of sample adjustment equation}\label{appendix:sampleadjeq}
\begin{proof}[Proof of \Cref{existslams}]
We can write $g_p(\gamma)$ as
\begin{equation}\label{mao}
    \begin{aligned}
    g_p(\gamma)=&\frac{1}{p} \sum_{i: d_i \neq 0} \frac{1}{\frac{1}{p}\left(\sum_{j: h^{\prime \prime}\left(\hjatbtj\right) \neq+\infty, 0} \frac{d_i^2-\gamma}{\gamma+h^{\prime \prime}\left(\hjatbtj\right)}+\sum_{j: h^{\prime \prime}\left(\hjatbtj\right)=0} \frac{d_i^2-\gamma}{\gamma}\right)+1}\\
    &+\frac{1}{p}\sum_{i: d_i=0} \frac{1}{\frac{1}{p}\left(\sum_{j: h^{\prime \prime}\left(\hjatbtj\right) \neq 0,+\infty} \frac{-\gamma}{\gamma+h^{\prime \prime}\left(\hjatbtj\right)}-\sum_{j: h^{\prime \prime}\left(\hjatbtj\right)=0} 1\right)+1}
    \end{aligned}.
\end{equation}
Let us first consider the case where $d_i \neq 0$ for all $i$. In this case, only the first sum remain and the denominators of the summands are
\begin{equation*}
    \begin{aligned}
        &\frac{d_i^2}{p}\left(\sum_{j: h^{\prime \prime}\left(\hjatbtj\right) \neq+\infty, 0} \frac{1}{\gamma+h^{\prime \prime}\left(\hjatbtj\right)}+\sum_{j: h^{\prime \prime}\left(\hjatbtj\right)=0} \frac{1}{\gamma}\right)\\
        &\qquad+1-\frac{1}{p}\left(\sum_{j: h^{\prime \prime}\left(\hjatbtj\right) \neq+\infty, 0} \frac{\gamma}{\gamma+h^{\prime \prime}\left(\hjatbtj\right)}+\sum_{j: h^{\prime \prime}\left(\hjatbtj\right)=0} 1\right)
    \end{aligned}
\end{equation*}
Observe that
\[
1 -\frac{1}{p}\left(\sum_{j: h^{\prime \prime}\left(\hjatbtj\right)_{\neq+\infty, 0}} \frac{\gamma}{\gamma+h^{\prime \prime}\left(\hjatbtj\right)}+\sum_{j: h^{\prime \prime}\left(\hjatbtj\right)=0} 1\right) \geq 0
\]
and
$$
\begin{aligned}
& \sum_{j: h^{\prime \prime}\left(\hjatbtj\right) \neq+\infty, 0} \frac{1}{\gamma+h^{\prime \prime}\left(\hjatbtj\right)}+\sum_{j: h^{\prime \prime}\left(\hjatbtj\right)=0} \frac{1}{\gamma}=0 \\
&\qquad \qquad \Leftrightarrow \sum_{j: h^{\prime \prime}\left(\hjatbtj\right) \neq+\infty, 0} \frac{\gamma}{\gamma+h^{\prime \prime}\left(\hjatbtj\right)}+\sum_{j: h^{\prime \prime}\left(\hjatbtj\right)=0} 1=0.
\end{aligned}
$$
These two observations and the assumption that $d_i \neq 0$ for all $i$ implies that for all $i \in[p]$, $g_p$ is well-defined on $(0,+\infty)$. For the case where $d_i=0$ for some $i$, all the denominators in \eqref{mao} are non-zero (and thus $g_p$ is well defined on $(0,+\infty)$) if
\begin{equation}\label{sdfaldfk}
    1-\frac{1}{p} \sum_{j: h^{\prime \prime}\left(\hjatbtj\right) \neq 0,+\infty} \frac{\gamma}{\gamma+h^{\prime \prime}\left(\hjatbtj\right)}-\frac{1}{p} \sum_{j: h^{\prime \prime}\left(\hjatbtj\right)=0} 1>0
\end{equation}
which is equivalent to $\frac{1}{p} \sum_{j: h^{\prime \prime}\left(\hjatbtj\right) \neq 0} 1>\frac{1}{p} \sum_{j: h^{\prime \prime}\left(\hjatbtj\right) \neq 0,+\infty} \frac{\gamma}{\gamma+h^{\prime \prime}\left(\hjatbtj\right)}$. The condition \eqref{sdfaldfk} is also necessary when $\exists i\in [p], d_i>0$. Meanwhile, we have that
$$
\frac{1}{p} \sum_{j: h^{\prime \prime}\left(\hjatbtj\right) \neq 0} 1 \stackrel{(a)}{\geq} \frac{1}{p} \sum_{j: h^{\prime \prime}\left(\hjatbtj\right) \neq 0,+\infty} 1 \stackrel{(b)}{\geq} \frac{1}{p} \sum_{j: h^{\prime \prime}\left(\hjatbtj\right) \neq 0,+\infty} \frac{\gamma}{\gamma+h^{\prime \prime}\left(\hjatbtj\right)} .
$$
Therefore, \eqref{sdfaldfk} holds if and only if at least one of $(a),(b)$ is strict. Note that $(a)$ is strict if and only if $\frac{1}{p} \sum_{j: h^{\prime \prime}\left(\hjatbtj\right)=+\infty} 1>0$ and $(b)$ is strict if and only if
$$
\frac{1}{p} \sum_{j: h^{\prime \prime}\left(\hjatbtj\right) \neq 0,+\infty}\left(1-\frac{\gamma}{\gamma+h^{\prime \prime}\left(\hjatbtj\right)}\right)>0 \Leftrightarrow \frac{1}{p} \sum_{j: h^{\prime \prime}\left(\hjatbtj\right) \neq 0,+\infty} 1>0.
$$
Note that $\frac{1}{p} \sum_{j: h^{\prime \prime}\left(\hjatbtj\right) \neq 0} 1>0$ if and only if $\frac{1}{p} \sum_{j: h^{\prime \prime}\left(\hjatbtj\right) \neq 0,+\infty} 1>0$ or $\frac{1}{p} \sum_{j: h^{\prime \prime}\left(\hjatbtj\right)=+\infty} 1>0$. This shows that \eqref{sdfaldfk} holds if and only if there exists some $i \in[p]$ such that $h^{\prime \prime}\left(\hiatbti\right) \neq 0$. The latter statement holds if $\|d\|_0+\left\|h^{\prime \prime}(\hatbt)\right\|_0>p$.

From now on, suppose that $g_p$ is well-defined. It follows from \eqref{sdfaldfk} that it is differentiable. Taking derivative of \eqref{sdfaldfk} yields
\begin{equation}
    \begin{aligned}
        g^{\prime}_p(\gamma)= &\frac{1}{p} \sum_{i: d_i \neq 0} \frac{\frac{1}{p}\left(\sum_{j: h^{\prime \prime}\left(\hjatbtj\right) \neq+\infty, 0} \frac{h^{\prime \prime}\left(\hjatbtj\right)+d_i^2}{\left(\gamma+h^{\prime \prime}\left(\hjatbtj\right)\right)^2}+\sum_{j: h^{\prime \prime}\left(\hjatbtj\right)=0} \frac{d_i^2}{\gamma^2}\right)}{\left(\frac{1}{p}\left(\sum_{j: h^{\prime \prime}\left(\hjatbtj\right) \neq+\infty, 0} \frac{d_i^2-\gamma}{\gamma+h^{\prime \prime}\left(\hjatbtj\right)}+\sum_{j: h^{\prime \prime}\left(\hjatbtj\right)=0} \frac{d_i^2-\gamma}{\gamma}\right)+1\right)^2} \\
        &+\frac{1}{p} \sum_{i: d_i=0} \frac{\frac{1}{p}\left(\sum_{j: h^{\prime \prime}\left(\hjatbtj\right) \neq 0,+\infty} \frac{h^{\prime \prime}\left(\hjatbtj\right)}{\left(\gamma+h^{\prime \prime}\left(\hjatbtj\right)\right)^2}\right)}{\left(\frac{1}{p}\left(\sum_{j: h^{\prime \prime}\left(\hjatbtj\right) \neq 0,+\infty} \frac{-\gamma}{\gamma+h^{\prime \prime}\left(\hjatbtj\right)}-\sum_{j: h^{\prime \prime}\left(\hjatbtj\right)=0} 1\right)+1\right)^2}>0
    \end{aligned}
\end{equation}
We claim that given $\gamma \mapsto g(\gamma)$ is well-defined, $g_p^{\prime}(\gamma)>0, \forall \gamma \in(0,+\infty)$ if and only if for some $j, \frac{1}{p} \sum_{j: h^{\prime \prime}\left(\hjatbtj\right) \neq+\infty} 1>0$. Note that if $\frac{1}{p} \sum_{j: h^{\prime \prime}\left(\hjatbtj\right) \neq 0,+\infty} 1>0$, then 
\begin{equation*}
    \frac{1}{p} \sum_{j: h^{\prime \prime}\left(\hjatbtj\right) \neq 0,+\infty} \frac{h^{\prime \prime}\left(\hjatbtj\right)}{\left(\gamma+h^{\prime \prime}\left(\hjatbtj\right)\right)^2}>0
\end{equation*}
and the above will be positive. Also note that if $\frac{1}{p} \sum_{j: h^{\prime \prime}\left(\hjatbtj\right)=0} 1>0$, then the assumption $D \neq$ 0 implies that there exists some $i \in[p]$ such that $\frac{1}{p} \sum_{j: h^{\prime \prime}\left(\hjatbtj\right)=0} \frac{d_i^2}{\gamma^2}>0$ and the above will be positive. Note that $\frac{1}{p} \sum_{j: h^{\prime \prime}\left(\hjatbtj\right) \neq+\infty} 1>0$ if and only if $\frac{1}{p} \sum_{j: h^{\prime \prime}\left(\hjatbtj\right) \neq 0,+\infty} 1>0$ or $\frac{1}{p} \sum_{j: h^{\prime \prime}\left(\hjatbtj\right)=0} 1>0$. Therefore, the positivity of the above follows from the assumption that there exists some $j \in[p]$ such that $h^{\prime \prime}\left(\hjatbtj\right) \neq+\infty$. Conversely, if $h^{\prime \prime}\left(\hjatbtj\right)=+\infty, \forall j, g_p(\gamma)=$ $1, \forall \gamma \in(0,+\infty).$

Note that if $\left\|h^{\prime \prime}(\hatbt)\right\|_0<p$ and for all $i, d_i\neq 0$, $\lim _{\gamma \rightarrow 0} g_p(\gamma)=0$; if $\left\|h^{\prime \prime}(\hatbt)\right\|_0=0$ and for some $i, d_i=0$, $g_p$ is not well-defined per discussion above; if $0<\left\|h^{\prime \prime}(\hatbt)\right\|_0<p$ and for some $i, d_i=0$, 
$$\lim _{\gamma \rightarrow 0} g_p(\gamma)=\frac{p-\|d\|_0}{\left\|h^{\prime \prime}(\hatbt)\right\|_0}<1$$
given that $\|d\|_0+\left\|h^{\prime \prime}(\hatbt)\right\|_0>p$; if $\left\|h^{\prime \prime}(\hatbt)\right\|_0=p,$
$$\lim _{\gamma \rightarrow 0} g_p(\gamma)=\frac{1}{p}\left(\sum_{i: d_i \neq 0} \frac{1}{\frac{1}{p}\left(\sum_{j: h^{\prime \prime}\left(\hjatbtj\right) \neq+\infty, 0 } h^{\prime \prime}\left(\hjatbtj\right)\right)+1}+\sum_{i: d_i=0} 1\right)<1$$ since $\Dbm \neq 0$. We also have that 
$$\lim _{\gamma \rightarrow+\infty} g_p(\gamma)=\frac{1}{1-\left(\frac{1}{p} \sum_{j: h^{\prime \prime}\left(\hjatbtj\right) \neq+\infty}1\right)} \in(1,+\infty]$$ 
if for some $i$, $h^{\prime \prime}\left(\hiatbti\right) \neq+\infty$. The proof is complete after combining these facts. 
\end{proof}

\subsubsection{Population limit of the adjustment equation}\label{appendix:popuadjeq}
\begin{proof}[Proof of \Cref{well}]
We can write $g_\infty (\gamma)$ as
\begin{equation}\label{zdgm}
\begin{aligned}
    & g_\infty(\gamma)=\mathbb{E} \frac{\mathbb{I}\left(\D^2>0\right)}{\left(\D^2-\gamma\right) \mathbb{E} \frac{\mathbb{I}(\Us \neq+\infty, 0)}{\gamma+\Us}+\left(\D^2-\gamma\right) \frac{1}{\gamma} \mathbb{P}(\Us=0)+1}\\
    & \qquad \qquad \qquad +\frac{\mathbb{P}\left(\D^2=0\right)}{\mathbb{E} \frac{-\gamma \mathbb{I}(\Us \neq+\infty, 0)}{\gamma+\Us}-\mathbb{P}(\Us=0)+1}
\end{aligned}
\end{equation}
Note that the denominators of both terms in \eqref{zdgm} are non-zero (and thus $g_\infty$ is well-defined) if
\begin{equation}\label{zdgm1}
    1-\E \frac{\gamma \mathbb{I}(\Us \neq+\infty, 0)}{\gamma+\Us}-\mathbb{P}(\Us=0)>0
\end{equation}
which is equivalent to $\mathbb{P}(\Us \neq 0)>\E \frac{\gamma \mathbb{I}(\Us \neq+\infty, 0)}{\gamma+\Us}$. Meanwhile we have that
$$
\mathbb{P}(\Us \neq 0) \stackrel{(a)}{\geq} \mathbb{P}(\Us \neq 0,+\infty) \stackrel{(b)}{\geq} \E \frac{\gamma \mathbb{I}(\Us \neq+\infty, 0)}{\gamma+\Us}
$$
Therefore, \eqref{zdgm1} holds if at least one of $(a),(b)$ is strict. Note that $(a)$ is strict if and only if $\mathbb{P}(\Us=+\infty)>0$ and $(b)$ is strict if and only if
$$
\E \mathbb{I}(\Us \neq+\infty, 0)\left(1-\frac{\gamma}{\gamma+\Us}\right)>0 \Leftrightarrow \mathbb{P}(\Us \neq 0,+\infty)>0.
$$
Note that $\mathbb{P}(\Us \neq 0)>0$ if and only if $\mathbb{P}(\Us \neq 0,+\infty)>0$ or $\mathbb{P}(\Us=+\infty)>0$. This shows that \eqref{zdgm1} holds and thus $g_\infty$ is well-defined  since $\mathbb{P}(\Us \neq 0)>0$ by \Cref{noid}. 

It follows from \eqref{RCa}, \eqref{RCc} and \eqref{eq:Jacprox} that $\gamma_*$ is a solution of the equation $g_\infty(\gamma)=1$ . We prove that $\gamma_*$ is a unique solution by showing $g_\infty$ is strictly increasing. Applying \cite[Proposition A.2.1]{talagrand2010mean}, we obtain that $g_\infty$ is differentiable and can be differentiated inside the expectation as follows
$$
\begin{aligned}
g_{\infty}^{\prime}(\gamma)=&\mathbb{E} \frac{\mathbb{I}\left(\D^2>0\right)\left(\mathbb{E} \frac{\Us \mathbb{I}(\Us \neq+\infty, 0)}{(\gamma+\Us)^2}+\D^2 \mathbb{E} \frac{\mathbb{I}(\Us \neq+\infty, 0)}{(\gamma+\Us)^2}+\left(\D^2 \frac{1}{\gamma^2}\right) \mathbb{P}(\Us=0)\right)}{\left(\left(\D^2-\gamma\right) \mathbb{E} \frac{\mathbb{I}(\Us \neq+\infty, 0)}{\gamma+\Us}+\left(\D^2-\gamma\right) \frac{1}{\gamma} \mathbb{P}(\Us=0)+1\right)^2} \\
&+\mathbb{E} \frac{\mathbb{I}\left(\D^2=0\right)\left(\mathbb{E} \frac{\Us \mathbb{I}(\Us \neq+\infty, 0)}{(\gamma+\Us)^2}\right)}{\left(\mathbb{E} \frac{-\gamma \mathbb{I}(\Us \neq+\infty, 0)}{\gamma+\Us}-\mathbb{P}(\Us=0)+1\right)^2}
\end{aligned}
$$
To prove $g_{\infty}^{\prime}(\gamma)>0, \forall \gamma \in(0,+\infty)$, note that if $\mathbb{P}(\Us \neq+\infty, 0)>0$, then $\mathbb{E} \frac{\Us \mathbb{I}(\Us \neq+\infty, 0)}{(\gamma+\Us)^2}>0$ and the above will be positive. Also note that if $\mathbb{P}(\Us=0)>0$, then $\mathbb{I}\left(\D^2>0\right)\left(\D^2 \frac{1}{\gamma^2}\right) \mathbb{P}(\Us=0)>$ 0 with positive probability and the above will be positive. Note that $\mathbb{P}(\Us \neq+\infty)>0$ if and only if $\mathbb{P}(\Us \neq 0$ and $\Us \neq+\infty)>0$ or $\mathbb{P}(\Us=0)>0$. Therefore, the positivity of $g_\infty^{\prime}(\gamma)$ follows from $\mathbb{P}(\Us \neq+\infty)>0$ which holds by \Cref{noid}. The proof is now complete.  
\end{proof}

\begin{proof}[Proof of \Cref{ptconv}]
We first note that
\begin{equation}\label{651}
    \hatbt=\operatorname{Prox}_{\gamma_*^{-1} h}\left(\rstar\right).
\end{equation}
This follows from $\rstar \in \hatbt+\frac{1}{\gamma_*} \partial h(\hatbt)$ and the equivalence relation $\rstar \in \hatbt+\frac{1}{\gamma_*} \partial h(\hatbt) \Leftrightarrow \hatbt=\operatorname{Prox}_{\gamma_*^{-1} h}\left(\rstar\right)$. The former is a consequence of the KKT condition $\X^{\top}(\y-\X \hatbt) \in \partial h(\hatbt)$ and the latter follows from \Cref{prop:proxp}, $(a)$. Also note that for any $\gamma>0$,
\begin{equation}\label{sa51}
    \mathbb{P}\left(\sqrt{\taustar} \Zs+\Xstar \in\left\{x \in \mathbb{R}: \frac{1}{\gamma+ h^{\prime \prime}\left(\operatorname{Prox}_{\gamma_*^{-1} h}(x)\right)} \text { is continuous at } x\right\}\right)=1
\end{equation}
which follows from that $x\mapsto \frac{1}{\gamma+ h^{\prime \prime}\left(\operatorname{Prox}_{\gamma_*^{-1} h}(x)\right)}$ has only finitely many discontinuities (cf. \Cref{Extend}) and that $\taustar>0$. Then, almost surely,
\begin{equation}\label{hassabids}
    \begin{aligned}
    \lim _{p \rightarrow \infty} \frac{1}{p} \sum_{i=1}^p \frac{1}{\gamma+h^{\prime \prime}\left(\hiatbti\right)} & \stackrel{(a)}{=} \lim _{p \rightarrow \infty} \frac{1}{p} \sum_{j=1}^p \frac{1}{\gamma+h^{\prime \prime}\left(\operatorname{Prox}_{\gamma_*^{-1} h}\left(r_{*, j}\right)\right)}\\
    &\stackrel{(b)}{=} \mathbb{E} \frac{1}{\gamma+h^{\prime \prime}\left(\operatorname{Prox}_{\gamma_*^{-1} h}\left(\sqrt{\taustar} \Zs+\Xstar\right)\right)}
\end{aligned}
\end{equation}
where $(a)$ follows from \eqref{651} and $(b)$ follows from \Cref{thm:empmain}, \Cref{prop:asW} and \eqref{sa51}. An immediate consequence is that almost surely for all sufficiently large $p$, there must exist some $i\in [p]$ such that $h^{\prime \prime} (\hat{\beta}_i)\neq +\infty$. This is because the RHS is bounded away from $0$ for any fixed $\gamma>0$; for if not, we must have $h^{\prime \prime}\left(\operatorname{Prox}_{\gamma_*^{-1} h}\left(\sqrt{\taustar} \Zs+\Xstar\right)\right)=+\infty$ almost surely, which implies that 
$$\frac{1}{\eta_*}=\gamma_*^{-1} \E \operatorname{Prox}_{\gamma_{*}^{-1} h}^{\prime}\left(\Xstar+\sqrt{\taustar} \Zs\right)=\E \frac{1}{\gamma_*+h^{\prime \prime}\left(\operatorname{Prox}_{\gamma_*^{-1} h}\left(\Xstar+\sqrt{\taustar} \Zs\right)\right)}=0,$$
contradicting \Cref{Assumpfix}. By this and \Cref{Assumpgp}, we know that \Cref{Assumpgpweak} holds almost surely for all sufficiently large $p$. By \Cref{CORweak}, almost surely for sufficiently large $p$, $g_p$ is well-defined, strictly increasing and equation \eqref{solution} admits a unique solution on $(0,+\infty)$. 

Now, \eqref{hassabids}, along with \Cref{AssumpD} and \Cref{prop:combW}, implies that almost surely
\begin{equation}\label{thisffdf}
\left(\operatorname{diag}\left(\Dbm^{\top} \Dbm\right)-\gamma\right)\left(\frac{1}{p} \sum_{j=1}^p \frac{1}{\gamma+h^{\prime \prime}\left(\hjatbtj\right)}\right) \stackrel{W_2}{\rightarrow}\left(\D^2-\gamma\right) \mathbb{E} \frac{1}{\gamma+\Us}.
\end{equation}
Almost sure convergence \eqref{desired} follows from \eqref{thisffdf}, \Cref{prop:asW} and the fact from \Cref{well} that $1+\left(\D^2-\gamma\right) \mathbb{E} \frac{1}{\gamma+\Us}>0$ almost surely.
\end{proof}

\subsection{Finite or single coordinate inference under exchangeability}\label{appendix:single}

\begin{proof}[Proof of \Cref{sgoods}]
We only show that $\frac{{r}_{*, i}-\sti}{\sqrt{\taustar}} \Rightarrow N(0,1)$ for $\mathbf{r}_*$ defined in \eqref{defr1r2}. \eqref{werbaoz} then follows from consistency of $\hat{\tau}_{*}$ and $\adj$ (cf. \Cref{neig}) and the Slutsky's theorem. Let $\bUm \in \mathbb{R}^{p \times p}$ denote a permutation operator drawn uniformly at random independent of $\st, \X, \epbm$. We have that
$$
\left(\X \bUm^\top, \bUm \st, \epbm\right) \stackrel{L}{=}\left(\X, \st, \epbm\right)
$$
where we use $\stackrel{L}{=}$ to denote equality in law. Note that
$$
\begin{aligned}
& \hatbt=\underset{\bm{\beta}}{\operatorname{argmin}} \frac{1}{2}\left\|\X \bUm^\top \bUm\left(\st-\bm{\beta}\right)+\epbm\right\|^2+h\left(\bUm^\top \bUm \bm{\beta}\right) \\
&\qquad =\bUm^\top \underset{\bUm \beta}{\operatorname{argmin}} \frac{1}{2}\left\|\X \bUm^\top\left(\bUm \st-\bUm \bm{\beta} \right)+\epbm\right\|^2+h(\bUm \bm{\beta})
\end{aligned}
$$
where $h$ applies entry-wise to its argument. The above then implies
\begin{equation}\label{aaaf}
\begin{aligned}
& \left(\bUm \hatbt, \X \bUm^\top, \bUm \st, \epbm\right) \\
& \qquad =\left(\underset{\bm{\beta}}{\operatorname{argmin}} \frac{1}{2}\left\|\X \bUm^\top\left(\bUm \st-\bm{\beta}\right)+\epbm\right\|^2+h(\bm{\beta}), \X \bUm^\top, \bUm \st, \epbm\right) \\
& \qquad \stackrel{L}{=}\left(\underset{\beta}{\operatorname{argmin}} \frac{1}{2}\left\|\X\left(\st-\bm{\beta}\right)+\epbm\right\|^2+h(\bm{\beta}), \X, \st, \epbm\right) \\
& \qquad =\left(\hatbt, \X, \st, \epbm\right)
\end{aligned}
\end{equation}
Below we prove the Corollary for $\mathcal{L}=\{i,k\},i\neq k$. The general case is analogous. For standard basis $\mathbf{e}_i, \mathbf{e}_k$, and any constant $c_1, c_2 \in \R$,
$$
\begin{aligned}
 \mathbb{P}&\left(\frac{\mathbf{e}_i^{\top} \rstar-\mathbf{e}_i^{\top} \st}{\sqrt{\taustar}}<c_1, \frac{\mathbf{e}_k^{\top} \rstar-\mathbf{e}_k^{\top} \st}{\sqrt{\taustar}}<c_2 \right) \\
&\stackrel{(a)}{=} \mathbb{P}\left(\frac{\mathbf{e}_i^{\top} \bUm \rstar-\mathbf{e}_i^{\top} \bUm \st}{\sqrt{\taustar}}<c_1, \frac{\mathbf{e}_k^{\top} \bUm \rstar-\mathbf{e}_k^{\top} \bUm \st}{\sqrt{\taustar}}<c_2\right) \\
& \stackrel{(b)}{=} \mathbb{E}\left(\mathbb{P}\left(\frac{\mathbf{e}_i^{\top} \bUm \rstar-\mathbf{e}_i^{\top} \bUm \st}{\sqrt{\taustar}}<c_1, \frac{\mathbf{e}_k^{\top} \bUm \rstar-\mathbf{e}_k^{\top} \bUm \st}{\sqrt{\taustar}}<c_2 \mid \mathcal{F}\left(\st, \epbm, \X\right)\right)\right) \\
& \stackrel{(c)}{=} \mathbb{E} \frac{1}{p(p-1)} \sum_{j_1\neq j_2\in [p]} \mathbb{I}\left(\frac{1}{\sqrt{\taustar}}\left(r_{*, j_1}-\beta_{j_1}^{\star}\right)<c_1\right) \mathbb{I}\left(\frac{1}{\sqrt{\taustar}}\left(r_{*, j_2}-\beta_{j_2}^{\star}\right)<c_2\right)
\end{aligned}
$$
where in $(a)$ we used \eqref{aaaf} above, in $(b)$ we used $ \mathcal{F}\left(\st, \epbm, \X\right)$ to denote sigma-field generated by $\st, \epbm, \X$ and in $(c)$ we used that $\mathbf{U}$ is a permutation operator drawn uniformly at random.

Note that almost surely as $p\to \infty$,
\begin{equation*}
    \begin{aligned}
    &\bigg|\frac{1}{p(p-1)} \sum_{j_1\neq j_2\in [p]} \mathbb{I}\left(\frac{1}{\sqrt{\taustar}}\left(r_{*, j_1}-\beta_{j_1}^{\star}\right)<c_1\right) \mathbb{I}\left(\frac{1}{\sqrt{\taustar}}\left(r_{*, j_2}-\beta_{j_2}^{\star}\right)<c_2\right)- \\
    &\qquad \frac{1}{p^2} \sum_{j=1}^{p} \mathbb{I}\left(\frac{1}{\sqrt{\taustar}}\left(r_{*, j}-\beta_j^{\star}\right)<c_1\right) \sum_{j=1}^{p} \mathbb{I}\left(\frac{1}{\sqrt{\taustar}}\left(r_{*, j}-\beta_j^{\star}\right)<c_2\right)\bigg|\to 0.
    \end{aligned}
\end{equation*}
Note also that for $\iota=1,2$ almost surely
$$
\lim _{p\to \infty} \frac{1}{p} \sum_{j=1}^{p} \mathbb{I}\left(\frac{1}{\sqrt{\taustar}}\left(r_{*, j}-\beta_j^{\star}\right)<c_\iota\right)=\mathbb{P}(\Zs<c_\iota)
$$
where $\Zs\sim N(0,1)$. Here, we used \Cref{neig} and \Cref{prop:asW}. Using dominated convergence theorem, we conclude that
$$
 \mathbb{P}\left(\frac{\mathbf{e}_i^{\top} \rstar-\mathbf{e}_i^{\top} \st}{\sqrt{\taustar}}<c_1, \frac{\mathbf{e}_k^{\top} \rstar-\mathbf{e}_k^{\top} \st}{\sqrt{\taustar}}<c_2 \right) \rightarrow \mathbb{P}(\Zs<c_1)\mathbb{P}(\Zs<c_2)
$$
as required.
\end{proof}

\subsection{Hypothesis testing and confidence intervals}\label{appendix:testing}
\begin{proof}[Proof of \Cref{weeren}]
To see $(a)$, We have that almost surely
$$
\begin{gathered}
\lim _{p\to \infty} \frac{\frac{1}{p} \sum_{j=1}^{p} \mathbb{I}\left(P_j \leq \alpha, \beta_j^{\star}=0\right)}{\frac{1}{p} \sum_{j=1}^{p} \mathbb{I}\left(\beta_j^{\star}=0\right)}=\lim _{p\to \infty} \frac{\frac{1}{p} \sum_{j=1}^{p} \mathbb{I}\left(\left|\frac{\hat{r}_{*, j}-\beta_j^{\star}}{\sqrt{\hat{\tau}_{*}}}\right| \geq \Phi^{-1}\left(1-\frac{\alpha}{2}\right), \abs{\beta_j^{\star}} \leq \frac{\mu_0}{2}\right)}{\frac{1}{p} \sum_{j=1}^{p} \mathbb{I}\left(\abs{\beta_j^{\star}} \leq \frac{\mu_0}{2}\right)} \\
=\frac{\mathbb{P}\left(|\Zs| \geq \Phi^{-1}\left(1-\frac{\alpha}{2}\right), \abs{\Xstar} \leq \frac{\mu_0}{2}\right)}{\mathbb{P}\left(\abs{\Xstar} \leq \frac{\mu_0}{2}\right)}=\mathbb{P}\left(|\Zs| \geq \Phi^{-1}\left(1-\frac{\alpha}{2}\right)\right)=\alpha
\end{gathered}
$$
by \Cref{neig} and \Cref{prop:asW}. Using exchangeability of columns of $\X$
$$
\mathbb{E} \frac{\frac{1}{p} \sum_{j=1}^{p} \mathbb{I}\left(P_j \leq \alpha, \beta_j^{\star}=0\right)}{\frac{1}{p} \sum_{j=1}^{p} \mathbb{I}\left(\beta_j^{\star}=0\right)}=\frac{\mathbb{P}\left(T_i=1\right) \frac{1}{p} \sum_{j=1}^{p} \mathbb{I}\left(\beta_j^{\star}=0\right)}{\frac{1}{p} \sum_{j=1}^{p} \mathbb{I}\left(\beta_j^{\star}=0\right)}=\mathbb{P}\left(T_i=1\right)
$$
The the coordinate-wise result follows from an application of the dominated convergence theorem. 

To see $(b)$, note that by \Cref{neig} and \Cref{prop:asW}, almost surely
    $$
\lim _{p\to \infty} \frac{1}{p} \sum_{i=1}^{p} \mathbb{I}\left(\sti \in \mathsf{CI}_i\right)=\lim _{p\to \infty} \frac{1}{p} \sum_{i=1}^{p} \mathbb{I}\left(a<\frac{\sti-\bhetahi}{\sqrt{\hat{\tau}_{*}}}<b\right)=\mathbb{P}(a<\Zs<b)=1-\alpha.
$$
\end{proof}

\begin{Remark}[Asymptotic limit of TPR]\label{TPRlimit}
Note that we can further calculate the exact asymptotic limit of the TPR as follows. Under the assumption of \Cref{weeren} (a), we have that almost surely
$$
\begin{aligned}
 \lim _{p\to \infty} \mathsf{TPR}(p) &= \lim _{p\to \infty} \frac{\sum_{j=1}^{p} \mathbb{I}\left(P_j \leq \alpha,\left|\beta_j^{\star}\right| \geq \mu_0\right)}{\sum_{j=1}^p \mathbb{I}\left(\abs{\beta_j^{\star}} \geq \mu_0\right)}\\
&=\lim _{p\to \infty} \frac{\frac{1}{p} \sum_{j=1}^{p} \mathbb{I}\left(\left|\frac{r_{*, j}}{\sqrt{\hat{\tau}_{*}}}\right| \geq \Phi^{-1}\left(1-\frac{\alpha}{2}\right),\left|{\beta_j^{\star}}\right| \geq \mu_0\right)}{\frac{1}{p} \sum_{j=1}^{p} \mathbb{I}\left(\abs{\beta_j^{\star}} \geq \mu_0\right)} \\
&=\frac{\mathbb{P}\left(\left|\frac{1}{\sqrt{\taustar}} \Xstar+\Zs\right| \geq \Phi^{-1}\left(1-\frac{\alpha}{2}\right),\left|\Xstar\right| \geq \mu_0\right)}{\mathbb{P}\left(\left|\Xstar\right| \geq \mu_0\right)}
\end{aligned}
$$
where we used in the second line \Cref{neig} and \Cref{prop:asW}.
\end{Remark}

\section{Proofs for PCR-Spectrum-Aware Debiasing}\label{appendix:pcrdeb}
\subsection{Pseudo-code for PCR-Spectrum-Aware Debiasing}

\Cref{algodebiaspcr} below summarizes the PCR-Spectrum-Aware Debiasing procedure in algorithmic format. 

\begin{algorithm}[t]
\caption{PCR-Spectrum-Aware Debiasing}\label{algodebiaspcr}
\begin{algorithmic}[1]
\Require Response and design $(\y,\X)$, a penalty function $h$ and an index set of PCs $\Js \subset \mathcal{N}$ (see \eqref{defJJ}). 

\State \textbf{Conduct} eigen-decomposition: $\X^\top \X =\Obm^\top \Dbm^\top \Dbm \Obm$ and let $\Obm_\Js, \Obm_{\Jsb}$ be PCs indexed by $\Js$ and $\Jsb=\mathcal{N}\setminus \Js$ respectively. 
\State \textbf{Compute} alignment PCR estimator
$$\pcrt\gets \Obm_{\Js}^\top \left(\bm{W}_\Js^{\top} \bm{W}_\Js\right)^{-1} \bm{W}_\Js^{\top}\y$$
where $\bm{W}_\Js:=\X \Obm_\Js^{\top}.$
\State \textbf{Construct} new data
\begin{equation*}
        \ynew \gets \qty(\Dbm_{\Jsb}^\top \Dbm_{\Jsb})^{1/2} \left(\bm{W}_{\Jsb}^{\top} \bm{W}_{\Jsb}\right)^{-1} \bm{W}_{\Jsb}^{\top} \y, \quad \Xnew \gets \qty(\Dbm_{\Jsb}^\top \Dbm_{\Jsb})^{1/2} \Obm_{\Jsb}
    \end{equation*}
where $\bm{W}_{\Jsb}=\X \Obm_{\Jsb}^\top $ and $\Dbm_{\Jsb}$ consists of columns of $\Dbm$ indexed by $\Jsb$. 
\State \textbf{Find} minimizer $\hatbt$ of $\mathcal{L}(\bm{\cdot} \; ;\Xnew,\ynew)$ for $\mathcal{L}$ defined in \eqref{deflasso}
\State \textbf{Compute} the eigenvalues $(d_i^2)_{i=1}^p$ of $\Xnew^\top\Xnew$
\State \textbf{Find} solution $\adj(\Xnew,\ynew,h)$ of \eqref{gammasolvea} and compute complement PCR estimator
    \begin{equation}
        \pcrc \gets \hatbt+\adj^{-1} \Xnew^\top (\ynew-\Xnew \hatbt)
    \end{equation}
    and $\hat{\tau}_{*}(\Xnew,\ynew,h)$ from \eqref{DEFEFD}
\Ensure PCR-Spectrum-Awaure estimator $$\pcrdb \gets \pcrt+\pcrc$$ and the associated variance estimator $\hat{\tau}_{*}\gets \hat{\tau}_{*}(\Xnew,\ynew,h)$. 
\end{algorithmic}
\end{algorithm}

\subsection{Asymptotic normality}
\begin{proof}[Proof of \Cref{PCRTHM}]
(a) \textbf{Alignment PCR.} Let $\Dbm_{\Js} \in \mathbb{R}^{n \times J}$ consist of columns of $\Dbm$ indexed by $\Js, \Obm_{\Js} \in \mathbb{R}^{J \times p}$ consist of rows of $\Obm$ indexed by $\Js$, and $\mathbf{P}_{\Js}=\Obm_{\Js}^{\top} \Obm_{\Js}$. Note that
\begin{equation}\label{algepcrt}
\begin{aligned}
\pcrt(\Js)&=\Obm_{\Js}^{\top} \pcr(\Js)\\
&=\Obm_{\Js}^{\top}\left(\mathbf{W}_{\Js}^{\top} \mathbf{W}_{\Js}\right)^{-1} \mathbf{W}_{\Js}^{\top} \y\\
&=\Obm_{\Js}^{\top}\left(\Dbm_{\Js}^{\top} \Dbm_{\Js}\right)^{-1} \Dbm_{\Js}^{\top}\left(\Dbm \Obm \st+\Qbm \epbm\right) \\
& =\Obm_{\Js}^{\top} \Obm_{\Js} \st+\Obm_{\Js}^{\top}\left(\Dbm_{\Js}^{\top} \Dbm_{\Js}\right)^{-1} \Dbm_{\Js}^{\top} \Qbm \epbm\\
&=\stal+\Obm_{\Js}^{\top} \Obm_{\Js} \zetr+\Obm_{\Js}^{\top}\left(\Dbm_{\Js}^{\top} \Dbm_{\Js}\right)^{-1} \Dbm_{\Js}^{\top} \Qbm \epbm 
\end{aligned}
\end{equation}
where we used that
\begin{equation*}
\mathbf{W}_{\Js}=\Qbm^{\top} \Dbm \Obm \Obm_{\Js}^{\top}=\Qbm^{\top} \Dbm_{\Js}, \quad \y=\X \st+\epbm=\Qbm^{\top} \Dbm \Obm \st+\epbm
\end{equation*}
in the penultimate equality and \eqref{stwerif} in the last equality.

Using rotational invariance of $\Obm$, we have 
\begin{equation}\label{useuse}
\begin{aligned}
    \E \qty[\qty(\frac{1}{p}\left\|\mathbf{P}_{\Js} \zetr\right\|_2^2)^2\mid \zetr] &=\frac{1}{p^2} \E \qty[\left\|\Obm_{\Js} \zetr\right\|_2^4\mid \zetr] \\
    & = \qty(\frac{\left\|\zetr\right\|_2^2}{p})^2 \E \qty[\sum_{i=1}^J O_{1 i}^2]^2=O\qty(\frac{1}{p^2})
\end{aligned}
\end{equation}
where we used that $J$ is finite not growing with $p$ and basic moment property of entries of $\Obm$ (see e.g. \cite[Proposition 2.5]{meckes2014concentration}). It follows from a straightforward application of Markov inequality and Borel-Cantelli lemma that almost surely
\begin{equation}\label{asffsnone}
\lim_{p\to\infty}\frac{1}{p}\left\|\mathbf{P}_{\Js} \zetr\right\|_2^2=0.
\end{equation}
Meanwhile, using $\Qbm \epbm \stackrel{L}{=} \epbm$, we have 
\begin{equation}\label{useuse2}
\frac{1}{p}\left\|\Obm_{\Js}^{\top}\left(\Dbm_{\Js}^{\top} \Dbm_{\Js}\right)^{-1} \Dbm_{\Js}^{\top} \Qbm \epbm\right\|_2^2 \stackrel{L}{=} \frac{1}{p} \epbm^{\top} \Dbm_{\Js}\left(\Dbm_{\Js}^{\top} \Dbm_{\Js}\right)^{-2} \Dbm_{\Js}^{\top} \epbm=\frac{1}{p} \sum_{i \in \Js} \frac{\epi^2}{d_i^2}.
\end{equation}
Using $J$ is finite and the assumption that $\limsup_{p\to \infty} \max_{i\in \Js } d_i^{-2}/p\to 0$, we obtain that almost surely,
\begin{equation}\label{asffstwo}
    \lim_{p\to\infty} \frac{1}{p}\left\|\Obm_{\Js}^{\top}\left(\Dbm_{\Js}^{\top} \Dbm_{\Js}\right)^{-1} \Dbm_{\Js}^{\top} \Qbm \epbm\right\|_2^2 =0.
\end{equation}
The result then follows from \eqref{asffsnone} and \eqref{asffstwo}. 

(b) \textbf{Complement PCR.} Similarly to \eqref{algepcrt}, we have that
\begin{equation}\label{algepcrtBBC}
\pcr(\Jsb)=\Obm_{\Jsb} \zetr+\left(\Dbm_{\Jsb}^{\top} \Dbm_{\Jsb}\right)^{-1} \Dbm_{\Jsb}^{\top} \Qbm \epbm.
\end{equation}
It follows that
$$\ynew=\left(\Dbm_{\Jsb}^{\top} \Dbm_{\Jsb}\right)^{\frac{1}{2}} \pcr(\Jsb) \in \mathbb{R}^{N-J}, \qquad \Xnew=\left(\Dbm_{\Jsb}^{\top} \Dbm_{\Jsb}\right)^{\frac{1}{2}} \Obm_{\Jsb} \in \mathbb{R}^{(N-J) \times p}$$ 
defined in \eqref{xnewYnew} satisfy the following relation:
\begin{equation}\label{filteredpcr}
    \ynew=\Xnew \zetr+\epnew
\end{equation}
for 
\begin{equation*}
    \epnew=\left(\Dbm_{\Jsb}^{\top} \Dbm_{\Jsb}\right)^{-\frac{1}{2}} \Dbm_{\Jsb}^{\top} \Qbm \epbm \sim N\left(\bm{0}, \sigma^2 \mathbf{I}_{N-J}\right).
\end{equation*}
Note that the new design matrix $\Xnew$ admits singular value decomposition
$$
\Xnew=\Qnew^{\top} \Dnew \Obm
$$
where 
$$
\Qnew=\mathbf{I}_{N-J}, \qquad \Dnew=\left[\left(\Dbm_{\Jsb}^{\top} \Dbm_{\Jsb}\right)^{\frac{1}{2}}, \bm{0}_{(N-J) \times(p+J-N)}\right] \in \mathbb{R}^{(N-J) \times p}.
$$
Note that since $J$ is finite not growing with $n,p$, $$\Dnew^\top \bm{1}_{(N-J)\times 1} \equiv \mathbf{d}_{\Js^c}\stackrel{W_2}{\to} \D.$$ 
The above, along with the assumption we made in \Cref{PCRTHM}, reduces the new regression problem defined by \eqref{filteredpcr} to the same one considered in \Cref{subsectionmrdeb}. Since $\pcrc$ is Spectrum-Aware debiased estimator with respect to the new regression problem, the result follows from \Cref{NEIGMAIN}. The consistency of $\hat{\sigma}^2$ follows a similar reasoning. 

(c) \textbf{Debiased PCR.} By the definition of $\pcrdb$, we have that
$$
\tauc^{-1/2}\qty(\pcrdb-\st)=\tauc^{-1/2}\qty(\pcrt-\stal)+\tauc^{-1/2}\qty(\pcrc-\zetr).
$$
The result then follows from (a), (b) above and \Cref{prop:combW}. 
\end{proof}

\subsection{Finite or single coordinate inference}\label{appendix:finitecoorpcr}
\begin{Corollary}\label{CORPCRTHM}
    Suppose Assumptions \ref{Assumph}---\ref{AssumpPriorAL} hold. If $\qty(\zetrj)_{j=1}^p$ are exchangeable as in \Cref{exchangedef}, then for any fixed, finite index set $\mathcal{I}\subset [p],$ we have that almost surely as $p\to \infty$, 
    \begin{equation}\label{centralclaim}
        \begin{aligned}
            & \pcrtI(\Js)\to \stalI,\;\;  \tauc^{-1/2}\qty(\pcrcI(\Jsb)-\zetr_\mathcal{I})\Rightarrow N(\bm{0},\mathbf{I}_{|\mathcal{I}|})\\  & \tauc^{-1/2}\qty(\pcrdbI-\st_\mathcal{I})\Rightarrow N \qty(\bm{0},\mathbf{I}_{|\mathcal{I}|}).
        \end{aligned}
    \end{equation}
\end{Corollary}

\begin{proof}[Proof of \Cref{CORPCRTHM}]
To see the first result in \eqref{centralclaim}, recall from \eqref{algepcrt}, we have that
\begin{equation}
\pcrt(\Js)=\stal+\Obm_{\Js}^{\top} \Obm_{\Js} \zetr+\Obm_{\Js}^{\top}\left(\Dbm_{\Js}^{\top} \Dbm_{\Js}\right)^{-1} \Dbm_{\Js}^{\top} \Qbm \epbm 
\end{equation}
Note that when $\zetr$ is exchangeable, we have that for any fixed $i \in[p]$
\begin{equation}\label{dawg11}
\begin{aligned}
    \mathbb{E}\left[\left(\left(\Obm_{\Js}^{\top} \Obm_{\Js} \zetr\right)_i\right)^2 \right]&=\E \qty[\qty(\mathbf{e}_i^\top \bUm \Obm_{\Js}^\top \Obm_{\Js} \bUm^\top \bUm \zetr )^2] \\
    &=\mathbb{E}\left[\left(\frac{1}{p}\left\|\mathbf{P}_{\Js} \zetr\right\|_2^2\right)^2 \right] =O\left(\frac{1}{p^2}\right)
\end{aligned}
\end{equation}
where we used that for a permutation matrix $\bUm\in \R^{p\times p}$ drawn uniformly, $(\Obm_{\Js} \bUm^\top, \bUm \zetr)\stackrel{L}{=}(\Obm_{\Js},\zetr)$ and \eqref{useuse}. And by rotational invariance of $\Obm$,
\begin{equation}\label{dawg12}
\mathbb{E}\left(\Obm_{\Js}^{\top}\left(\Dbm_{\Js}^{\top} \Dbm_{\Js}\right)^{-1} \Dbm_{\Js}^{\top} \Qbm \epbm\right)_i^2=\mathbb{E} \frac{1}{p}\left\|\Obm_{\Js}^{\top}\left(\Dbm_{\Js}^{\top} \Dbm_{\Js}\right)^{-1} \Dbm_{\Js}^{\top} \Qbm \epbm\right\|_2^2=O \qty(\frac{1}{p^2})
\end{equation}
where we used \eqref{useuse2} at the last equality. The first result in \eqref{centralclaim} then follows from Markov inequality and Borel-Cantelli lemma. The second result in \eqref{centralclaim} can be proved similarly to \Cref{sgoods}. The third result in \eqref{centralclaim} follows from the first two results and an application of the Slutsky's theorem. 
\end{proof}

\subsection{Alignment test}\label{apptestcor}
\begin{proof}[Proof of \Cref{testCor}]
Similarly to \eqref{algepcrt}, we have that
\begin{equation}\label{algepcrtBB}
\pcr(\Js)-\alphstar=\Obm_{\Js} \zetr+\left(\Dbm_{\Js}^{\top} \Dbm_{\Js}\right)^{-1} \Dbm_{\Js}^{\top} \Qbm \epbm
\end{equation}
Now note that by basic properties of Haar measure on orthogonal groups \cite{meckes2014concentration}, as $p \rightarrow \infty$,
\begin{equation*}
\Obm_{\Js} \zetr \Rightarrow N\left(\bm{0}, \E\left(\Zetr\right)^2\cdot \mathbf{I}_{J}\right)
\end{equation*}
where we used the assumption that $\zetr \stackrel{W_2}{\to} \Zetr$, and that
\begin{equation*}
\left(\Dbm_{\Js}^{\top} \Dbm_{\Js}\right)^{-1} \Dbm_{\Js}^{\top} \Qbm \epbm \sim N\left(\bm{0}, \sigma^2 \cdot \left(\Dbm_{\Js}^{\top} \Dbm_{\Js}\right)^{-1} \right).
\end{equation*}
By independence of $\Obm$ and $\epbm$, we have that
\begin{equation*}
\Obm_{\Js} \zetr+\left(\Dbm_{\Js}^{\top} \Dbm_{\Js}\right)^{-1} \Dbm_{\Js}^{\top} \Qbm \epbm \Rightarrow N\left(\bm{0}, \E\left(\Zetr\right)^2\cdot \mathbf{I}_{J}+\sigma^2\cdot \left(\Dbm_{\Js}^{\top} \Dbm_{\Js}\right)^{-1}\right)
\end{equation*}
Desired result then follows from the fact that $\magzetr$ consistently estimates $\E\left(\Zetr\right)^2$ and $\hat{\sigma}^2$ consistently estimate $\sigma^2$. That is, almost surely
$$\hat{\omega}= p^{-1} \norm{\pcrc}^2 -\tauh \to \E (\Zetr)^2, \quad \hat{\sigma}^2\to \sigma^2$$
as $p\to \infty$. The former follows from the fact that almost surely $\pcrc \toW \Zetr+\sqrt{\tau_*}\mathsf{Z}$ for $\mathsf{Z}$ independent of $\Zetr$. 
\end{proof}

\section{Further remarks on right-rotationally invariant designs} \label{RIDint}
As discussed in the main text, assuming right singular vectors $\Obm$ of the design $\X$ to be Haar lands $\X$ in the class of right-rotationally invariant designs (Definition \ref{def:Rotinv}). Varied research communities realized the strength of such designs \cite{takeda2006analysis,takeuchi2019rigorous,rangan2019vector,barbier2018mutual,opper2001tractable,ma2017orthogonal,fan2021replica,rush2015capacity}. In particular, \cite{dudeja2022spectral,wang2022universality} established that properties of high-dimensional systems proven under such designs continue to hold for a broad class of designs (including nearly deterministic designs as observed in compressed sensing \cite{donoho2009observed}) as long as they satisfy certain spectral properties. In fact, the universality class for such designs is far broader than that for Gaussians, suggesting that these may serve as a solid prototype for modeling high-dimensional phenomena arising in non-Gaussian data. Despite such exciting developments, there are hardly any results when it comes to debiasing or inference under such designs (with the exception of \cite{takahashi2018statistical} which we discuss later). This paper develops this important theory and methodology. 

Despite the generality of right-rotationally invariant designs, studying these presents new challenges. For starters, analogs of the leave-one-out approach \cite{mezard1987spin,talagrand2003spin,bean2013optimal,el2018impact,sur2019likelihood,sur2019modern,chen2021spectral,jiang2022new} and Stein's method \cite{stein1981estimation,chatterjee2010spin,bellec2019biasing,anastasiou2023stein}, both of which form fundamental proof techniques for Gaussian designs, are nonexistent or under-developed for this more general class. To mitigate this issue, we resort to an algorithmic proof strategy that the senior authors' earlier work and that of others have used in the context of Gaussian designs. 
To study $\bhetah$, we observe that it depends on the regularized estimator $\boldsymbol{\hat{\beta}}$. However, $\boldsymbol{\hat{\beta}}$ does not admit a closed form in general, thus studying these turns out difficult. To circumvent this, we create surrogate estimators using vector approximate message passing (VAMP) algorithms \cite{rangan2019vector} (see details in \Cref{vamp} and \Cref{section:DIST}). { The proof relies on several new theoretical developments for VAMP algorithms, including a Cauchy convergence guarantee (cf. \Cref{Corc}), existence of fixed points (cf. \eqref{fp} and the discussion that follows), and a universality result (cf. \Cref{bigsecuni}). For the challenging case of debiasing the Lasso, we introduce a novel covering argument to control the behavior of the design submatrix (see \Cref{secmainLasso} and \ref{ssvods} for details). We believe these technical contributions may be of independent interest to the signal processing \cite{takeda2006analysis}, probability \cite{tulino2013support}, statistical physics \cite{takeuchi2019rigorous}, and information and coding theory \cite{rangan2019vector,pandit2020inference,xu2023capacity} communities, where right-rotationally invariant designs arise in a range of problems.}

Among the literature related to right-rotationally invariant designs, two prior works are the most relevant for us. Of these, \cite{gerbelot2022asymptotic}
initiated a study of the risk of $\boldsymbol{\hat{\beta}}$ under right-rotationally invariant designs using the VAMP machinery. However, their characterization is partially heuristic, meaning that they assume certain critical exchange of limits is allowed and that limits of certain fundamental quantities exist. 
The former assumption may often not hold, and the latter is unverifiable without proof (see  Remark \ref{rem:cedric} for further details). { In addition, they simply assumed that the system of fixed point equations (cf. \eqref{fp}) admit a solution in their proof.} As a by-product of our work on debiasing, we provide a complete rigorous characterization of the risk of regularized estimators under right-rotationally invariant designs (\Cref{thm:empmain}) without these unverifiable assumptions. { We also extend the result to the broader spectral universality class identified in \cite{dudeja2022spectral,dudeja2023universal,wang2022universality}. }
The second relevant work is \cite{takahashi2018statistical}, which conjectures a population version of a debiasing formula for the Lasso using non-rigorous statistical physics tools.  To be specific, they conjecture a debiasing formula that involves unknown parameters related to the underlying limiting spectral distribution of the sample covariance matrix. This formula does not provide an estimator that can be calculated from the observed data. In contrast, we develop a complete data-driven pipeline for debiasing and  develop a consistent estimator for its asymptotic variance.

\section{Extensions}\label{section:extention}

{ \subsection{Universality}\label{bigsecuni}
We show that \Cref{NEIGMAIN} and \Cref{PCRTHM}, i.e. asymptotic normality of the Spectrum-Aware debiased and the PCR-Spectrum-Aware debiased estimators, hold for a broader universality class of designs proposed in \cite{dudeja2022spectral}. 

Let us first review the spectral universality class defined in \cite{dudeja2022spectral}, Definition 1. 
\begin{Definition}[Spectral Universality Class]\label{uniclass}
Given a compactly supported probability measure $\mu$ on $[0, \infty)$, we say that a sensing matrix $\X$ lies in the universality class $\mathscr{U}(\mu)$ if:
\begin{enumerate}
    \item[(i)] Random Signs. $\X=\Jbm \Sbm$ where $\Jbm \in \mathbb{R}^{n \times p}$ is a deterministic matrix and $\Sbm=\operatorname{diag}\left(s_{1: p}\right)$ is a diagonal matrix of i.i.d. Rademacher signs $s_{1: p} \stackrel{\text { i.i.d. }}{\sim} \operatorname{Unif}(\{ \pm 1\})$.
    \item[(ii)] Bounded Operator norm. $\|\Jbm\|_{\mathrm{op}} \lesssim 1$.
    \item[(iii)] Convergence of Empirical Spectral Measure. For any fixed $k \in \mathbb{N}$,
    \begin{equation*}    \operatorname{Tr}\left[\left(\Jbm^{\top} \Jbm\right)^k\right] / p \rightarrow \int \lambda^k \mu(\mathrm{~d} \lambda) \quad \text { as } p \rightarrow \infty
    \end{equation*}
    \item[(iv)] Generic Right Singular Vectors. For any fixed $k \in \mathbb{N}, \epsilon>0$,
    \begin{equation*}
    \left\|\left(\Jbm^{\top} \Jbm\right)^k-\frac{\operatorname{Tr}\left[\left(\Jbm^{\top} \Jbm\right)^k\right]}{p} \mathbf{I}_p\right\|_{\infty} \lesssim p^{-1 / 2+\epsilon} .
    \end{equation*}
    This means that for any $k \in \mathbb{N}, \epsilon>0$ there are constants $C(k, \epsilon)>0, p_0(k, \epsilon) \in \mathbb{N}$ such that:
    \begin{equation*}
    \left\|\left(\Jbm^{\top} \Jbm\right)^k-\frac{\operatorname{Tr}\left[\left(\Jbm^{\top} \Jbm\right)^k\right]}{p} \mathbf{I}_p\right\|_{\infty} \leq C(k, \epsilon) \cdot p^{-1 / 2+\epsilon} \quad \forall p \geq p_0(k, \epsilon)
    \end{equation*}
\end{enumerate}
In the above display, for a matrix $\mathbf{A} \in \mathbb{R}^{p \times p},\|\mathbf{A}\|_{\infty} \stackrel{\text { def }}{=} \max _{i, j \in[p]}\left|A_{i j}\right|$ is the entry-wise infinity norm.
\end{Definition}

\begin{Example}[Examples of Spectral Universality Class]

The spectral universality class includes right-rotationally invariant design matrices defined in \Cref{AssumpD} along with a variety of design matrices. We include the following examples from \cite{dudeja2022spectral} Section 2. 

\begin{itemize}
\item \textbf{Linear transformations of i.i.d.\ matrices.}
  \[
    \mathbf{X}=\mathbf{T}\mathbf{Z},
  \]
  where $\mathbf{T}\in\mathbb{R}^{n\times n}$ is deterministic with $\|\mathbf{T}\|_{\mathrm{op}}\lesssim 1$ and the empirical distribution of the eigenvalues of $\mathbf{T}\mathbf{T}^\top$ converges; $\mathbf{Z}\in\mathbb{R}^{n\times p}$ has i.i.d.\ entries with $\sqrt{p}\,(\mathbf{Z})_{ij}$ mean $0$, variance $1$, finite moments, and a symmetric distribution ($( \mathbf{Z})_{ij}\stackrel{d}{=}-(\mathbf{Z})_{ij}$). This unifies the standard i.i.d.\ model ($\mathbf{T}=\mathbf{I}_n$), the \emph{elliptic} model ($\mathbf{T}=\operatorname{diag}(t_1,\dots,t_n)$), and more general preconditioned designs. 
  
\item \textbf{Sign- and permutation-invariant matrices.}
Let $\mathbf{X}=\mathbf{Q}^\top \mathbf{D}\,\mathbf{O}$, where $\mathbf{Q}\in O(n)$ is deterministic (here $O(n)$ denotes the group of $n\times n$ orthogonal matrices), $\mathbf{D}\in\mathbb{R}^{n\times p}$ is a deterministic rectangular diagonal matrix (with $\|\mathbf{D}\|_{\mathrm{op}}=O(1)$ and a well-defined limiting spectral measure), and
\[
  \mathbf{O}=\mathbf{S}\mathbf{V}\mathbf{P},
\]
with $\mathbf{V}\in O(p)$ deterministic and delocalized ($\|\mathbf{V}\|_\infty\lesssim p^{-1/2+\varepsilon}$, $\varepsilon$ is arbitrarily small constant), $\mathbf{S}=\operatorname{diag}(s_1,\ldots,s_p)$ a diagonal Rademacher sign matrix, and $\mathbf{P}$ a uniformly random permutation matrix (independent of $\mathbf{S}$). One important example is the randomized partial Hadamard–Walsh matrix, $$\mathbf{X}=[\mathbf{I}_n,\mathbf{0}]\,\mathbf{P}^\top \mathbf{H}_p \mathbf{S}$$
where $\mathbf{H}_p$ is the orthonormal Hadamard–Walsh matrix. This type of matrix is commonly used as a structured dimension-reduction map in numerical linear algebra and high-dimensional data analysis \cite{halko2011finding,nguyen2009fast}. 

 \item \textbf{Randomized/subsampled orthogonal matrices.}
  With an integer aspect ratio $L=p/n$ (fixed), define
  \[
    \mathbf{X}=\big[\,\mathbf{D}_1\mathbf{O}\ \ \mathbf{D}_2\mathbf{O}\ \ \cdots\ \ \mathbf{D}_L\mathbf{O}\,\big]\mathbf{S},
  \]
  where $\mathbf{O}\in\mathbb{R}^{n\times n}$ is deterministic, delocalized, orthogonal ($\|\mathbf{O}\|_\infty\lesssim n^{-1/2+\varepsilon}$); for each $\ell\in[L]$, $\mathbf{D}_\ell=\operatorname{diag}(d_{\ell,1},\dots,d_{\ell,n})$ has i.i.d.\ bounded, symmetric entries and $\mathbf{D}_1,\dots,\mathbf{D}_L$ are independent; and $\mathbf{S}$ is a diagonal Rademacher sign matrix. Instances of this form arise in compressed sensing (e.g., subsampled randomized Hadamard-Walsh sensing matrices, masked orthogonal sensing matrices), fast random features (ORF/SORF), SRHT-based subspace embeddings and least-squares preconditioning, and locally private frequency estimation via the Hadamard Response \citep{dudeja2022spectral,rauhut2010structured,yu2016orthogonal,tropp2011srht,acharya2019hadamard}.

\item \textbf{Signed incoherent tight frames.}
Let $\mathbf{F}_{n,p}\in\mathbb{R}^{n\times p}$ satisfy $\mathbf{F}\mathbf{F}^\top=I_n$, $(\mathbf{F}^\top\mathbf{F})_{ii}=n/p$, and $\|\mathbf{F}^\top\mathbf{F}-(n/p)I_p\|_\infty\lesssim p^{-1/2+\varepsilon}$. For a diagonal Rademacher matrix $\mathbf{S}$, set $\mathbf{X}=\mathbf{F}\mathbf{S}$. Instances of this form arise in compressed sensing (partial Fourier/Hadamard and low-coherence/equiangular constructions), randomized numerical linear algebra (SRHT/SRFT subspace embeddings for least squares/PCA), and coding/line packing via equiangular tight frames \citep{rauhut2010structured,halko2011randomized,strohmer2003grassmannian,foucart2013cs}.
%
% (Related signed embeddings such as Achlioptas binary JL projections and feature hashing are not tight frames but are closely connected in practice \cite{achlioptas2003,weinberger2009}.)
  %
\end{itemize}
\end{Example}

The proposition below is a direct consequence of \cite{dudeja2022spectral}, Lemma 3 and Example 2.
\begin{Proposition}\label{runri}
    If $\X$ satisfies \Cref{AssumpD}, then $\X \in \mathscr{U}(\mu)$ where $\mu$ is probability measure of random variable $\mathsf{D}^2$. 
\end{Proposition}

The first of our main results in this section is the following. We note that the convergence is in probability as opposed to almost sure as in \Cref{NEIGMAIN}. This is because the universality result in \cite{dudeja2022spectral} (i.e. Theorem 3) is only established in in-probability sense. Another difference is that we require entries of the signal vector $\st$ to be i.i.d. draws from $\Xstar$, instead of $\st \stackrel{W_2}{\to} \Xstar$. This stronger requirement is a consequence of condition required for the universality result (i.e. \cite{dudeja2022spectral}, Theorem 3). We defer the proof to \Cref{proofsunidf}.

\begin{Theorem}\label{universethm}
    Suppose that $\X \in \mathscr{U}(\mu)$ for a compactly supported probability measure $\mu$ and \Cref{AssumpPrior}---\ref{Assumpfix2} hold. Furthermore, suppose that entries of $\st$ are i.i.d.~draws from the law of $\Xstar$ (defined in \Cref{AssumpPrior}). Then in probability as $p\to \infty$,
\begin{equation*}
        \tauh^{-1/2}(\bhetah-\st)\stackrel{W_2}{\to} N(0,1).
\end{equation*}
\end{Theorem}

We now introduce a universality principle for the PCR-Spectrum-Aware estimator. To start, we present an analog of \Cref{AssumpPriorAL}. Note that in \eqref{wdfm} we assumed that the projection of $\zetr$ onto the low-dimensional subspace spanned by the eigenvectors indexed by $\Js$ is small. Under \Cref{AssumpPriorAL}, this was automatically true since when $\X$ is right-rotationally invariant, the projection subspace is chosen uniformly at random. When we work the general spectral universality class studied in this section, one needs to assume this projection property separately, hence condition \eqref{wdfm} below. 
%while imposing the design condition \Cref{uniclass} on $\Xnew$. 
\begin{Assumption}\label{univ8}
We assume that $\Js$ is of finite size and for some real-valued vectors $\alphr \in \R^{J},\zetr\in \R^{p}$,
    \begin{equation}\label{stwerif}
        \st=\stal+\zetr, \qquad  \stal=\sum_{i=1}^{J} \alphstari \cdot \mathbf{o}_{\Js(i)}.
    \end{equation}
where we used $\Js(i)$ to denote the $i$-th index in $\Js$. We assume that enties of $\zetr$ are i.i.d. copies of a  random variable $\Zetr$ with finite variance. Furthermore, we assume that $\Xnew$, defined in \eqref{xnewYnew}, lies in the spectral universality class from  \Cref{uniclass}. %$\Xnew \in \mathscr{U}(\mu)$ for a compactly supported probability measure $\mu$ and spectral universality class $\mathscr{U}(\mu)$ defined in \Cref{uniclass} and $\Xnew$ defined in \eqref{xnewYnew}.
In probability, the projection of $\Zetr$ satisfies
%We assume in addition that in probability
\begin{equation}\label{wdfm}
    \lim_{p\to \infty}p^{-1} \norm{\mathbf{P}_{\Js} \zetr}^2=0.
\end{equation}
Finally, as in \Cref{AssumpPriorAL}, the sizes of eigenvalues corresponding to indices in $\Js$ satisfy 
$\limsup_{p\to \infty} \max_{i\in \Js } d_i^{-2}/p\to 0.$
\end{Assumption}

We now state our universality result for the PCR-Spectrum-Aware estimator. We defer the proof to \Cref{proofsunidf}.

\begin{Theorem}\label{univPCRthm}
    Suppose Assumptions \ref{Assumph}---\ref{Assumpfix2} and \Cref{univ8} hold. Then, inprobability as $p\to \infty$, we have the following: (a) Alignment PCR: $\frac{1}{p}\norm{\pcrt(\Js)-\stal}^2\to 0$; (b) Complement PCR: $\tauc^{-1/2}\qty(\pcrc(\Jsb)-\zetr)\stackrel{W_2}{\to} N(0,1)$; (c) Debiased PCR: $\tauc^{-1/2}\qty(\pcrdb-\st)\stackrel{W_2}{\to} N(0,1).$
\end{Theorem}

\subsection{Lasso Penalty}\label{secmainLasso}
Thus far, we have operated under \Cref{Assumpgp}, i.e., the penalty needs to be strongly convex or that $\X^\top \X$ needs to be positive definite, which excludes the case of the Lasso penalty in high dimensions. In this section, we extend our results to the Lasso under mild additional assumptions. Our first assumption 

\begin{Assumption}\label{AssumpLassoo}
    There exists some $c>0$ such that for all sufficiently large $p$, the minimum \textit{positive} eigenvalues of $\X^\top \X$ is bounded away from zero, i.e.
    \begin{equation}\label{hahabao}
         \lim_{p \to \infty} \min \{d_i^2: d_i^2>0, i\in [p]\}>c.
    \end{equation}
\end{Assumption}

\begin{Assumption}\label{AssumpLassoo2}
    We require that either all eigenvalues of $\X^\top \X$ are bounded away from zero, i.e. for some $c>0$,
    $d_-\equiv\lim_{p\to \infty }\min_{i\in [p]}(d_i^2)>c$
    or there exists a solution to the system of fixed point equation \eqref{fp} such that
    \begin{equation}\label{shichi}
        \frac{\gamma_*}{\eta_*}\equiv \mathbb{P}\left(\left|\sqrt{\tau_*} \Zs+\Xstar\right| \geq \gamma_*^{-1} \lambda_1\right)<C_{\mathrm{abs}}\cdot \qty(\frac{\E \D^2}{d_+})^3.
    \end{equation} 
\end{Assumption}

\begin{Remark}
The condition \eqref{shichi} in \Cref{AssumpLassoo2} is not explicit. However, we note that under mild conditions, $\gamma_* / \eta_*$ tends to zero as $\lambda_1 \rightarrow 0$ (see \Cref{jiros} for precise statement). So, condition \eqref{shichi} is satisfied for all  $\lambda_1$ above a certain threshold. Meanwhile, we note that if conditions of \Cref{Smdfkf} below are satisfied, then $\frac{\hat{s}}{p}\to \frac{\gamma_*}{\eta_*}$ almost surely as $n,p\to \infty$. Therefore, \eqref{shichi} may be interpreted as requiring the active set to be not too large, that is, we are in the regime where the Lasso solutions are suitably sparse, which is quite natural to assume for the lasso. Analogous assumptions have also appeared in the prior Lasso literature (c.f.~ \cite{lockhart2014significance,ChatterjeeLahiri2013AdaptiveLassoRates,LiuYu2013LassoMLS}). 
% \ps{Reference to other high-dim stat literature where such assumption arises will be useful. Explicit threshold for lambda is also useful.}
%$\lambda_1$ to be above a threshold determined by...

%\ps{can this be written in terms of $\lambda_1$ being above a threshold determined by what, what and what? Instead of ``Lasso solution to be sufficiently sparse"?} \yl{fdf}
%hyperparamLasso solution to be sufficiently sparse. 
\end{Remark}

The following result is proved in \Cref{wakakafff}. Our proof begins with the strategy used in \cite{bayati2011lasso}; we seek to show that the VAMP iterates converge  to the Lasso solution by proving that the design submatrix consisting of columns selected by the active support of VAMP iterates has its smallest singular value bounded away from zero. However, this turns out to be a significantly difficult task for right rotationally invariant designs where \cite{bayati2011lasso}'s argument no longer applies. We establish the result using a novel $\varepsilon$-net argument.

%\ps{There needs to be a main theorem in this subsection like Theorem 5.2 has been written in the previous subsection.}
\begin{Theorem}\label{Smdfkf}
Assume that Assumptions \ref{AssumpD}, \ref{AssumpPrior}, \ref{AssumpLassoo} and \ref{AssumpLassoo2} hold and that the penalty function is given by 
$h(x)=\lambda_1 |x|$
for some \(\lambda_1>0\). Then, all of our aforementioned results, i.e. Theorems \ref{NEIGMAIN}, \ref{PCRTHM} \ref{universethm}, \ref{univPCRthm} and Corollaries \ref{sgoods}, \ref{weeren}, \ref{testCor}, hold without requiring Assumption \ref{Assumpgp}.
% is no longer required for theoretical results throughout the paper. 
%\ps{You should write this as ``Then, Theorems blah to blah hold. That is, Assumption 4 is not longer required"  } \yl{Done. See above}
\end{Theorem}
}

{
\section{Proofs for Extensions to the Spectral Universality Class of Designs} \label{proofsunidf}
Analogously to \cite{dudeja2022spectral} Definition 2, we introduce the asymptotic equivalence of two random vectors for the ease of presentation. 
\begin{Definition}\label{defPW2}
Let $\left(\boldsymbol{v}^{(1)}, \ldots, \boldsymbol{v}^{(k)}\right)$ and $\left(\widetilde{\boldsymbol{v}}^{(1)}, \ldots, \widetilde{\boldsymbol{v}}^{(k)}\right)$ be two collections of $p$-dimensional vectors. We say that $\left(\boldsymbol{v}^{(1)}, \ldots, \boldsymbol{v}^{(k)}\right)$ and $\left(\widetilde{\boldsymbol{v}}^{(1)}, \ldots, \widetilde{\boldsymbol{v}}^{(k)}\right)$ are asymptotically equivalent with respect to the Wasserstein-2 metric if for any continuous test function $h: \mathbb{R}^{k} \rightarrow \mathbb{R}$ (independent of $p$) that satisfies:
\begin{equation*}
|h(\vbf)-h(\vbf^\prime)| \leq C\|\vbf-\vbf^\prime\|\left(1+\|\vbf\|+\|\vbf^\prime\|\right) ,\quad \forall \vbf, \vbf^\prime \in \mathbb{R}^k
\end{equation*}
for some finite constants $C \geq 0$, we have, in probability,
\begin{equation*}
\frac{1}{p} \sum_{i=1}^p h\left(v_i^{(1)}, v_i^{(2)}, \ldots, v_i^{(k)} \right)-\frac{1}{p} \sum_{i=1}^p h\left(\widetilde{v}_i^{(1)}, \widetilde{v}_i^{(2)}, \ldots, \widetilde{v}_i^{(k)} \right) \to 0.
\end{equation*}
We denote equivalence in the above sense using the notation $\left(\boldsymbol{v}^{(1)}, \boldsymbol{v}^{(2)}, \ldots, \boldsymbol{v}^{(k)} \right) \stackrel{PW_2}{\simeq}\left(\widetilde{\boldsymbol{v}}^{(1)}, \ldots, \widetilde{\boldsymbol{v}}^{(k)} \right)$.
\end{Definition}
We also denote singular value decomposition of $\Jbm$ as 
$$\Jbm=\Qbm^\top \Dbm \Rbm$$
whereby $\X$ can be decomposed as
$$\X=\Qbm^\top \Dbm \Obm, \quad \Obm=\Rbm \Sbm.$$
Note that the only source of randomness in the design now comes from $\Sbm$. 

We now proceed to prove \Cref{universethm}. 

\begin{proof}[Proof of \Cref{universethm}]
We note that the main task is to establish a universality principle for the state evolution of the VAMP algorithm, i.e. we want to prove \Cref{SEovamp} for $\X \in \mathscr{U}(\mu)$. The rest of the proof is identical to that of \Cref{NEIGMAIN} since right-rotationally invariance of the design $\X$ is only used to establish \Cref{SEovamp}. That is, it suffices to prove the following for the VAMP algorithm.

\begin{Claim}\label{universalSE}
Under the assumptions of \Cref{universethm}, as $p,n\to \infty$, we have in probability
    \begin{equation}\label{vv1}
        \begin{aligned}
            \left(\xonet, \ronet, \st\right) \stackrel{W_2}{\to}\left(\operatorname{Prox}_{\gamma_*^{-1} h}\left(\sqrt{\taustar} \mathsf{Z}+\Xstar\right), \sqrt{\taustar} \mathsf{Z}+\Xstar, \Xstar\right).
        \end{aligned}
    \end{equation}
    Furthermore, as $p,n\to \infty$, we have in probability
\begin{equation}\label{vv2}
\begin{aligned}
    &\frac{1}{p}\left\|\X \rtwot-\y\right\|^2 \rightarrow \tau_{**} \mathbb{E} \D^2+\delta \\
    & \frac{1}{p} \norm{\y-\X \hat{\mathbf{x}}_{2 t}}^2 \to \tau_{**} \cdot \E \frac{\D^2(\eta_*-\gamma_*)^2}{(\D^2+\eta_*-\gamma_*)^2}+\frac{n-p}{p}+\E \qty(\frac{\eta_*-\gamma_*}{\D^2+\eta_*-\gamma_*})^2.
\end{aligned}
\end{equation}
\end{Claim}

Let us recall the oracle VAMP algorithm \eqref{empvamp} may be written, upon change of variable 
\begin{equation}\label{werbadfd}
    \ampxt=\rtwot-\st, \quad \ampyt=\ronet-\st-\mathbf{e}
\end{equation}
as the following iterations: with initialization $\mathbf{q}^0\sim N(\rm{0},\taustar \cdot \mathbf{I}_p)$, $ \ampxone = F(\mathbf{q}_0,\st)$, for $t=1,2,3,\ldots,$
\begin{equation}\label{vmpwer}
\mathbf{y}^{t}=\Obm^\top \bm{\Lambda} \Obm \ampxt, \quad \mathbf{x}^{t+1}=F(\ampyt+\mathbf{e},\st)
\end{equation}
Recall that here, $F$ is a scalar-valued function defined in \eqref{Fdef} applied entry-wise, and $\bm{\Lambda}, \mathbf{e}$ defined in \eqref{Quant}.
To ease notation, let us further define vector 
$$\bm{\nu}:=\Obm^\top (\Dbm^\top \Dbm)^{1/2} (\Dbm^\top \Dbm+(\eta_*-\gamma_*)\cdot \mathbf{I}_p)^{-1}\Dbm^\top \Qbm \epbm 
$$
and functions $h_1:\R \mapsto \R, h_2:\R \mapsto \R$ and $h_3:\R \mapsto \R$
\begin{equation*}
    \begin{aligned}
        &h_1(x):=\frac{\sqrt{x}}{x+(\eta_*-\gamma_*)},  \quad h_2(x):=\frac{x}{x+(\eta_*-\gamma_*)}\\
        &h_3(x):=\frac{\eta_*\left(\eta_*-\gamma_*\right)}{\gamma_*\left(x+\left(\eta_*-\gamma_*\right) \right)}-\left(\frac{\eta_*-\gamma_*}{\gamma_*}\right).
    \end{aligned}
\end{equation*}
We observe that for some $\mathbf{z}\sim N(0, \mathbf{I}_p)$ independent of $\X, \st$,
$$ (\mathbf{e}, \X\epbm, \bm{\nu})|_{\X}= \qty(\frac{\eta_*}{\gamma_*}h_1(\X^\top \X) \mathbf{z}, \sqrt{\X^\top \X}\mathbf{z},  h_2(\X^\top \X) \mathbf{z}), \quad \Obm^\top \bm{\Lambda} \Obm=h_3(\X^\top \X).$$ 
Therefore, the iterations in \eqref{vmpwer} may be written as follows: with initialization $\mathbf{q}^0\sim N(\rm{0},\taustar \cdot \mathbf{I}_p)$, $\mathbf{y}^1=h_2(\X^\top \X) F(\mathbf{q}_0, \st)$, for $t=1,2,3,\dots,$
\begin{equation}\label{itAA}
    \mathbf{e}=\frac{\eta_*}{\gamma_*}h_1(\X^\top \X) \mathbf{z}, \qquad \mathbf{y}^{t+1}=h_3(\X^\top \X)F(\ampyt+\mathbf{e},\st).
\end{equation}
Meanwhile, we introduce the following auxiliary iterates:
\begin{equation}\label{itBB}
    \begin{aligned}
    &\mathbf{c}=\sqrt{\X^\top \X} \mathbf{z}, \qquad \mathbf{d}=h_2({\X^\top \X}) \mathbf{z},\\
    &\mathbf{r}^t=\sqrt{\X^\top \X}F(\ampyt+\mathbf{e},\st), \qquad \mathbf{w}^t=h_1(\X^\top  \X) F(\ampyt+\mathbf{e},\st). 
\end{aligned}
\end{equation}
For ease of notation, we let
$$\mathbf{Y}_t:=(\mathbf{y}^1,...,\mathbf{y}^t), \quad \mathbf{R}_t:=(\mathbf{r}^t,...,\mathbf{r}^t), \quad \mathbf{W}_t:=(\mathbf{w}^1,...,\mathbf{w}^t). $$

Now we claim the following universality principle regarding the iteration above: Let $\X, \widetilde{\X}$ be two independent design matrices in the same universality class $\mathscr{U}(\mu)$ and 
\begin{equation}\label{rbsm}
\begin{aligned}
&(\mathbf{Y}_t(\X^\top \X), \mathbf{R}_t(\X^\top \X),\mathbf{W}_t(\X^\top \X), \mathbf{c}(\X^\top \X), \mathbf{d}(\X^\top \X),\mathbf{e}(\X^\top \X), \st)\\
&\qquad \stackrel{PW_2}{\simeq}(\mathbf{Y}_t(\widetilde{\X}^\top \widetilde{\X}), \mathbf{R}_t(\widetilde{\X}^\top \widetilde{\X}),\mathbf{W}_t(\widetilde{\X}^\top \widetilde{\X}), \mathbf{c}(\widetilde{\X}^\top \widetilde{\X}), \mathbf{d}(\widetilde{\X}^\top \widetilde{\X}),\mathbf{e}(\widetilde{\X}^\top \widetilde{\X}),\st) 
\end{aligned}
\end{equation}
where $\stackrel{PW_2}{\simeq}$ is defined in \Cref{defPW2}. This equivalence follows from the universality principle for general first order methods as in \cite{dudeja2022spectral}, Theorem 3.
To apply \cite{dudeja2022spectral}, Theorem 3, we need to verify condition 1---3 for the above iteration. Condition 2 and 3 are immediate from our assumption that entries of $\st$ are iid copies of $\Xstar$ and that $(q,x)\mapsto F(q,x)$ are Lipschitz continuous. For condition 1, we use the same argument as in the proof of \cite{dudeja2022spectral}, Theorem 1. That is, we may first approximate preconditioning matrices $h_1(\X^\top \X), h_2(\X^\top \X), h_3(\X^\top \X), \sqrt{\X^\top \X}$ with polynomials of $\X^\top \X$. By \cite{dudeja2022spectral}, Remark 8, the preconditioning matrices form a strongly semi-random ensemble for any such approximation with finite $k$-degree polynomials, which implies that the equivalence \eqref{rbsm} holds for the iterates for any $k\in \N$. We may then obtain \eqref{rbsm} by invoking  Stone-Weierstrass theorem and taking $k\to \infty$. 

Now, \eqref{vv1} follows immediately from \eqref{rbsm}, \eqref{werbadfd} and \Cref{SEovamp}. To show \eqref{vv2}, note that LHS may be expressed in terms of iterates in \eqref{itAA},\eqref{itBB}
\begin{equation*}
\begin{aligned}
    \frac{1}{p}\left\|\X \rtwot-\y\right\|^2 &= \frac{1}{p}\norm{\epbm}^2+\frac{1}{p}\norm{\mathbf{r}^t}^2+\frac{1}{p}\mathbf{c}^\top\mathbf{r}^t,  \\
     \frac{1}{p} \norm{\y-\X \hat{\mathbf{x}}_{2 t}}^2 &= \frac{1}{p}\big(\norm{\epbm}^2+\norm{\bm{\nu}}^2+(\eta_*-\gamma_*)^2\norm{\mathbf{w}^t}^2+(\eta_*-\gamma_*)\bm{\nu}^\top 
    \mathbf{w}^t\\
    & \qquad \qquad -2(\eta_*-\gamma_*)\mathbf{c}^\top\mathbf{w}^t-2\mathbf{c}^\top\mathbf{\nu}\big).
\end{aligned}
\end{equation*}
where for the second line we used the identities \eqref{identv1} and \eqref{identv2}. Therefore, \eqref{vv2} follows from \eqref{rbsm} and \Cref{SEovamp}. This concludes the proof of \Cref{universalSE} and thus \Cref{universethm}. 
\end{proof}

As a direct corollary of the proof above, we obtain universality of the distributional characterization. 
\begin{Corollary}[Universality of distributional characterizations]\label{thm:empmainuniv}
Suppose that $\X \in \mathscr{U}(\mu)$ for a compactly supported probability measure $\mu$ and \Cref{AssumpPrior}---\ref{Assumpfix2} hold. Furthermore, we require that entries of $\st$ are i.i.d. copies of the random variable $\Xstar$ from \Cref{AssumpPrior}. Then in probability as $p\to \infty$,
\begin{equation}\label{waka1}
\left(\hatbt, \rstar, \st\right) \stackrel{W_2}{\rightarrow}\left(\operatorname{Prox}_{\gamma_*^{-1} h}\left(\sqrt{\taustar} \Zs+\Xstar \right), \sqrt{\taustar} \Zs+\Xstar, \Xstar \right),
\end{equation}
where $\Zs\sim N(0,1)$ is independent of $\Xstar$. Furthermore, in probability as $p\to \infty$
\begin{equation}\label{waka2}
    \begin{aligned}
        &\frac{1}{p}\left\|\X \rstst-\y\right\|^2 \to \tau_{**} \cdot \mathbb{E}\D^2+\sigma^2\cdot \delta, \\
        &\frac{1}{p} \norm{\y-\X \hatbt}^2\to \tau_{**} \cdot \E \frac{\D^2(\eta_*-\gamma_*)^2}{(\D^2+\eta_*-\gamma_*)^2}+\sigma^2\cdot \qty(\frac{n-p}{p}+\E \qty(\frac{\eta_*-\gamma_*}{\D^2+\eta_*-\gamma_*})^2).
    \end{aligned}
\end{equation}
\end{Corollary}

\begin{proof}[Proof of \Cref{univPCRthm}]
    The proof of part (a) follows the same lines as in \Cref{PCRTHM}, except that \eqref{asffsnone} now immediately follows from our assumption. Similarly, the proof of part (b) mirrors that in \Cref{PCRTHM}, with the asymptotic normality of the Spectrum-Aware debiased estimator for the new problem \eqref{filteredpcr} now derived from \Cref{universethm} rather than \Cref{NEIGMAIN}. Finally, part (c) is an immediate consequence of (a) and (b).
\end{proof}

\section{Proofs for Extension to the Lasso}\label{sec:Lasso}

\subsection{Existence and properties of fixed points}\label{sec:Lassoexist}
\begin{Lemma}\label{tobound}
Let \(\Zs\sim N(0,1)\). Then, for any \(a>0\), we have
\[
\mathbb{E}\Bigl[(|\Zs|-a)^2\mid |\Zs|>a\Bigr]<1.
\]
\end{Lemma}

\begin{proof}[Proof of \Cref{tobound}]
We begin by writing the conditional expectation as
\[
\mathbb{E}\Bigl[(|\Zs|-a)^2\mid |\Zs|>a\Bigr]=\frac{\mathbb{E}\Bigl[(|\Zs|-a)^2\,\mathbb{I}({|\Zs|>a})\Bigr]}{\mathbb{P}(|\Zs|>a)}.
\]
Since \(\Zs\) is symmetric, we can write the numerator as
\[
\mathbb{E}\Bigl[(|\Zs|-a)^2\,\mathbb{I}({|\Zs|>a})\Bigr]
=2\int_a^\infty (x-a)^2\,\phi(x)\,dx.
\]
Next, expand the square:
\[
\int_a^\infty (x-a)^2\,\phi(x)\,dx=\int_a^\infty x^2\,\phi(x)\,dx-2a\int_a^\infty x\,\phi(x)\,dx+a^2\int_a^\infty \phi(x)\,dx.
\]
Using the standard identities for the normal distribution:
\[
\int_a^\infty \phi(x)\,dx=1-\Phi(a),\quad
\int_a^\infty x\,\phi(x)\,dx=\phi(a),\quad
\int_a^\infty x^2\,\phi(x)\,dx=a\phi(a)+1-\Phi(a),
\]
we obtain:
\[
\begin{split}
\int_a^\infty (x-a)^2\,\phi(x)\,dx 
&=\Bigl[a\phi(a)+1-\Phi(a)\Bigr]-2a\,\phi(a)+a^2\Bigl[1-\Phi(a)\Bigr]\\[1mm]
&=(1-\Phi(a))(1+a^2)-a\phi(a).
\end{split}
\]
Thus, the numerator of the conditional expectation is
\[
2\Bigl[(1-\Phi(a))(1+a^2)-a\phi(a)\Bigr],
\]
and the probability in the denominator is
\[
\mathbb{P}(|\Zs|>a)=2(1-\Phi(a)).
\]
Hence,
\[
\mathbb{E}\Bigl[(|\Zs|-a)^2\mid |\Zs|>a\Bigr]
=\frac{2\Bigl[(1-\Phi(a))(1+a^2)-a\phi(a)\Bigr]}{2(1-\Phi(a))}
=1+a^2-\frac{a\phi(a)}{1-\Phi(a)}.
\]
Using the classical inequality bound for Mills ratio,
\[
\frac{a\phi(a)}{1-\Phi(a)}\ge a^2
\]
we conclude that
\[
\mathbb{E}\Bigl[(|\Zs|-a)^2\mid |\Zs|>a\Bigr]<1.
\]
This completes the proof.
\end{proof}

\begin{Lemma}\label{lassodfdf}
Let \(\D\) be the random variable defined in \Cref{AssumpD}. There exists a random variable \(\D_+^2\) with \(\mathbb{P}(\D_+^2=0)=0\) such that for every Borel set \(A\subseteq [0,\infty)\)
\[
\mathbb{P}(\D^2\in A)= (1-w)\cdot \delta_0(A)+ w\cdot \mathbb{P}(\D_+^2\in A).
\]
where $w=\mathbb{P}(\D^2>0).$ 
\end{Lemma}

\begin{proof}[Proof of \Cref{lassodfdf}]
We have \(w>0\) or else \(\D^2=0\) violating \Cref{AssumpD}. Define \(\D_+^2\) as the random variable whose law is the conditional distribution of \(\D^2\) given \(\D^2>0\); that is, for any Borel set \(A\subset (0,\infty)\)
\[
\mathbb{P}(\D_+^2\in A)=\frac{\mathbb{P}(\D^2\in A)}{w}.
\]
Since \(\D_+^2\) is supported on \((0,\infty)\), we have \(\mathbb{P}(\D_+^2=0)=0\). By the law of total probability, for every Borel set \(A\subseteq [0,\infty)\) we obtain
\[
\begin{aligned}
    \mathbb{P}(\D^2\in A)&= \mathbb{P}(\D^2\in A\mid \D^2=0)\,(1-w)+\mathbb{P}(\D^2\in A\mid \D^2>0)w\\
&= (1-w)\,\delta_0(A) + w\,\mathbb{P}(\D_+^2\in A).
\end{aligned}
\]
This completes the proof.
\end{proof}

We first require an additional condition on $\D^2$. We note that this condition is mild. It rules out the edge case where $\D^2$ has no mass at zero but eigenvalues of $\X^\top \X$ are not bounded away zero. When $\X$ has i.i.d. sub-Gaussian entries and $n/p\to\varsigma$, $\D^2$ follows the Marchenko–Pastur law; this condition then excludes only the edge case $\varsigma= 1$.

\begin{Assumption}\label{AssumpDLasso}
Let \(\D\) be the random variable defined in \Cref{AssumpD}. Recall from \Cref{lassodfdf} that there exists a random variable \(\D_+^2\) with \(\mathbb{P}(\D_+^2=0)=0\) such that for every Borel set \(A\subseteq [0,\infty)\)
\begin{equation}\label{XL11}
\mathbb{P}(\D^2\in A)= (1-w)\cdot \delta_0(A)+ w\cdot \mathbb{P}(\D_+^2\in A).
\end{equation}
where $w=\mathbb{P}(\D^2>0).$ We require that either $w<1$ (i.e. the distribution of $\D^2$ has a positive probability mass on 0), or if $w=1$, we must have $d_->0$ (i.e. if $\D^2$ has zero probability mass on 0, it must then be bounded away from 0). 
\end{Assumption}

It is easy to see that \Cref{AssumpDLasso} above is implied by \Cref{AssumpLassoo} under \Cref{AssumpD}. 

Below is a restatement of \Cref{fixexist}'s Lasso case. 
\begin{Proposition}\label{fixexistLasso}
    Let $\mathsf{D}^2$ be the random variable defined in \Cref{AssumpD} and satisfy \Cref{AssumpDLasso}. Then \Cref{Assumpfix} holds for $h(x)=\lambda_1 \abs{x}$ for any $\lambda_1>0$. 
\end{Proposition}

\begin{proof}[Proof of \Cref{fixexistLasso} (\Cref{fixexist} Lasso case)]
Recall from the proof of \Cref{fixexist}, we may obtain a new system of fixed equation
\begin{subequations}\label{fpstren}
\begin{align}
& \gamma_*^{-1}=\frac{1}{-R\left(\eta_*^{-1}\right)} \label{RSaLasso}\\
& \eta_*^{-1}=\gamma_*^{-1} \mathbb{P}\left(\left|\frac{1}{\gamma^{-1}_*} \Xstar+\frac{1}{\alpha_*} \mathsf{Z}\right|>\lambda_1\right)  \label{RSbLasso} \\
& 1=\alpha_*^2 R^{\prime}\left(\eta_*^{-1}\right)\mathbb{E}\left(\operatorname{ST}_{\gamma_*^{-1} \lambda_1}\left(\Xstar+\frac{\gamma_*^{-1}}{\alpha_*} \mathsf{Z}\right)-\Xstar\right)^2+\sigma^2\frac{\alpha_*^2}{\gamma_*^{-1}}\left[1+\frac{\eta_*^{-1} R^{\prime}\left(\eta_*^{-1}\right)}{R\left(\eta_*^{-1}\right)}\right] \label{RScLasso}
\end{align}
\end{subequations}
 from \eqref{fp} by eliminating $\tau_{**}$ and introducing a change of variable $\tau_*=\gamma_*^{-2} \alpha_*^{-2}$. 

 Similarly to the proof of \Cref{fixexist}, we also introduce 
 $$\gamma_{+}^{-1}:=\lim _{z \rightarrow G\left(-d_{-}\right)} \frac{1}{-R(z)}.$$ 
 Recall from \eqref{iff} in the proof of \Cref{fixexist},
 \begin{equation}
    \gamma_+^{-1}=+\infty \text{   if and only if   }  G(-d_{-})=+\infty \text{ and } d_{-}=0.
\end{equation}
Combining this with \Cref{AssumpDLasso}, we observe that there are only two possible cases:
\begin{itemize}
    \item[(i)]: $\gamma_+^{-1}<+\infty$: $w=1$, $d_{-}>0$ and \begin{equation}\label{ggin}
        G(-d_-)/\gamma_+^{-1}>1;
    \end{equation}
    \item[(ii)]: $\gamma_+^{-1}=+\infty$: $w\in (0,1)$, $G(-d_{-})=+\infty$ and $d_{-}=0$;
\end{itemize}
 In case (i) above, we obtained \eqref{ggin} from \eqref{cc1} and \eqref{cc2}. 

  We will consider case (i) first. We now proceed to consider finding a solution $\gamma^{-1}=\gamma^{-1}(\alpha)$ from the equation
 \begin{equation}\label{smdfLasso}
    \gamma R^{-1}(-\gamma)=\mathbb{P}\left(\left|\frac{1}{\gamma^{-1}} \Xstar+\frac{1}{\alpha} \mathsf{Z}\right|>\lambda_1\right)
\end{equation}
for $\alpha \in (0,+\infty).$ This amounts to solving for $\gamma_*^{-1}, \eta_*^{-1}$ in terms of $\alpha_*$ from \eqref{RSaLasso} and \eqref{RSbLasso}. 

We have already showed in the proof of \Cref{fixexist} that the LHS is a strictly increasing function in $\gamma^{-1} \in [\frac{1}{\E \D^2},\gamma_+^{-1})$ whereas we know that RHS is a non-increasing function in $\gamma^{-1} \in [\frac{1}{\E \D^2},\gamma_+^{-1})$. We also have that
$$\E \D^2 R^{-1}\qty(-\E \D^2)=0\le  \inf_{\alpha \in (0,+\infty)}\mathbb{P}\left(\left|(\E \D^2) \cdot \Xstar+\frac{1}{\alpha} \mathsf{Z}\right|>\lambda_1\right)$$
and that
 \begin{equation}
    \lim_{\gamma^{-1} \to \gamma_+^{-1}}\gamma R^{-1}(-\gamma)=\frac{G(-d_-)}{\gamma_+^{-1}}>1\ge\lim_{\gamma^{-1} \to \gamma_+^{-1}}\sup_{\alpha \in (0,+\infty)}\mathbb{P}\left(\left|\frac{1}{\gamma^{-1}} \Xstar+\frac{1}{\alpha} \mathsf{Z}\right|>\lambda_1\right)
\end{equation}
where we used \eqref{thething} and \eqref{ggin}. The above ensures that there exists a solution $\gamma^{-1}(\alpha) \in [\frac{1}{\E \D^2},\gamma_+^{-1})$ and that
\begin{equation}\label{bddd1}
    \sup_{\alpha \in (0,+\infty)} \gamma^{-1}(\alpha) < \gamma_+^{-1}.
\end{equation}
Let $\eta^{-1}(\alpha)=R^{-1}(-\gamma(\alpha))$. Since $\gamma^{-1}\mapsto R^{-1} \qty(-\frac{1}{\gamma^{-1}})$ is strictly increasing, we have from \eqref{thething} that
\begin{equation}\label{bddd2}
    \sup_{\alpha \in (0,+\infty)} \eta^{-1}(\alpha) < G(-d_-).
\end{equation}

The next step is to plug $\gamma^{-1}(\alpha)$ and $\eta^{-1}(\alpha)$ into the RHS of \eqref{RScLasso} to obtain the function $v:(0,+\infty)\mapsto (0,+\infty)$
$$\begin{gathered}v(\alpha)=\alpha^2 R^{\prime}\left(\eta^{-1}(\alpha)\right)\left[\mathbb{E}\left(\operatorname{ST}_{\gamma^{-1}(\alpha) \lambda_1}\left(\Xstar+\frac{\gamma^{-1}(\alpha)}{\alpha} \mathsf{Z}\right)-\Xstar\right)^2\right] \\ +\sigma^2 \alpha^2 \frac{1}{\gamma^{-1}(\alpha)}\left[1+\frac{\eta^{-1}(\alpha) R^{\prime}\left(\eta^{-1}(\alpha)\right)}{R\left(\eta^{-1}(\alpha)\right)}\right]\end{gathered}$$
and show that the RHS of \eqref{RSc}, i.e. $v(\alpha)$, diverges to $+\infty$ as $\alpha \to +\infty$ and goes to some value less than $1$ as $\alpha \to 0$. Given \eqref{bddd1} and \eqref{bddd2}, This step is identical to the same step in the proof of \Cref{fixexist}. 

We now proceed to consider case (ii). From \eqref{lem:cauchy}, we see that under case (ii), $G^{-1}(z)$ and $R(z)$ are defined on the domain $(0,+\infty)$ and $z\mapsto R^{-1}(-1/z)$ is defined on the domain $[\frac{1}{\E \D^2},+\infty)$. Before proving the existence of fixed points, we first prove the asymptotic statements in \eqref{blefd} and \eqref{thebig2}.

We let $x\left(\eta^{-1}\right)=\eta^{-1} G^{-1}\left(\eta^{-1}\right)$ for $\eta^{-1} \in(0,+\infty)$. We then have that
\begin{equation}\label{yadd}
\mathbb{E}\left[\frac{1}{\frac{x\left(\eta^{-1}\right)}{\eta^{-1}}+\D^2}\right]=\eta^{-1} \Leftrightarrow(1-w)+w \mathbb{E}\left[\frac{x\left(\eta^{-1}\right)}{x\left(\eta^{-1}\right)+\eta^{-1} \D_{+}^2}\right]=x\left(\eta^{-1}\right)
\end{equation}
Note that since $\lim _{\eta^{-1} \rightarrow+\infty} G^{-1}\left(\eta^{-1}\right)=0$, dominated convergence theorem implies that as $\eta^{-1} \rightarrow+\infty$
\begin{equation*}
\mathbb{E}\left[\frac{x\left(\eta^{-1}\right)}{x\left(\eta^{-1}\right)+\eta^{-1} \D_{+}^2}\right]=\mathbb{E}\left[\frac{G^{-1}\left(\eta^{-1}\right)}{G^{-1}\left(\eta^{-1}\right)+\D_{+}^2}\right] \rightarrow 0
\end{equation*}
Combining this and \eqref{yadd} implies that
\begin{equation*}
\lim _{\eta^{-1} \rightarrow+\infty} x\left(\eta^{-1}\right) \equiv \lim _{\eta^{-1} \rightarrow+\infty} \eta^{-1} G^{-1}\left(\eta^{-1}\right) \rightarrow 1- w.
\end{equation*}
Hence, 
\begin{equation}\label{blefd}
\lim _{\gamma^{-1} \rightarrow+\infty} \gamma R^{-1}(-\gamma)=\lim _{\eta^{-1} \rightarrow+\infty}-\eta^{-1} R\left(\eta^{-1}\right)=1-\lim _{\eta^{-1} \rightarrow+\infty} \eta^{-1} G^{-1}\left(\eta^{-1}\right)=w.
\end{equation}
We then have that
\begin{equation}\label{blefdB}
\begin{aligned}
& \lim _{\eta^{-1} \rightarrow \infty} \mathbb{E} \frac{1}{\left(\eta^{-1} \D^2+\eta^{-1} G^{-1}\left(\eta^{-1}\right)\right)^2} \\
& =(1-w) \lim _{\eta^{-1} \rightarrow \infty} \frac{1}{\left(\eta^{-1} G^{-1}\left(\eta^{-1}\right)\right)^2}+w \lim _{\eta^{-1} \rightarrow \infty} \mathbb{E} \frac{1}{\left(\eta^{-1} \D_{+}^2+\eta^{-1} G^{-1}\left(\eta^{-1}\right)\right)^2} \\
& =\frac{1}{1-w}
\end{aligned}
\end{equation}
where we used \eqref{blefd} and dominated convergence theorem for the last line. It follows that
\begin{equation}\label{thebig}
\begin{aligned}
& \lim _{\eta^{-1} \rightarrow \infty}-\frac{\eta^{-1} R^{\prime}\left(\eta^{-1}\right)}{R\left(\eta^{-1}\right)} \\
& =\lim _{\eta^{-1} \rightarrow \infty} \frac{\left(\mathbb{E} \frac{1}{\left(\eta^{-1} \D^2+\eta^{-1} G^{-1}\left(\eta^{-1}\right)\right)^2}\right)^{-1} \frac{1}{\eta^{-1} G^{-1}\left(\eta^{-1}\right)}-\frac{1}{\eta^{-1} G^{-1}\left(\eta^{-1}\right)}}{1-\frac{1}{z G^{-1}\left(\eta^{-1}\right)}} \\
& =\frac{(1-w) \frac{1}{1-w}-\frac{1}{1-w}}{1-\frac{1}{1-w}} \\
& =1
\end{aligned}
\end{equation}
where the second line can be seen from the proof of \Cref{lem:cauchy} and the third line uses \eqref{blefd} and \eqref{blefdB}. Combining \eqref{thebig} and \eqref{blefd}, we also have that
\begin{equation}\label{thebig2}
    \lim _{\eta^{-1} \rightarrow \infty} \eta^{-2} R^{\prime}\left(\eta^{-1}\right) \rightarrow w.
\end{equation}

Let us define $\alpha_{\min}=\alpha_{\min}(\lambda_1, w)$ as the solution of the following equation (in terms of $\alpha$
 \begin{equation}\label{amindef}
    w =\mathbb{P}\left(\left|\mathsf{Z}\right|>\alpha \lambda_1\right).
\end{equation}
We note that $\alpha_{\min}\in (0,+\infty)$ is well-defined since under case (ii), $w\in (0,1)$ and RHS is strictly decreasing in $\alpha$ for any $\lambda_1>0$. We now proceed to consider finding a solution $\gamma^{-1}=\gamma^{-1}(\alpha)$ from the equation
 \begin{equation}\label{bbb}
    \gamma R^{-1}(-\gamma)=\mathbb{P}\left(\left|\frac{1}{\gamma^{-1}} \Xstar+\frac{1}{\alpha} \mathsf{Z}\right|>\lambda_1\right)
\end{equation}
for $\alpha \in (\alpha_{\min},+\infty).$ This amounts to solving for $\gamma_*^{-1}, \eta_*^{-1}$ in terms of $\alpha_*$ from \eqref{RSaLasso} and \eqref{RSbLasso}. We have already showed in the proof of \Cref{fixexist} that the LHS is a strictly increasing function in $\gamma^{-1} \in [\frac{1}{\E \D^2},+\infty)$ whereas we know that RHS is a non-increasing function in $\gamma^{-1} \in [\frac{1}{\E \D^2},+\infty)$. We also have that for any $\alpha \in (\alpha_{\min},+\infty)$
\begin{equation*}
    \E \D^2 R^{-1}\qty(-\E \D^2)=0<\mathbb{P}\left(\left|(\E \D^2) \cdot \Xstar+\frac{1}{\alpha} \mathsf{Z}\right|>\lambda_1\right)
\end{equation*}
and that
 \begin{equation*}
    \lim_{\gamma^{-1} \to +\infty}\gamma R^{-1}(-\gamma)=w>\mathbb{P}\left(\left|\mathsf{Z}\right|>\alpha \lambda_1\right)=\lim_{\gamma^{-1} \to +\infty} \mathbb{P}\left(\left|\frac{1}{\gamma^{-1}} \Xstar+\frac{1}{\alpha} \mathsf{Z}\right|>\lambda_1\right)
\end{equation*}
where we used definition of $\alpha_{\min}$ via \eqref{amindef} and the fact that $\mathbb{P}\left(\left|\mathsf{Z}\right|>\alpha \lambda_1\right)$ is strictly decreasing on $\alpha \in (\alpha_{\min},+\infty)$. The above ensures that there exists a solution $\gamma^{-1}(\alpha) \in [\frac{1}{\E \D^2},+\infty)$. 

Let us define $\eta^{-1}(\alpha)=R^{-1}(-\gamma(\alpha))$. The next step is to plug $\gamma^{-1}(\alpha)$ and $\eta^{-1}(\alpha)$ into the RHS of \eqref{RScLasso} to obtain the function $v:(\alpha_{\min},+\infty)\mapsto (0,+\infty)$
$$\begin{gathered}v(\alpha)=\alpha^2 R^{\prime}\left(\eta^{-1}(\alpha)\right)\left[\mathbb{E}\left(\operatorname{ST}_{\gamma^{-1}(\alpha) \lambda_1}\left(\Xstar+\frac{\gamma^{-1}(\alpha)}{\alpha} \mathsf{Z}\right)-\Xstar\right)^2\right] \\ +\sigma^2 \alpha^2 \frac{1}{\gamma^{-1}(\alpha)}\left[1+\frac{\eta^{-1}(\alpha) R^{\prime}\left(\eta^{-1}(\alpha)\right)}{R\left(\eta^{-1}(\alpha)\right)}\right]\end{gathered}$$
and show that the RHS of \eqref{RSc}, i.e. $v(\alpha)$, diverges to $+\infty$ as $\alpha \to +\infty$ and goes to some value less than $1$ as $\alpha \to \alpha_{\min}$. 

First consider any positive increasing sequence $\left(\alpha_m\right)_{m=1}^{+\infty}$ such that $\alpha_m \rightarrow+\infty$ as $m \rightarrow \infty$. We must have that 
$$C_1=\limsup _{m \rightarrow \infty} \gamma^{-1}\left(\alpha_m\right)<+\infty.$$ 
If not, we would have a subsequence $\alpha_{m_t}$ such that
$$\lim_{t\to \infty} \gamma(\alpha_{m_t})R^{-1}(-\gamma(\alpha_{m_t}))\to w<1$$
while
$$\lim_{t\to \infty} \mathbb{P}\left(\left|\frac{1}{\gamma(\alpha_{m_t})} \Xstar+\frac{1}{\alpha_{m_t}} \mathsf{Z}\right|>\lambda_1\right) \to 1.
$$
It follows from this and monotonicity of $z\mapsto R^{-1}(-1/z)$ that
\begin{equation*}
\limsup _{m \rightarrow \infty} \eta^{-1}\left(\alpha_m\right)\le R^{-1}\qty(-\frac{1}{C_1})<+\infty
\end{equation*}
from which we conclude that
\begin{equation*}
C_2:=\liminf _{m \rightarrow \infty} 1+\frac{\eta^{-1}\left(\alpha_m\right) R^{\prime}\left(\eta^{-1}\left(\alpha_m\right)\right)}{R\left(\eta^{-1}\left(\alpha_m\right)\right)}>0
\end{equation*}
This follows from the fact that $\lim _{x \rightarrow 0} 1+\frac{x R^{\prime}(x)}{R(x)}=1$ using \Cref{lem:cauchy}, (f) and continuity of the function $x \mapsto$ $1+\frac{x R^{\prime}(x)}{R(x)}$ on $\left(0, G\left(-d_{-}\right)\right)$. Note that by the above discussion, we have $\liminf _{\alpha \rightarrow+\infty} \frac{v(\alpha)}{\alpha^2} \geq \sigma^2 \frac{C_2}{C_1}$ by lower-bounding second summand in $v(\alpha)$ which then implies that
\begin{equation}\label{liminffixeLasso}
\liminf _{\alpha \rightarrow+\infty} v(\alpha) \rightarrow+\infty.
\end{equation}

Now consider any positive decreasing sequence $\left(\alpha_m\right)_{m=1}^{+\infty}$ such that $\alpha_m \rightarrow \alpha_{\min}$ as $m \rightarrow \infty$. 

First let us define $\gamma^{-1}_0(\alpha)$ as the unique solution of the following equation (in terms of $\gamma^{-1}$)
$$\gamma R^{-1}(-\gamma)=\mathbb{P}\left(\left|\mathsf{Z}\right|>\alpha \lambda_1\right).$$
Similarly to $\gamma^{-1}(\alpha)$, we can show that there $\gamma^{-1}_0(\alpha)\in [\frac{1}{\E \D^2},\infty)$ is well-defined for any $\alpha\in (\alpha_{\min},+\infty)$. Meanwhile, we note that RHS of \eqref{bbb} is non-increasing in $\gamma^{-1}$ and that it converges to $\mathbb{P}\left(\left|\mathsf{Z}\right|>\alpha \lambda_1\right)$ as $\gamma^{-1} \to \infty$ for each fixed $\alpha\in (\alpha_{\min},+\infty)$. It follows that for any $\alpha\in (\alpha_{\min},+\infty)$
\begin{equation}\label{weiwu1}
    \gamma^{-1}(\alpha)\ge \gamma^{-1}_0(\alpha).
\end{equation}
We also that 
\begin{equation}\label{weiwu2}
    \lim_{\alpha \to \alpha_{\min}}\gamma^{-1}_0(\alpha)=+\infty
\end{equation}
which follows from (i) LHS of \eqref{bbb} is strictly increasing in $\gamma^{-1}$ and converges to $w$ as $\gamma^{-1}\to +\infty$ and (ii) $\mathbb{P}\left(\left|\mathsf{Z}\right|>\alpha \lambda_1\right)$ is strictly decreasing in $\alpha$ and converges to $w$ as $\alpha \to \alpha_{\min}$. Combining \eqref{weiwu1} and \eqref{weiwu2}, we obtain that
\begin{equation}\label{weiwu}
    \lim_{m \to+\infty}\gamma^{-1}(\alpha_m)=+\infty 
\end{equation}
This, \eqref{blefd} and \eqref{thebig2} imply that
\begin{equation}\label{weiwueta}
    \begin{aligned}
        &\lim_{m \to+\infty} \eta^{-1}(\alpha_m)=+\infty, \quad \lim_{m \to+\infty} -\eta^{-2}(\alpha_m)R^{\prime}(\eta^{-1}(\alpha_m))=w\\
        &\lim_{m \to+\infty} -\eta^{-1}(\alpha_m)R(\eta^{-1}(\alpha_m))=w
    \end{aligned}
\end{equation}

We first show that the second summand of $v(\alpha_m)$ vanishes as $\alpha_m \rightarrow \alpha_{\min}$. Using \Cref{lem:cauchy}, (d) and $\lim_{m\to +\infty}\gamma^{-1}(\alpha_m)\to +\infty$ and $\lim_{m\to +\infty} \alpha_m \to \alpha_{\min}<+\infty$, we have that
\begin{equation}\label{ppinfd1}
\lim _{m \rightarrow+\infty} \frac{\sigma^2 \alpha_m^2}{\gamma^{-1}\left(\alpha_m\right)}\left[1+\frac{\eta^{-1}\left(\alpha_m\right) R^{\prime}\left(\eta^{-1}\left(\alpha_m\right)\right)}{R\left(\eta^{-1}\left(\alpha_m\right)\right)}\right]=0
\end{equation}
as required. 

We now proceed to show that the first summand of $v(\alpha_m)$ converges to a constant less than 1 as $\alpha_m \rightarrow 0$. We note that the first summand of $v(\alpha_m)$ can be rewritten as follows
\begin{equation*}
\begin{aligned}
& \alpha_m^2 R^{\prime}\left(\eta^{-1}\left(\alpha_m\right)\right) \mathbb{E}\left(\operatorname{ST}_{\gamma^{-1}\left(\alpha_m\right) \lambda_1}\left(\Xstar+\frac{\gamma^{-1}\left(\alpha_m\right)}{\alpha_m} \Zs\right)-\Xstar\right)^2 \\
& =\eta^{-2}\left(\alpha_m\right) R^{\prime}\left(\eta^{-1}\left(\alpha_m\right)\right) \mathbb{E}\Bigg(\operatorname{sgn}\left(\frac{\alpha_m}{\eta^{-1}\left(\alpha_m\right)} \Xstar +\frac{1}{-\eta^{-1}\left(\alpha_m\right) R\left(\eta^{-1}\left(\alpha_m\right)\right)} \Zs\right) \\
& \qquad \times \left(\left|\frac{\alpha_m}{\eta^{-1}\left(\alpha_m\right)} \Xstar+\frac{1}{-\eta^{-1}\left(\alpha_m\right) R\left(\eta^{-1}\left(\alpha_m\right)\right)} \Zs\right|\right. \\
& \qquad \qquad \qquad \left.\left.-\frac{\alpha_m}{-\eta^{-1}\left(\alpha_m\right) R\left(\eta^{-1}\left(\alpha_m\right)\right)} \lambda_1\right)_{+}-\frac{\alpha_m}{\eta^{-1}\left(\alpha_m\right)} \Xstar\right)^2 \Bigg).
\end{aligned}
\end{equation*}
Using \eqref{weiwueta}, we have that
\begin{equation}\label{ppinfd2}
    \begin{aligned}
        \lim_{m\to +\infty}\alpha_m^2 R^{\prime}\left(\eta^{-1}\left(\alpha_m\right)\right) \mathbb{E}\left(\operatorname{ST}_{\gamma^{-1}\left(\alpha_m\right) \lambda_1}\left(\Xstar+\frac{\gamma^{-1}\left(\alpha_m\right)}{\alpha_m} \Zs\right)-\Xstar\right)^2\\
        =w \mathbb{E}\left(\operatorname{sgn}(\Zs)\left(\left|\frac{1}{w} \Zs \right|-\frac{\alpha_{\min }}{w} \lambda_1\right)_{+}\right)^2=\E \qty[\qty(\abs{\Zs}-\alpha_{\min}\lambda_1)^2\mid \abs{\Zs}>\alpha_{\min}\lambda_1]<1
    \end{aligned}
\end{equation}
as required. The last inequality follows from \Cref{tobound} along with the fact that $\alpha_{\min}\lambda_1>0$. Combining \eqref{ppinfd1} and {ppinfd2}, we have that
\begin{equation}\label{supvlimLasso}
    \limsup_{\alpha \to \alpha_{\min}} v(\alpha) <1.
\end{equation}

Combine \eqref{liminffixeLasso} and \eqref{supvlimLasso}. By continuity of $\alpha \mapsto v(\alpha)$ on $(0,+\infty)$, we know that there exists a solution $\alpha_* \in(\alpha_{\min},+\infty)$ to the equation $v\left(\alpha_*\right)=1$. Therefore, a solution of \eqref{fpstren} is $(\gamma^{-1}, \eta^{-1}, \alpha)=\left(\gamma^{-1}\left(\alpha_*\right), \eta^{-1}\left(\alpha_*\right), \alpha_*\right)$ by construction. This concludes the proof.
\end{proof}

As a Corollary of the proof of \Cref{fixexistLasso}, we have the following. 

\begin{Corollary}\label{jiros}
Assume that both \Cref{AssumpD} and \Cref{AssumpDLasso} are satisfied and   $h(x)=\lambda_1|x|,\forall \lambda_1>0.$ Furthermore, assume that $\E \Xstar, \V(\D^2)$ are both finite. 
Let \(\gamma_*=\gamma_*(\lambda_1)\) and \(\eta_*=\eta_*(\lambda_1)\) be any fixed points defined in \(\eqref{fp}\) (whose existence is guaranteed by \Cref{fixexistLasso}). Then,  
\begin{equation}\label{lkf1}
    \lim_{\lambda_1\to+\infty}\;\;\frac{\gamma_*(\lambda_1)}{\eta_*(\lambda_1)}=0.
\end{equation}
Meanwhile, we have that 
\begin{equation}\label{lkf2}
    \frac{\gamma_*(\lambda_1)}{\eta_*(\lambda_1)}<w,\qquad \forall \gamma_1>0
\end{equation}
when $w<1$.
\end{Corollary}

\begin{proof}[Proof of \Cref{jiros}]
    We continue from the proof of \Cref{fixexistLasso}. Recall that we have shown that for each fixed $\alpha\in (\alpha_{\min},+\infty)$ (case (ii) when $\gamma_+^{-1}$) or $\alpha\in (0,+\infty)$ (case (i) when $\gamma_+^{-1}<+\infty$), 
    the following equation 
    \begin{equation*}
\gamma R^{-1}(-\gamma)=\mathbb{P}\left(\left|\frac{1}{\gamma^{-1}} \Xstar+\frac{1}{\alpha} \Zs\right|>\lambda_1\right)
\end{equation*}
admits a unique solution $\gamma^{-1}\left(\alpha, \lambda_1\right)$ on $\left[\frac{1}{\mathbb{E} \D^2}, \gamma_{+}^{-1}\right)$. Since RHS is a non-increasing function in $\gamma^{-1}$, we have that
$$\gamma\left(\alpha, \lambda_1\right) R^{-1}\left(-\gamma\left(\alpha, \lambda_1\right)\right) \leq \mathbb{P}\left(\left|\Xstar\cdot\left(\mathbb{E} \D^2\right)+\frac{1}{\alpha} \Zs\right|>\lambda_1\right).$$
Note that that $\gamma^{-1} \mapsto \gamma R^{-1}(-\gamma)$ is strictly increasing on $\left[\frac{1}{\E \D^2}, \gamma_{+}^{-1}\right)$, taking value 0 as $\gamma^{-1}=$ $\frac{1}{\E \D^2}$ and that $\lim _{\lambda_1 \rightarrow \infty} \mathbb{P}\left(\left|\Xstar\cdot\left(\E \D^2\right)+\frac{1}{\alpha} \Zs\right|>\lambda_1\right)=0$ for each fixed $\alpha$. We must have that for each fixed $\alpha$
\begin{equation}\label{shid2}
\lim _{\lambda_1 \rightarrow \infty} \gamma^{-1}\left(\alpha, \lambda_1\right)=\frac{1}{\E \D^2}
\end{equation}
which implies that $\eta^{-1}(\alpha,\lambda_1):=\gamma(\alpha,\lambda_1)R^{-1}(\gamma(\alpha,\lambda_1))$ satisfies that
\begin{equation}\label{shid}
\lim _{\lambda_1 \rightarrow \infty} \eta^{-1}\left(\alpha, \lambda_1\right)=0
\end{equation}
Recall that we showed that there exists a solution $\alpha_*=\alpha_*(\lambda_1)$ for the equation
\begin{equation*}
v\left(\alpha_*\left(\lambda_1\right), \lambda_1\right)=1
\end{equation*}
where
$$
\begin{gathered}
v(\alpha,\lambda_1)=\alpha^2 R^{\prime}\left(\gamma^{-1}(\alpha,\lambda_1)\right)\left[\mathbb{E}\left(\operatorname{ST}_{\gamma^{-1}(\alpha,\lambda_1) \lambda_1}\left(\Xstar+\frac{\gamma^{-1}(\alpha,\lambda_1)}{\alpha} \mathsf{Z}\right)-\Xstar\right)^2\right] \\ +\sigma^2 \alpha^2 \frac{1}{\gamma^{-1}(\alpha,\lambda_1)}\left[1+\frac{\gamma^{-1}(\alpha,\lambda_1) R^{\prime}\left(\gamma^{-1}(\alpha,\lambda_1)\right)}{R\left(\gamma^{-1}(\alpha,\lambda_1)\right)}\right],
\end{gathered}
$$
By \eqref{shid2}, \eqref{shid} and \Cref{lem:cauchy}, (f),  we obtain that for each fixed $\alpha$,
\begin{equation*}
\lim _{\lambda_1 \rightarrow \infty} v\left(\alpha, \lambda_1\right)=\alpha^2 \mathbb{V}\left(\D^2\right) \mathbb{E}\left(\Xstar\right)+\sigma^2 \alpha^2 \mathbb{E} \D^2.
\end{equation*}
Thus,
\begin{equation*}
\lim _{\lambda_1 \rightarrow \infty} \alpha_*\left(\lambda_1\right)=\frac{1}{ \mathbb{V}\left(\D^2\right) \mathbb{E}\left(\Xstar\right)+\sigma^2 \mathbb{E} \D^2}.
\end{equation*}
It follows that $\gamma_*\left(\lambda_1\right):=\gamma\left(\alpha_*\left(\lambda_1\right), \lambda_1\right), \eta_*\left(\lambda_1\right):=\eta\left(\alpha_*\left(\lambda_1\right), \lambda_1\right)$ satisfies that
\begin{equation*}
\lim _{\lambda_1 \rightarrow \infty} \frac{\gamma_*\left(\lambda_1\right)}{\eta_*\left(\lambda_1\right)} \leq \lim _{\lambda_1 \rightarrow \infty} \mathbb{P}\left(\left|\Xstar \cdot\left(\mathbb{E} \D^2\right)+\frac{1}{\alpha_*\left(\lambda_1\right)} \Zs\right|>\lambda_1\right)=0.
\end{equation*}
The proof of \eqref{lkf1} is complete. 

From \eqref{fp} (c), we have that
\begin{equation*}
\begin{aligned}
& 1=\mathbb{E} \frac{\eta_*}{\D^2+\eta_*-\gamma_*}=(1-w) \cdot \frac{1}{1-\frac{\gamma_*}{\eta_*}}+w \mathbb{E} \frac{\eta_*}{\D_{+}^2+\eta_*-\gamma_*} \\
& \Leftrightarrow \frac{w-\frac{\gamma_*}{\eta_*}}{1-\frac{\gamma_*}{\eta_*}}=w \mathbb{E} \frac{1}{\D_{+}^2 \eta_*^{-1}+1-\frac{\gamma_*}{\eta_*}}
\end{aligned}
\end{equation*}
which implies that
\begin{equation*}
\frac{w-\frac{\gamma_*}{\eta_*}}{1-\frac{\gamma_*}{\eta_*}}>0 \Leftrightarrow \frac{\gamma_*}{\eta_*}<w
\end{equation*}
The proof of \eqref{lkf2} is complete. 
\end{proof}

\subsection{Convergence of VAMP to Lasso solution}\label{secmfmdfma}
In this section, we establish \Cref{prop:sds} specifically for the Lasso penalty under additional conditions. We state our main result in \Cref{prop:sdsLasso} below where $\xonet,\mathbf{r}_{j t}$ are oracle VAMP iterates defined in \eqref{empvamp}. 

\begin{Proposition}\label{prop:sdsLasso}
Suppose that Assumptions \ref{AssumpD}, \ref{AssumpPrior}, \ref{AssumpLassoo} and \ref{AssumpLassoo2} hold and $h(x)=\lambda_1 |x|$ for some $\lambda_1>0$. Then  for $j=1,2$,
$$
\lim _{t \rightarrow \infty} \lim _{p \rightarrow \infty} \frac{1}{p}\left\|\hatbt-\hat{\mathbf{x}}_{j t}\right\|_2^2=\lim _{t \rightarrow \infty} \lim _{p \rightarrow \infty} \frac{1}{p}\left\|\mathbf{r}_{j t}-\mathbf{r}_{j *}\right\|_2^2=0.
$$
where the inner limits exist almost surely for each fixed $t$. 
\end{Proposition}

We first restate Lemma 3.1 from \cite{bayati2011lasso}.

\begin{Lemma}[\cite{bayati2011lasso}, Lemma 3.1]\label{B5317}
Let us recall from \eqref{deflasso} that
$
\mathcal{L}(\mathbf{x})=\frac{1}{2} \norm{\y-\X \mathbf{x}}^2+\norm{\mathbf{x}}_1.
$
There exists a function $\xi\left(\vartheta, c_1, \ldots, c_5\right)$ such that the following happens.
If $\mathbf{x}, \mathbf{r} \in \mathbb{R}^p$ satisfy the following conditions
\begin{enumerate}
    \item $\|\mathbf{r}\|_2 \leq c_1 \sqrt{p}$;
    \item $\mathcal{L}(\mathbf{x}+\mathbf{r}) \leq \mathcal{L}(\mathbf{x})$;
    \item There exists some subgradient of $\mathcal{L}$ evaluated at $\mathbf{x}$, i.e. $\operatorname{sg}(\mathcal{L}, \mathbf{x}) \in \partial \mathcal{L}(\mathbf{x})$ such hat $\|\operatorname{sg}(\mathcal{L}, \mathbf{x})\|_2 \leq \sqrt{p} \vartheta$;
    \item Let $\mathbf{v} \equiv(1 / \lambda_1)\left[\X^\top (\y-\X \mathbf{x})+\operatorname{sg}(\mathcal{L}, \mathbf{x})\right] \in \partial\|\mathbf{x}\|_1$, and $S\left(c_2\right) \equiv\left\{i \in[p]:\left|v_i\right| \geq 1-c_2\right\}$. Then, for any $S^{\prime} \subseteq[p],\left|S^{\prime}\right| \leq c_3 p$, the minimum singular value of submatrix of $\X$ consisting of columns indexed by $S(c_1)\cup S^{\prime}$ is bounded away from zero, i.e.
    $s_{\min }\left(\X_{S\left(c_2\right) \cup S^{\prime}}\right) \geq c_4$;
    \item The maximum singular value of $\X$ is bounded: $s_{\max}(\X) \leq c_5$.
\end{enumerate}
Then $\|\mathbf{r}\|_2 \leq \sqrt{p} \xi\left(\vartheta, c_1, \ldots, c_5\right)$. Further for any $c_1, \ldots, c_5>0, \xi\left(\vartheta, c_1, \ldots, c_5\right) \rightarrow 0$ as $\vartheta \rightarrow 0$. Further, if $\operatorname{ker}(\X)=\{0\}$, the same conclusion holds under conditions 1, 2, 3, 5 above.
\end{Lemma}

\begin{proof}[Proof of \Cref{prop:sdsLasso}]

We apply \Cref{B5317} to $\mathbf{x}=\xonet$, the VAMP estimate iterate in \Cref{subsectionVAMP} and $\mathbf{r}=\hatbt-\xonet$ the distance from the LASSO optimum $\hatbt$. The thesis follows by checking conditions $1-5$. Namely we need to show that there exists constants $c_1, \ldots, c_5>0$ and, for each $\vartheta>0$ some $t=t(\vartheta)$ exists such that condition $1-5$ hold almost surely as $p \rightarrow \infty$.

We first show Condition 1 holds. First note that  
$$\norm{\mathbf{r}}_2\le \norm{\xonet}_2+\norm{\hatbt}_2.$$
So it suffices to show that there exists some constant $C>0$ such that almost surely
\begin{equation}
    \lim_{t\to \infty } \lim_{p \to \infty} \norm{\xonet}^2_2<C, \qquad \lim_{p \to \infty} \norm{\hatbt}^2_2<C.
\end{equation}
The first statement follows from \Cref{SEovamp}. The second statement can be proved in the same way as in \cite{bayati2011lasso}, under \Cref{AssumpD} and \ref{AssumpLassoo}. 

Condition 2 holds because $\mathbf{x}+\mathbf{r}=\hatbt$ minimizes $\mathcal{L}(\cdot)$. 

Condition 3 follows from \Cref{subgConv} with $\vartheta$ arbitrarily small for $t$ large enough. Here, we have chosen the subgradient to be
$$
\operatorname{sg}(\mathcal{L}, \mathbf{x})\equiv\mathcal{L}^{\prime}\left(\xonet\right)=\X^{\top}\left(\X \xonet-\y\right)+\gamma_*\left(\ronetm-\xonet\right)
$$
as in \Cref{subgConv}. 

We now proceed to consider Condition 4. Note that it is not needed for the case where $d_->0$, since in this case, kernel space of $\X$ is $\{0\}$ for all sufficiently large $p$. So we prove that it holds for our choices of $\operatorname{sg}(\mathcal{L}, \mathbf{x})$
$$\mathbf{v} \equiv \mathbf{v}_t=\frac{\gamma_*}{\lambda_1}\left(\ronetm-\xonet\right)$$
for any $t\ge 1$, when \eqref{shichi} holds. We have
$$S(\psi)\equiv S_t(\psi):=\{i\in [p]: \left|v_{t,i}\right| \geq 1-\psi\}$$
for $\psi\in (0,1)$. From \Cref{SEovamp}, we have that
almost surely
\begin{equation*}
\begin{aligned}
\lim _{p \rightarrow+\infty} \frac{\left|S_t(\psi)\right|}{p} & =\lim _{p \rightarrow+\infty} \frac{1}{p} \sum_{i=1}^p \mathbb{I}\left(\left|\frac{\gamma_*}{\lambda_1}\left(r_{1, t-1, i}-\hat{x}_{1 t, i}\right)\right| \geq 1-\psi\right) \\
& =\mathbb{P}\left(\frac{\gamma_*}{\lambda_1}\left|\sqrt{\tau_*} \Zs+\Xstar-\mathrm{ST}_{\gamma_*^{-1} \lambda_1}\left(\sqrt{\tau_*} \Zs+\Xstar\right)\right| \geq 1-\psi\right).
\end{aligned}
\end{equation*}
Note that
\begin{equation*}
\left|\sqrt{\tau_*} \Zs+\Xstar-\mathrm{ST}_{\gamma_*^{-1} \lambda_1}\left(\sqrt{\tau_*} \Zs+\Xstar\right)\right|= \begin{cases}\gamma_*^{-1} \lambda_1 & \text { when }\left|\sqrt{\tau_*} \Zs+\Xstar\right| \geq \gamma_*^{-1} \lambda_1 \\ \left|\sqrt{\tau_*} \Zs+\Xstar\right| & \text { otherwise }\end{cases}.
\end{equation*}
Therefore, from the law of total probability,
\begin{equation*}
\begin{aligned}
& \mathbb{P}\left(\frac{\gamma_*}{\lambda_1}\left|\sqrt{\tau_*} \Zs+\Xstar-\mathrm{ST}_{\gamma_*^{-1} \lambda_1}\left(\sqrt{\tau_*} \Zs+\Xstar\right)\right| \geq 1-\psi\right) \\
& =\mathbb{P}\left(\left|\sqrt{\tau_*} \Zs+\Xstar\right| \geq \gamma_*^{-1} \lambda_1\right)+\mathbb{P}\left(\left|\sqrt{\tau_*} \Zs+\Xstar\right|<\gamma_*^{-1} \lambda_1\right) \\
& \qquad \qquad \times \mathbb{P}\bigg(\frac{\gamma_*}{\lambda_1}\left|\sqrt{\tau_*} \Zs+\Xstar\right| \geq 1-\psi \bigg| \abs{\sqrt{\tau_*} \Zs+\Xstar} <\gamma_*^{-1} \lambda_1\bigg) \\
& =\mathbb{P}\left(\left|\sqrt{\tau_*} \Zs+\Xstar\right| \geq \gamma_*^{-1} \lambda_1\right)+\mathbb{P}\left(1-\psi \leq \frac{1}{\gamma_*^{-1} \lambda_1}\left|\sqrt{\tau_*} \Zs+\Xstar\right| \leq 1\right).
\end{aligned}
\end{equation*}
Note that the second term goes to $0$ as $\psi\to 0$. So we have that almost surely
\begin{equation*}
\lim _{\psi \rightarrow 0} \lim _{p \rightarrow+\infty} \frac{\left|S_t(\psi)\right|}{p}=\mathbb{P}\left(\left|\sqrt{\tau_*} \Zs+\Xstar\right| \geq \gamma_*^{-1} \lambda_1\right)=\frac{\gamma_*}{\eta_*}<C_{\mathrm{abs}}\left(\frac{\mathbb{E}\D^2}{d_{+}}\right)^3
\end{equation*}
The last inequality follows from \Cref{AssumpLassoo}. Condition 4 follows from this and \Cref{34ba}. We state and prove \Cref{34ba} separately in the next section. 

Condition 5 follows from \Cref{AssumpD}.
    
\end{proof}

\subsection{Smallest singular value of design submatrix}\label{ssvods}
Recall that our proof of Condition 4 of \Cref{B5317} requires \Cref{34ba}, which controls the smallest singular value of certain design submatrix. 

For the following, we define the sigma-field generated by outputs of the VAMP algorithm $\mathcal{G}_t, t\ge 1$  (in the probability space of $\mathbf{O}, \st$ and $\epbm$) as
    $$\mathcal{G}_t:= \mathcal{G}(\mathbf{H},\X_{t},\mathbf{S}_t,\mathbf{Y}_t)$$
    where $\mathbf{H}$ is defined in \Cref{prop:AMPparamconverge} and $\X_{t},\mathbf{S}_t,\mathbf{Y}_t$ are stacked VAMP iterates defined in \Cref{thm:ampSE}.
The following the matrix quantity 
    \begin{equation*}
\X|\mathcal{G}_t=\mathbf{Q}^{\top} \mathbf{D U}\left(\mathbf{V}^{\top} \mathbf{V}\right)^{-1} \mathbf{V}^{\top}+\mathbf{Q}^{\top} \mathbf{D} \boldsymbol{\Pi}_{\mathbf{U}^{\perp}} \widetilde{\mathbf{O}} \boldsymbol{\Pi}_{\mathbf{V}^{\perp}}^{\top}
\end{equation*}
is the design matrix $\X$ conditioned on $\mathcal{G}_t$ (cf. proof of \Cref{thm:ampSE}). Here, $\mathbf{U}, \mathbf{V}, \boldsymbol{\Pi}_{\mathbf{U}^{\perp}},$ and $ \boldsymbol{\Pi}_{\mathbf{V}^{\perp}}$ are measurable to $\mathcal{G}_t$ with
$$\mathbf{U}=\left(\mathbf{e}_{b}, \mathbf{S}_{t},
\bm{\Lambda} \mathbf{S}_{t}\right), \quad \mathbf{V}=\left(\mathbf{e},\X_{t}, \mathbf{Y}_{t}\right)$$
for $\mathbf{e}_{b}, \mathbf{e}$ defined in \eqref{Quant}, $\tilde{\Obm} \sim \Haar(\mathbb{O}(p-(2 t+1)))$ an independent copy of Haar matrix and $\bm{\Pi}_{\mathbf{U}^{\perp}},
\bm{\Pi}_{\mathbf{V}^{\perp}} \in \mathbb{R}^{p \times(p-(2 t+1))}$. For convenience, we further introduce notations for the projections 
\begin{equation*}
\mathbf{P}_{\mathbf{V}}=\mathbf{V}\left(\mathbf{V}^{\top} \mathbf{V}\right)^{-1} \mathbf{V}^{\top}, \quad \mathbf{P}_{\mathbf{U}}=\mathbf{U}\left(\mathbf{U}^{\top} \mathbf{U}\right)^{-1} \mathbf{U}^{\top}, \quad \mathbf{P}_{\mathbf{V}}^{\perp}=\boldsymbol{\Pi}_{\mathbf{V}^{\perp}} \boldsymbol{\Pi}_{\mathbf{V}^{\perp}}^{\top}, \quad \mathbf{P}_{\mathbf{U}}^{\perp}=\boldsymbol{\Pi}_{\mathbf{U}^{\perp}} \boldsymbol{\Pi}_{\mathbf{U}^{\perp}}^{\top}.
\end{equation*}

The following is a counterpart of Lemma 5.3 of \cite{bayati2011lasso}. Due to the difficulty of studying a submatrix of right-rotationally invariant design, we resort to a covering argument. This allows us to establish the same result as Lemma 5.3 of \cite{bayati2011lasso} for design submatrices of sufficiently small number of columns.

\begin{Lemma}\label{baorbaoer}
    Fix $S \subset [p]$. There exists absolute constant $C_{\mathrm{abs}}>0$ such that if
    $$\frac{|S|}{p}<C_{\mathrm{abs}}  \qty (\frac{\E \D^2}{d_+})^3$$
    we have for some $\alpha_1>0,\alpha_2>0$ such that for any fixed $t\ge 1$,
    $$
    \begin{aligned}
        & \mathbb{P}\left\{\min _{\|\mathbf{v}\|_2=1, \operatorname{supp} (\mathbf{v}) \subseteq S}\left\|\X \mathbf{v}\right\|_2 \leq \alpha_2 \mid \mathcal{G}_t\right\}\\
        &\equiv \mathbb{P}\left\{\min _{\|\mathbf{v}\|_2=1, \operatorname{supp} (\mathbf{v}) \subseteq S}\left\|\mathbf{Q}^{\top} \mathbf{D} \mathbf{U}\left(\mathbf{V}^{\top} \mathbf{V}\right)^{-1} \mathbf{V}^{\top} \mathbf{v}+\mathbf{Q}^{\top} \mathbf{D} \boldsymbol{\Pi}_{\mathbf{U}^{\perp}} \widetilde{\mathbf{O}} \boldsymbol{\Pi}_{\mathbf{V}^{\perp}}^{\top} \mathbf{v}\right\|_2 \leq \alpha_2 \mid \mathcal{G}_t\right\} \\
    & <\exp(-p \alpha_1).
    \end{aligned}
    $$
    almost surely as $p\to \infty$.    
\end{Lemma}

\begin{Remark}
    Our proof shows that the result holds when \(C_{\mathrm{abs}} \le 0.00148\). The estimate comes from the absolute constants in several concentration inequalities and covering number estimates. Although we do not expect this bound to be optimal, it appears challenging to improve it significantly with the current argument.
\end{Remark}

This lemma immediately implies the following, which is the counterpart of Lemma 3.4, \cite{bayati2011lasso}. 
\begin{Lemma}\label{34ba}
    Let $S \subset [p]$ be measurable on $\mathcal{G}_t$. If
    $$\frac{|S|}{p}<C_{\mathrm{abs}} \qty (\frac{\E \D^2}{d_+})^3$$
    we have for some $a_1>0,a_2>0$ such that for any fixed $t\ge 1$,
    \begin{equation*}
\min _{S^{\prime}}\left\{s_{\min }\left(\X_{S \cup S^{\prime}}\right): \quad S^{\prime} \subseteq[N],\left|S^{\prime}\right| \leq a_1 N\right\} \geq a_2
\end{equation*}
almost surely as $p\to \infty$. 
\end{Lemma}

\begin{proof}[Proof of \Cref{34ba}]
By Borel-Cantelli, it is sufficient to show that, for $S$ measurable on $\mathcal{G}_t$ and $|S|$ satisfying conditions given in the lemma statement, there exist $a_1>0$ and $a_2>0$, such that
\begin{equation*}
\mathbb{P}\left\{\min _{\left|S^{\prime}\right| \leq a_1 p} \min _{\|\mathbf{v}\|=1, \operatorname{supp}(\mathbf{v}) \subseteq S \cup S^{\prime}}\|\X \mathbf{v}\|<a_2\right\} \leq \frac{1}{p^2}.
\end{equation*}
for all $p$ large enough. Note that cardinality of the set of all possible $S^\prime$ that satisfies $|S^\prime|\le a_1 p$ can be estimated by 
$$\sum_{k=1}^{pa_1}{{p}\choose{k}}\le \exp(p)h(a_1)$$
where $h(x)=-x \log x-(1-x) \log (1-x), x\in [0,1]$ is binary entropy function. Then we have from union bound
\begin{equation*}
\begin{aligned}
&\mathbb{P}\left\{\min _{\left|S^{\prime}\right| \leq a_1 p} \min _{\|\mathbf{v}\|=1, \operatorname{supp}(\mathbf{v}) \subseteq S \cup S^{\prime}}\|\X \mathbf{v}\|<a_2\right\} \\
&\qquad \leq e^{p h\left(a_1\right)} \mathbb{E}\left\{\max _{\left|S^{\prime}\right| \leq a_1 p} \mathbb{P}\left\{\min _{\|\mathbf{v}\|=1, \operatorname{supp}(\mathbf{v}) \subseteq S \cup S^{\prime}}\|\X \mathbf{v}\|<a_2 \mid \mathcal{G}_t\right\}\right\}.
\end{aligned}
\end{equation*}
From this and \Cref{baorbaoer}, we see that we can take some small enough $a_1$ such that $|S\cup S^\prime|< C_{\mathrm{abs}}\left(\frac{\mathbb{E }\D^2}{d_{+}}\right)^3$ and $h(a_1)<\alpha_1$ from \Cref{baorbaoer}. 
\end{proof}

\begin{proof}[Proof of \Cref{baorbaoer}]

Note that for any $\mathbf{v}$
\begin{equation*}
\begin{aligned}
&\left\|\mathbf{Q}^{\top} \mathbf{D U}\left(\mathbf{V}^{\top} \mathbf{V}\right)^{-1} \mathbf{V}^{\top} \mathbf{v}+\mathbf{Q}^{\top} \mathbf{D} \boldsymbol{\Pi}_{\mathbf{U}^{ \top}} \widetilde{\mathbf{O}} \boldsymbol{\Pi}_{\mathbf{v}^{\perp}}^{\top} \mathbf{v}\right\|_2^2 =A_2(\mathbf{v})+A_1(\mathbf{v})
\end{aligned}
\end{equation*}
where  
$$A_1(\mathbf{v}):=\mathbf{v}^{\top} \mathbf{V}\left(\mathbf{V}^{\top} \mathbf{V}\right)^{-1} \mathbf{U}^{\top} \mathbf{D}^{\top} \mathbf{D U}\left(\mathbf{V}^{\top} \mathbf{V}\right)^{-1} \mathbf{V}^{\top} \mathbf{v}$$
and
$$A_2(\mathbf{v}):=\mathbf{v}^{\top} \boldsymbol{\Pi}_{\mathbf{v}^{\perp}}^{\top} \widetilde{\mathbf{O}} \Pi_{\mathbf{U}^{\perp}}^{\top} \mathbf{D}^{\top} \mathbf{D} \boldsymbol{\Pi}_{\mathbf{U}^{\perp}} \widetilde{\mathbf{O}} \boldsymbol{\Pi}_{\mathbf{v}^{\perp}}^{\top} \mathbf{v}+2 \mathbf{v}^{\top} \mathbf{V}\left(\mathbf{V}^{\top} \mathbf{V}\right)^{-1} \mathbf{U}^{\top} \mathbf{D}^{\top} \mathbf{D} \boldsymbol{\Pi}_{\mathbf{U}^{\perp}} \widetilde{\mathbf{O}} \boldsymbol{\Pi}_{\mathbf{V}^{\perp}}^{\top} \mathbf{v}.$$
It follows from \Cref{SEovamp} that almost surely as $p\to \infty$
\begin{equation*}
\frac{1}{p} \mathbf{U}^{\top} \mathbf{D}^{\top} \mathbf{D U} \rightarrow \mathbb{E} \D^2 \cdot\left(\begin{array}{ccc}
b_* & 0 & 0 \\
0 & \Delta_t & 0 \\
0 & 0 & \kappa_* \Delta_t
\end{array}\right)
\end{equation*}
where the RHS is positive positive definite matrix defined in \Cref{SEovamp}. Using this and the identity that $\mathbf{V}^\top \mathbf{V}=\mathbf{U}^\top \mathbf{U}$ (cf. proof of \Cref{SEovamp}), it follows that almost surely as $p\to \infty$,
$$A_1(\mathbf{v})\to \E \D^2 \cdot \norm{\mathbf{P}_{\mathbf{V}}\mathbf{v}}_2^2.$$

We now establish the following claim: there exists some constant $c_1,c_2>0$ such that the event
\begin{equation*}
\mathcal{E}:=\left\{A_2(\mathbf{v})<c_1\cdot \norm{\mathbf{P}_{\mathbf{V}^\bot}\mathbf{v}}_2^2, \forall \mathbf{v} \text { s.t. }\|\mathbf{v}\|_2=1, \operatorname{supp}(\mathbf{v}) \subseteq S\right\}
\end{equation*}
satisfies 
$$\mathbb{P}(\mathcal{E}\mid \mathcal{G}_t)\ge 1-\exp(-c_2 p).$$ 
almost surely as $p\to \infty$. Observe that if $\mathcal{E}$ holds, we then have
\begin{equation*}
\begin{aligned}
    \min _{\|\mathbf{v}\|_2=1, \operatorname{supp}(\mathbf{v}) \subseteq S} A_1(\mathbf{v})+A_2(\mathbf{v}) &\geq \min _{\|\mathbf{v}\|_2=1, \operatorname{supp}(\mathbf{v}) \subseteq S} \mathbb{E} \D^2 \cdot\left\|\mathbf{P}_{\mathbf{V}} \mathbf{v}\right\|_2^2+c_1\left\|\mathbf{P}_{\mathbf{V}}^{\perp} \mathbf{v}\right\|_2^2 \\
    &\geq \min \left(\mathbb{E} \D^2, c_1\right)>0
\end{aligned}
\end{equation*}
which conclude the proof. 

We will now prove the claim above. We will apply a covering argument. For any \textit{fixed} $\mathbf{v}$ such that $\norm{\mathbf{v}}_2=1$, we have the following equality in law
\begin{equation*}
\widetilde{\mathbf{O}} \boldsymbol{\Pi}_{\mathbf{v}^{\perp}}^{\top} \mathbf{v} \stackrel{d}{=} \frac{\left\|\mathbf{P}_{\mathbf{V}}^{\perp} \mathbf{v}\right\|_2}{\left\|\mathbf{P}_{\mathbf{U}}^{\perp} \mathbf{z}\right\|_2} \boldsymbol{\Pi}_{\mathbf{U}^{\perp} }^{\top} \mathbf{z}
\end{equation*}
where $\mathbf{z}\sim N(\bm{0}, \mathbf{I}_p)$. From this, we obtain that
\begin{equation*}
{A }_2(\mathbf{v}) \stackrel{d}{=} \frac{\frac{1}{p} \mathbf{z}^{\top} \mathbf{P}_{\mathbf{V}}^{\perp} \mathbf{D}^{\top} \mathbf{D} \mathbf{P}_{\mathbf{V}}^{\perp} \mathbf{z}-\frac{2}{p} \mathbf{a}^{\top} \mathbf{z}\left\|\mathbf{P}_{\mathbf{U}}^{\perp} \mathbf{z}\right\|_2}{\frac{1}{p}\left\|\mathbf{P}_{\mathbf{U}}^{\perp} \mathbf{z}\right\|_2^2}\left\|\mathbf{P}_{\mathbf{V}}^{\perp} \mathbf{v}\right\|_2^2
\end{equation*}
where $\mathbf{a}=\left\|\mathbf{P}_{\mathbf{V}}^{\perp} \mathbf{v}\right\|_2^{-1} \mathbf{P}_{\mathbf{U}}^{\perp} \mathbf{D}^{\top} \mathbf{D} \mathbf{U}\left(\mathbf{V}^{\top} \mathbf{V}\right)^{-1} \mathbf{V}^{\top} \mathbf{v}$ satisfies that
\begin{equation}\label{pgor}
\|\mathbf{a}\|_2 \leq \max _{i \in[p]}\left(d_i^2\right) \cdot \frac{\left\|\mathbf{U}\left(\mathbf{V}^{\top} \mathbf{V}\right)^{-1} \mathbf{V}^{\top} \mathbf{v}\right\|_2}{\left\|\mathbf{P}_{\mathbf{V}}^{\perp} \mathbf{v}\right\|_2}=\max _{i \in[p]}\left(d_i^2\right) \frac{\left\|\mathbf{P}_{\mathbf{V}}^{\perp} \mathbf{v}\right\|_2}{\left\|\mathbf{P}_{\mathbf{V}}^{\perp} \mathbf{v}\right\|_2}=\max _{i \in[p]}\left(d_i^2\right).
\end{equation}
Let us denote
$$R:=\frac{1}{p} \mathbf{z}^{\top}\left(\mathbf{P}_{\mathbf{V}}^{\perp} \mathbf{D}^{\top} \mathbf{D} \mathbf{P}_{\mathbf{V}}^{\perp}\right) \mathbf{z}.$$
Using Hanson-Wright inequality, we have that for any $\epsilon>0$,
\begin{equation*}
\begin{aligned}
& \mathbb{P}\left(\left|R-\E R\right| \geq \epsilon \E R\right) \\
& \quad \leq 2 \exp \left(-p C_{\mathrm{HW}}\left(\frac{\epsilon^2 (\E R)^2}{\max _{i \in[p]}\left(d_i^4\right)} \wedge \frac{\epsilon \E R}{\max _{i \in[p]}\left(d_i^2\right)}\right)\right).
\end{aligned}
\end{equation*}
Using standard Gaussian tail upper bound, we have that
\begin{equation*}
\mathbb{P}\left(\left|\frac{1}{p} \mathbf{z}^{\top} \mathbf{a}\right| \geq \frac{\E R}{\sqrt{p}}\right) \leq 2 \exp \left(-p\frac{(\E R)^2 }{2 \max _{i \in[p]}\left(d_i^4\right)}\right).
\end{equation*}
Using standard concentration inequality for chi-squared distribution, we obtain that for any $\epsilon^\prime>0$,
\begin{equation*}
\mathbb{P}\left(\frac{1}{p}\left\|\mathbf{P}_{\mathbf{U}}^{\perp} \mathbf{z}\right\|_2^2 \geq 1+\frac{2 \sqrt{\epsilon^\prime p}}{p}+2 \epsilon^\prime\right) \leq \exp (-p \epsilon^\prime)
\end{equation*}
where we have used \eqref{pgor}. 

Combining the above concentration inequalities, we obtain the concentration inequality,
\begin{equation}\label{bigpg}
\begin{aligned}
    &\mathbb{P}\qty(\frac{\frac{1}{p} \mathbf{z}^{\top} \mathbf{P}_{\mathbf{U}}^{\perp} \mathbf{D}^{\top} \mathbf{D} \mathbf{P}_{\mathbf{U}}^{\perp} \mathbf{z}-\frac{2}{p} \mathbf{a}^{\top} \mathbf{z}\left\|\mathbf{P}_{\mathbf{U}}^{\perp} \mathbf{z}\right\|_2}{\frac{1}{p}\left\|\mathbf{P}_{\mathbf{U}}^{\perp} \mathbf{z}\right\|_2^2} \geq \mathbb{E} R \frac{1-\epsilon-2\left(1+2 \sqrt{\frac{\epsilon^{\prime}}{p}}++2\epsilon^{\prime}\right)^{\frac{1}{2}} \frac{1}{\sqrt{p}}}{1+2 \sqrt{\frac{\epsilon^{\prime}}{p}}+2\epsilon^{\prime}})\\
    &\qquad \le 5 \exp \left(-\min \left(\epsilon^{\prime}, C_{\mathrm{HW}} \frac{\left(\mathbb{E} R\right)^2 \epsilon^2}{d_{+}^2}\right) p\right).
\end{aligned}
\end{equation}
Let $\mathcal{N}$ by a $r$-net that covers the set $\mathcal{B}:=\{\mathbf{v}\in \R^{p}:\|\mathbf{v}\|_2=1, \operatorname{supp}(\mathbf{v}) \subseteq S\}$ such that for any $\mathbf{v}\in \mathcal{B}$, there exists a point $\mathbf{v}_r\in \mathcal{B}$ such that $\norm{\mathbf{v}-\mathbf{v}_r}_2<r$. It follows from the definition of $A_2$ and the basic inequality $\left|\mathbf{a}^{\top} \mathbf{A} \mathbf{a}-\mathbf{b}^{\top} \mathbf{A b}\right| \leq\|\mathbf{A}\|_{\mathrm{o p}}\left(\|\mathbf{a}\|_2+\|\mathbf{b}\|_2\right)\|\mathbf{a}-\mathbf{b}\|_2$ for any real-valued matrix and vectors $\mathbf{a},\mathbf{b}$ that
$$\abs{A_2(\mathbf{v})-A_2(\mathbf{v}_r)}\le 6 \max_{i\in [p]}(d_i^2)\cdot r.$$ It is well-established that there exists $r$-net $\mathcal{N}$ such that
$$|\mathcal{N}|\le \qty(1+\frac{2}{r})^{|S|}.$$
Combining \eqref{bigpg} with the above, we obtain that for any $r,\epsilon,\epsilon^{\prime}>0$,
\begin{equation}
    \begin{aligned}
 &\mathbb{P}\qty(A_2(\mathbf{v})\le \left(\mathbb{E} R \frac{1-\epsilon-2\left(1+2 \sqrt{\frac{\epsilon^{\prime}}{p}}+2\epsilon^{\prime}\right)^{\frac{1}{2}} \frac{1}{\sqrt{p}}}{1+2 \sqrt{\frac{\epsilon^{\prime}}{p}}+2\epsilon^{\prime}}-6 \max _{i \in[p]} d_i^2 \cdot r\right)\left\|\mathbf{P}_{\mathbf{V}}^{\perp} \mathbf{v}\right\|_2^2)\\
 &\qquad \le 5 \exp \left(-p\left[\min \left(\epsilon^{\prime}, C_{\mathrm{HW}} \frac{(\mathbb{E} R)^2 \epsilon^2}{\max_{i\in[p]}(d_i^4)}\right)-\frac{|S|}{p} \log \left(1+\frac{2}{r}\right)\right]\right).
    \end{aligned}
\end{equation}

We note that as $p\to \infty$
$$\E R= \frac{1}{p} \operatorname{Tr}(\mathbf{P}_{\mathbf{V}}^{\perp} \mathbf{D}^{\top} \mathbf{D} \mathbf{P}_{\mathbf{V}}^{\perp})\to \E \D^2$$
which follows from \Cref{AssumpD} and trace inequalities. Meanwhile we also have from \Cref{AssumpD} that as $p\to \infty$
$$d_{+}:=\limsup _{p \rightarrow \infty} \max _{i \in[p]} d_i^2<+\infty.$$
Set $\epsilon^{\prime}=C_{\mathrm{HW}}$. For any constant $c\in (0,1)$ and any $\epsilon\in (0,1-c)$, for all $p$ sufficiently large we have that
\begin{equation*}
\begin{aligned}
& \mathbb{E} R \frac{1-\epsilon-2\left(1+2 \sqrt{\frac{\epsilon^{\prime}}{p}}+2 \epsilon^{\prime}\right)^{\frac{1}{2}} \frac{1}{\sqrt{p}}}{1+2 \sqrt{\frac{\epsilon^{\prime}}{p}}+2 \epsilon^{\prime}}-6 \max _{i \in[p]} d_i^2 \cdot r \\
& \qquad \geq(1-c)\left(\mathbb{E}\D^2\right) \frac{1-c-\epsilon}{1+3 C_{\mathrm{HW}}}-6(1+c) d_{+} \cdot r \\
& \min \left(\epsilon^{\prime}, C_{\mathrm{HW}} \frac{(\mathbb{E} R)^2 \epsilon^2}{d_{+}^2}\right)-\frac{|S|}{p} \log \left(1+\frac{2}{r}\right) \geq C_{\mathrm{HW}} \frac{(1-c)\left(\mathbb{E }\D^2\right)^2 \epsilon^2}{d_{+}^2}-\frac{|S|}{p} \frac{2}{r}.
\end{aligned}
\end{equation*}
Thus, the claim is proved if we can find some $r>0, c\in (0,1), \epsilon \in (0,1-c)$ such that
\begin{equation*}
(1-c)\left(\mathbb{E}\D^2\right) \frac{1-c-\epsilon}{1+3 C_{\mathrm{HW}}}-6(1+c) d_{+} \cdot r>0, \quad C_{\mathrm{HW}} \frac{(1-c)\left(\mathbb{E }\D^2\right)^2 \epsilon^2}{d_{+}^2}-\frac{|S|}{p} \frac{2}{r}>0.
\end{equation*}
Rearranging the first term gives an upper bound on $r$ and the second term a lower bound on $r$. Thus, it suffices for some $c\in (0,1), \epsilon \in (0,1-c)$, the lower bound is smaller than the upper bound
\begin{equation*}
\begin{aligned}
    \frac{2|S|}{p C_{\mathrm{HW}}} \frac{d_{+}^2}{(1-c)\left(\mathbb{E}\D^2\right)^2 \epsilon^2}<\frac{(1-c)\left(\mathbb{E}\D^2\right)}{6(1+c) d_{+}} \frac{1-c-\epsilon}{1+3 C_{\mathrm{HW}}} \\
    \iff \frac{|S|}{p}<\frac{C_{\mathrm{HW}}}{1+3 C_{\mathrm{HW}}} \frac{(1-c)^2}{12(1+c)}(1-c-\epsilon) \epsilon^2 \frac{\left(\mathbb{E }\D^2\right)^3}{d_{+}^3}.
\end{aligned}
\end{equation*}
Select $c=0.0001, \epsilon=3/4$ and note that $C_{\mathrm{HW}}=0.145$ using estimates from \cite{moshksar2021absolute}. The second line above is satisfied if
$$\frac{|S|}{p}<C_{\mathrm{abs}} \frac{\left(\mathbb{E }\D^2\right)^3}{d_{+}^3}, \qquad C_{\mathrm{abs}} =0.00148$$
which is guaranteed by the assumption. 
\end{proof}

\subsection{Proof of main results}\label{wakakafff}
In this section, we prove \Cref{Smdfkf} using the main result, \Cref{prop:sdsLasso}, from \Cref{secmfmdfma}. 

% \begin{Proposition}\label{chong}
% Suppose that Assumptions \ref{AssumpD}, \ref{AssumpPrior} and \ref{AssumpLassoo} hold and $h(x)=\lambda_1 |x|$ for some $\lambda_1>0$. Then, there exists some positive constant $\lambda_{\min}=\lambda_{\min}(\D^2)$ depending only on $\D^2$, such that for all $\lambda_1>\lambda_{\min}$, we have that for $j=1,2$,
% \begin{equation}\label{dangwang}
%     \lim _{t \rightarrow \infty} \lim _{p \rightarrow \infty} \frac{1}{p}\left\|\hatbt-\hat{\mathbf{x}}_{j t}\right\|_2^2=\lim _{t \rightarrow \infty} \lim _{p \rightarrow \infty} \frac{1}{p}\left\|\mathbf{r}_{j t}-\mathbf{r}_{j *}\right\|_2^2=0.
% \end{equation}
% where the inner limits exist almost surely for each fixed $t$. 
% \end{Proposition}

\begin{proof}[Proof of \Cref{Smdfkf}]
We first note that \eqref{hahabao} in \Cref{AssumpLassoo} is a stronger condition than \Cref{AssumpDLasso}. Thus, by \Cref{fixexistLasso},  we know that under \eqref{hahabao} and \Cref{AssumpD}, there always exists a solution $\gamma_*,\eta_*, \tau_*, \tau_{**} \in (0,+\infty)$ with $\eta_*>\gamma_*$ to the fixed point equation \eqref{fp}. Similarly to the proof under \Cref{Assumpgp}, we do not require fixed points to be unique. We consider the oracle VAMP with respect to one set of such fixed points and show that it eventually tracks the Lasso solution. 

Inspecting the rest of the proof, \Cref{prop:sds} and \Cref{ptconv} are the only remaining instances where we require  \Cref{Assumpgp}. In particular, the distribution characterization \Cref{thm:empmain} only requires us to show \Cref{prop:sds} holds in the Lasso case under Assumptions \ref{AssumpD}, \ref{AssumpPrior}, \ref{AssumpLassoo} and \ref{AssumpLassoo2}. 

For \Cref{prop:sds}, we have from \Cref{prop:sdsLasso} that the claims \eqref{tianshdf} still hold under Assumptions \ref{AssumpD}, \ref{AssumpPrior}, \ref{AssumpLassoo} and \ref{AssumpLassoo2}. 

For  \Cref{ptconv}, \Cref{Assumpgp} is required so that \Cref{thm:empmain} and \Cref{Assumpgpweak} holds. Note that we just showed that \Cref{thm:empmain} holds for Lasso under Assumptions \ref{AssumpD}, \ref{AssumpPrior}, \ref{AssumpLassoo} and \ref{AssumpLassoo2}. For \Cref{Assumpgpweak}, we showed in the proof of \Cref{ptconv} that if \Cref{thm:empmain} holds, then almost surely for all sufficiently large $p$, there must exist some $i\in [p]$ such that $h^{\prime \prime} (\hat{\beta}_i)\neq +\infty$. When $d_->0$, we have that $\X^\top \X$ is non-singular for all sufficiently large $p$. Otherwise, we must have $w<1$. Then, we have that almost surely as $p\to \infty$
    $$\frac{1}{p} \norm{d}_0+\frac{1}{p}\norm{h^{\prime \prime}(\hatbt)}\to w+1-\frac{\gamma_*}{\eta_*}>1.$$
Here, the convergence follows from  \Cref{thm:empmain}, \eqref{eq:Jacprox} and \eqref{fp} (a) and \eqref{XL11}; the inequality follows from \Cref{jiros}. Therefore, we showed that \Cref{Assumpgpweak} holds almost surely for all sufficiently large $p$. 
\end{proof}
}

\section{Conjectures for Ellipsoidal Models}\label{section:CONJE}
We conjecture that debiasing is possible in a more general settings than considered in this paper. Namely, one would like to consider the design matrix $\X=\Qbm^\top \Dbm \Obm \bm{\Sigma}^{1/2}$ where $\bm{\Sigma}\in \R^{p\times p}$ is non-singular, $\Qbm \in \R^{n\times n}, \Obm \in \R^{p\times p}$ are orthogonal matrices and $\Dbm \in \R^{n\times p}$ is diagonal matrix. We assume that $\bm{\Sigma}\in \R^{p\times p}$ is observed and $\Obm$ is drawn uniformly from the orthogonal group $\O(p)$ independent of $\epbm,\Dbm,\Qbm$. We refer to this class of random design matrices as ellipsoidal invariant designs. The special case where $\Qbm^\top \Dbm \Obm$ is an isotropic Gaussian matrix is studied extensively in prior literature \cite{bellec2019biasing,celentano2020lasso,javanmard2014confidence,javanmard2014hypothesis,bellec2022biasing}. Similarly to the anisotropic Gaussian design case, the challenge in applying such a model arises when $\bm{\Sigma}$ is unknown. We discuss this in \Cref{unknownSigma} at the end. 

Furthermore, one would like to consider the case where the convex penalty function $\hv:\R^p\mapsto \R$ is non-separable (e.g. SLOPE, group-Lasso) and  
$
\hatbt \in \underset{\mathbf{b} \in \mathbb{R}^p}{\arg \min } \frac{1}{2}\|\y-\X \mathbf{b}\|^2+\hv\left(\mathbf{b}\right).
$
where $\hv$ is assumed to be proper and closed. To illustrate, we give debiasing formulas under the case $\sigma^2=1$: 
\begin{equation}\label{dbbdbnon}
    \bhetah=\hatbt+\frac{1}{\adj} \bm{\Sigma}^{-1} \X^{\top}(\y-\X \hatbt)
\end{equation}
where $\adj$ is solution of the following equation
\begin{equation}\label{solvegammanon}
\frac{1}{p} \bigsum_{i=1}^p \frac{1}{\frac{d_i^2-\adj}{p} \Tr\left(\left(\adj\cdot \mathbf{I}_p+\bm{\Sigma}^{-1}\left(\nabla^2\hv(\hatbt)\right)\right)^{-1}\right)+1}=1.
\end{equation}
Here, we assumed that $\hv$ is twice-differentiable or that it admits a twice-differentiable extension as in \Cref{Extend}. Notice that the equation \eqref{solvegammanon} becomes \eqref{gammasolvea} if one let $\bm{\Sigma}=\mathbf{I}_p$ and $(\hv(x))_i=h(x_i),\forall i \in p$ for some $h:\R\mapsto \R$. Analogous to \eqref{DEFEFD}, we define
\begin{equation}\label{DEFEFDnon}
    \begin{aligned}
&\hat{\eta}_* (p) :=\left(  \frac{1}{p} \Tr \qty(\adj\cdot \mathbf{I}_p+\bm{\Sigma}^{-1} \nabla^2 \hv (\hatbt))  \right)^{-1}\\
& \hatrstst(p):=\hatbt+\frac{1}{\hat{\eta}_*-\adj} \bm{\Sigma}^{-1} \X^{\top}(\X \hatbt-\y),\quad \hat{\tau}_{**}(p) :=\frac{\frac{1}{p}\left\|\X \hatrstst-\y\right\|^2-\frac{n}{p}}{\frac{1}{p} \sum_{i=1}^p d_i^2}\\
& \tauh(p):=\left(\frac{\hat{\eta}_*}{\adj}\right)^2 \frac{1}{p} \bigsum_{i=1}^p \frac{d_i^2}{\left(d_i^2+\hat{\eta}_*-\adj \right)^2}\\
& \qquad \qquad +\left(\frac{\hat{\eta}_*-\adj}{\adj}\right)^2\left(\frac{1}{p} \bigsum_{i=1}^p\left(\frac{\hat{\eta}_*}{d_i^2+\hat{\eta}_*-\adj}\right)^2-1\right) \hat{\tau}_{**}
\end{aligned}
\end{equation}
One can then make the following conjecture on the distribution of $\bhetah$.
\begin{Conjecture} \label{conjectnon}
Under suitable conditions, there is a unique solution $\adj$ of \eqref{solvegammanon} and
$$\tauh^{-1/2}(\bhetah-\st) = \bm{\Sigma}^{1/2}\mathbf{z}+O\qty(p^{-1/2})$$
where $\mathbf{z} \sim N(\bm{0},\mathbf{I}_p)$ and $O\qty(p^{-1/2})$ denotes a vector $\mathbf{v}\in \R^p$ satisfying $\frac{1}{p} \norm{\mathbf{v}}^2\to 0$ almost surely as $p\to \infty$. 
\end{Conjecture}
The derivation of the above is by considering a change of variable $\tilde{\bm{\beta}}=\bm{\Sigma}^{1/2}\hatbt$ whereby $\tilde{\bm{\beta}}\in \underset{\mathbf{b} \in \mathbb{R}^p}{\arg \min } \frac{1}{2}\|\y-\Qbm^\top \Dbm \Obm \mathbf{b} \|^2+h\left(\bm{\Sigma}^{-\frac{1}{2}} \mathbf{b} \right)$ and using the iterates of the VAMP algorithm (for non-separable penalties \cite[Algorithm 1]{fletcher2018plug}) to track $\tilde{\bm{\beta}}$. One can then obtain \eqref{DEFEFDnon} and \Cref{conjectnon} from the state evolution of the VAMP algorithm \cite[Eq. (19), Theorem 1]{fletcher2018plug}. If \Cref{conjectnon} holds, it will be straightforward to develop inference procedure for $\st$. A main gap to prove \Cref{conjectnon} in our opinion is to establish an analogue of \Cref{prop:sds}, i.e. the non-separable VAMP iterates indeed tracks $\hatbt$. We leave the proof of \Cref{conjectnon} as an open problem. 

\begin{Remark}\label{unknownSigma}
    When $\bm{\Sigma}$ is unknown, we require access to a large unlabeled dataset $\X_0 \in \mathbb{R}^{n_0 \times p}$, with $n_0 \gg p$, similar to the anisotropic Gaussian case. Let $\X_0= \mathbf{Z}_0 \bm{\Sigma}^{1/2}$ where $ \mathbf{Z}_0=\Qbm_0^\top \Dbm_0 \Obm_0, \Qbm_0 \in \R^{n \times n}, \Obm_0 \in \R^{p\times p}, \Dbm_0 \in \R^{n \times p}$. We require the modeling assumption that $\Dbm_0^\top \Dbm_0 \approx \mathbf{I}_p$. This assumes without loss of generality that as $n \to \infty$ and $p$ is fixed, the spectrum of $\mathbf{Z}_0^\top \mathbf{Z}_0$ converges to a point mass.
    
    Under this assumption, we can estimate $\bm{\Sigma}$ using the standard sample covariance estimator $\hat{\bm{\Sigma}} = \X_0^\top \X_0$. The ellipsoidal-invariant assumption then requires that $\X \hat{\bm{\Sigma}}^{-1/2}$ can be modeled as a right-rotationally invariant random matrix. Unlike the anisotropic Gaussian assumption, we do not require the spectrum of the sample covariance matrix of $\X \hat{\bm{\Sigma}}^{-1/2}$ to converge to the Marchenko-Pastur law and is expected to lead to more robust debiasing performance.
\end{Remark}

\section{Numerical Experiments}\label{appendix:Additional}
\subsection{Details of the design matrices}
Throughout the paper, we have illustrated our findings using different design matrices. We provide additional details in this section. 

\begin{Remark}[Notations used in caption]\label{Notationforfig}
    we use $\mathsf{InverseWishart}\qty(\bm{\Psi}, \nu)$ to denote inverse-Wishart distribution \cite{WikipediaInvWishart} with scale matrix $\bm{\Psi}$ and degrees-of-freedom $\nu$, $\mathsf{Mult}$-$\mathsf{t}(\nu, \bm{\Psi})$ to denote multivariate-t distribution \cite{WikipediamultT} with location $\bm{0}$, scale matrix $\bm{\Psi}$, and degrees-of-freedom $\nu$.
\end{Remark} 

\begin{Remark}[Right-rotationally invariant]\label{remark:RO}
    All design matrices in \Cref{fig1}, \ref{figPCRA} satisfies that $\X \stackrel{L}{=} \X \Obm$ for $\Obm\sim \Haar(\O(p))$ independent of $\X$. It is easy to verify that this is equivalent to right-rotational invariance as defined in \Cref{def:Rotinv}. 
\end{Remark}

\begin{Remark}[Comparison between designs in \Cref{fig1} and \Cref{figPCRA}]\label{CompGroupAB}
    The designs featured in \Cref{figPCRA} can be seen as more challenging variants of the designs in \Cref{fig1}, characterized by heightened levels of correlation, heterogeneity, or both. 
    
    Specifically, $\mathbf{\Sigma}^{\mathrm{(col)}}$ under $\mathsf{MatrixNormal}$-$\mathsf{B}$ has a higher correlation coefficient (0.9) compared to the correlation coefficient (0.5) in $\mathsf{MatrixNormal}$. This results in a stronger dependence among the rows of the matrix $\X$. Concurrently, the $\mathbf{\Sigma}^{\mathrm{(row)}}$ in $\mathsf{MatrixNormal}$-$\mathsf{B}$ is sampled from an inverse-Wishart distribution with fewer degrees of freedom, leading to a more significant deviation from the identity matrix compared to the MatrixNormal design presented in \Cref{fig1}.
    
    In $\mathsf{Spiked}$-$\mathsf{B}$, there are three significantly larger spikes when compared to $\mathsf{Spiked}$ in \Cref{fig1}, which contains 50 spikes of smaller magnitudes. Consequently, issues related to alignment and outlier eigenvalues are much more pronounced in the case of $\mathsf{Spiked}$-$\mathsf{B}$.
    
    Design under $\mathsf{LLN}$-$\mathsf{B}$ is product of four independent isotropic Gaussian matrices whereas $\mathsf{LLN}$-$\mathsf{B}$ contains 20th power of the same $\X_1$. The latter scenario presents greater challenge for DF or Spectrum-Aware Debiasing, primarily because the exponentiation step leads to the emergence of eigenvalue outliers.
    
    Larger auto-regressive coefficients are used in $\mathsf{VAR}$-$\mathsf{B}$, leading to stronger dependence across rows.
    
    When designs are sampled from $\mathsf{MultiCauchy}$, it is equivalent to scaling each row of an isotropic Gaussian matrix by a Cauchy-distributed scalar. This results in substantial heterogeneity across rows, with some rows exhibiting significantly larger magnitudes compared to others.
\end{Remark}

\begin{Definition}[Simulated Designs Specification]\label{GroupS}
Below we give more detailed information on simulated designs from \Cref{figPCRA}, top-left experiment. Without loss of generality, all designs below are re-scaled so that average of the eigenvalues of $\X^\top \X$ is 1. 
\begin{itemize}
    \item [(i)] $\mathsf{MatrixNormal}$-$\mathsf{B}$: $\X\sim N(\rm{0},\mathbf{\Sigma}^{\mathrm{(col)}}\otimes \mathbf{\Sigma}^{\mathrm{(row)}})$ where $\mathbf{\Sigma}^{\mathrm{(col)}}_{ij}=0.9^{|i-j|},\forall i,j\in [n]$ and $\bm{\Sigma}^{\mathrm{(row)}}\sim \mathsf{InverseWishart}\qty(\mathbf{I}_p, 1.002\cdot p)$ (see \Cref{Notationforfig} for notation);
    \item [(ii)] $\mathsf{Spiked}$-$\mathsf{B}$: $\X= \mathbf{V} \mathbf{R} \mathbf{W}^\top +n^{-1}N(\rm{0}, \mathbf{I}_n\otimes \mathbf{I}_p)$ where $\mathbf{V}, \mathbf{W}$ are drawn randomly from Haar matrices of dimensions $n,p$ respectively with 3 columns retained, and $\mathbf{R}=\diag(500,250,50)$;
    \item [(iii)]$\mathsf{LNN}$-$\mathsf{B}$: $\X=\X_1^{15}\cdot \X_2$ where $\X_1\in \R^{n\times n},\X_2\in \R^{n\times p}$ have iid entries from $N(0,1)$;
    \item[(iv)]$\mathsf{VAR}$-$\mathsf{B}$: $\X_{i,\bullet}=\sum_{k=1}^{\tau\vee i}\alpha_k \X_{i-k,\bullet}+\epbm_i$ where $\X_{i,\bullet}$ denotes the $i$-th row of $\X$. Here, $\epbm_i \sim N(\mathbf{0}, \mathbf{\Sigma})$ with $\mathbf{\Sigma} \sim \mathsf{InverseWishart}(\mathbf{I}_p, 1.1\cdot p)$. We set $\tau=3,\mathbf{\alpha}=\qty(0.7, 0.14, 0.07)$, $\X_1=0$;
    \item[(v)] $\mathsf{MultiCauchy}$: rows of $\X$ are sampled iid from $\mathsf{Mult}$-$\mathsf{t}(1, \mathbf{I}_p)$ (see \Cref{Notationforfig} for notation).
\end{itemize} 
\end{Definition}

\begin{Definition}[Real-data Designs Specification]\label{GroupD}
Below we give more detailed information on real-data designs from \Cref{figPCRA}, bottom-left experiment. Without loss of generality, all designs below are re-scaled so that average of the eigenvalues of $\X^\top \X$ is 1. 
\begin{itemize}
    \item [(i)] $\mathsf{Speech}$: $200\times 400$ with each row being i-vector (see e.g. \cite{ibrahim2018vector}) of the speech segment of a English speaker. We imported this dataset from the OpenML repository \cite{OpenMLSpeech} (ID: 40910) and retained only the last 200 rows of the original design matrix. The original dataset is published in \cite{goldstein2016comparative}. 
    \item [(ii)] $\mathsf{DNA}$: $100\times 180$ entries with each row being one-hot representation of primate splice-junction gene sequences (DNA). We imported this dataset from the OpenML repository \cite{OpenMLDNA} (ID: 40670) and retained only the last 100 rows of the original design matrix. The original dataset is published in \cite{noordewier1990training}. 
    \item [(iii)] $\mathsf{SP500}$: $300\times 496$ entries where each column representing a time series of daily stock returns (percentage change) for a company listed in the S\&P 500 index. These time series span 300 trading days, ending on January 1, 2023.. We imported this dataset from Yahoo finance API \cite{YahooFin};
    \item[(iv)]$\mathsf{FaceImage}$: $1348\times 2914$ entries where each row corresponds to a JPEG image of a single face. We imported this dataset from the scikit-learn package, using the handle sklearn.datasets.fetch\_lf2\_people \cite{sklearn}. The original dataset is published in \cite{LFWTech}
    \item[(v)] $\mathsf{Crime}$: $50 \times 99$ entries where each column corresponds to a socio-economic metric in the UCI communities and crime dataset \cite{misc_communities_and_crime_183}. Only the last 50 rows of the dataset is retained. We also discarded categorical features: state, county, community, community name, fold from the original dataset.  
\end{itemize} 
\end{Definition}

\subsection{QQ plots}
\Cref{figQQA}, \ref{figQQB} and \ref{figQQD} are QQ-plots of \Cref{fig1}, \ref{figPCRA} top and bottom row experiments respectively.

\begin{figure}[H]
    \centering
    \includegraphics[width=\linewidth]{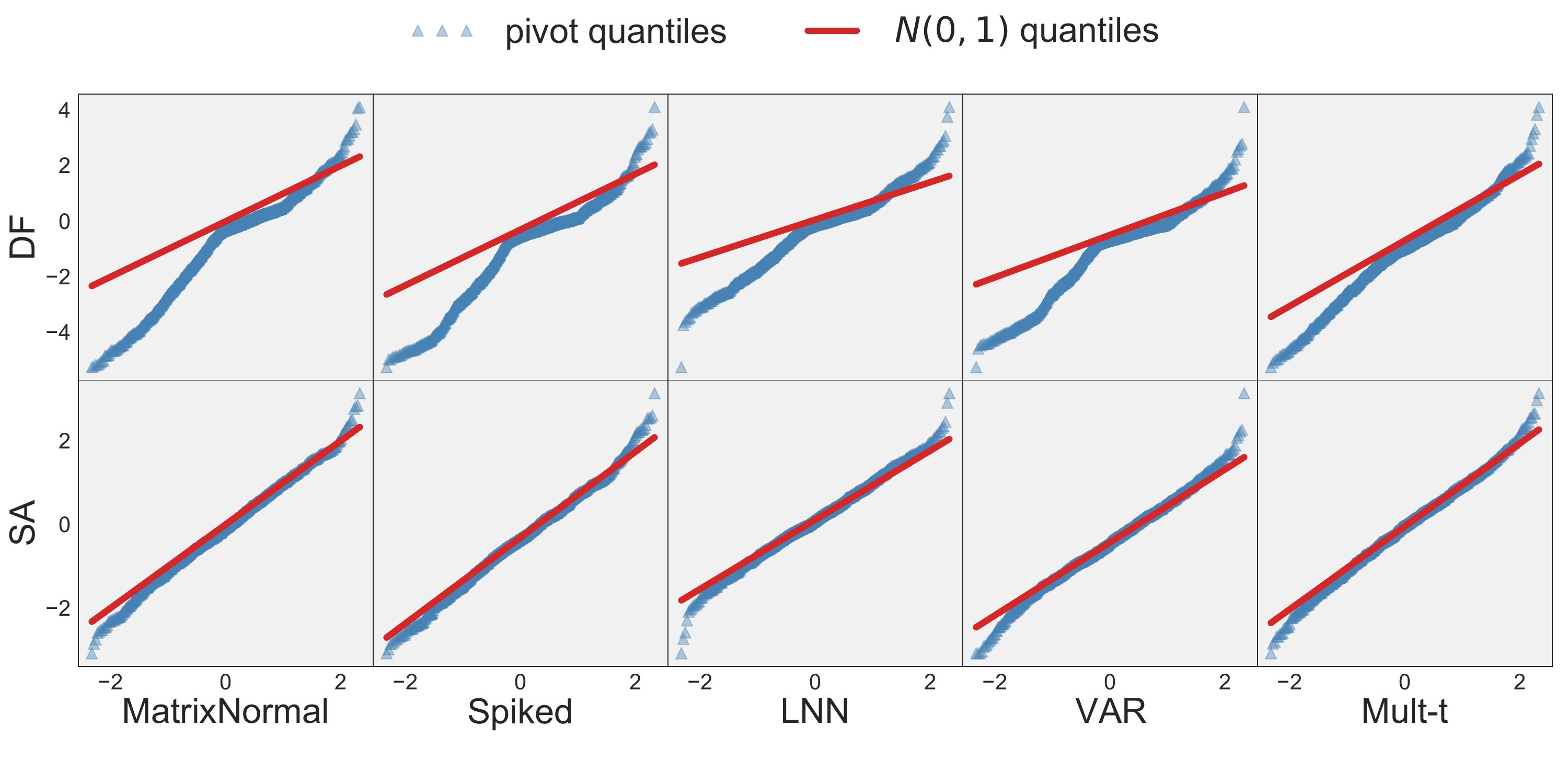}
    \caption{QQ plots corresponding to  \Cref{fig1}.}
    \label{figQQA}
\end{figure}

\begin{figure}[H]
    \centering\includegraphics[width=\linewidth]{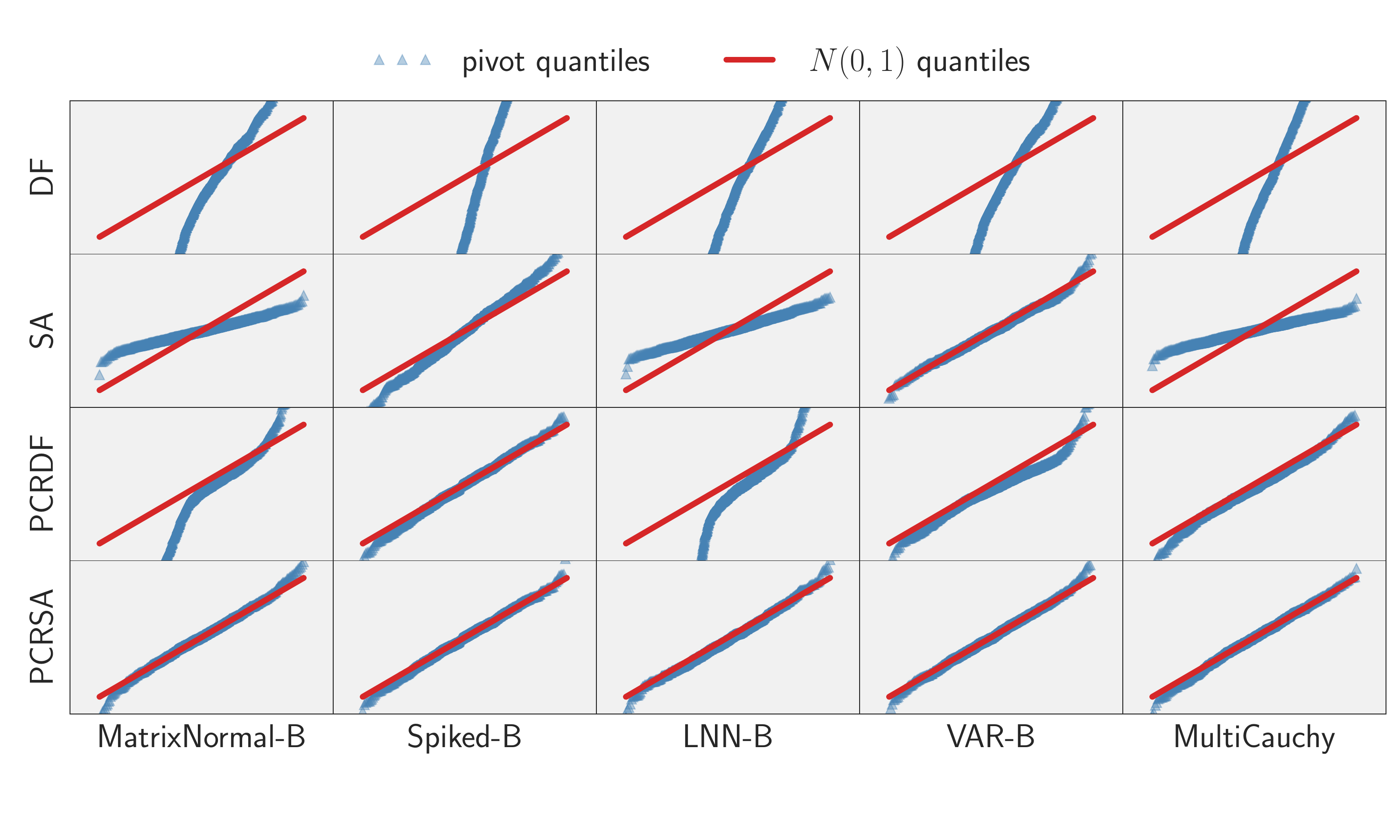}
    \vspace{-6mm}
    \caption{QQ plots corresponding to \Cref{figPCRA}, top-left.}
    \label{figQQB}
\end{figure}

\begin{figure}[H]
    \centering
    \includegraphics[width=\linewidth]{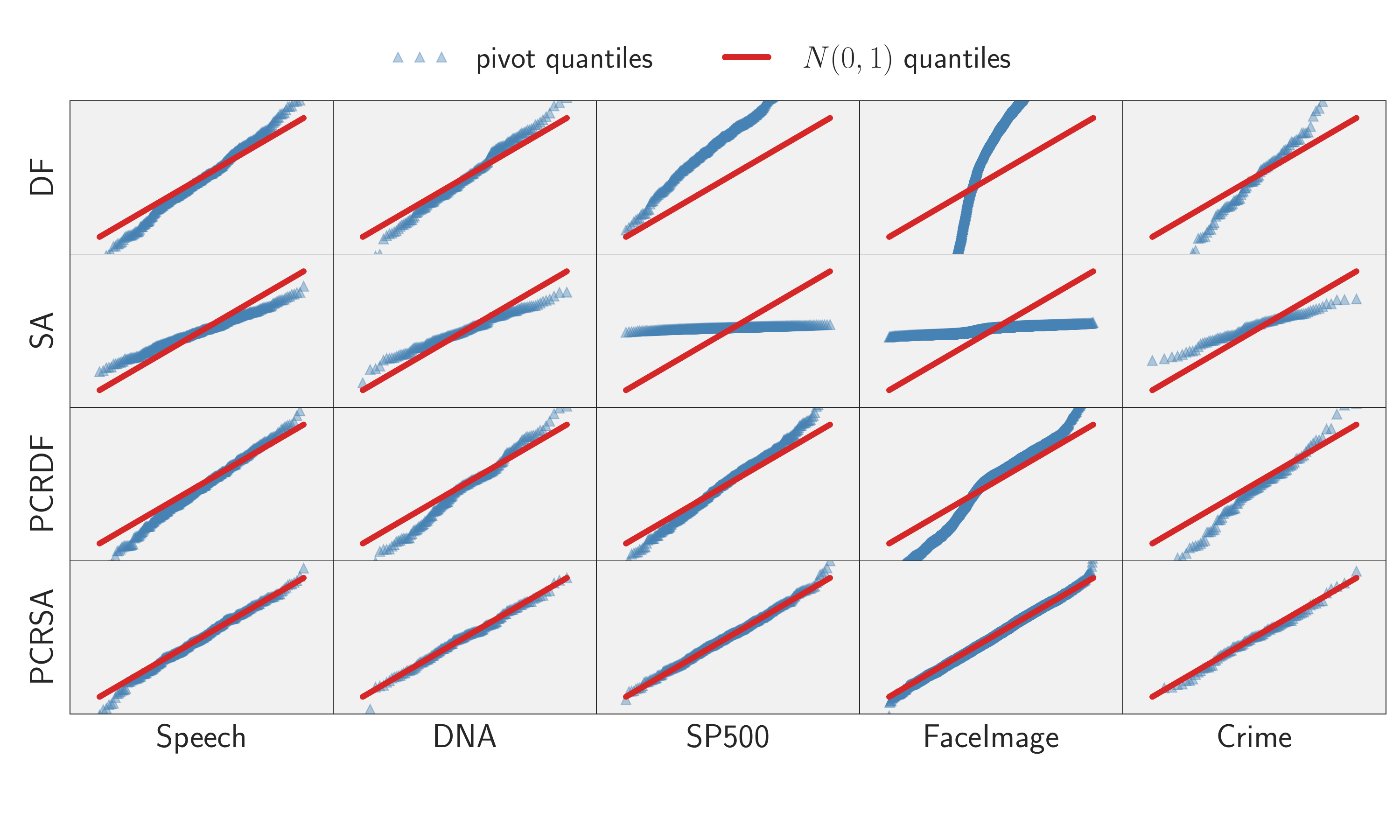}
    \vspace{-6mm}
    \caption{QQ plots corresponding to \Cref{figPCRA}, bottom-left.}
    
    \label{figQQD}
\end{figure}

\subsection{Marginal inference under exchangeability}\label{margex}
\Cref{fig3} below shows an illustration for this result focusing on $\mathcal{I}=\{1\}$. Observe that we once again outperform degrees-of-freedom debiasing.
\begin{figure}[H]
    \centering
    \includegraphics[width=\linewidth]{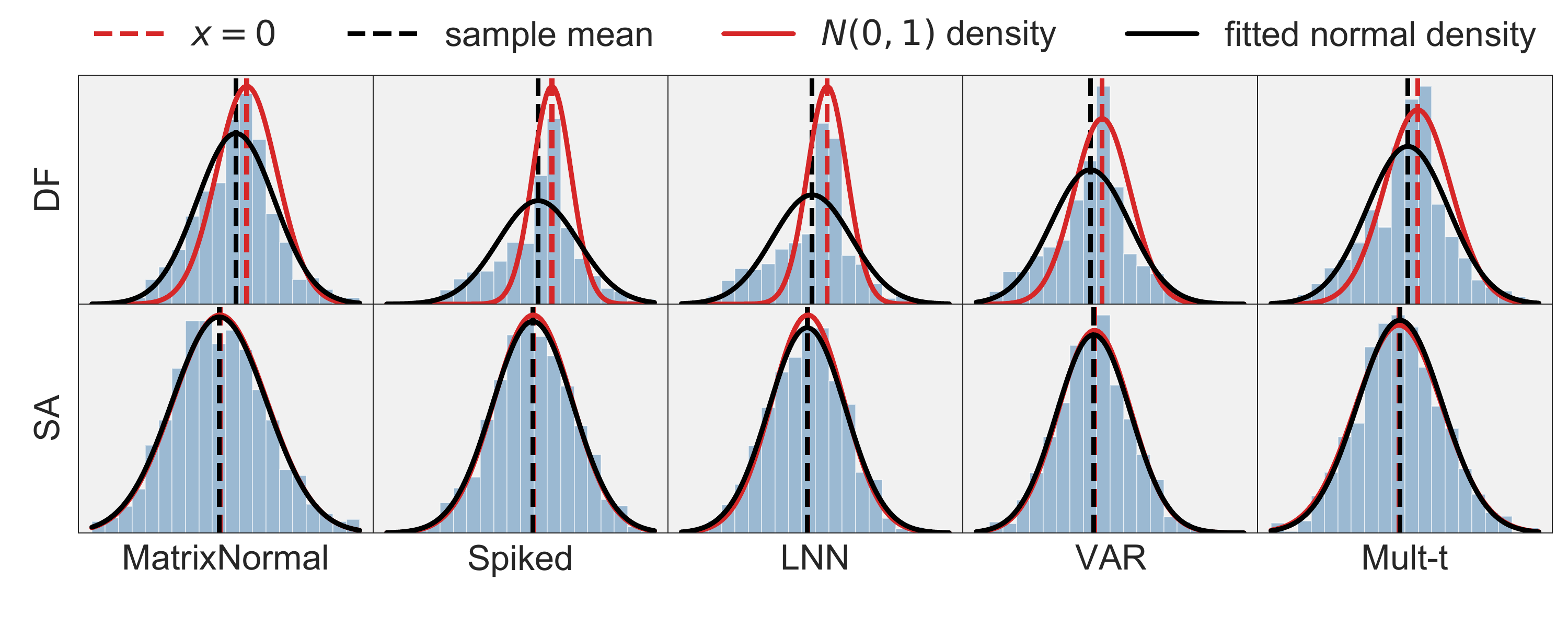}
    \vspace{-0.6cm}
    \caption{Histograms of $\frac{\hat{\beta}_1-\beta_1^{\star}}{\sqrt{\tauh}}$ across from 1000 Monte-Carlo trials using DF and Spectrum-Aware Debiasing. The setting is identical to \Cref{fig1} except that here we set $n=100, p=200$ for computational tractability. }
    \label{fig3}
\end{figure}

% \subsection{Inference for debiased PCR}\label{appendix:Ininf}
% We include here plots for the inference procedures described in \Cref{section:inference}. 

\subsection{Alignment tests for simulated designs}\label{dnf}
\Cref{figtestsimu} shows results of hypothesis tests for the alignment coefficients $\alphstari, i=1,...,6$ for experiments in \Cref{figPCRA}. 

\begin{figure}[H] % The asterisk (*) is crucial for spanning both columns 
  \centering

  \includegraphics[width=0.32\columnwidth]{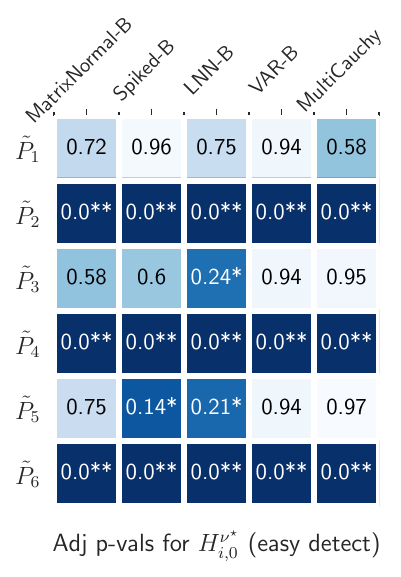} % Adjust the width as needed
    \includegraphics[width=0.4\columnwidth]{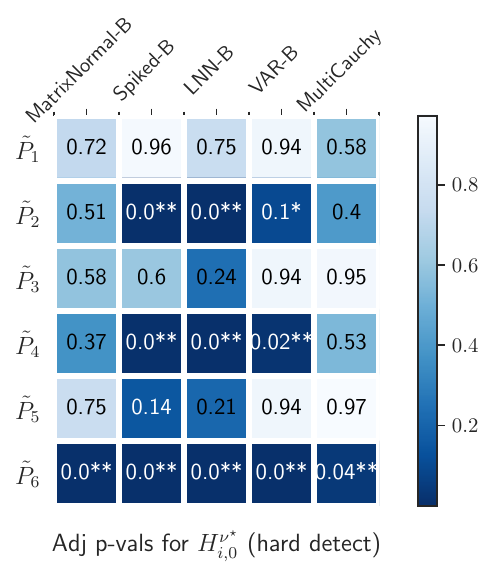} % Adjust the width as needed
  
\includegraphics[width=0.32\columnwidth]{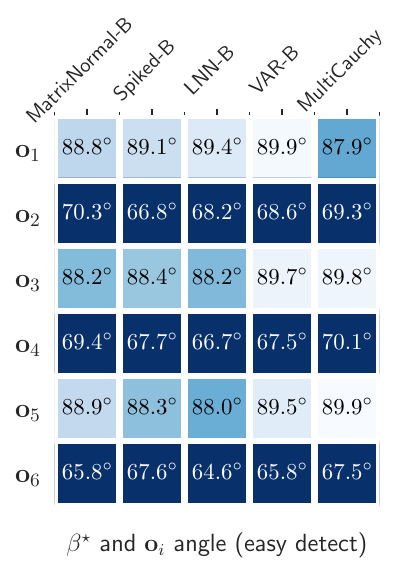} % Adjust the width as needed  
  % \vspace{-0.cm} % This adds vertical space between the rows
  % % Second row
  \includegraphics[width=0.4\columnwidth]{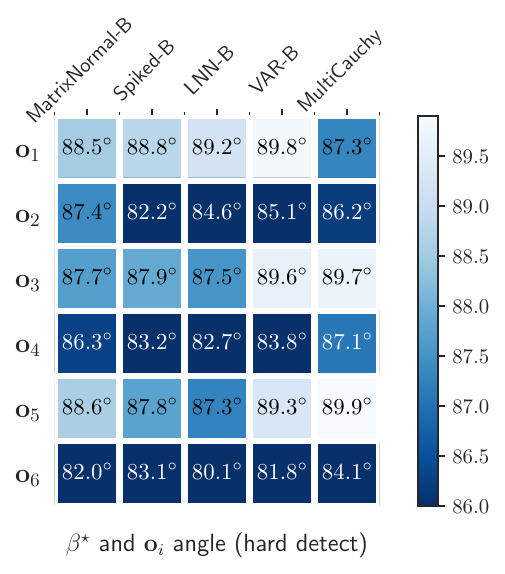} % Adjust the width as needed

  \caption{\textbf{top-row}: Same setting as \Cref{figPCRA}, bottom-left experiment. \textbf{bottom-row}: Same setting as \Cref{figPCRA}, bottom-left experiment, except that we increase difficulty of alignment detection by considering $v_i^*=\sqrt{p}, i \in\{2,4,6\}$ as opposed to $v_i^*=5\cdot \sqrt{p}, i \in\{2,4,6\}$. \textbf{left-column}: Benjamini-Hochberg adjusted p-values $\tilde{P}_i$ for $H_{i,0}^{\alphstar},i=1,...,6$. ** indicates rejection under FDR level 0.05 and * rejection under FDR level 0.1. Recall from \Cref{section:pcar}.\ref{Altest} that rejection of $H_{i,0}^{\alphstar}$ indicates alignment between $\st$ and $\mathbf{o}_i$. \textbf{right-column}: True alignment angles  between the signal $\st$ and $i$-th PC $\mathbf{o}_i$ calculated using $\mathrm{argcos}\{\mathbf{o}_i^\top \st/(\norm{\mathbf{o}_i}_2 . \norm{\st}_2 )\}$.}
  %(this is unobserved { by the model})  }
  \label{figtestsimu}
\end{figure}

% \subsubsection{Inference for $\zetri$}\label{dsdfnf}
% \Cref{figztCIA} depict the True Positive Rate (TPR), False Positive Rate (FPR) of the hypothesis tests $H_i: \zetri=0$, and the False Coverage Proportion (FCP) of confidence intervals for the entries of $\zetri$. These plots illustrate the changes in TPR, FPR, and FCP as we systematically vary the targeted FPR/FCP level $\alpha$ from 0 to 1. 

% \begin{figure}[H]
%     \centering
%     \includegraphics[width=\linewidth]{Image_pnas/panel_pcr_sim_zeta_FCP.pdf}
%     \caption{Under the setting of \Cref{figPCRA}, we plot the True Positive Rate (TPR) and False Positive Rate (FPR) corresponding to hypothesis tests outlined in \eqref{sdfasmtest}, and the False Coverage Proportion (FCP) of confidence intervals defined in \eqref{skdtest} for $(\zetri)_{i=1}^p$. The $x$-axis spans $\alpha$ values from 0 to 1, while the $y$-axis ranges between 0 and 1. The dotted black line represents the 45-degree reference line.}
%     \label{figztCIA}
% \end{figure}

\subsection{Misspecified setting}
Our theory assumes that all relevant covariates are observed. In practice, there may be unobserved features $\mathbf{Z}$ and responses is generated from $\y=\X \st +\mathbf{Z} \bm{\theta}^\star+\epbm$. We found that our method is in fact relatively robust to such mis-specification. We demonstrate by introducing unobserved features $\mathbf{Z}$ of shape $n\times \frac{p}{2}$ (i.e. half as many as the observed features $\X\in \R^{n \times p}$) drawn from different distributions. 

\begin{figure}[H]
    \centering
    \includegraphics[width=0.7\columnwidth]{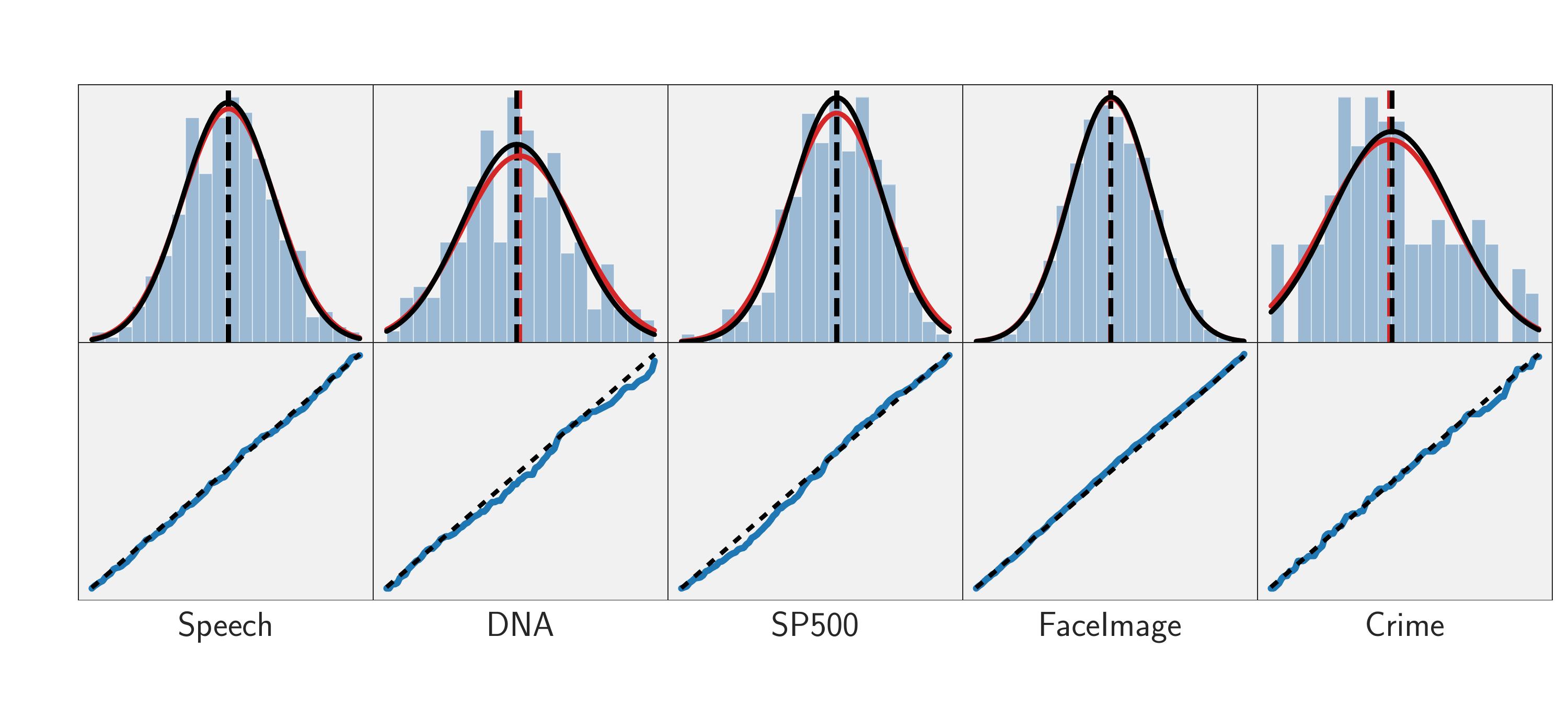}
    \caption{\textbf{Z from MatrixNormal} Same settings and legends as \Cref{figPCRA},  bottom-left experiment, except that the response is generated from a misspecified model $\y=\X \st +\mathbf{Z} \bm{\theta}^\star+\epbm$. We set ${\theta}^\star_i \sim 0.2\cdot N(-5,1)+0.3\cdot N(2,1)+0.5\cdot \delta_0$. Here, we let $\mathbf{Z}$ be sampled from MatrixNormal as described in \Cref{fig1} with dimension half as many as $\X$.  }
\end{figure}

\begin{figure}[H]
    \centering
    \includegraphics[width=0.7\columnwidth]{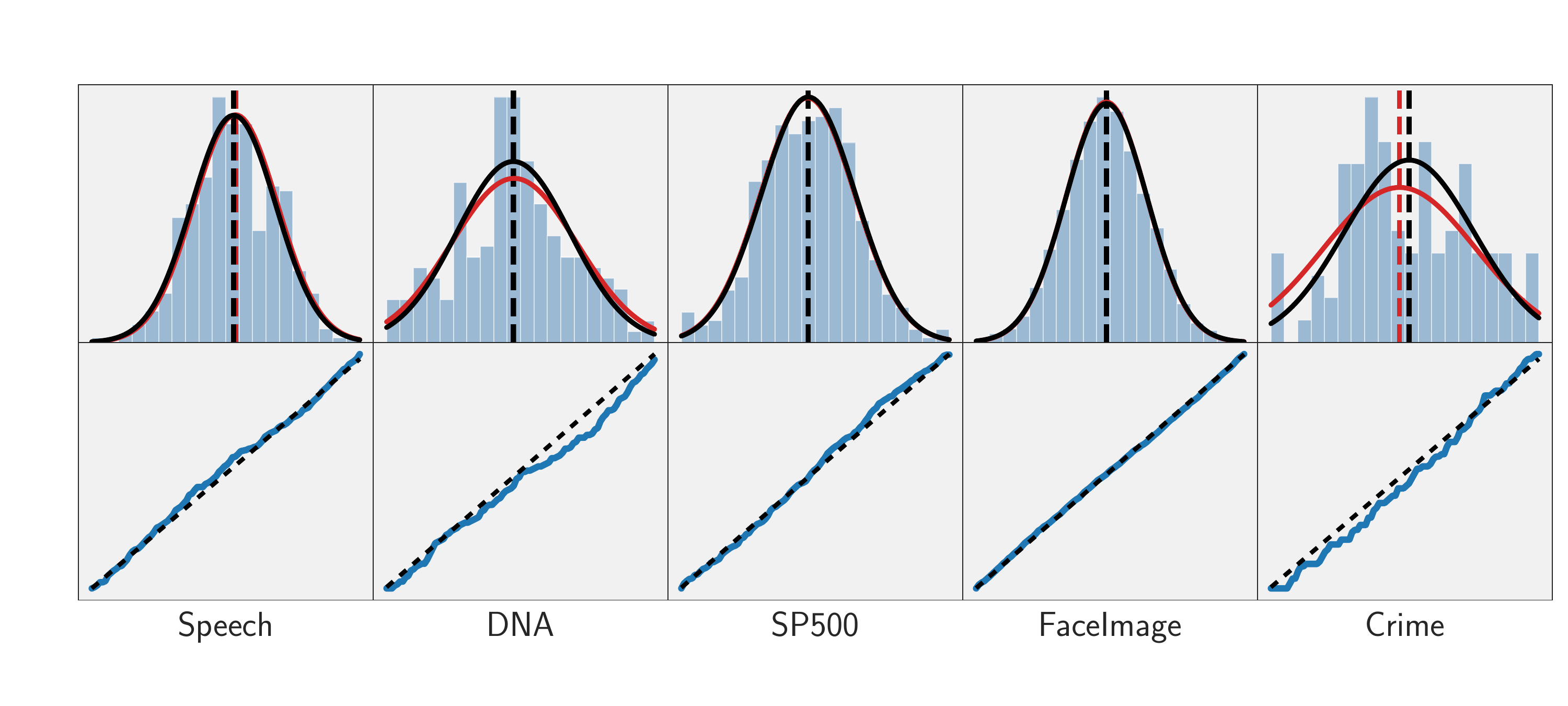}
    \caption{ \textbf{Z from Spike.} Same settings and legends as \Cref{figPCRA},  bottom-left experiment, except that the response is generated from a misspecified model $\y=\X \st +\mathbf{Z} \bm{\theta}^\star+\epbm$. We set ${\theta}^\star_i \sim 0.2\cdot N(-5,1)+0.3\cdot N(2,1)+0.5\cdot \delta_0$. Here, we let $\mathbf{Z}$ be sampled from Spike as described in \Cref{fig1} with dimension half as many as $\X$.}
\end{figure}

\begin{figure}[H]
    \centering
    \includegraphics[width=0.7\columnwidth]{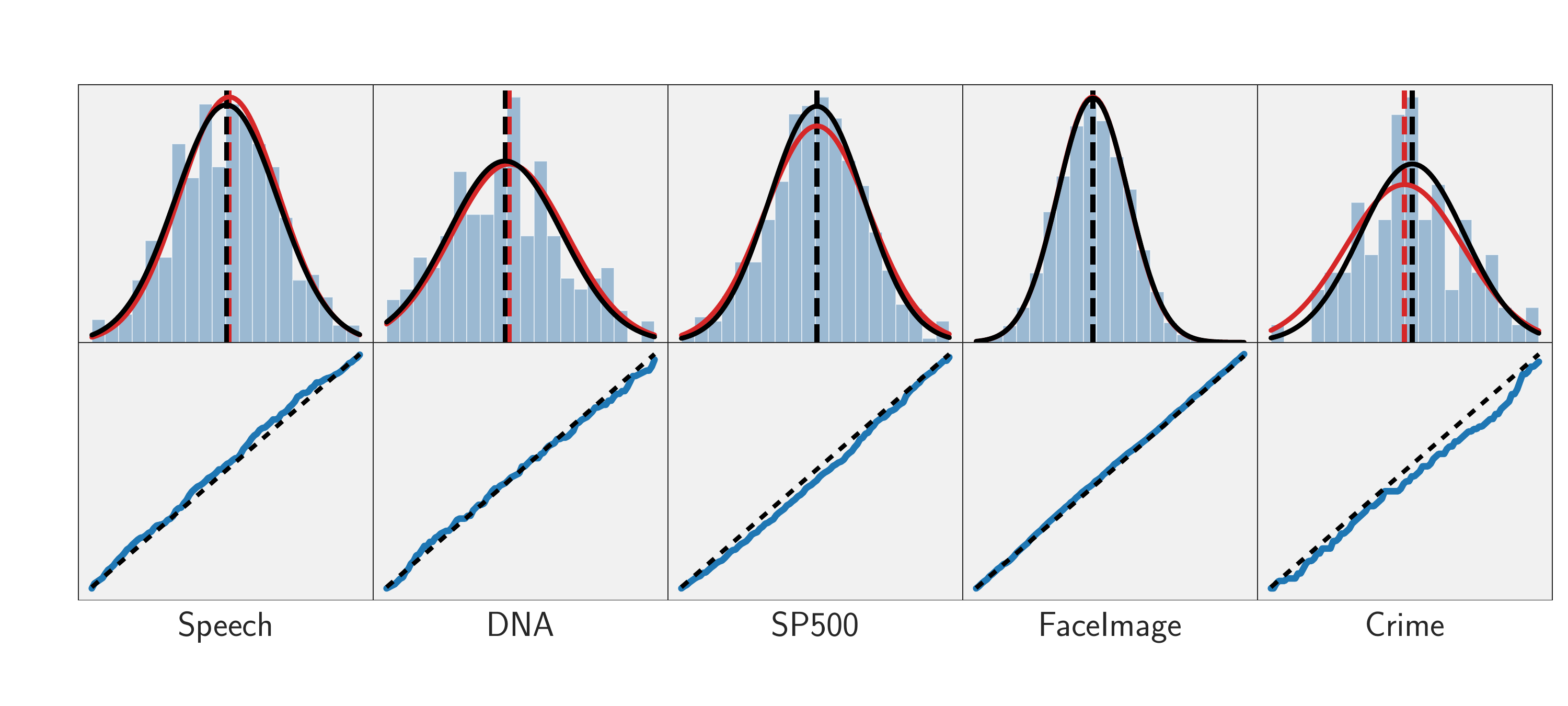}
    \vspace{-6mm}
    \caption{ \textbf{Z from LNN.} Same settings and legends as \Cref{figPCRA},  bottom-left experiment, except that the response is generated from a misspecified model $\y=\X \st +\mathbf{Z} \bm{\theta}^\star+\epbm$. We set ${\theta}^\star_i \sim 0.2\cdot N(-5,1)+0.3\cdot N(2,1)+0.5\cdot \delta_0$. Here, we let $\mathbf{Z}$ be sampled from LNN as described in \Cref{fig1} with dimension half as many as $\X$.}
\end{figure}

\begin{figure}[H]
    \centering
    \includegraphics[width=0.7\columnwidth]{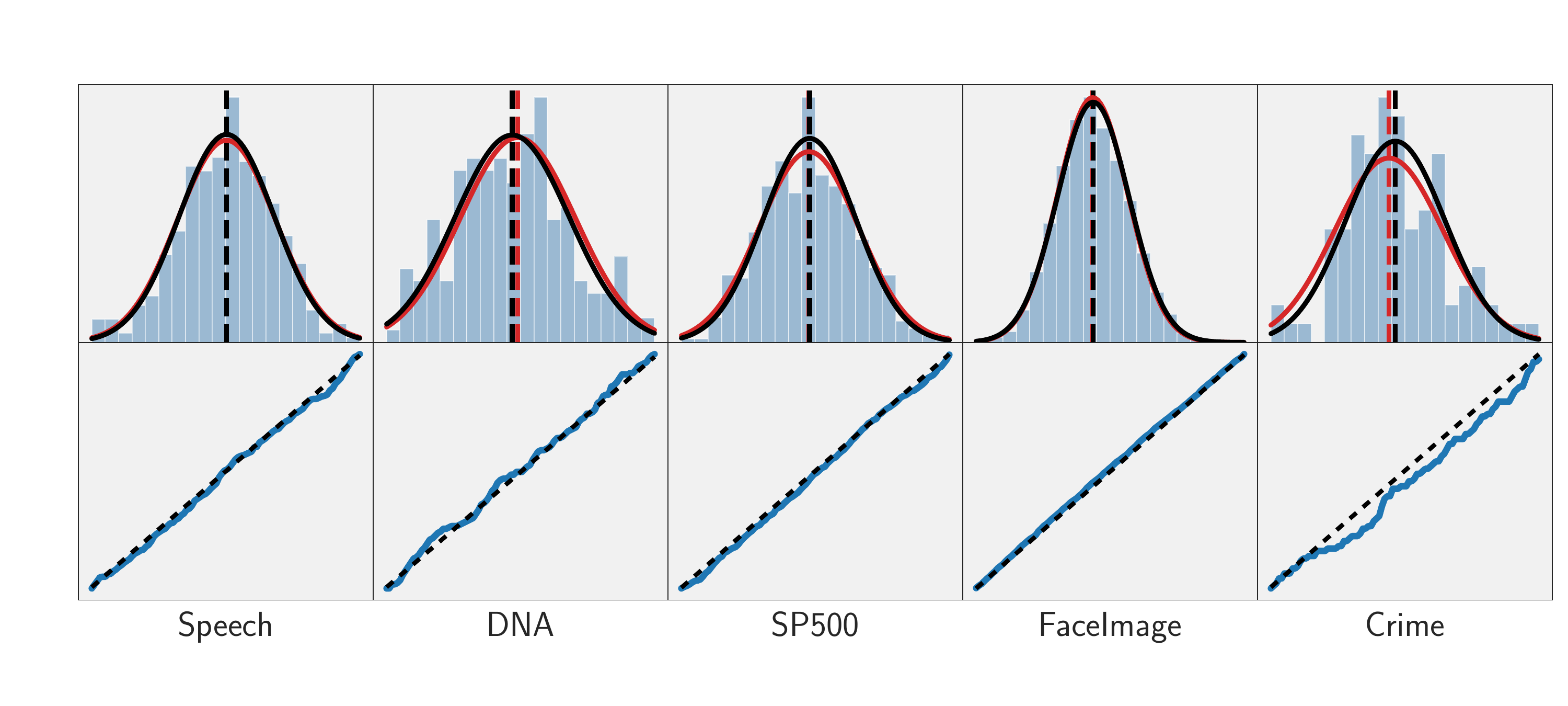}
    \vspace{-6mm}
    \caption{ \textbf{Z from VAR.} Same settings and legends as \Cref{figPCRA},  bottom-left experiment, except that the response is generated from a misspecified model $\y=\X \st +\mathbf{Z} \bm{\theta}^\star+\epbm$. We set ${\theta}^\star_i \sim 0.2\cdot N(-5,1)+0.3\cdot N(2,1)+0.5\cdot \delta_0$. Here, we let $\mathbf{Z}$ be sampled from VAR as described in \Cref{fig1} with dimension half as many as $\X$.}
\end{figure}

\begin{figure}[H]
    \centering
    \includegraphics[width=0.7\columnwidth]{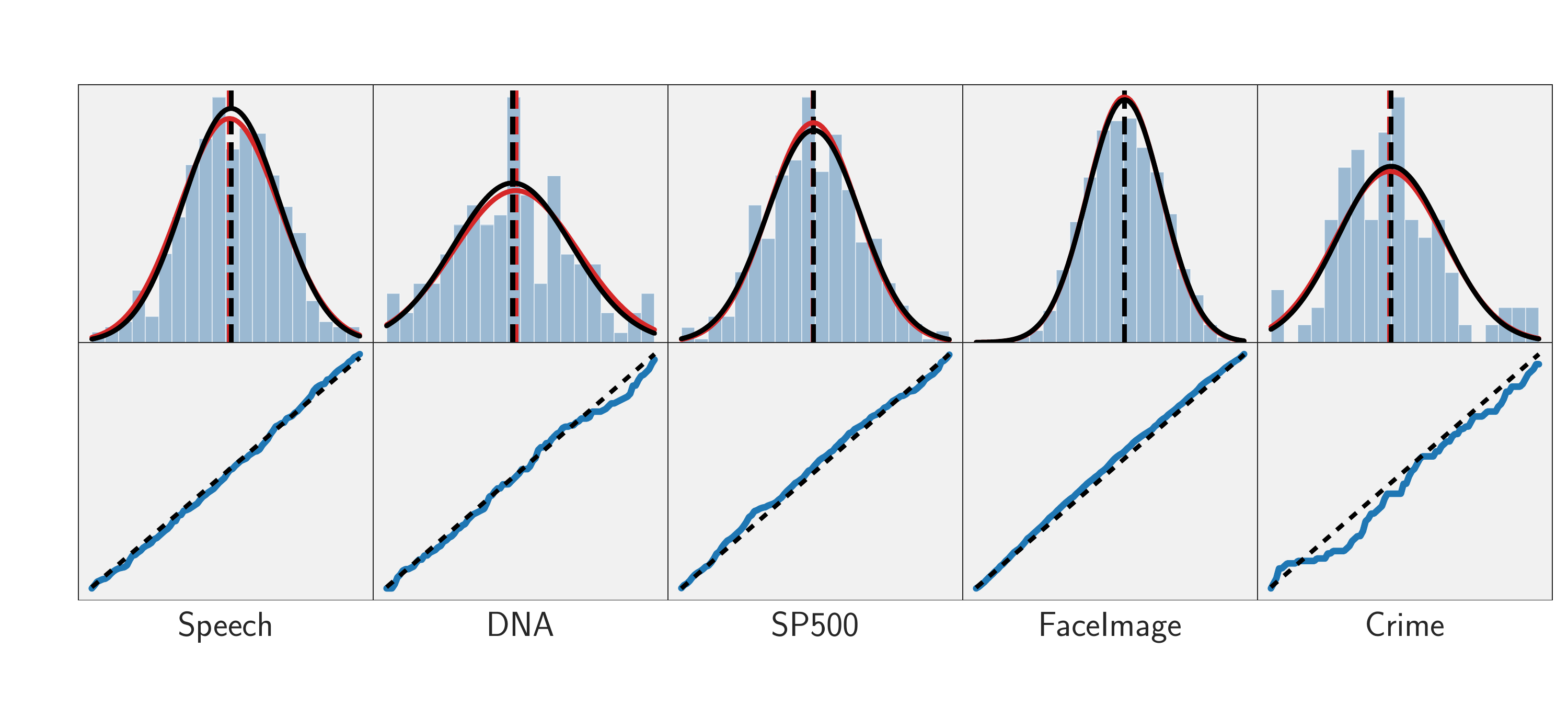}
    \vspace{-6mm}
    \caption{ \textbf{Z from Mult-t.} Same settings and legends as \Cref{figPCRA},  bottom-left experiment, except that the response is generated from a misspecified model $\y=\X \st +\mathbf{Z} \bm{\theta}^\star+\epbm$. We set ${\theta}^\star_i \sim 0.2\cdot N(-5,1)+0.3\cdot N(2,1)+0.5\cdot \delta_0$. Here, we let $\mathbf{Z}$ be sampled from Mult-t as described in \Cref{fig1} with dimension half as many as $\X$.}
\end{figure}

%\begin{funding}
%The first author was supported by NSF Grant DMS-??-??????.

%The second author was supported in part by NIH Grant %???????????.
%\end{funding}

% \bibliographystyle{imsart-number} 
% \bibliography{ref}

\end{document}